\documentclass[aap]{imsart}
\usepackage{amsmath}
\usepackage{amsfonts}
\usepackage{amsthm}
\usepackage{hyperref}
\usepackage{bbm}
\usepackage{scrextend}
\usepackage{amssymb}

\usepackage{comment}

\usepackage{graphicx}
\graphicspath{ {./images/} }
\usepackage[dvipsnames]{xcolor}
\usepackage{dsfont}
\usepackage{natbib}
\usepackage{graphicx}
\usepackage{float}

\usepackage{dsfont}
\usepackage{mathrsfs}
\usepackage{bbm}
\usepackage{esint}
\usepackage{cancel}

\newcommand{\ind}{1\hspace{-0,9ex}1}

\numberwithin{equation}{section}
\newtheorem{theorem}{Theorem}[section]
\newtheorem{lemma}[theorem]{Lemma}

\newtheorem{proposition}[theorem]{Proposition}
\newtheorem{corollary}[theorem]{Corollary}

\newtheorem{condition}[theorem]{Condition}

\newtheorem{example}[theorem]{Example}
\newtheorem{remark}[theorem]{Remark}%
\newtheorem{alg}[theorem]{Algorithm}%

\usepackage[utf8]{inputenc}

\DeclareMathOperator{\Cov}{Cov}
\DeclareMathOperator{\Var}{Var}
\DeclareMathOperator{\E}{E}

\DeclareMathOperator{\tr}{tr}

\DeclareMathOperator{\dd}{d}
\DeclareMathOperator{\diag}{diag}
\DeclareMathOperator{\rank}{rank}

\def\C{\mathbb{C}}

\def\E{\mathbb{E}}

\def\N{\mathbb{N}}
\def\Q{\mathbb{Q}}
\def\R{\mathbb{R}}
\def\S{\mathbb{S}}
\def\Z{\mathbb{Z}}

\def\FF{\mathcal{F}}

\def\LL{\mathcal{L}}

\def\NN{\mathcal{N}}

\def\CC{\mathcal{C}}

\def\N{\mathbb{N}}
\def\P{\mathbb{P}}

\def\X{\mathfrak{X}}

\date{}

\usepackage{tikz}
\usepackage{needspace}

\begin{document}

\begin{frontmatter}
\title{Computationally tractable nonparametric bootstrap of high-dimensional sample covariance matrices}
\runtitle{Bootstrapping high-dimensional sample covariance matrices}
\begin{aug}
\author[A]{\fnms{Holger}~\snm{Dette}\ead[label=e1]{holger.dette@rub.de}}
\and
\author[B]{\fnms{Angelika}~\snm{Rohde}\ead[label=e2]{angelika.rohde@stochastik.uni-freiburg.de}}
\address[A]{Department of Mathematics, Ruhr-Universit\"at Bochum\printead[presep={,\ }]{e1}}
\address[B]{Mathematical Institute, Albert-Ludwigs-Universit\"at Freiburg\printead[presep={,\ }]{e2}}
\end{aug}
\begin{abstract}
We introduce a new ``$(m,mp/n)$ out of $(n,p)$'' sampling-with-replace\-ment bootstrap for  eigenvalue statistics of high-dimensional sample covariance matrices based on $n$ independent $p$-dimensional random vectors.  As it only uses $q=\lfloor mp/n\rfloor $ coordinates of the observations in a subsample of size $m \ll n $  from the original data, it is computationally tractable for  large scale data.
In the high-dimensional scenario $p/n\rightarrow c\in (0,\infty)$, this fully nonparametric bootstrap is shown to consistently reproduce the empirical spectral measure if $m/n\rightarrow 0$. If $m^2/n\rightarrow 0$, it approximates correctly the distribution of linear spectral statistics. The crucial component is a suitably defined Representative Subpopulation Condition which is shown to be verified in a large variety of situations. Our proofs are conducted under minimal moment requirements for the spectral distribution and a fourth moment condition for linear spectral statistics, and incorporate  delicate results on non-centered quadratic forms, combinatorial trace moments estimates as well as a conditional bootstrap martingale CLT which may be of independent interest.
\end{abstract}
\end{frontmatter}

\section{Introduction}  \label{sec1}
  \def\theequation{1.\arabic{equation}}	
	\setcounter{equation}{0}

Let $Y_1,\dots, Y_n$ be independent, identically distributed $p$-dimensional centered random vectors with covariance matrix $\Sigma_n$ and corresponding sample covariance matrix
\begin{equation} \label{1.0}
\widehat \Sigma_n=\frac{1}{n}\sum_{i=1}^nY_iY_i^\top.
\end{equation}
Its eigenvalues
$\widehat{\lambda}_1,\dots,\widehat{\lambda}_p$ are fundamental quantities in principal component analysis and many other multivariate procedures. In contemporary applications, however, both the sample size $n$ and the dimension $p$ may be very large and of comparable order. In this regime, approximations derived under a fixed-dimensional asymptotic framework can be inaccurate \citep[see][]{johnstone2006}, while the computational cost of statistical inference becomes a major concern.

Our goal is to bootstrap linear spectral statistics of high-dimensional sample covariance matrices in a computationally tractable way, in view of the rarely available accessible form of their  (asymptotic) distribution.  Being computationally infeasible,  the  classical  $n$ out of $n$  bootstrap
is even untrustworthy when the problem is genuinely high-dimensional. More precisely, \cite{karpur2016} characterized the limiting spectral distribution (LSD) of the bootstrapped covariance matrix (Theorem S2.2 in their supplementary material) and showed for Gaussian data (their Section 3.1.2) that it generally differs  from that of $\widehat \Sigma_n$. As the LSD occurs explicitly in the parameters of the  limiting distribution of linear spectral statistics, there is no hope that the classical  bootstrap consistently approaches the correct distribution for linear spectral statistics if it fails for approximating the LSD.

In this article, we provide a
  powerful, fully nonparametric and computationally tractable tool  to obtain accurate approximations for the distribution of
  linear spectral statistics of the sample covariance matrix   in the  high-dimensional context. Our approach is
based on  the   traditionally in a wider range  applicable 
$m$ out of $n$  bootstrap \citep{PolitisRomano1994,bickel1997}, which has already
been investigated to approximate the eigenvalue distribution in the case where the dimension is fixed.
However, the use of this approach in the high-dimensional setting presents another challenge as it
 does not even preserve the limiting ratio $c$ of dimension and sample size if $m\ll n$, which
   appears already explicitly in the characterizing Mar\u{c}enko-Pastur equation for the Stieltjes transform
   of the LSD     \citep[see][]{Marcenko1967,silverstein1995}.
To address this difficulty, we propose to also select (possibly by a  random  mechanism) 
$q=\lfloor mp/n\rfloor $ coordinates from the estimator  for the covariance matrix obtained from the   subsample of $m$ observations
such that the  ratio of dimension and sample size remains (asymptotically) unchanged.
This    procedure will be called {\it ``$(m,mp/n)$ out of $(n,p)$  bootstrap''}  throughout this paper and  is based on the crucial observation that in many situations of interest, a subvector of $Y_1$,    selected according to an appropriate random sampling mechanism,  provides a covariance matrix, say $\tilde \Sigma_n$, with
a similar spectral distribution as  the covariance matrix $\Sigma_n$ of the full vector $Y_1$. 
We will prove that under the so-called Representative Subpopulation Condition and the fourth moment condition $\E X_{11}^4=3$,  the   ``$(m,mp/n)$ out of $(n,p)$''  bootstrap  consistently mimics the distribution of linear spectral statistics (LSS') of the sample covariance matrix 
if $m = o(\sqrt{n})$. Appealingly, the simultaneously reduced dimension and sample size make its implementation  computationally tractable even if original dimension and sample size are very large. 
\\
While most of the work on bootstrapping eigenvalue statistics in the high-dimensional regime 
has its focus on the extreme eigenvalues \citep[see, for example,][]{HanXoZhou2018,yao2022rates,ding2023extreme,yuzhaozhou2024,Nina2025}, the bootstrap for  linear spectral statistics  is much less explored. 
\cite{lopblaaue2019} proposed   a  parametric type  bootstrap method in the high-dimensional setting,    sampling
 bootstrap data from a proxy distribution that is parameterized by estimates of the eigenvalues and kurtosis.  Roughly speaking, these authors suggested 
to generate  a matrix of the form \eqref{1.0} from  independent random vectors  with  iid  Pearson distributed entries (matching  the first four moments asymptotically)
 and to multiply the   resulting matrix  from the left and the right by a square root of the  diagonal matrix containing the spectrum. 
 We also  mention the paper of \cite{wanglopes2022} who developed a parametric type bootstrap for linear spectral statistics in the high-dimensional elliptical model, which uses the specific structure of this model and  also requires the estimation of a diagonal matrix containing the spectrum. These  bootstrap approaches  are statistically powerful  and provide accurate approximations within their respective modeling frameworks.
However, their construction relies on repeated estimation of spectral characteristics from  high-dimensional  sample covariance matrices based on $n$ bootstrap observations. As a consequence, their computational complexity can become substantial. More precisely, their algorithm requires $O\big ( (np^2 + p^3) (B+1) \big ) $ operations, where $B$ is the number of bootstrap replications. In contrast, our  novel  and fully nonparametric   ``$(m,mp/n)$ out of $(n,p)$  bootstrap''  scales with  only $
O\big ( (np^2 + p^3)   \big ) + O\big ( (mq^2 + q^3) B   \big  )
$
operations as it only uses $q=\lfloor mp/n\rfloor $ coordinates of the observations in a subsample of size $m = o(\sqrt{n})$. 
In particular, 
we  do not need  assumptions  on the limiting spectrum of the population covariance matrix  which are usually required to   make  its estimation possible.   

\section{Preliminaries}
\label{prel} 
  \def\theequation{2.\arabic{equation}}	
	\setcounter{equation}{0}

For any Hermitian matrix $A\in\C^{p\times p}$ with eigenvalues $\lambda_1(A),\dots, \lambda_p(A)$, 
$$
\mu^{A}=\frac{1}{p}\sum_{i=1}^p\delta_{\lambda_i(A)}
$$
denotes its (normalized) spectral measure. For any matrix $A$, we write $A^\top$ for the  transpose of $A$ and $\bar{A}$ for its complex conjugate.  
For $1 \leq r \leq \infty$, we   denote by $\| A \|_{S_r}= \big ( \sum_{j=1}^p\lambda_j(A\bar A^{\top})^{r/2} \big )^{1/r} $ the 
Schatten-$r$-norm  of the matrix $A$ (with the usual convention $\|A\|_{S_\infty}=\max_j\lambda_j(A\bar A^\top)^{1/2}$; for a $p\times\infty$ matrix $A$, $\lambda_j(A\bar A^\top)$ are the eigenvalues of the $p\times p$ matrix $A\bar A^\top$). For probability measures $\mu,\nu$ on $\R$, $d_{BL}(\mu,\nu)=\sup\{\int f\,d\mu-\int f\,d\nu:\ \max(\|f\|_{\sup},\|f\|_{L})\le1\}$ denotes the dual bounded Lipschitz metric, which metrizes weak convergence (see \eqref{eq: dBL} in the appendix). The Stieltjes transform of a distribution $G$ on the real line is given by  
  $m_G:    \C^+ \rightarrow \C^+$ with
\begin{equation*}
  m_G (z) = \int {1 \over \lambda -z } dG(\lambda),
\end{equation*}
where  $\mathbb{C}^+=\{z\in\C\arrowvert\, \Im(z)>0\}$ denotes the upper complex half-plane. 
 If $\mu_n$, $n\in\N$, and $\mu$ are finite signed measures on a common measurable space, $\mu_n\Rightarrow \mu$ denotes weak convergence of $(\mu_n)_{n\in\N}$ to  $\mu$.

\smallskip

{\it Model assumptions. } Aligning with the common framework in random matrix theory, we shall work under the same type of conditions and  study a triangular array of $p=p(n)$-dimensional observations $Y_1,\dots, Y_n$ of the form
\begin{equation}\label{eq: Annahme_an_Y}
  Y_i= A_n X_i, \ \ \ i=1,\dots,n.
\end{equation}
Here, 
$X_i = (X_{i1} ,X_{i2} , \ldots )^{\top }$  $(i\in\N)$  are independent identically distributed (iid) infinite dimensional  random vectors  and $A_n$ is a  $p \times \infty  $ matrix   such that the following assumptions  are satisfied:
\begin{itemize}
\item[(A1)] The  $(p \times \infty)  $-matrix $A_n$ has  square summable rows and
$\sup_{n \in \mathbb{N} } \| A_n \|_{S_\infty} < \infty$. 
\item[(A2)]  
$p/n\rightarrow c$ for some real constant $c>0$ as $n\rightarrow\infty$. 
\item[(A3)] The  vector $X_1$ has   iid entries $X_{1k}$, $k\in\N  $, with $\E X_{11}=0$ and $\E X_{11}^2=1$.
\end{itemize}
Note that under these  conditions, the random variable $Y_1 = A_n X_1$ is well defined as 
   limit in $L^2(\mathbb{P})$ 
   with covariance matrix
   $$
   \Sigma_n= \E [Y_1Y_1^\top ] = A_n A_n^{\top }~. 
   $$ 
   As concerns normal  approximation of linear spectral statistics, the existence of the  fourth moment $\E X_{11}^4<\infty$ is known to be necessary (consider $f(x)=x$ and $\Sigma_n=I_p$, where the linear spectral statistic is a normalized sum of the iid variables $X_{ik}^2$). Therefore, we shall impose in that case the stronger assumption
\begin{itemize}
    \item[(A3+)] In addition to assumption (A3),  $\E X_{11}^4=3$. 
\end{itemize}
Coincidence of the  fourth moment with those of the standard normal distribution can be avoided in the CLT for linear spectral statistics of high-dimensional covariance matrices at the expense of additional assumptions on the eigenvectors of the population covariance matrix \citep[Theorem~1.4]{panzhou2008}, or by a Gaussian approximation whose mean and variance depend on $n$ and need not converge \citep{najimyao2016}. Here, we refrain from this generalization to keep the technical expenditure as small as possible.
We  emphasize  that model \eqref{eq: Annahme_an_Y}  was also considered in 
\cite{zou2022clt} and contains the commonly
  used
  model
\begin{equation}
\label{eq: Annahme_an_Y1}
Y_i=\Sigma_n^{1/2}\tilde X_i,
 \end{equation}
 where $\tilde X_i$ is a $p$-dimensional random vector with iid entries $\tilde  X_{ik}$  ($\E \tilde X_{11}=0$, $\E \tilde X_{11}^2=1$)
 and  $\Sigma_n^{1/2}$  is the square root of the $p \times p $ matrix  $\Sigma_n$.
 For model \eqref{eq: Annahme_an_Y1}, it is well-known  that if  $(\mu^{\Sigma_n})$ is weakly convergent
  as $p,n \to \infty$ and   $p/n\to c\in (0,\infty)$, that is\begin{equation}\label{eq: Annahme_an_mu}
\mu^{\Sigma_n} \Rightarrow H\ \ \text{as}\ n \rightarrow\infty
\end{equation}
for some  distribution  $H$, the limiting spectral distribution (LSD)  of the sample covariance matrix exists and is given by the  Mar\u{c}enko-Pastur law 
{ $\mu_{c,H}^0$}  whose  Stieltjes transform
can be characterized   as the unique solution of the Mar\u{c}enko-Pastur (MP) equation 
\begin{align} \label{hd10}
 {m_{\gamma,H}^0} (z) = \int \frac{1}{\lambda (1-\gamma-\gamma z m_{\gamma,H}^0  (z))-z} dH(\lambda)
 \end{align}
 for $\gamma =c$ in the set $\{m\in\C:\,-(1-\gamma)/z+\gamma m\in\C^+\}$, $z\in\C^+$ \citep[Theorem~1.1]{silverstein1995}.
These results were extended to model  \eqref{eq: Annahme_an_Y}
 by  \cite{zou2022clt}. 
 Finally, we define  for a
distribution   $G$ on $[0,\infty)$ ${\underline{m}_{\gamma,G}^0}$ 
as  the unique solution in $\C^+$ of the equation
\begin{equation}
\label{eq: mbarnull}
{\underline{m}_{\gamma,G}^0}(z)=-\bigg(z-\gamma\int\frac{t}{1+t\underline{m}_{\gamma,G}^0(z)}\dd G(t)\bigg)^{-1}.%
\end{equation}
Note that 
$$
{\underline{m}_{\gamma,G}^0}(z)=-\frac{1-\gamma}{z}+\gamma {{m}_{\gamma,G}^0}(z), 
$$
where ${{m}_{\gamma,G}^0}(z)$ is the solution of 
the equation \eqref{hd10} for $H=G$.

\section{Representative subpopulations and the ``$(m,mp/n)$ out of $(n,p)$''  bootstrap}
 \label{sec2}
   \def\theequation{3.\arabic{equation}}	
	\setcounter{equation}{0}

The $m$ out of $n$ sampling-with-replacement bootstrap with $m = o(n)$
  provides a powerful methodology  in situations where the classical bootstrap  ``resampling with replacement''  does not work \citep{PolitisRomano1994,bickel1997}.  Moreover, modern massive data sets clearly favor a comparatively small subsampling size
in view of computational advantages.
  However, in
   the high dimensional regime $p/n \to c >0$ as $n\to \infty$,  the  properties of
the LSD  and  LSS' depend sensitively of the ``proportion'' $c$
 and the application of this methodology is  questionable as a  sample covariance matrix based
on a random sample of $m$ observations from $Y_{1}, \ldots , Y_{n}$  with $m=o(n)$  would exhibit an asymptotic behavior as in the case $c= \infty$.
To address this difficulty, we propose to also sample (by a  possibly random  coordinate projection $\Pi_n$) $q=\lfloor mp/n\rfloor $  coordinates of each of the $m$ observations in the bootstrap sample 
such that  $q/m \to c $ as $n\to \infty$, $m\to \infty$, $m=o(n)$ (note that $0\le p/n-q/m<1/m$). This    procedure will be called {\it  ``$(m,mp/n)$ out of $(n,p)$  bootstrap''}   throughout this paper. Interpreting  the entries $Y_{i1},\dots, Y_{ip}$  of each vector $Y_i$ as data of $p$ individuals in some population, our approach originates from the idea of selecting a {\bf representative subpopulation} of size $q$ which shares the statistical properties of interest  with the full population. An appropriate strategy to pick the subpopulation implements prior knowledge or rather a model assumption on the data generating process.
\begin{center}
\hspace{-8cm}
 \begin{tikzpicture}[
    node distance=1.2cm,
    every node/.style={align=center}
]
\node (top) {\hspace{1.7cm}\, $Y_1, \dots, Y_n 
    \overset{\text{iid}}{\sim} 
    \left(0, \Sigma_n \right)$%
    };
\node (middle) [below of=top] {
    $\hspace{3.7cm} \Pi_n Y_1, \dots, \Pi_n Y_n \mid \Pi_n
    \overset{\text{iid}}{\sim} 
    \left(0, \Pi_n \Sigma_n \Pi_n^{\top} \right)$%
};
\node (bottom) [below of=middle] {
    $\hspace{9.1cm}\, Z_1^*, \dots, Z_m^* \mid Y_1,\dots, Y_n,\Pi_n
    \overset{\text{iid}}{\sim} 
    \widehat{\mathbb{P}}_n^{\Pi_n} 
    = \frac{1}{n} \sum_{i=1}^{n} 
    \delta_{\Pi_n Y_i}
    $\quad($q/m\approx p/n$)
};
\draw[->] (top) -- (middle);
\draw[->] (middle) -- (bottom);
\end{tikzpicture}
\end{center} 

 In terms of covariance matrices, moving from $Y_i$ to $\Pi_n Y_i$ corresponds to randomly select a principal submatrix $\Pi_n\Sigma_n\Pi_n^\top$ out of the original population covariance matrix. The fact that the spectral distribution explicitly enters the normal approximation of linear spectral statistics  necessarily  requires  that the spectral distributions  of the population covariance matrix and the randomly selected principal submatrix are (approximately) the same: 
  $$
  \mu^{\Sigma_n}\approx \mu^{\Pi_n\Sigma_n\Pi_n^\top}.
  $$ 
  As a principal submatrix usually does not share the original spectral distribution however, the question raises whether an appropriate (random) coordinate selection strategy $\Pi_n$ can be realized -- especially without knowledge of $\Sigma_n$ -- in a significant number of problems of interest. We illustrate by the examples below that in many situations, structural assumptions on the population covariance matrix are plausible, which either correspond to a (composite) null hypothesis in a statistical test  or simply to a model assumption on the state of nature, under which such a representative subpopulation selection strategy $\Pi_n$ actually does exist.

\smallskip
In model \eqref{eq: Annahme_an_Y} satisfying (A1) -- (A3), suppose that $(\Sigma_n)_n$ with $\Sigma_n=A_nA_n^\top \in\R^{p\times p}$ is a sequence of covariance matrices and $(\Pi_{n})_n$ with $\Pi_{n}:\R^p\rightarrow\R^q$ and $q=o(p)$ is a possibly random sequence of coordinate projections, independent of $(X_n)_{n\in \N}$.
\begin{condition}[Representative subpopulation Condition]\label{def: rsc} The sequence  $(\Sigma_n,\Pi_n)_n$ is said to satisfy the Representative Subpopulation Condition   if the following is satisfied:
\begin{itemize}
    \item[(1)] With $\tilde\Sigma_n=\Pi_n\Sigma_n\Pi_n^\top$,  \begin{equation}\label{eq: similarity3}
d_{BL} \Big ( \mu^{\tilde\Sigma_n}, \mu^{ \Sigma_n }  \Big)  \longrightarrow 0\ \    \text{in probability}\ \ \text{as}\ n \rightarrow \infty, 
\end{equation}
where $d_{BL}$ denotes the dual bounded Lipschitz metric (cf.~\eqref{eq: dBL} in Section \ref{sec4}). \item[(2)] For almost all realizations of $(\Pi_n)$, there exists a decomposition of the form 
\begin{equation}\label{decompproj}
\Pi_n A_n  =  L_n  + R_n ~,
\end{equation}
where the sets of non-zero entries of the matrices $L_n$ and $R_n$ are disjoint, the matrix $ L_n $ has at most $q'$ non-zero columns with $q'\le Kq$ for a constant $K$ which depends neither on $n$ nor on the realization of $\Pi_n$, and $\mathbb{E}_{\Pi_n}
\big [ ||  R_n   ||_{S_{2}}^{2} \big ] = o (1)$ as $n\rightarrow\infty$ (here $\mathbb{E}_{\Pi_n}$ denotes the expectation with respect to $\Pi_n$). 
\end{itemize}
\end{condition}

Assumption (A1)  implies in particular that the spectral norm of the matrix $
\tilde \Sigma_n = \Pi_n \Sigma_n \Pi_n^\top$ is uniformly bounded in $n \in \mathbb{N}$. 

\begin{remark}[Stability under perturbations] If $(\Pi_n,\Sigma_n)$ satisfies spectral similarity~\eqref{eq: similarity3} and $\Sigma_n=A_nA_n^\top$ and $\Gamma_n=B_nB_n^\top$ in model \eqref{eq: Annahme_an_Y} fulfill  $\tr \big((\Gamma_n-\Sigma_n)(\Gamma_n-\Sigma_n)^{\top}\big)=o(q)$ or
 $\rank(\Sigma_n-\Gamma_n)=o(q)$,
then  also $(\Pi_n,\Gamma_n)$ satisfies \eqref{eq: similarity3} due to Theorem A.43 and Corollary~ A.41 in \cite{baisilverstein2010}, together with Lemma C.13 in \cite{JurczakRohde}.  Note that the light-tail condition \eqref{decompproj} is always satisfied under appropriate summability conditions on the rows of $A_n$ resp.~$B_n$, independently of $\Pi_n$.
\end{remark}

\begin{example}[Diagonal covariance matrices]\normalfont \label{ex2}
Let
$$ (\Sigma_n  )_{n\in \N}=  (\text{diag} (g_{1}, \ldots , g_{p}))_{n\in \N} $$
denote a sequence of positive
semi-definite diagonal matrices satisfying \eqref{eq: Annahme_an_mu}, that is
\begin{equation}\label{hd15}
  \frac{1}{p} \sum^p_{i=1} \delta_{g_i} \Rightarrow H
\end{equation}
for some distribution  $H$. Let $\Pi_n:\R^p\rightarrow\R^q$ be the random coordinate projection which picks $q$ out of $p$ components uniformly at random. Then
$$
\E\, \mu^{\Pi_n\Sigma_n\Pi_n^\top}=\mu^{\Sigma_n},
$$
and by Theorem 1
of \cite{chatterjee2009}  and  \eqref{hd15}, it follows that \eqref{eq: similarity3} is fulfilled. Moreover, for the canonical choice $A_n=(\Sigma_n^{1/2},0)\in\R^{p\times\infty}$,\eqref{decompproj} also holds using $L_n  = \Pi_n  A_n \in \mathbb{R}^{q \times \infty}, R_n = 0  \in \mathbb{R}^{q \times \infty}$.
\end{example}

More generally, for an arbitrary Hermitian matrix $M$ of order $N$ and $k\le N$, \cite{chatterjee2009} prove that the empirical spectral distribution of a principal submatrix of order $k$, chosen uniformly at random, concentrates around its expectation: the expected Kolmogorov distance is at most $(13+\sqrt8\log k)/\sqrt k$, uniformly in $M$ and $N$ (their Theorem~1).
Note  that  in general, this distribution does not coincide with the spectral distribution of $M$. However, this is true for diagonal matrices $M$.

\begin{example}[Symmetric Toeplitz and block Toeplitz matrices] \label{ex1}\normalfont  ~~~
\begin{itemize}
\item[(i)]
If the components  of the vectors  $Y_i=(Y_{i,\ell})_{\ell=1, \ldots p}$  are  defined by  a purely nondeterministic stationary process, then it follows by Wold's theorem \citep[see][]{BrockwellDavis}
that
 \begin{equation}\label{ho1}
   Y_{i,\ell}  = \sum^\infty_{j=0} b_j X_{i,p-\ell+j+1}
 \end{equation}
 where $\sum_{j=0}^{\infty} b_j^2  < \infty$ and for each $i$  the random variables in the vector   $X_{i}= (X_{i,j})_{j \in \mathbb{N}}$ are uncorrelated.
If the random variables $X_{i,j}$ are independent with $\E X_{11}=0$, $\E X_{11}^2=1$
(as assumed in the present paper)
we obtain a representation of the form \eqref{eq: Annahme_an_Y}, where the $p \times \infty $ matrix $A_n$ is given by
  \begin{eqnarray}\label{antoep}
   A_n  &=& \left( \begin{array}{ccccccccc}
                       0 & 0  & \ldots & \ldots & {\ldots} &0 & b_0 & b_1 & {\ldots}\\
                      0 & 0  & \ldots & \ldots & {\ldots} & b_0 & b_1 & b_2 & {\ldots}\\
                      \vdots  &  \vdots & \ddots & \ddots & \ddots & \ddots &\ddots & \ddots & {\cdots} \\
                   0   &  0 & b_0 & b_1 & \cdots  &  b_{p-4} &  b_{p-3}  & b_{p-2}  &\cdots \\
              0   & b_0 & b_1 & b_2 & \cdots  & b_{p-3}& b_{p-2}  & b_{p-1} & \cdots \\
                     b_0 & b_1 & b_2 & b_3    & \cdots & b_{p-2} &  b_{p-1}  & b_{p}  & {\ldots}
                            \end{array} \right) . 
 \end{eqnarray}
Then the $p\times p$ autocovariance matrix
 $
 \Sigma_n = A_n A^\top_n  = (t_{|i-j|})^p_{ij=1}
 $
is a Toeplitz matrix, where 
 $$
 t_k= {\rm Cov} (Y_{i,\ell}, Y_{i,\ell - k}) = \sum^\infty_{j=k} b_j b_{j-k}= \sum^\infty_{j=0} b_j b_{j+k}
 $$
 (note that $t_{- \ell} = t_\ell$). In particular,  $\Sigma_n $
 is a $p\times p$ principal submatrix of the   fixed (infinite)   Toeplitz matrix  $(t_{i-j})_{ij\in \N }$.
 Now, if $\sum^\infty_{\ell = 0}|t_\ell| < \infty$, it follows from Szeg\"{o}'s theorem \citep [see][]{gresze1958} that the normalized spectral distribution of $\Sigma_n$ satisfies  \eqref{eq: Annahme_an_mu},
  where the limiting distribution $H$ is supported on  the range $[\operatorname{ess\,inf}_t T(e^{it}),\operatorname{ess\,sup}_t T(e^{it})]$ of the symbol $T$. More precisely, $\mu^{\Sigma_n} \Rightarrow H$ as $n \to \infty $, where the measure $H$ is defined by
 $$
 H ( ( \alpha, \beta ]) = \frac {1}{2 \pi} \lambda \big ( \big \{ t \in (-\pi, \pi]  \big |  \alpha < T(e^{it}) \leq \beta  \big \}  \big )
 $$
with $\lambda$ denoting the Lebesgue measure
 and 
 $$
 T(z) = t_0 + \sum^\infty_{\ell = 1} t_\ell  ( z^\ell +z^{-\ell } ), \quad\text{i.e.}\quad T(e^{it})=t_0+2\sum_{\ell=1}^\infty t_\ell\cos(\ell t),
 $$
 is the Laurent series with coefficients $(t_\ell)_{\ell \in \mathbb{Z}}$.
If $\tilde \Sigma_n   $ is the $q \times q$ principal submatrix of $\Sigma_n$ corresponding to $q$ consecutive indices (which equals $\Sigma_q$),  then obviously
its spectral distribution converges weakly to $H$ as $q \to \infty$. Consequently, if  $Y_{1,sub}$ consists of  $q$ consecutive entries of
$Y_1$, its covariance matrix is equal to $\Sigma_q$,  and spectral similarity holds in the sense of \eqref{eq: similarity3}.\\
Note that it is not even necessary to rely on a random sampling mechanism. Now, let $ \Pi_n :\R^p\rightarrow\R^q$ denote any coordinate projection which selects q consecutive  components of  the vector $Y_{1}\in\R^p$,
defined  in \eqref{ho1}. Then, with the definition of the matrix $A_n $ in \eqref{antoep}, a  decomposition of the form \eqref{decompproj}  holds with
  \begin{eqnarray*} 
\label{hd4}
  ~~~~~
   L_n  &=& \left( \begin{array}{ccccccccccc}
                0  &0  & 0& \cdots& 0 & b_0 & \cdots & b_q & 0 & 0 & \cdots \\
                0  & 0 &0 & \cdots& b_{0} & b_1 & \cdots & b_{q+1} & 0 & 0 & \cdots \\
                   \vdots   &  \vdots  & \vdots & \ddots &\vdots&  \vdots &  \ddots & \vdots  & \vdots & \vdots &  \ddots \\
                                 0  & 0  & b_0 &   \cdots  &  b_{q-4}&  b_{q-3} & \cdots & b_{2q-3} & 0 & 0 & \cdots \\
                             0    & b_0   & b_1 &  \cdots&   b_{q-3}& b_{q-2}& \cdots & b_{2q-2} & 0 & 0 & \cdots \\
                             b_0 & b_1  & b_2 & \cdots  & b_{q-2}& b_{q-1}&\cdots & b_{2q-1} & 0 & 0 & \cdots
                \end{array}  \right) ~
 \end{eqnarray*}
 and 
 \begin{eqnarray*} 
   R_n   &=& \left( \begin{array}{ccccccc}
                     0 & 0  & \cdots & 0 & b_{q+1} & b_{q+2} & \cdots \\
                 0 & 0  & \cdots & 0 & b_{q+2} & b_{q+3} & \cdots \\
              \vdots & \vdots &  \ddots  & \vdots & \vdots & \vdots & \ddots \\
                     0 & 0 & \cdots & 0 & b_{2q-2} & b_{2q-1} & \cdots \\
                     0 & 0 & \cdots & 0 & b_{2q-1} & b_{2q} & \cdots \\
		 0 & 0 & \cdots & 0
		  & b_{2q}  & b_{2q+1} & \cdots
                \end{array}  \right)~
 \end{eqnarray*}
 (note that  both matrices have $q$ rows and  that  the first $2q$ columns of the matrix $R_n  $ have zero entries).
 Now, under the additional assumption that
 \begin{equation*}  
\sum^\infty_{ \ell=1}  \ell b^2_\ell < \infty,
\end{equation*}
  it is easy to see that
 \begin{eqnarray*}
  \| L_n \|^2_{S_2} &=&   q \sum^q_{i=0}   b^2_i    +  \sum^{2q-1}_{i=q+1}   (2q-i) b^2_i   =
   O(q) \\
    \| R_n   \| ^2_{S_2} &=& \sum^q_{i=1} \sum^\infty_{\ell = i} b^2_{q+ \ell} \leq
    q \sum^\infty_{\ell = q+1} b^2_\ell \leq \sum^\infty_{ \ell=q+1}  \ell b^2_\ell = o(1)
 \end{eqnarray*}
 as $n \to \infty$ (a similar argument applies if $q$ consecutive components of $Y_{1}$ are sampled from a uniformly distributed position  on the set $\{ 1, \ldots , p-q+1 \}$). Hence,
  the triangular array $(Y_1,\dots, Y_n)$   satisfies  the   Representative Subpopulation Condition. If  even $\sum_\ell\ell^2b_\ell^2<\infty$, also the stronger condition $\E_{\Pi_n}\|R_n\|_{S_2}^2=o(q^{-1})$ of Theorem~\ref{thm: new} below holds, since $\|R_n\|_{S_2}^2=\sum_{\ell>q}\min(\ell-q,q)b_\ell^2\le q^{-1}\sum_{\ell>q}\ell^2b_\ell^2$
  (note that the first assumption in Condition \ref{def: rsc} was shown above; the matrices $L_n$ and $R_n$ displayed above correspond to the selection of the last $q$ components of $Y_1$).

\item[(ii)] 
 If $p= \tilde p r$ for some $\tilde p, r \in \mathbb{N}$, similar results are available
 in the case where the components of the vectors $Y_i$ can be decomposed  into $\tilde p$ blocks of length $r$, that is
  $$
  Y_i = \big( Y^{(1)\top }_i, \ldots, Y^{(\tilde p)\top }_i \big)^\top  \qquad i=1,\ldots,n,
  $$
 which are defined by a vector moving average model of order $\ell \leq \tilde p - 1$, that is
  $$
  Y^{(s)}_i = \sum^\ell_{j=0} B_j \varepsilon_{i, \tilde p- s+j+1}~,~~s=1, \ldots , \tilde p .
  $$
  Here, ($\varepsilon_{i,j})_{i \in \mathbb{N}, j \in \mathbb{N}_0}$ is an array of independent $r$-dimensional vectors with $\mathbb{E} [\varepsilon_{ij}]=0 \in \mathbb{R}^r$ and Var$(\varepsilon_{ij})=I_r$ (to obtain a representation of the form \eqref{eq: Annahme_an_Y} with (A3), the components of the $\varepsilon_{ij}$ have to be iid, distributed as $X_{11}$), $I_r$ denotes the $r$-dimensional identity matrix and $B_0, \ldots, B_\ell$ are given $r \times r$ matrices. In this case, it is easy to see that the population covariance matrix $\Sigma_{n}$ of $Y_i$
is a banded block Toeplitz matrix, that is
 \begin{equation}\label{hd11}
   \Sigma_n = \big( T_{i-j} I \{ |i-j| \leq \ell \} \big)^{\tilde p}_{i,j=1}~, \qquad T_{-s}:=T_s^\top,
 \end{equation}
 where $T_0, \ldots, T_{\ell}$ are $r \times r$ matrices defined by
 $$
T_s = \Cov\big(Y_i^{(t)},Y_i^{(t-s)}\big) = \sum_{j=0}^{\ell-s} B_{j+s}B_{j}^\top ~,~~s=0,\ldots ,\ell
 $$
($T_0$ is symmetric, $T_s$ for $s>0$ in general not; the blocks with $|i-j|>\ell$ vanish).
 If $\tilde p \to \infty$,  the LSD of the population covariance matrix
 exists and  can be
 characterized in terms of an equilibrium problem under certain technical assumptions \citep[see][]{delvaux2012}. However, an explicit form is only possible in very special cases. For example, if $\ell=1$ and $T_1$ is non-singular, the LSD is absolutely continuous, with a density which can be expressed through the roots $z_1(x),\ldots,z_{2r}(x)$ of the algebraic equation $\det\big(T_1z+T_0+T_1^\top z^{-1}-xI_r\big)=0$ \citep[Theorem~2.5 and formula (2.14)]{delvauxdette2012}. If, in addition, $T_1$ is symmetric, the LSD is $\frac1r\tr X_{T_1,T_0}$, where $X_{T_1,T_0}$ denotes the matrix measure of orthogonality of the matrix Chebyshev polynomials of the first kind with recurrence coefficients $T_1$ and $T_0$ \citep[Theorem~1.1]{durlopsaff1999}; for positive definite $T_1$, an explicit formula for its density is given in Theorem~3.1 of \cite{durlopsaff1999}.
 Now, if  $\tilde \Sigma_n   $ is the $q \times q$ principal submatrix of $\Sigma_n$ corresponding to $q/r$ consecutive blocks,  then   \eqref{eq: similarity3}
 obviously   holds. Moreover, $\Pi_nA_n$ has at most $q+\ell r$ non-zero columns, so that \eqref{decompproj} holds with $L_n=\Pi_nA_n$ and $R_n=0$.
\end{itemize}

\end{example}

\begin{example}[Representative subpopulations]\normalfont \label{ex3}
In a recent paper,   \cite{fan2019} investigated properties of  the  LSD of   variance components in linear random  effect models.
In the  simplest case of a random effect  ANOVA  model  with $k$ factors, we have
\begin{equation}\label{m1}
   \tilde Y_{ij} = M_i + s_{ii} X_{ij} \qquad \qquad j=1,\ldots,p_i \ ; \quad i=1,\ldots,k,
 \end{equation}
 $s_{11}, \ldots , s_{kk}>0$  are constants,
 $ \{X_{ij} | ~j=1,\ldots,p_i; i=1,\ldots,k\}$ are independent random variables with $\mathbb{E} [X_{ij}]=0$, ${\rm Var}(X_{ij})=1 $ and  $M=(M_1,\ldots,M_k)^\top  $ is a $k$-dimensional random vector with covariance matrix $T=(\tau_{ij})^k_{i,j=1}$ representing the group effects.  In this case,  the vector
 $Y_1$ can be decomposed in $k$ groups, that is
 $$
Y_1 = (\tilde Y_{11}, \ldots, \tilde Y_{1p_1}, \tilde Y_{21}, \ldots, \tilde Y_{2p_2}, \ldots, \tilde Y_{k1}, \ldots, \tilde Y_{kp_k})^\top  \in \R^p~,
 $$
 where $p= \sum_{i=1}^k p_i$.
Using the notation $\sigma_{ii}:=s^2_{ii} + \tau_{ii}$ $(i=1,\ldots,k)$,
  the covariance matrix  of the vector $Y_{1}$ is 
  the   positive semi-definite symmetric block matrix
 $$
  \Sigma_n  =
\left(   \begin{matrix}
 G_{11}  & G_{12} & G_{13}  &  \dots  &  G_{1k}\\
 G_{21}  & G_{22}  & G_{23}  &  \dots  &  G_{2k}\\
  \vdots   &  \vdots &  \vdots   &  \ddots  &   \vdots \\
  G_{k1}  & G_{k2}  & G_{k3}  &  \dots  &  G_{kk}
\end{matrix}
\right)
$$
  with  blocks
$$
G_{ii} =
\left(   \begin{matrix}
 \sigma_{ii}  & \tau_{ii}
  &  \dots  &   \tau_{ii} \\
  \tau_{ii}   & \sigma_{ii}
   &  \dots  &   \tau_{ii} \\
  \vdots   &  \vdots
   &  \ddots  &   \vdots \\
   \tau_{ii}   & \tau_{ii}
    &  \dots  &  \sigma_{ii}
\end{matrix}
\right)\in \mathbb{R}^{ p_i \times p_i},~
G_{ij} = \tau_{ij}
\left(   \begin{matrix}
1 & 1
  &  \dots  &  1 \\
1 & 1
  &  \dots  &  1 \\
  \vdots   &  \vdots
    &  \ddots  &   \vdots \\
1 & 1
 &  \dots  &  1 \\
\end{matrix}
\right) \in \mathbb{R}^{ p_i \times p_j}
 ~(i \not = j ).
$$
Simple calculation shows that  the matrix $  \Sigma_n$
 has at most $2k$ different eigenvalues and that there are $k$ eigenvalues $\lambda_i = \sigma_{ii}-\tau_{ii}$ of multiplicity
 $p_i-1$ ($i=1, \ldots , k$).  Therefore, if 
 $k$ is fixed, the constants $s_{ii}$ (and hence $\lambda_i=s_{ii}^2$) do not depend on $n$, and there exist nonnegative constants $ \omega_1,   \ldots , \omega_k $ such that   $\sum_{i=1}^{k } \omega_{i} =1 $ and
 $$\max^k_{i=1} \Big| \frac {p_i}{p}- \omega_i \Big| = o(1)\ \ \text{ as }n \to \infty ,$$
 the sequence $(\mu^{\Sigma_n})_{n \in \mathbb{N}} $ of spectral measures converges weakly to the  discrete measure $\sum^k_{i=1} \omega_i \delta_{\lambda_i}$. 
If  $\Pi_n$ is such that $Y_{1,sub}=\Pi_nY_1$ consists of $q$ out of $p$ different entries of the vector $Y_1$ uniformly sampled at random without replacement, its covariance matrix ${\tilde \Sigma_n}=\Pi_n\Sigma_n\Pi_n^\top$   has  again at most $2k$ different eigenvalues and  there exist $k$ eigenvalues $\lambda_i = \sigma_{ii}-\tau_{ii}$ of multiplicity
 $\max(q_i-1,0)$ ($i=1, \ldots , k$), where $(q_1, \ldots , q_k)$ is  a multivariate hypergeometrically distributed random variable with parameters $((p_1, \ldots , p_k),q)$.
Hence,   property \eqref{eq: similarity3} is satisfied. Note that model \eqref{m1} in  Example \ref{ex3} can be rewritten as 
 \begin{equation}\label{m1a}
   Y_1 = EM+SX~,
 \end{equation}
 where
 \begin{eqnarray*}
 X &
 =& \big  ( X_{11}, \ldots, X_{1p_{1}}, \ldots,  X_{k1}, \ldots, X_{kp_{k}} )^{\top } \in \mathbb{R}^{p },\\
  M
  & =&  \big  ( M_{1},  \ldots, M_{k}
 )^{\top } \in \mathbb{R}^{k },
 \end{eqnarray*}
  the $p \times k$ matrix $E$ and the $ p  \times p$ matrix $S$ are  defined by
 \begin{eqnarray*}
   E &=& \left ( \begin{array}{cccc}
                   1_{p_1} & & & \\
                   & 1_{p_2} & & \\
                   & & \ddots  \\
                   & & & 1_{p_k}
                 \end{array} \right ) \in \mathbb{R}^{p\times k} ~,~~
                  \label{emat}
   S = \left ( \begin{array}{cccc}
                  s_{11} I_{p_1} & & & \\
                   & s_{22} I_{p_2} & & \\
                   & & \ddots  \\
                   & & & s_{kk} I_{p_k}
                 \end{array} \right )  \in \mathbb{R}^{p\times p}~,
 \end{eqnarray*}
  respectively, and
 $1_{p_\ell} = (1, \ldots, 1)^\top   \in \mathbb{R}^{p_\ell}$ (all other entries in the  matrices $E$ and $S$ are $0$).
Model \eqref{m1a} can alternatively be represented (in distribution) as
 \begin{equation}\label{m1c}
   Y_1 = \tilde EZ
 \end{equation}
 where $Z= (X^\top  , U^\top  )^\top   \in \mathbb{R}^{p+k}, U = (U_1,\ldots,U_k)^\top  $ is a vector with iid entries independent of $X$, distributed as $X_{11}$ (the representation holds provided that $M\overset{d}{=}T^{1/2}U$, e.g.\ in the Gaussian case) and the $p \times (p+k)$ matrix $\tilde E$ is given by
 $$
 \tilde E  = ( S  : ET^{1/2})~.
 $$
Model \eqref{m1c} can obviously be written in the form \eqref{eq: Annahme_an_Y}. Moreover, the matrix  $L_n $ in the decomposition \eqref{decompproj} is given by
$q$ randomly drawn rows from the matrix $\tilde E$, while the matrix $R_n  $ has only $0$ entries. Note, however, that $\|\Sigma_n\|_{S_\infty}=\|S^2+ETE^\top\|_{S_\infty}\ge s_{ii}^2+p_i\tau_{ii}$ (take the vector $p_i^{-1/2}1_{p_i}$ in group $i$), so that assumption (A1) requires $T=T_n$ with $\sup_n\|D_n^{1/2}T_nD_n^{1/2}\|_{S_\infty}<\infty$, $D_n=\diag(p_1,\ldots,p_k)$ (e.g.\ $T_n=T/p$); for fixed $T$ with $\tau_{ii}>0$ the matrix $\Sigma_n$ has $k$ spikes of order $p$. As the matrix   $S$ is diagonal,  $L_n$  has at most $q+k = O(q)$ non-zero columns, and 
 the sequence  $(\Sigma_n,\Pi_n)_n$ satisfies the  Representative Subpopulation Condition.

\end{example}

The Representative Subpopulation Condition being granted for $q=\lfloor mp/n\rfloor$ with $m\ll n$, we propose the following resampling scheme.

\needspace{8\baselineskip}
\begin{alg}[   $(m,mp/n)$ out of $(n,p)$ Bootstrap]  \label{alg1}
Let $Y_1,\dots, Y_n$ be drawn from \eqref{eq: Annahme_an_Y}.
\begin{itemize}
\item[(i)] For the coordinate selection projection $\Pi_n$ of the Representative Subpopulation Condition, form the $q$-dimensional random variables $\Pi_nY_1,\dots,\Pi_nY_n$. 
\item[(ii)] Conditional on $Y_1,\dots, Y_n$ and $\Pi_n$, draw an iid-sample $Z_1^*,\dots, Z_m^*$ from the
measure
$$
\hat{\P}_n^{\Pi_n}=\frac{1}{n}\sum_{i=1}^n\delta_{\Pi_nY_i}.
$$
 \item[(iii)]  Output  the estimator
$\displaystyle \hat \Sigma_{n}^{*} = {1 \over m} \sum_{i=1}^{m}Z_i^*{Z_i^*}^{\top }$
 and its  spectral distribution $\mu^{\hat \Sigma_{n}^{*}} $.
\end{itemize}
\end{alg}
Throughout, $\P^*=\P(\,\cdot\,|Y_1,\dots,Y_n,\Pi_n)$ and $\E^*$ denote the conditional probability and expectation in the bootstrap world, given the data and the projection $\Pi_n$.

Note that the Representative Subpopulation Condition is a structural model assumption on the data generating process, which either corresponds to a (composite) null hypothesis in a statistical test  or simply to a model assumption on the state of nature (which might be plausible or can be tested). Exploiting structure allows to significantly reduce dimension and enables to run our new bootstrap at low computational cost, but necessarily comes at the expense of a qualitative model hypothesis. While appealingly no specific  type of structure on $\Sigma_n$ is required in the Representative Subpopulation Condition which allows applicability in a large variety of situations, a suitable selection strategy $\Pi_n$ always implements prior knowledge on the data generating process.
In the following section, we will show that under appropriate assumptions, Algorithm \ref{alg1}  yields a consistent bootstrap estimate of the  LSD and of the distribution of linear spectral statistics.

\section{Probabilistic properties of the ``$(m,mp/n)$ out of $(n,p)$'' bootstrap}  \label{sec3}
  \def\theequation{4.\arabic{equation}}	
	\setcounter{equation}{0}

In view of the failure of the classical sampling-with-replacement bootstrap, it is apparent that independence of the bootstrap observations $(Z_1^*, \ldots , Z_m^*) = (\Pi_nA_nX_1^* , \ldots , \Pi_nA_n X_m^*)$ 
 conditionally on the original data $Y_1, \ldots , Y_n$ and $\Pi_n$ cannot be sufficient to successfully run classical arguments in the conditional bootstrap world for proving  consistency of our new approach. 
Indeed,  the vectors $X_i^*$ do not satisfy Assumption (A3) any longer (conditional on $X_1,\dots, X_n$); in particular, they do not possess the essential structure of independent components  which however is a crucial requirement for   the classical MP law and the CLT of linear spectral statistics to hold. 

\subsection{Spectral distribution}Our first result demonstrates that $\widehat\Sigma_n^*$ mimics the sample covariance matrix in terms of spectral distributions. This is a necessary ingredient  for the CLT for linear spectral statistics studied later as the limiting spectral distribution of the sample covariance matrix explicitly enters the limiting variance expression of linear spectral statistics.

\begin{theorem}[Spectral distribution]\label{thm: Boot1}
Grant assumptions (A1) -- (A3).
Assume that the   Representative Subpopulation Condition~\ref{def: rsc} is satisfied with $q=\lfloor mp/n\rfloor$. If $m\to\infty$ and $m = o(n)$, then
\begin{align*}
d_{BL} \Big ( 
\mu^{\widehat \Sigma_n}, \mu^{\widehat \Sigma_n^*}\Big ) \longrightarrow  0 ~~\text{ in probability. } 
\end{align*}

\end{theorem}
As concerns the  proof of Theorem \ref{thm: Boot1},
note that 
the derivation of the classical MP-law via the Stieltjes transform method has  two major steps:
\begin{itemize}
    \item[(1)] to establish the concentration of the Stieltjes transform  of the bootstrap spectral measure around its conditional expectation (cf.~equation \eqref{eq: 5.1} in the appendix); 
\item[(2)]to prove that the conditional expectation approaches the solution of a particular MP equation  (see equation \eqref{eq: 5.1b} in the appendix).
\end{itemize} 
Whereas (1)  can be carried out by adapting classical martingale arguments due to the conditional independence of the bootstrap observations, carrying out  (2) is substantially more involved. 
 At this point,  it starts to matter that  there may be  ties in the bootstrap sample when studying quadratic forms of the type
$$
Z_1^{*\top} A(Z_2^*, \ldots , Z_m^*) Z_1^* -   {\rm tr}  \Big( \Pi_n\Sigma_n\Pi_n^{\top}    A(Z_2^*, \ldots , Z_m^*) \Big). 
$$
Here, $A(Z_2^*,\dots, Z_m^*)$  is a matrix containing the resolvent of $\frac{1}{m}\sum_{j=2}^mZ_j^*{Z_j^*}^{\top}$ as a building block. Although $Z_1^*$ is conditionally independent of $A(Z_2^*,\dots, Z_m^*)$, these expressions are not centered any longer and therefore do not satisfy classical moment  bounds for centered quadratic forms. Moreover, both, the vector $Z_1^*$ as well as the matrix  $ A(Z_2^*, \ldots , Z_m^*)$,
depend in an intricate way on the sample $X_1, \ldots , X_n$, which makes estimates on the unconditional expectation rather delicate, see Section \ref{sec52a}. When performing  the ``$(m,mp/n)$ out of $(n,p)$''  bootstrap, the probability of generating ties in the bootstrap sample turns out to be sufficiently small for the required approximation quality if $m=o(n)$. 

\begin{remark}
{\rm 
Note that the finite second moment of $X_{11}$  is necessary to  define $\Sigma_n$. 
Thus, it is the weakest possible requirement for the comparison of the spectral distribution of $\widehat\Sigma_n$ with the MP law $\mu^0_{p/n,\mu^{\Sigma_n}}$ to be meaningful.}
\end{remark}

\subsection{Extremal eigenvalues}

A further important step in the proof of the CLT for linear spectral statistics are estimates on the probability of exceedance for the extremal eigenvalues of $\widehat{\Sigma}_{n}^{*}$ from the support of the limiting spectral measure. Our next result  
 shows that $\widehat{\Sigma}_n^*$ even shares these properties with the sample covariance matrix to a large extent. Note that for the latter, \cite{BAI1988166} proved finiteness of $\E X_{11}^4$ to be the weakest condition to ensure that the   $\limsup $ of its spectral norm stays finite almost surely.
For the proof of Theorem~\ref{thm: new}, the original variables are replaced by truncated, centered and standardized ones (Section~\ref{sec82aa} of the appendix). The following results are only needed in this setting, which we formalize as
\begin{itemize}
\item[(T)] For each $n$, the random variables $X_{ij}=X_{ij}^{(n)}$ are iid with $\E X_{11}=0$, $\E X_{11}^2=1$, $\sup_n\E X_{11}^4<\infty$, and $|X_{ij}|\le\delta_m\sqrt m$ for a sequence $\delta_m\searrow0$ with $\delta_m m^{1/8}\to\infty$.
\end{itemize}
Condition (T) replaces (A3) in Theorem~\ref{thm: EW1}(b), (c) and in Corollary~\ref{thm: EW1a}; part (a) of Theorem~\ref{thm: EW1} shows that, under (A3) and $\E X_{11}^4<\infty$ only, the truncation does not affect the extremal eigenvalues asymptotically.

\begin{theorem}[Extremal eigenvalues]
\label{thm: EW1}
Grant  assumptions   (A1) -- (A3), $\E X_{11}^4<\infty$  and assume  that 
 $A_n=I_{p}$, canonically embedded into $\R^{p\times \infty}$. Let $\Pi_n:\R^p\rightarrow \R^q$ be a possibly random coordinate projection with $q=\lfloor mp/n\rfloor$, independent of $X_1,\dots, X_n$, and assume $m= o(\sqrt{n})$.
\begin{itemize}
    \item[(a)] For every $K_{\rm left}<(1-\sqrt{c})^2 \mathds{1}_{(0,1) } (c)$ there exists a constant $K_{\rm  right}$ such that 
    \begin{align} \label{ang12a}
    \P\Big( \lambda_{\min}\big(\widehat{\Sigma}_{n}^{*}\big)<K_{\rm left} \text{ or } \big\Arrowvert \widehat{\Sigma}_{n}^{*}\big\Arrowvert_{S_{\infty}}>K_{\rm  right}\Big)= o(1)\ \ \ 
    \end{align}
    \item[(b)]  If the entries satisfy Condition (T) (in place of (A3)), then
    \begin{itemize}
        \item[(b1)] for every $l\in\N$ there exists a constant $K_{\rm  right}=K_{\rm right}(l)>0$
    such that  
\begin{align} \label{ang12}
\P\Big(\big\Arrowvert \widehat{\Sigma}_{n}^{*}\big\Arrowvert_{S_{\infty}}>K_{\rm  right}\Big)= o(m^{-l})\ \ \ \text{ (for this $l$)}.
\end{align}
Moreover, if $m = o(\log n) $, then 
\eqref{ang12} even holds  
for every $K_{\rm  right}>(1+\sqrt{c})^2$ and every $l\in\N$.
    \item[(b2)] 
  we have   for every $K_{\rm left}<(1-\sqrt{c})^2 \mathds{1}_{(0,1) } (c) $
 \begin{align*}
\P\Big(\lambda_{\min}\big(\widehat{\Sigma}_{n}^{*}\big)<K_{\rm left}\Big)= o(m^{-l})\ \ \ \text{for every $l\in\N$}.
 \end{align*}
     \end{itemize}
 \item[(c)] Under the assumptions of (b),
 $\limsup_{n\to \infty} \E  \big\Arrowvert \widehat{\Sigma}_{n}^{*}\big\Arrowvert_{S_{\infty}}^\ell  < \infty$ for all $\ell \in \N$. 
 \end{itemize}
 \end{theorem}

A few comments on the proof are in order. For the classical covariance matrix, corresponding bounds  in \cite{yinbaikri1988} and \cite{baiyin1993} are based on trace moment estimates, which are deduced by   graph theory involving combinatorial arguments. 
Since our results do not contain bounds conditional (and potentially uniform) on $X_1,\dots, X_n$, we  were able to develop essentially two types of manipulation of their original combinatorial arguments in order to extend their results for the sample covariance matrix  to the bootstrap setting  as follows. 
  \begin{itemize}
  \item[(1)]  Let  $(W_1,\cdots, W_n)^\top $ 
  denote a  vector with a multinomial distribution with parameter $(m,(\frac{1}{n},\cdots,\frac{1}{n}))$, independent of $X_1,\dots, X_n$, that is
 $(W_1,\cdots, W_n)^\top \sim \mathcal{M} 
 (m,(\frac{1}{n},\cdots,\frac{1}{n}))$. Using the representation 
  $$
  \widehat{\Sigma}_{n}^{*}= \frac{1}{m}\sum_{i=1}^n W_i(\Pi_n X_i)(\Pi_n X_i)^{\top},
  $$
we aim at bounding expectations of the type
  \begin{align*}
  \E\tr \big(\widehat{\Sigma}_{n}^{*}\big)^k&=  \E \tr \Big(\frac{1}{m}\sum_{j=1}^n W_j(\Pi_n X_j)(\Pi_n X_j)^{\top}\Big)^k\\
  &=\frac{1}{m^k}\sum_{\substack{i_1,\dots, i_k\in\{1,\dots, q\}\\ j_1,\dots, j_k\in\{1,\dots, n\}}} 
\E \big [ W_{j_1}\dots W_{j_k} \big ] 
\E \big [ X_{i_1j_1}X_{i_2j_1}X_{i_2j_2}\cdots X_{i_kj_k}X_{i_1j_k}
\big ]
\end{align*}
The difference to the analysis of 
\cite{yinbaikri1988}
for $\widehat \Sigma_n$ are the additional factors $\E \big( W_{j_1}\dots W_{j_k}\big) $ as well as the range of indices $\{1,\dots, q\}, \{1,\dots, n\}$ instead of $\{1,\dots, p\}, \{1,\dots, n\}$ in the above expression. Note that $n/p=\mathcal{O}(1)$ while $n/q\rightarrow\infty$. To address these problems, we prove in Section \ref{sec7}  the following result
by deriving sharp bounds on mixed  moments of a multinomial distribution with parameters $m$ and $(\frac{1}{n},\cdots,\frac{1}{n})$.

\begin{lemma} \label{multi}
Assume 
that $(W_1, \ldots, W_n) \sim \mathcal{M} (m, \frac {1}{n},\ldots,\frac {1}{n})$ and denote 
 $k_m = \lfloor \gamma  \log m  \rfloor $  for some   $\gamma>0$, and assume that  $ m^{1+\gamma/4} \le n $. Then 
there exists a constant $c_\gamma > 0 $ (one may take $c_\gamma=e-2$)
such that  
  $$
  \max_{ k  \leq  k_m   }  ~
   \max_{ \substack{
   s_1,\ldots,s_n \in \N_{0} \\  \sum_{j=1}^ns_j =k} }
   \Big(\frac {n}{m}\Big)^{\sum_{j=1}^n \mathds{1} \{ s_j \geq 1 \} }
  \mathbb{E}[W^{s_1}_1 \ldots W^{s_n}_n ] \leq 
  \Big  ( \ 
1 +   c_\gamma  {m^{1+\gamma}   \over n }  
 \Big )  ^{k_m} . 
  $$ 
\end{lemma}
This yields part (b1) of Theorem \ref{thm: EW1}.
  \item[(2)]  For the smallest eigenvalue we use that $\widehat{\Sigma}_{n}^{*}\succeq \frac1m\sum_{i:\,W_i\ge1}(\Pi_n X_i)(\Pi_n X_i)^{\top}$. Conditionally on the set $\mathcal D=\{i:\,W_i\ge1\}$ of distinct indices in the bootstrap sample, the right-hand side is, up to the factor $|\mathcal D|/m$, a classical sample covariance matrix built from $|\mathcal D|$ iid vectors, and $m-|\mathcal D|\le\sqrt m$ with probability $1-o(m^{-l})$ for every $l\in\N$ if $m^2\le n$. Hence the bound for the smallest eigenvalue of \cite{baiyin1993} in the form (1.9b) of \cite{baisil2004} applies and yields part (b2). Such a conditioning argument is not available for the largest eigenvalue: an index which is drawn $w\ge2$ times contributes a rank-one term whose norm is of the order $wq/m$, and if $n$ grows polynomially in $m$, the probability of such an event is only polynomially small in $m$. This is the reason why the constant $K_{\rm right}$ in (b1) has to depend on $l$.

\end{itemize}

 \begin{corollary}
\label{thm: EW1a}
Grant  assumptions   (A1), (A2), (T)  and $\E X_{11}^4<\infty$.  Let
$$
\widehat\Sigma^*_{n,L_n}=\frac1mL_nX^*X^{*\top}L_n^\top,
$$
which equals $\widehat\Sigma^*_n$ if $R_n=0$. If $R_n\neq0$, the bounds below hold uniformly in the realizations of $\Pi_n$ with $\|R_n\|_{S_2}\le\varepsilon_n$, where $\varepsilon_n\to0$ is a sequence with $\P_{\Pi_n}(\|R_n\|_{S_2}>\varepsilon_n)\to0$ (such a sequence exists, since $\E_{\Pi_n}\|R_n\|_{S_2}^2=o(1)$); for $R_n=0$ no restriction is needed. 

Assume that the   Representative Subpopulation Condition~\ref{def: rsc} is satisfied with $q=\lfloor mp/n\rfloor$.
 Let $c' = \limsup (q'/m)$,  where $q'$ is the number  of  non-zero columns of  the matrix $ L_n $
in the decomposition \eqref{decompproj} of the Representative Subpopulation Condition.
\begin{itemize}
    \item[(a)] 
    If $m= o(\sqrt{n}) $,  then for every $l\in\N$ there exists a constant $K_{\rm  right} =K_{\rm right}(l) >0$
    such that  
\begin{align} \label{ang12b}
\P\Big(\big\Arrowvert \widehat{\Sigma}_{n,L_n}^{*}\big\Arrowvert_{S_{\infty}}>K_{\rm  right} \Big | ~\Pi_n \Big)= o(m^{-l})\ \ \ \text{(for this $l$)}.
\end{align}
If $m = o(\log n) $, then 
\eqref{ang12b} holds even  
for every $K_{\rm  right}> \limsup_{n \in \mathbb{N} } \| \Sigma_n \|_{S_\infty}
(1+\sqrt{c'})^2$ and every $l\in\N$.
    \item[(b)] 
  If $m= o(\sqrt{n}) $, then  we have   for any $K_{\rm left}< \liminf_{n \in \mathbb{N} } \lambda_{\min}(\Sigma_n )  (1-\sqrt{c'})^2 \mathds 1_{(0,1)}(c')$ 
 \begin{align*}
\P\Big(\lambda_{\min}\big(\widehat{\Sigma}_{n,L_n}^{*}\big)<K_{\rm left} \Big | ~\Pi_n \Big)= o(m^{-l})\ \ \ \text{for every $l\in\N$ }.
 \end{align*}
 \end{itemize}
 \end{corollary}

\subsection{Linear spectral statistics}
\label{sec43}
Finally, we are equipped to study   linear spectral statistics\begin{align}\label{eq: AR1003}
\hat{T}_n^*(f)
=\sum_{j=1}^{q}f\big(\hat{\lambda}_j^*\big)
=q\int f(x)\dd
\mu^{\widehat\Sigma_n^*}
(x)
,
\end{align}
where $\hat{\lambda}_1^*,  \dots, \hat{\lambda}_{q}^*$ denote the eigenvalues of the matrix $\widehat\Sigma_n^*$. To keep the technical expenditure as small as possible,  we restrict attention to functions  $f$ which are  analytic
in a region of the complex plane containing the support of $\mu^{\widehat\Sigma_n^*}$  with probability tending to one. As shown in \cite{najimyao2016}, this restriction on $f$  can be 
relaxed in the CLT for classical sample covariance matrices by representing the linear spectral statistic with the help of Helffer–Sj\"ostrand's formula instead of the Cauchy integral formula. Note that a finite fourth moment $\E X_{11}^4<\infty$ is necessary for the CLT on 
\begin{align}
    \label{revd1}
\hat T_n(f)=\sum_{j=1}^{p}f\big(\hat{\lambda}_j\big)
=p\int f(x)\dd
\mu^{\widehat\Sigma_n } (x) 
\end{align}
to hold.

For any given $\Pi_n $ and $j=1,\ldots,m$, let  
$$
T_{n,m}^{(j)} (f, \Pi_n) =q\int f\,\dd\mu^{S_j} ,
$$
where  $S_j=\frac1m\sum_{i=(j-1)m+1}^{jm}\Pi_nY_i(\Pi_nY_i)^\top$  denotes the linear spectral statistic built from the observations $\Pi_n Y_{(j-1)m+1}, \ldots , \Pi_n Y_{jm}$ and we set $T_{n,m}^{(j)}(f,\Pi_n):=0$ if $f$ is not defined at some eigenvalue of the $j$th block (by Corollary~\ref{thm: EW1a} and its analogue for the blocks, this concerns a vanishing fraction of the blocks).
Define their median by 
$$
\overline{T}_{n,m}^{\Pi_n } (f)  :=  \operatorname{med}\big\{T_{n,m}^{(1)}(f,\Pi_n),\dots,T_{n,m}^{(m)}(f,\Pi_n)\big\}.
$$
Here $\operatorname{med}$ denotes the $\lceil m/2\rceil$-th order statistic of the $m$ block statistics. Its computation requires $O\big(m(mq^2+q^3)\big)$ operations, i.e.\ the same order as $m$ bootstrap replications. 

\begin{theorem}[Linear spectral statistics]\label{thm: new}
Grant assumptions (A1) -- (A3+).
Assume that the   Representative Subpopulation Condition~\ref{def: rsc} is satisfied with $q=\lfloor mp/n\rfloor$ and  $\E_{\Pi_n}\|R_n\|_{S_2}^2=o(q^{-1})$.
Let  $f$ be a real-valued function which is analytic  in  a region
of 
 the complex plane containing the interval
$
I= \big [ K_{\rm left} , 
K_{\rm  right}  \big]
$, 
where $K_{\rm left}$ and $K_{\rm  right}$ are the constants in Corollary \ref{thm: EW1a} (with $l=\ell_0$, the fixed integer in \eqref{det33} of the appendix). Furthermore, assume that $m\to\infty$ and $  m= o(\sqrt{n})$. 
Then 
\begin{align} 
&d_{BL}\Big[\LL\Big(\hat{T}_n^*(f)
 - \overline{T}_{n,m}^{\Pi_n} (f)  ~
\Big \arrowvert ~ Y_1,\dots, Y_n , \Pi_n \Big)  , \nonumber\\ &\hspace{30mm}\label{det205a}\LL\Big(\hat{T}_n(f) 
- p\int f\,\dd\mu^0_{p/n,\mu^{\Sigma_n}} - d_n(f) 
\Big)\Big]\longrightarrow_{\P}0 ,
\end{align}
where $d_{BL}$ denotes the dual bounded Lipschitz metric , and, with $\underline m=\underline m^0_{p/n,\mu^{\Sigma_n}}$ and a contour enclosing $I$ in the region of analyticity of $f$,
$$
d_n(f)=-\frac1{2\pi i}\ointctrclockwise f(z)\,\frac{\frac pn\int\underline m(z)^3t^2\big(1+t\underline m(z)\big)^{-3}\dd\mu^{\Sigma_n}(t)}{\Big(1-\frac pn\int\underline m(z)^2t^2\big(1+t\underline m(z)\big)^{-2}\dd\mu^{\Sigma_n}(t)\Big)^2}\,\dd z
$$
is the mean of the limiting normal distribution of \cite{baisil2004}. 
\end{theorem}

\begin{remark}
{\rm
 It is interesting to note that although the ``$(m,mp/n)$ out of $(n,p)$''-bootstrap  consistently mimics the spectral distribution of the sample covariance matrix  if $m=o(n)$, consistently matching expectation and variance in the CLT of linear spectral statistics  in our proof requires $m^2=o(n)$. Again, this prerequisite comes from  moment bounds on non-centered quadratic forms, this time however of uniform type over a specific sequence of curves $({\cal C}_n)_{n \in \mathbb{N}}$ in the complex plane, namely on
\begin{align*}
 & \sup_{z_1,z_2 \in  {\cal C}_n} 
\E  \big\arrowvert Z_1^{*\top} A_{z_1,z_2 }(Z_2^*, \ldots , Z_m^*) Z_1^* -   {\rm tr}  \big( \Pi_n\Sigma_n\Pi_n^{\top}    A_{z_1,z_2 }(Z_2^*, \ldots , Z_m^*) \big)  \big\arrowvert^{\kappa}%
\end{align*}
for $\kappa\geq 2$ (see Proposition \ref{lemma: formel 2.1a}) and
\begin{align*}
  & \sup_{z_1,z_2 \in  {\cal C}_n} 
\E^*  \big\arrowvert Z_1^{*\top} A_{z_1,z_2 }(Z_2^*, \ldots , Z_m^*) Z_1^* -   {\rm tr}  \big( \Pi_n\Sigma_n\Pi_n^{\top}    A_{z_1,z_2 }(Z_2^*, \ldots , Z_m^*) \big)  \big\arrowvert^{{\kappa}}
\end{align*}
 for ${\kappa}=2,4$ (see Proposition~\ref{lemrev5}). 
 As the expressions in there appear in the proof  with an additional factor $q$ as compared to the proof of Theorem \ref{thm: Boot1}, they cause the requirement $m^2= o(n)$, see, for example,  \eqref{det111}, \eqref{det112}, \eqref{det113}, \eqref{det114}   and \eqref{det115}. } 
\end{remark}

 Building on Theorems \ref{thm: Boot1} and  \ref{thm: EW1}, the core of the proof of   Theorem  \ref{thm: new} 
consists in proving a  functional central limit theorem  for (an appropriately  truncated version of) the bootstrap  Stieltjes process conditional  on the original sample in probability (Propositions ~\ref{prop}  and \ref{revprop1} in the appendix) as follows:
\begin{itemize}
    \item[(i)] We formulate and prove a (conditional) bootstrap version of the classical Martingale CLT (Theorem~\ref{thm: bmclt} in the appendix).   
    \item[(ii)] We represent the  centered bootstrap Stieltjes process as a  martingale difference sum (conditional on the original observations and the projection $\Pi_n$) and verify the conditions of Theorem~\ref{thm: bmclt} in (i).     The crux is, however, to prove stochastic convergence of the sum of conditional squared moments  in equation \eqref{eq: 3.10}  -- corresponding to \eqref{eq: cond 1} in Theorem~\ref{thm: bmclt} -- to the {\it right} limit (required for bootstrap consistency).
 \item[(iii)] Given weak convergence of the conditional finite dimensional distributions in probability, we continue with proving 
 conditional tightness in probability  
 in Section \ref{sec953}  of the appendix that is sufficient to deduce the functional central limit theorem for  the bootstrap Stieltjes
process (Proposition~\ref{prop} in the appendix). 
    \item[(iv)] As concerns verification of conditional tightness in probability, we cannot rely our analysis on the  quadratic moment estimates as in the proof of \cite{baisil2004} or  \cite{najimyao2016} of the spectral CLT for high-dimensional   sample covariance matrices because they are evaluated under the conditional distribution in our case and therefore  still random. To this aim, we derive uniform quadratic  moment bounds of Glivenko-Cantelli type on the increments of the  bootstrap Stieltjes process. Their derivation makes essential use of Corollary \ref{thm: EW1a} and the above mentioned Propositions \ref{lemma: formel 2.1a} and \ref{lemrev5}.   
    \item[(v)]  In the same spirit we prove uniform convergence of the (random) conditional expectation of the bootstrap Stieltjes process in probability and derive the explicit limit, see Section \ref{proofrevprop1} for details. 
    \end{itemize}
Rewriting  $f(x)$ on the right-hand side in \eqref{eq: AR1003} by the Cauchy integral formula as complex curve integral and applying Fubini's theorem (see equation \eqref{eq: red 1} in the appendix), the statement of Theorem \ref{thm: new} then follows 
by an application of the continuous mapping theorem. 

 \begin{remark}[Beyond $\E X_{11}^4=3$]\normalfont
    For the mathematical analysis of the new ``$(m,mp/n)$ out of $(n,p)$''  bootstrap, we have restricted attention to random variables $X_{11}$ in model \eqref{eq: Annahme_an_Y}  with $\E X_{11}^4=3$, corresponding to the fourth moment of the $\mathcal{N}(0,1)$  distribution. As clarified in \cite{baisil2004} with formula (1.15), this allows to significantly simplify covariances of quadratic forms, which in our case require a much more sophisticated consideration nevertheless. Relaxing the assumption $\E X_{11}^4=3$ requires primarily a strengthened  version of the Representative Subpopulation Condition \ref{def: rsc}, which  has to guarantee that $\Pi_n\Sigma_n\Pi_n^\top$ mimics all features of $\Sigma_n$ that contribute to the Gaussian approximation of linear spectral statistics. Whereas  solely the spectral distribution enters expectation and variance of the Gaussian approximation under conditions (A1), (A2) and (A3+), properties of the eigenvectors play also a role if $\E X_{11}^4$ is finite but $\not=3$ as pointed out in \cite{najimyao2016}. Specifically, in model \eqref{eq: Annahme_an_Y1}, with  
$$
T_{\Sigma_n}(z)=\bigg(-zI_p+\Big(1-\frac{p}{n}\Big)\Sigma_n-z\frac{p}{n}m_{\frac{p}{n},\,\mu^{\Sigma_n}}^0(z)\Sigma_n\bigg)^{-1}
$$
embodying some kind of deterministic equivalent to the resolvent $(\widehat\Sigma_n-zI_p)^{-1}$, the term
$$
\Theta_{\Sigma_n}(z_1,z_2)= \big(\E X_{11}^4-3\big)\frac{p}{n}\frac{1}{p}\sum_{i=1}^p\frac{\partial}{\partial z_1}\big( z_1T_{\Sigma_n}(z_1)\big)_{ii}\frac{\partial}{\partial z_2}\big(z_2 T_{\Sigma_n}(z_2)\big)_{ii}
$$
enters the approximating covariance structure of the Stieltjes process, which not only depends on the spectrum of $\Sigma_n$ but also on its eigenvectors. (In the general model \eqref{eq: Annahme_an_Y} the corresponding term depends on $A_n$ and not only on $\Sigma_n$, and the CLT of \cite{zou2022clt} requires additional structure of $A_n$ if $\E X_{11}^4\neq3$.) Likewise, an additional term $\Gamma_{\Sigma_n}$ involving both, eigenvectors and spectrum, enters the expectation if $\E X_{11}^4\not=3$. Thus, in addition to the spectral similarity condition \eqref{eq: similarity3}, the approximations 
\begin{equation*}%
\Arrowvert \Theta_{\Sigma_n}-\Theta_{\Pi_n\Sigma_n\Pi_n^\top}\Arrowvert_{\mathcal{C}_n\times\mathcal{C}_n}\longrightarrow 0\quad\text{and}\quad \Arrowvert \Gamma_{\Sigma_n}- \Gamma_{\Pi_n\Sigma_n\Pi_n^\top}\Arrowvert_{\mathcal{C}_n}\longrightarrow 0
\end{equation*}
as $n\rightarrow\infty$ then necessarily enter the Representative Subpopulation Condition. This being granted  we expect that \eqref{det205a} continues to hold.
\end{remark}

\section{An application to hypotheses testing} \label{sec6}
  \def\theequation{5.\arabic{equation}}	
	\setcounter{equation}{0}
In this section we study the performance  of the ``$(m,mp/n)$ out of $(n,p)$'' bootstrap for LSS in the context of hypotheses testing. More specifically, we are interested in the problem of testing if  the covariance matrix $\Sigma$ from a sample of iid $p$-dimensional random vectors is the identity matrix, that is 
\begin{align}
    \label{data1}
    H_0: \Sigma_n=I_{p} ~ \text{ versus } ~ H_1: \Sigma_n\not = I_{p}
\end{align}
Several tests have been suggested for this problem, and for the sake of brevity we consider the statistic $\hat T_n=p\,V$, where $V=\frac1p\tr\big((\widehat\Sigma_n-I_p)^2\big)$ is the statistic of \cite{nagao1973}. If $p/n\to c$, its classical null distribution is no longer valid: for normally distributed data, \cite{ledoitwolf2022} show that $nV-p\to\mathcal N(1,4+8c)$ under $H_0$ and propose a corrected statistic. Here, $\hat T_n$ is calibrated by the bootstrap, and the null hypothesis is rejected for large values of
\begin{align}
    \label{dsim2}
    \hat T_n = {\rm tr} \big ( \big ( \widehat \Sigma_n - I_p \big )^2\big )  =  {\rm tr} \ \big ( \widehat \Sigma_n ^2\big ) -2 {\rm tr} \ \big ( \widehat \Sigma_n \big ) + p ,
\end{align}
and corresponds to the choice $f(x) = x^2-2x+1$ in \eqref{revd1}.   We have implemented a bootstrap test based on the approximation \eqref{det205a}, where under the null hypothesis, that is  $\Sigma=I_p$, the centering term for $\hat T_n$ is given by 
$$
p \int f (x) d \mu_{p/n,\delta_{\{1 \}}}^0 =  p \Big ( {p \over n } +1 \Big )  -2 p  + p = {p^2 \over n }.
$$ 
Moreover, for this choice of $f$ and $\Sigma_n=I_p$ the mean correction in Theorem \ref{thm: new} is $d_n(f)=p/n$. Therefore, the resulting test is given by  
\begin{align}
    \label{test}
    \hat T_n - {p^2 \over n } {}- {p \over n} >  q_{1-\alpha}^*~, 
\end{align}
where $q_{1-\alpha}^*$ denotes the $(1-\alpha)$-quantile of the bootstrap  distribution of $\hat{T}_{m,n}^*(f) - \overline{T}_{n,m}^{\Pi_n } (f)$.  We  now make the dependence of the bootstrap statistic on the sample size $m$ explicit in the notation $\hat{T}_{m,n}^*(f)$ and 
 use a  random coordinate projection picking  $q$ out of $p$ components uniformly at random.
To demonstrate that the ``$(m,mp/n)$ out of $(n,p)$'' bootstrap  is computationally feasible even for large sample sizes, we  choose $n=80 000$ and the dimension $p=20000$.

For the choice of the size $m$ of the subsample in the ``$(m,mp/n)$ out of $(n,p)$ bootstrap''  we  consider   two methods. The first is  an  adaptive rule introduced by \cite{BickelSakov2008}
that consists of the following steps.
\begin{enumerate}
    \item[(i)] Fix $K \in \mathbb{N}$. For each \( j = 0, 1, 2, \ldots , K \), define \( m_j = \lceil \psi^j n \rceil \), where \( \psi  \in (0,1) \) is some fixed parameter, and let \( {\cal L}^*_j \) denote the distribution of  the bootstrap  statistic \( \hat T^*_{m_j,n} \) conditional on the observed sample \(  Y_1,  \ldots,  Y_n \).
    
    \item[(ii)] For some metric \( d \) consistent with weak convergence, let \( J \) be the smallest \( j \) that minimizes the distance \(  d({\cal L}^*_j, {\cal L}^*_{j+1}) \).
       \item[(iii)] Use \( {\cal L}^*_J \) as the bootstrap approximation.
\end{enumerate}
An alternative adaptive method for choosing $m$ was recently introduced by \cite{dettekroll}, who proposed to replace step (ii)  in this procedure  and works as follows:
\begin{enumerate}
    \item[(ii')] For some metric $d $ consistent with weak convergence let $J$ denote the smallest $j$  minimizing the sum of distances $\sum_{k=1}^K d ({\cal L}_j^*, {\cal L}_{k}^* )$
\end{enumerate}
In the following discussion we compare both methods for the  calculation of  the quantiles of  the test \eqref{test} by  the ``$(m,mp/n)$ out of $(n,p)$  bootstrap'', where we use the Kolmogorov  distance as metric $d$ and choose  $\psi=0.75$ and $K=30$ (starting in step (i) with $j=20$) .

We first generate   $p+1$-dimensional vectors $X_i=(X_{i,1}, \ldots , X_{i,p+1})^\top$ with independent identically distributed entries.  From these vectors we calculate the data $Y_i=(Y_{i1}, \ldots , Y_{ip})^\top$, where  for $r \in \{0, 1, \ldots , p \} $
\begin{equation} \label{simmod1}
 Y_{i,j} = 
 \begin{cases}
     a X_{i,j+1} + b X_{ij} & \text {  if  } j=1 , \ldots  ,   r  \\ 
     X_{i,j+1}  & \text {  if  } j= r
     +1 , \ldots , p
 \end{cases}   
\end{equation}
and the constants $a \geq 0 $ and $b \geq 0 $ are chosen such that   $\text{Var} (Y_{i,j} )= a^2+b^2 =1$  and $\text{Cov} (Y_{i,j-1}, Y_{i,j} ) = \rho$, that is
$$
a^2= \frac{1+\sqrt{1-4\rho^2}}{2}, \quad  b^2 = \frac{\rho^2}{a^2} = \frac{1-\sqrt{1-4\rho^2}}{2}.
$$
This means that the population covariance matrix $\Sigma_n$ is a tri-diagonal matrix with diagonal elements given by  $1$ and $r-1$ entries equal to $\rho $ on the first off-diagonal (all other elements are $0$).  Note that the case $r=0 $ corresponds to the null hypothesis in \eqref{data1}. We choose $\rho=0.05$ and for the distribution of the random variables $X_{ij}$ we use a standard normal distribution and a $\chi^2$ with $20$ degrees of freedom normalized such that it has expectation $0$ and variance $1$. Note that in  this case $\E [X_{ij}^4] = {18 \over 5}  \not = 3$. 

The simulated rejection probabilities of the test \eqref{test}, which have been calculated by $1000$  simulation runs and $500$ bootstrap replications,   are displayed in Table \ref{tab1} for different values of $r$, where the nominal level is $\alpha=0.05$. The two selection rules of \cite{BickelSakov2008}
 and \cite{dettekroll} for choosing  the size $m$ of the subsample yield  very similar results. From the left  part of Table~\ref{tab1} we observe a good approximation of the nominal level by the bootstrap test \eqref{test} for   the normal distribution.  The test also exhibits reasonable power under the alternative. For example, if {$r=200$ (i.e.\ $r/p=1\%$), the test rejects in $44.9\%$ of the cases if the method \cite{BickelSakov2008} is used and in $44.2\%$ of the cases if  the method \cite{dettekroll} is used to choose $m$. 

In the right part of Table \ref{tab1} we show the corresponding results  for the $\chi^2$-distribution. The results are qualitatively the same with a slightly worse approximation of the nominal level under the null hypothesis by the method of \cite{dettekroll} (and as a consequence slightly lower power). In particular the bootstrap tests  exhibit some robustness against the violation of the assumption $\E [ X_{ij}^4]=3$.

\begin{table}[h]
\begin{center}
\begin{tabular}{c|ccccc||ccccc}
 & \multicolumn{5}{c||}{normal distribution} & \multicolumn{5}{c}{$\chi^2$-distribution} \\$r/p$ & 0\% & 1\% & 2.5\% & 5\% & 10\% &   0\% & 1\% & 2.5\% & 5\% & 10\% \\
\hline
BS & 0.039 & 0.449 & 0.984 & 1.000 &  1.000&  0.037 & 0.361 &  0.974&  1.000&  1.000  \\
DK     & 0.033 & 0.442 &  0.984 & 1.000 &  1.000 & 0.036  & 0.364  & 0.964  & 1.000  & 1.000    
\end{tabular}
\end{center}

\caption{Empirical rejection probabilities of the bootstrap test \eqref{test}. The size $m$ of the subsample in the ``$(m,mp/n)$ out of $(n,p)$'' bootstrap was chosen by the method of \cite{BickelSakov2008} (BS) and the method of \cite{dettekroll} (DK) and $r/p$ represents the proportion of elements in the first off-diagonal equal to $\rho=0.05\not=0$.
Left part: standard normal distribution. Right part: standardized $\chi^2$-distribution. \label{tab1}}
\end{table}

To our best knowledge, there are two high-dimensional bootstrap methods for linear spectral statistics, which could be considered in a comparison with our approach, namely the ``parametric'' bootstrap proposed by \cite{lopblaaue2019} and   its extension by \cite{wanglopes2022} to elliptical models. As both methods are very similar in spirit and have the same  high computational complexity (see the discussion below), we restrict ourselves to a comparison with the method in \cite{lopblaaue2019}, which is applicable without the assumption of an elliptical model. The following table contrasts the conditions under which bootstrap validity is proved.

\begin{table}[h!]
\begin{center}
\begin{tabular}{|c|c|}
\hline
\cite{lopblaaue2019} & ``$(m,mp/n)$ out of $(n,p)$''  \\
\hline  \hline
$p/n\rightarrow c\in (0,\infty)\setminus\{1\}$& $p/n\rightarrow c>0$\\
\hline
$\mathbb{E} X_{11}=0$, $\mathbb{E} X_{11}^2=1$ &$\mathbb{E} X_{11}=0$, $\mathbb{E} X_{11}^2=1$ \\
\hline
 $\phantom{\overset{,}{I}}\mathbb{E} X_{11}^8<\infty$&\\
 + regularity condition on&$\mathbb{E}X_{11}^4=3$\\
 eigenvectors if  $\mathbb{E}X_{11}^4\not=3$ &\\
\hline
$\limsup_n\Arrowvert\Sigma_n\Arrowvert_{S_{\infty}}<\infty$&$\limsup_n\Arrowvert\Sigma_n\Arrowvert_{S_{\infty}}<\infty$\\
\hline
$\textrm{supp}(H)=$ finite union & Representative\\
of closed intervals, & Subpopulation Condition \\
bounded away from $0$ and $\infty$ & \\
\hline
$f\in C^3$ & $f$ analytic\\
\hline
\end{tabular}
\end{center}
    \caption{Assumptions under which Bootstrap validity is guaranteed.} 
\end{table} 

It is not possible to run simulations for the procedure in \cite{lopblaaue2019} with the sample sizes considered in Table \ref{tab1} (even on an HPC  cluster), because their algorithm requires $O\big ( (np^2 + p^3) (B+1) \big ) $ operations. In contrast, our method   scales with  $$
O\big ( (np^2 + p^3)   \big ) + O\big ( (mq^2 + q^3) B   \big )   = O\big ( (np^2 + p^3)   \big )
$$
operations (note that $m = o(\sqrt{n})$). To compare both methods, we consider the model \eqref{simmod1} with normally distributed entries,  smaller sample size  $n=10000$ and  dimension  $p=5000$, where we choose $\rho=0.1$. The rejection probabilities of the ``$(m,mp/n)$ out of $(n,p)$'' bootstrap  proposed in this paper and the test of \cite{lopblaaue2019} are  displayed in Table \ref{tab2}. Both procedures provide a reasonable approximation of the nominal level, while the test of \cite{lopblaaue2019} has slightly larger power. In the right column of the table we show the computation time (in seconds) of both procedures for {\bf one} simulation run. We observe that the ``$(m,mp/n)$ out of $(n,p)$'' yields substantial computational savings  although it uses several values of $m$ to identify the appropriate size of the subsample by the method in \cite{dettekroll}. We emphasize that we used a computation faster version of the bootstrap test of \cite{lopblaaue2019}, where we implemented the statistic \eqref{dsim2} directly instead of its eigenvalue version $ \sum_{i=1}^p \hat \lambda_i^2 - 2 \sum_{i=1}^p \hat \lambda_i + p $. If one uses this version,  one run of the bootstrap of \cite{lopblaaue2019} takes about $17.500$ seconds ($\approx 4.9$ hours), while the ``$(m,mp/n)$ out of $(n,p)$'' bootstrap needs $39$ seconds, if it calculates the eigenvalues explicitly.
 
\begin{table}[h]
\begin{center}
\begin{tabular}{c|ccccc|c|}
 $r/p$ & 0\% & 1\% & 2.5\% & 5\% & 10\% & time \\
\hline
\cite{lopblaaue2019} & 0.058 & 0.181 & 0.546 & 0.964 &  1.000 & 4938.2 \\
``$(m,mp/n)$ out of $(n,p)$''     & 0.054 & 0.161 &  0.470 & 0.957 &  1.000 & 11.4  
\end{tabular}
\end{center}
\caption{Empirical rejection probabilities of the  test \eqref{test}, where  the  quantile is obtained by the  bootstrap procedure proposed in  \cite{lopblaaue2019} and the   ``$(m,mp/n)$ out of $(n,p)$'' bootstrap proposed in this paper  (with $m$ chosen by the  method in \cite{dettekroll}). The sample size is $n=10000$, the dimension is $p=5000$ and 
$r/p$ represents the proportion of elements in the first off-diagonal equal to $\rho=0.1\not=0$. The right column in the table shows the computation time (in seconds) of {\bf one} simulation run.  }
\label{tab2}
\end{table}

\section{Conclusion} \label{sec6a}
  \def\theequation{6.\arabic{equation}}	
	\setcounter{equation}{0}
 
Thinking of the $p$ components $Y_{i,1},\dots Y_{i,p}$ of each observation vector $Y_i$ as data of the same $p$ individuals, our approach originates from the idea  of  selecting a subpopulation which is representative for the full population concerning  the statistics of interest --    here the spectral distribution. A suitable selection strategy implements prior knowledge or rather a structural model assumption on the state of nature, i.e.~the data generating process. Building on the so-called  Representative Subpopulation Condition, we have then introduced a fully nonparametric and computationally tractable bootstrap of  high-dimensional sample covariance matrices. This ``$(m,mp/n)$ out of 
$(n,p)$'' bootstrap   provably possesses  desirable consistency properties,  which we have exemplarily demonstrated for estimating the spectral distribution itself and for linear spectral statistics.  Besides obvious technical extensions of studying LSS' under less restrictive circumstances, let us conclude with two essential open problems which are left for future work: 
\begin{itemize}
\item[(i)]Our results on the extremal eigenvalues prompt the question whether the approach may even be successful for distributional approximation of the largest eigenvalue. Here, the particularly interesting feature is the phase transition in its limiting behavior, depending on whether some suitably separated spike in the population covariance matrix is present or not, see
\cite{10.1214/009117905000000233} for the complex  and \cite{Paul2007} for the  real Gaussian case. Although our current mathematical formalization of the Representative Sub\-population Condition  is insensitive for individual eigenvalues, it is worth being investigated if the  ``$(m,mp/n)$ out of $(n,p)$'' bootstrap is successful under this condition when  there are no spikes in the population  spectrum. 

\item[(ii)] Our proofs require the upper bounds $m=o(n)$ (for spectrum consistency) and $m^2=o(n)$ (for LSS' consistency), respectively; whether these conditions are necessary remains open. Moreover, our theoretical results do not provide any guidance on how to choose $m$ in an optimal way so far. Even  if the underlying population covariance matrix is a multiple of the identity such that there is no extra bias in the population spectrum by moving to a subpopulation, the optimal  choice of $m$ is a challenging open problem. The reason is that its investigation  requires sharp quantitative bounds on the distance between the conditional bootstrap and the original distribution, which we have derived so far only in parts.

\end{itemize}

\appendix

\section{Preliminaries and further notation}
  \def\theequation{A.\arabic{equation}}	
	\setcounter{equation}{0}
\label{sec4} For technical convenience,  we will assume throughout this appendix that 
the sequence of spectral measures $(\mu^{\Sigma_{n}})$ is weakly convergent, that is,
\begin{align} \label{det108}
\mu^{\Sigma_{n}}\Rightarrow H\ \ \text{as}\ n \rightarrow\infty
\end{align}
 for some  limiting distribution $H$.
 Note that this does not impose any further restriction on our results, because
  Assumption (A1) implies 
tightness of the sequence 
$(\mu^{\Sigma_{n}})$  such that we can restrict attention in all proofs to weakly convergent subsequences anyway.  

\medskip

Recall that our model \eqref{eq: Annahme_an_Y} extends the classical setting with observation vectors $Y_i=\Sigma_n^{1/2}X_i$, $i=1,\dots, n$. The next lemma states  that the well-known MP-limit for the spectrum of the sample covariance matrix remains valid in model \eqref{eq: Annahme_an_Y}.

\begin{lemma} 
\label{prop1} 
Grant assumptions (A2) and (A3). Let $\Sigma_n=A_nA_n^\top $ and $\sup_{n \in \mathbb{N} } \| A_n\|_{S_\infty} < \infty$, then
$$
\mu^{\widehat\Sigma_n}-\mu^{0}_{\Sigma_n}\Rightarrow 0\ \ \ \text{almost surely,}
$$
where $\mu^{0}_{\Sigma_n}$ is the measure corresponding to the solution of the 
MP-equation \eqref{hd10}  for $\gamma =p/n $ and $H=\mu^{\Sigma_n}$.
\end{lemma}

\begin{proof}
Assume first that the matrix $A_n$ has $k < \infty$ non-vanishing columns, let $A_{n}=(a_1\ldots, a_k)= UDV^\top  \in \mathbb{R}^{p \times k} $ be  the  singular value decomposition  of  the matrix $A_{n}$ and let $e_i\in \mathbb{R}^k$ 
 denote 
 the  $i$th unit vector, then we obtain for the $i$th column
$    a_i=A_{n}e_i =UDV^\top e_i $   of $A_{n}$. This implies   
\begin{align*}
    \| a_i\| ^2
    = e_i^\top  ~VDU^\top ~ UDV^\top e_i
    = e^\top _i V D^2 V^\top  e_i = \sum_{j} d^2_j v^2_{ij}~,
\end{align*}
where $(v_{i1}, v_{i2}, \ldots)^\top $ is the $i$th  row  of the matrix  $V$ and $d_{1} \geq  d_{2}  \geq \ldots  \geq 0 $  are the singular values of the matrix $A_{n}$.
Note that  $\sum^p_{j=1} v_{ij}^2 \le 1$, which implies  $  \| a_i\| ^2 \leq d_{1}^2$.
Consequently, if   $d_1 = \|  A_{n}\| _{S_\infty} < \infty $  it follows that  
\begin{align*}
    \max^k_{i=1}\| a _i\|  \leq d_{1} < \infty. 
\end{align*}
This statement is also correct in the case $k=\infty$ (just use a truncation and a limiting argument). 
With this inequality 
we obtain for  any  $\eta >0$  
\begin{align*}
   &  \frac{1}{np\eta^2}~\sum^k_{i=1} \sum^n_{j=1} \|  a_i \| ^2_2~\mathbb{E}\left[|X_{ij}|^2I\{\| a_i\|  | X_{ij}|>\sqrt{n}\eta\} \right] \\
    &  \qquad \quad \leq \frac{1}{p\eta^2} \sum^k_{i=1} \| a_i\| ^2 \mathbb{E} \left[ |X_{11}|^2 I\{\| A_n \| _{S_\infty} | X_{11}|>\sqrt{n} \eta\} \right]\\
    & \qquad \quad  = \frac{1}{\eta^2} \frac{1}{p} \| A_n\| _{S_2}^2 ~o(1) = o(1), 
\end{align*}
where we have used the fact that $p/n\rightarrow c$. Consequently 
the Lindeberg-type condition  in \cite{zou2022clt}  is satisfied, and the assertion follows from their Theorem~1.
\end{proof}

\subsection{Further notation} \label{sec51} 
In the following discussion, $\mathbb{E}$
denotes the expectation, $\mathbb{E}_X$
the expectation with respect to $X_1, \ldots , X_n$,  $\mathbb{E}_{\Pi_n} $
with respect to $\Pi_n $ (note that $X$ and $\Pi_n$ are independent),  
$\mathbb{E}^*$ denotes the expectation with respect to the (random) measure 
$\hat{\mathbb{P}}_n^{\otimes m} $ with $\hat\P_n=\frac1n\sum_{k=1}^n\delta_{X_k}$ (conditionally on $\Pi_n$) and   
$\mathbb{E}^*_j$ is the conditional expectation 
 operator corresponding to $\hat{\P}_n$ with respect to the $\sigma$-field generated by $X_1^*,\dots, X_j^*$.
Subsequently, $q$ equals $\lfloor m p/n\rfloor$. We will also frequently  make use of the abbreviations 
\begin{align}
r_j^*&=\frac{1}{\sqrt{m}} L_n X_{j}^*\in \mathbb{R}^q\\
D^*(z)&=\sum_{j=1}^mr_j^*r_j^{*^\top}-zI_q \in \mathbb{R}^{q \times q}
\label{ang2} \\
D_j^*(z)&=D^*(z)-r_j^*r_j^{*^\top}\\
\label{ang3}
\varepsilon_j^*(z) &=r_j^{*^\top}{D_j^{*}}(z)^{-1}r_j^*-\frac{1}{m}\tr\Big( L_n^\top {D_j^*}(z)^{-1}L_n\Big)\\
{\delta_j^*(z)}
&=r_j^{*^\top}{D_j^*}(z)^{-2}r_j^*-\frac{1}{m}\tr\Big( L_n^{^\top}{D_j^*}(z)^{-2} L_n\Big)=\frac{\dd}{\dd z}\varepsilon_j^*(z)\\
\label{ar11a}
\beta_j^*(z)&=\frac{1}{1+r_j^{*^\top}{D_j^*}(z)^{-1}r_j^*}\\
\bar{\beta}_j^*(z)&=\frac{1}{1+m^{-1}\tr\big(L_n L_n^{^\top}{D_j^*}(z)^{-1}\big)}\\
b_n^*(z)&=\frac{1}{1+m^{-1}\E^*\tr\big(
L_nL_n^\top{D_1^*}(z)^{-1}\big)}.
\\
\label{ar11}
D_{ij}^*(z)&=D_j^*(z)-r_i^* {r_i^{*}}^\top \\
\label{ar20}
\beta_{ij}^*(z)&=\frac{1}{1+ {r_i^{*}}^\top D_{ij}^*(z)^{-1}r_i^*}\\
\label{ar13}
b_{1}^*(z)
&{=\frac{1}{1+m^{-1}\E^*\tr \big( L_n L_n^\top {D_{12}^*}(z)^{-1}\big)}} \\
\label{ang3a}
\varepsilon_{ij}^*(z) &=r_j^{*^\top}{D_{ij}^{*}}(z)^{-1}r_j^*-\frac{1}{m}\tr\Big( L_n^\top {D_{ij}^*}(z)^{-1}L_n\Big) \\ 
\label{ar20a}
\bar \beta_{ij}& =\frac{1}{1+m^{-1}\tr\big(L_n L_n^\top{D_{ij}^*}(z)^{-1}\big)} \\
\label{gamma1} 
\gamma_j^* (z) &=   {1 \over \beta_j^* (z)}  - {1 \over b_n^*(z)}  =  r_j^{*^\top}{D_j^*}(z)^{-1}r_j^* -  {1 \over m} \E^*\tr\big(
L_nL_n^\top{D_j^*}(z)^{-1}\big) 
\end{align} 

\medskip
For any real-valued bounded function $f$, its supremum norm is denoted by $\Arrowvert f\Arrowvert_{\sup}$. If $f$ is defined on some metric space $(X,d_X)$ and Lipschitz in addition, then its bounded-Lipschitz norm is defined as
$\Arrowvert f\Arrowvert_{BL}=\max\big(\Arrowvert f\Arrowvert_{\sup},\Arrowvert f\Arrowvert_L\big)$ with
$$
\Arrowvert f\Arrowvert_L:=\sup_{x\not= y}\frac{\arrowvert f(x)-f(y)\arrowvert }{d_X(x,y)}.
$$
Correspondingly, we write $BL:=\big\{f:X\rightarrow \R\big\arrowvert\, \Arrowvert f\Arrowvert_{BL}<\infty\big\}$ for the  space of bounded Lipschitz functions. With slight abuse of notation, $d_{BL}$ denotes the dual bounded-Lipschitz metric on the space of probability measures on $(X,\mathcal{B}(\mathcal{X}))$, i.e.
\begin{align}\label{eq: dBL}
    d_{BL}(\mu,\nu)&:=\sup\bigg\{\int fd\mu-\int fd\nu:\ \Arrowvert f\Arrowvert_{BL}\leq 1\bigg\}.
    \end{align}
    If $(X,d_X)$ is separable, then $d_{BL}$ metrizes weak convergence for probability measures on $(X,\mathcal{B}(\mathcal{X}))$. 
    On the space of probability measures on $(\R,\mathcal{B}(\R))$, recall  furthermore    the Kolmogorov metric $d_K$ and the L\'evy metric $d_L$, given by
\begin{align*}
d_K(\mu,\nu)&:=\big\Arrowvert \mu((-\infty,\cdot])-\nu((-\infty,\cdot])\big\Arrowvert_{\sup}
\end{align*} 
and    \begin{align*}
d_L(\mu,\nu)&:= \inf \Big\{\varepsilon>0\Big\arrowvert \mu((-\infty,x-\varepsilon])-\varepsilon\leq \nu((-\infty,x])\\
&\hspace{25mm}\leq\mu((-\infty,x+\varepsilon])+\varepsilon\ \text{for all }x\in\R\Big\},
    \end{align*} 
respectively. We will frequently make use of the well-known relation $d_L\leq d_K$. Finally, $C$ and $K$ denote numerical constants which do not depend on the variable parameters in the respective expressions unless explicitly indicated. Their value may change from line to line.

\section{Proof of  Theorem \ref{thm: Boot1}} 
\label{sec5} 
  \def\theequation{B.\arabic{equation}}	
	\setcounter{equation}{0}

\subsection{Reduction to $L_n$} \label{sec52} 
In this subsection we will prove that  we  can replace the matrix $\Pi_n A_n$ in \eqref{decompproj}
by the matrix $L_n$. Moreover, we also show that we can restrict ourselves to centered and standardized random vectors with uniformly bounded iid components.

With
$\widehat \Sigma_{n,L_n}^*:=m^{-1} L_n X^*X^{*\top} L_n^\top $, where $ L_n X^* = (L_nX^*_1, \ldots ,  L_nX^*_m) \in \mathbb{R}^{q \times m} $
we shall prove that
\begin{equation}
    \label{hd65}
d_{BL}\left(\mu^{\widehat \Sigma_n^*},\mu^{\widehat \Sigma_{n,L_n}^*}\right)=  o_{\mathbb{P}} (1). 
\end{equation}
To this aim, note first that by the definition of the dual bounded Lipschitz metric and inequality (1.2) in \cite{limath1999}, 
$$
d_{BL}\left(\mu^{\widehat \Sigma_n^*},\mu^{\widehat \Sigma_{n,L}^*}\right)\leq\frac{1}{q}\left\|\frac{1}{m}
L_nX^*X^{*\top} L_n^\top - \widehat \Sigma_n^* \right\|_{S_1}.
$$
Next, 
\begin{eqnarray*}
  \left\|  \frac{1}{m} L_n X^*X^{*\top} L_n^\top  - \widehat \Sigma_n^* \right\|_{S_1}  \leq &  A_1 + A_2 + A_3
\end{eqnarray*}
with
\begin{eqnarray*}
  A_1 &=& \frac {1}{m} \big \|  L_n X^* X^{*\top}  R_n^\top  \big \|_{S_1} \\
  A_2 &=& \frac {1}{m} \big \|  R_n X^* X^{*\top} L^\top _n \big \|_{S_1} \\
  A_3 &=& \frac {1}{m} \big \|  R_n X^* X^{*\top}  R_n^\top  \big \|_{S_1}   
\end{eqnarray*}
and $L_n$ and $R_n$ are defined by the decomposition \eqref{decompproj}.
We now show that $A_1$, $A_2$ and $A_3$ are of order $o_{\mathbb{P}}(q)$ starting with $A_3$. By the Cauchy-Schwarz inequality for Schatten norms,
\begin{align*}
    \frac{1}{m}\E\Arrowvert  R_n X^* X^{*\top}  R_n^\top \Arrowvert_{S_1}&\leq \frac{1}{m}\E\Big(\Arrowvert R_n X^*\Arrowvert_{S_2}\Arrowvert X^{*\top} R_n^\top \Arrowvert_{S_2}\Big)
   \\&
    =\frac{1}{m}\E\tr\left( R_n X^* X^{*\top}  R_n^\top \right)
    = \E_{\Pi_n} \big [
    \Arrowvert  R_n\Arrowvert_{S_2}^2 
    \big ] = o(q),
\end{align*}
where we used Condition  \ref{def: rsc}.
As concerns $A_1$ and $A_2$, we obtain similarly
\begin{align*}
    \frac{1}{m}\E\Arrowvert  R_n X^*X^{*\top}  L^\top _n\Arrowvert_{S_1}&\leq \frac{1}{m}\E\Big(\Arrowvert  R_n  X^*\Arrowvert_{S_2}\Arrowvert X^{*\top}  L^\top _n\Arrowvert_{S_2}\Big)\\
    &\le \frac{1}{m}\E^{1/2}\tr\left( R_n  X^* X^{*\top}  R_n^\top \right)\E^{1/2}\tr(  L_nX^*X^{*\top}  L_n^\top )\\
    &= 
    \E_{\Pi_n}^{1/2} \big [ \Arrowvert  R_n\Arrowvert_{S_2}^2 \big ]
    \E_{\Pi_n}^{1/2} \big [
    \Arrowvert  L_n\Arrowvert_{S_2}^2
    \big ] = o(q). 
\end{align*}
Here, we used (A1) and the Representative  Subpopulation Condition \ref{def: rsc} to get
$$
\sup_{n\in\N}\E_{\Pi_n} 
    \Arrowvert  L_n\Arrowvert_{S_2}^2\leq q\cdot\sup_{n\in\N} \E_{\Pi_n} 
    \Arrowvert  \Pi_n A_n\Arrowvert_{S_\infty}^2 \leq q\cdot\sup_n\Arrowvert A_n\Arrowvert_{\S_{\infty}}^2=\mathcal{O}(q).  
    $$
Combining these estimates yields
\eqref{hd65}. Moreover, since $\E_{\Pi_n}\|R_n\|_{S_2}^2=o(1)$, the event $\mathcal R_n=\{\|R_n\|_{S_2}\le1\}$ has probability tending to one, and on $\mathcal R_n$ we have 
$$
\|L_n\|_{S_\infty}\le\|\Pi_nA_n\|_{S_\infty}+\|R_n\|_{S_2}\le\sup_n\|A_n\|_{S_\infty}+1=:\alpha.
$$
As all assertions are in probability, we may and do work on $\mathcal R_n$ in the sequel. Finally, 
\begin{align*}
d_{BL}(\mu^{L_nL_n^\top},\mu^{\Pi_n\Sigma_n\Pi_n^\top})& \le\frac1q\|L_nL_n^\top-\Pi_nA_nA_n^\top\Pi_n^\top\|_{S_1} \\
& \le\frac1q\big(2\|L_n\|_{S_2}\|R_n\|_{S_2}+\|R_n\|_{S_2}^2\big)=o_\P(q^{-1/2}),
\end{align*}
so that \eqref{eq: similarity3} remains valid with $L_nL_n^\top$ in place of $\Pi_n\Sigma_n\Pi_n^\top$. This justifies \eqref{hol1} below.
Therefore, we will assume in the following discussion  
that, given the random projection $\Pi_n$,
\begin{equation} \label{hol1}
\tilde{\Sigma}_n=\Pi_n \Sigma_n\Pi_n^\top 
=  L_n L_n^\top 
\end{equation}
and correspondingly 
$$
\widehat{\Sigma}_{n}^{*}  = 
{1 \over m } L_nX^*X^{*\top} L_n^\top ~,
$$
 where $L_n$  is a $q \times q'$ matrix 
 satisfying $\| L_n \|_{S_\infty} \leq \alpha < \infty $ (on $\mathcal R_n$)
 and $X^* $ is an $  q'\times m$ matrix and $q' =  O(q)$.
 Moreover, without loss of generality, we assume in the following discussion that the corresponding matrix $X$ is of dimension $ q'\times n $  and work conditionally on the projection $\Pi_n$. 
 
\subsection{Reduction to  uniformly bounded iid components} \label{sec52aa} 
Note that arguments in this section  do not depend on the projection matrix  $\Pi_n$.
We  now show that without loss of generality, we may assume that the random variables $X_{ij}$ are centered, standardized and bounded.  
To this aim, we will prove in what follows that
\begin{equation}\label{eq: truncation}
 \limsup_{p\rightarrow\infty} \E^* d_L^2 \left(\mu^{\widehat{\Sigma}_{n}^{*}},\mu^{\check\Sigma_n^{*K}}\right)\longrightarrow 0\ \ \text{in probability}
\end{equation}
as $K\rightarrow\infty$, where  $\check\Sigma_n^{*K}=m^{-1}  L_n \check X^*\check X^{*\top} L_n^\top $ is built from the truncated, centered and standardized random variables 
$$
\check{X}_{ij}=\check{X}_{ij}(K)=\frac{X_{ij}\mathds{1}\{\arrowvert X_{{ij}}\arrowvert \leq K\}-\E X_{ij}\mathds{1}
\{\arrowvert X_{{ij}}\arrowvert \leq K\}}{\sqrt{\mathrm{Var} ( X_{ij}\mathds{1}
\{\arrowvert X_{{ij}}\arrowvert \leq K\} ) }}
$$
 for an arbitrary constant $K>0$. 
This is sufficient as it will turn out that the weak limit of $\mu^{\check \Sigma_n^{*K}}$ in probability does not depend on $K$. 
Define the matrices $\tilde X=\tilde X(K)$ and
$\tilde\Sigma_{n}^{*K}=m^{-1}  L_n\tilde X^* \tilde  X^{*\top} L_n^\top  $, where
$$
\tilde X_{ij}=X_{ij}\mathds{1}
\{\arrowvert X_{{ij}}\arrowvert \leq K\}.
$$
Next, define for any $\delta>0$ the event
$$
\Delta_{j,p,n}=\Big\{\frac{1}{n}\Big\arrowvert \sum_{i=1}^nX_{ij}^2-\E X_{ij}^2\Big\arrowvert\vee \frac{1}{n}\Big\arrowvert \sum_{i=1}^n
\Big ( X_{ij}^2\mathds{1}\{\arrowvert X_{ij}\arrowvert>K\}-\E X_{ij}^2\mathds{1}\{\arrowvert X_{ij}\arrowvert>K\} \Big ) \Big\arrowvert <\delta\Big\}.
$$
With this notation, we introduce the Hermitian matrices
$$
\hat\Sigma_{n}^{*'}=\frac{1}{m}  L_n{X^*}'
{X^*}^{'\top }
 L_n^\top \ \ \text{and}\ \ \Sigma_{n}^{*'}=\frac{1}{m} L_n\tilde {X^*}'
{\tilde {X^*}}^{'\top }
L_n^\top 
$$
with
$
X'_{ij}=X_{ij}\mathds{1}_{\Delta_{j,p,n}} \ \text{and}\  \tilde X'_{ij}=\tilde X_{ij}\mathds{1}_{\Delta_{j,p,n}}.
$
Then
\begin{align}\label{eq: 5.1a}
d_L&\Big(\mu^{\widehat \Sigma_n^*},
\mu^{\check \Sigma_{n}^{*K} }
\Big)\leq d_L\Big(\mu^{\widehat \Sigma_n^*},\mu^{\hat\Sigma_{n}^{*'}}\Big)+d_L\Big(\mu^{\widehat \Sigma_{n}^{*'}},\mu^{\tilde\Sigma_{n}^{*'}}\Big)+d_L
\Big(\mu^{\tilde{\Sigma}_{n}^{*'}},
\mu^{\check \Sigma_{n}^{*K}} \Big)
\end{align}
For  the second term in \eqref{eq: 5.1a}, 
we have by Theorem A.38 in \cite{baisilverstein2010}, the Lidskii-Wielandt perturbation bound (1.2) in \cite{limath1999},
$
\| L_n\|_{S_{\infty}}\leq {\alpha}$ {on $\mathcal R_n$ (see Section~\ref{sec52})}, H\"older's inequality for Schatten norms, and the Cauchy-Schwarz inequality 
\begin{align*}
\E^*d_L^2&\Big(\mu^{ \widehat \Sigma_{n}^{*'}},\mu^{\tilde\Sigma_{n}^{*'}}\Big)\\
&\leq \E^*\frac{1}{q}\sum_{j=1}^q\Big\arrowvert \lambda_i( \widehat \Sigma_{n}^{*'})-\lambda_i(\tilde\Sigma_{n}^{*'})\Big\arrowvert\\
&\leq \E^*\frac{1}{q}\big\Arrowvert \widehat \Sigma_{n}^{*'} -\tilde\Sigma_{n}^{*'}\big\Arrowvert_{S_1}\\
&\leq \E^*\frac{1}{qm}\Big\Arrowvert X^{*'}X^{*'\top   }-\tilde X^{*'}\tilde  X^{*'\top }\Big\Arrowvert_{S_1} 
\Big(\sup_{n \in \mathbb{N}}
\Arrowvert  A_n\Arrowvert_{S_{\infty}}^2+o(1)\Big) \\
&= \E^*\frac{1}{qm}\Big\Arrowvert (X^{*' }-\tilde X^{*' })(X^{*'}-\tilde X^{*' })^\top +(X^{*' }-\tilde X^{*' })\tilde X^{*'\top}+\tilde X^{*' }(X^{*' }-\tilde X^{*' })^\top \Big\Arrowvert_{S_1}
  \\
 &\hspace{90mm}
\cdot\Big(\sup_{n \in \mathbb{N}}
\Arrowvert  A_n\Arrowvert_{S_{\infty}}^2+o(1)\Big)\\
&\leq \E^*\frac{1}{qm}\bigg(\Big\Arrowvert (X^{*' }-\tilde X^{*' })(X^{*' }-\tilde X^{*' })^\top \Big\Arrowvert_{S_1}+2 \big\Arrowvert X^{*' }-\tilde X^{*' } \big\Arrowvert_{S_2}\big\Arrowvert \tilde X^{*' }\big\Arrowvert_{S_2}\bigg)
\\
 &\hspace{90mm}
\cdot\Big(\sup_{n \in \mathbb{N}}
\Arrowvert  A_n\Arrowvert_{S_{\infty}}^2+o(1)\Big)\\
&\leq \frac{1}{qm}\bigg\{\E^*\tr\big((X^{*' }-\tilde X^{*' })(X^{*' }-\tilde X^{*' })^\top \big)\\
&\quad\quad\quad\quad +2\Big(\E^*\tr\big((X^{*' }-\tilde X^{*' })(X^{*' }-\tilde X^{*' })^\top \big)\Big)^{1/2}\big(\E^*\tr\big(\tilde X^{*' }\tilde{X}^{*'\top }\big)^{1/2}\bigg\}
\\
 &\hspace{90mm}
\cdot\Big(\sup_{n \in \mathbb{N}}
\Arrowvert  A_n\Arrowvert_{S_{\infty}}^2+o(1)\Big).
\end{align*}
But with $K'>0$ being a uniform upper bound on $q'/q$  we obtain 
\begin{align*}
\sup_p\frac{1}{qm}\E^*\tr\big((X^{*' }-\tilde X^{*' })(X^{*' }-\tilde X^{*' })^\top \big)&=\sup_p\frac{1}{qm}\sum_{i=1}^m\sum_{j=1}^{q'}\E^*\big(X_{ij}^{*' }-\tilde X_{ij}^{*' })^2\\
&=\sup_p\frac{1}{q}\sum_{j=1}^{q'}\frac{1}{n}\sum_{i=1}^n\big(X_{ij}'-\tilde X_{ij}' )^2\\
&\leq \sup_p\frac{1}{qn}\sum_{j=1}^{q'}\mathds{1}_{\Delta_{j,p,n}}\sum_{i=1}^nX_{ij}^2\mathds{1}\{\arrowvert X_{ij}\arrowvert>K\}\\
&\leq K'\big(\E X_{11}^2\mathds{1}\{\arrowvert X_{11}\arrowvert >K\big\}+\delta\big),
\end{align*}
while
\begin{align*}
\sup_p\frac{1}{qm}\E^*\tr\big(\tilde X^{*' }\tilde{X}^{*'\top }\big)&=\sup_p\frac{1}{qn}\sum_{j=1}^{q'}\mathds{1}_{\Delta_{j,p,n}}\sum_{i=1}^nX_{ij}^2\mathds{1}\{\arrowvert X_{ij}\arrowvert\leq K\}
\leq   K'\big(\E X_{11}^2+2\delta\big).
\end{align*}
Summarizing these calculations, we obtain for the second term in \eqref{eq: 5.1a}  the estimate
\begin{align*}
\limsup_{p \to \infty}
\E^*d_L^2\Big(\mu^{\widehat \Sigma_{n}^{*' }},\mu^{\tilde\Sigma_{n}^{*' }}\Big)
\leq 
&\  
K'
\sup_n \Arrowvert A_n\Arrowvert_{S_{\infty}}^2
\Big \{ 
 \big (\E X_{11}^2\mathds{1}\{\arrowvert X_{11}\arrowvert >K\big\}+\delta \big  )
\\ 
&\ \  + 
2 \big ( 
\E X_{11}^2\mathds{1}\{\arrowvert X_{11}\arrowvert >K\big\}+\delta 
 \big ) ^{1/2}
\big( 
\E X_{11}^2+2\delta  \big)^{1/2} 
\Big \}. 
\end{align*}
For a corresponding estimate of the first term in \eqref{eq: 5.1a} we note that $\P(\Delta_{j,p,n})\rightarrow 1$ as $n\rightarrow\infty$ by the weak law of large numbers. Moreover, for fixed $n$, the value $\P(\Delta_{j,p,n})$ is the same for all $j\in\{1,2,\dots, p\}$. Hence, for sufficiently large $n$,
\begin{align*}
\P\Big(\sum_{j=1}^{q'}\mathds{1}_{\Delta_{j,p,n}^c}\geq \delta q\Big)&\leq \P\Big(\sum_{j=1}^{q'}\big(\mathds{1}_{\Delta_{j,p,n}^c}-\P(\Delta_{j,p,n}^c)\geq \frac{1}{2}\delta q\Big)
\leq \exp \Big(-\frac{\delta^2q}{2K'}\Big)
\end{align*}
by Hoeffding's inequality (applied conditionally on $\Pi_n$).  Using that $d_L\leq d_K$, where $d_K$ denotes the Kolmogorov distance,  Theorem A.43 of \cite{baisilverstein2010}, the inequality $\rank (AB)\leq \min(\rank(A),\rank(B))$ we obtain 
\begin{align*}
 d_L\Big(\mu^{\widehat \Sigma_n^*},\mu^{\hat\Sigma_{n}^{*' }}\Big)&\leq d_K\Big(\mu^{\widehat \Sigma_n^*},\mu^{\hat\Sigma_{n}^{*' }}\Big)\\
&\leq \frac{1}{q}\rank \Big(\widehat \Sigma_n^*-\widehat \Sigma_{n}^{*' }\Big)\\
&\leq  \frac{1}{q}\rank \Big(X^*X^{*\top}-X^{*' }X^{*'\top }\Big)\\
&\leq \frac{{2}}{q}\sharp\Big\{j\in \{1,\dots, q'\}: \mathds{1}_{\Delta_{j,p,n}^c}=1\Big\}
\leq {2\delta} 
\end{align*}
{on an event whose probability tends to one.}
Here the  fourth inequality  follows from
    \begin{align*}
        \text{rank} \Big (X^* {X^{*\top}}-X^{*' } X^{*'\top } \Big ) &=  \text{rank} \Big (\sum^m_{i=1} (X^*_i X^{*\top}_i-X_i^{*' }X_i^{*'\top })\Big )\\
       & =  \text{rank}  \Big ( \sum^n_{i=1} \delta_i \big ({X_i X_i^\top  - {X_i'} X_{i}^{'\top } } \big ) \Big )
    \end{align*}
  where $\delta _i  {=} \sharp \{ j \in \{1, \ldots, {m} \} ~|~X_j^*=X_i \} 
  \in \{ 0,\ldots, m\}$ 
  with $\sum^n_{i=1} \delta _i=m$ and the fact that {all non-zero entries of the $q'\times q'$ matrix $\sum^n_{i=1}\delta_i(X_iX_i^\top-X_i'X_i'^\top)$ lie in the rows and columns with index in $J=\{j:\mathds 1_{\Delta^c_{j,p,n}}=1\}$ (the $(j,k)$ entry is $\sum_i\delta_iX_{ij}X_{ik}(1-\mathds 1_{\Delta_{j,p,n}}\mathds 1_{\Delta_{k,p,n}})$), so that its rank is at most $2|J|$}.

{The third term, with $\check\Sigma_n^{*K}$ replaced by the matrix $\tilde\Sigma_n^{*K}=m^{-1}L_n\tilde X^*\tilde X^{*\top}L_n^\top$ built from the truncated (but not yet centered and standardized) variables, can be bounded by $2\delta$ analogously; centering and standardization are treated separately below. The first and the third term do not depend on the bootstrap sample. Hence, by the triangle inequality and Jensen's inequality, with $\varepsilon_K=\E X_{11}^2\mathds 1\{|X_{11}|>K\}$,
\begin{align*}
\limsup_{p\to\infty}\E^*d_L\Big(\mu^{\widehat\Sigma_n^*},\mu^{\tilde\Sigma_n^{*K}}\Big)\le4\delta+\Big[K'\alpha^2\Big\{(\varepsilon_K+\delta)+2(\varepsilon_K+\delta)^{1/2}\big(\E X_{11}^2+2\delta\big)^{1/2}\Big\}\Big]^{1/2}
\end{align*}
with probability tending to one. Since $\delta>0$ may be chosen arbitrarily small and $\E X_{11}^2=1$,
\begin{align*}
\limsup_{p\to\infty}\E^*d_L\Big(\mu^{\widehat\Sigma_n^*},\mu^{\tilde\Sigma_n^{*K}}\Big)\le\Big[K'\alpha^2\big(\varepsilon_K+2\varepsilon_K^{1/2}\big)\Big]^{1/2}
\end{align*}
with probability tending to one (note that $\E^*d_L^2\le\E^*d_L$, as $d_L\le1$).} Now, the last expression can be made arbitrarily small for $K$ sufficiently large, independently of the projection $\Pi_n$. Since the centralization of the truncated random variables $\tilde X_{ij}$ leads to a finite rank perturbation of $\tilde\Sigma_{n}$ (uniformly in $p$), we may assume the entries $\tilde X_{ij}$ to be centered. Next, as in the truncation step by replacing there $\mathds{1}\{\arrowvert X_{ij}\arrowvert\leq K\}$ with
$1/ \sqrt{\mathrm{Var} ( X_{ij}\mathds{1}
\{\arrowvert X_{{ij}}\arrowvert \leq K\} ) }$
in the definition of $\tilde X$, we may assume the entries to be standardized since the variance of the truncated variables converges to one as the truncation level tends to infinity, which 
completes the proof of \eqref{eq: truncation}. 

\smallskip

 Note that the matrix $L_n$ can be a random matrix,  but it  is independent of $X_1, \ldots  , X_n$ as well as from $X_1^*,\dots, X_m^*$. As a consequence of Section \ref{sec52} and \ref{sec52aa}, 
 we assume from now on that  conditional on $\Pi_n$, the variables $X_{ij}$ are centered, standardized and bounded, that the vectors $X_i$ have $q'=O(q)$ components and that  the matrix $L_n$ is of dimension $q \times q' $. 

\subsection{A first non-standard result on quadratic forms} \label{sec52a} 

In this section, we derive moment bounds on 
$$
X_1^{*\top}
 C^*(X_2^*,\dots, X_m^*)X_1^*-\tr C^*(X_2^*,\dots, X_m^*)
$$
for 
     the particular matrices 
   \begin{eqnarray}
   \label{eq: C_i1}
   C^*  &=&L_n^\top  D_1^*(z)^{-1} L_n  ,   \\
 \label{eq: C_ia1}  
    C^*  & =& L_n{^\top} D_1^*(z)^{-1} 
   (\E^*\underline{m}_n^*(z)\tilde\Sigma_n+I_q)^{-1}
  L_n  ~.
   \end{eqnarray}

 \begin{proposition}
 \label{lemma: formel 2.1}
For any $p\geq 2$, there exists some constant $K_p(z) >0$, such that for any $n\in \N$, 
\begin{align}
   \label{ar3}
 \E
 \big\arrowvert X_1^{*\top}
 C^*X_1^*-\tr C^* \big\arrowvert^p  \leq  
 {K_p (z) }
 \Big( m^{p/2}+
 \frac{m^{p+1}}{n}
 \Big), 
\end{align}
where $C^*$ is  either given by  \eqref{eq: C_ia1}  or \eqref{eq: C_i1} and  the constant  $K_p(z)$ 
 depends only $z$ and $p$.
 \end{proposition}

A natural idea of proving this result is to
first 
condition on $X_1, \ldots , X_n$ and applying 
Lemma B.26 of \cite{baisilverstein2010} to 
$$
\E_X^* \big\arrowvert X_1^{*\top}
 C^*X_1^*-\tr C^* \big\arrowvert^p .
 $$
 However, this approach fails as the components 
 of the vector $X_1^*$ are conditional on $X_1, \ldots , X_n$ neither independent nor normalized. 
 Therefore, a proof of the estimate 
 \eqref{ar3} for the unconditional expectation  relies on a different argument, which  originates  the condition $m = o(n)$. Note that in the unconditional world, the vector $X_1^*$ and the matrix $C^*$ are not independent any longer.

      \begin{proof}[Proof of Proposition \ref{lemma: formel 2.1}]Since 
     \begin{align}\label{eq: anfang}
   \E \big [ \E^*  \big [  \arrowvert {X_1^*}^\top C^*X_1^*-\tr C^*\arrowvert^p ~ \big |  X_2^*, \ldots , X_m^* \big ] \big ] &=\E \Big [  \frac{1}{n}\sum_{j=1}^n\arrowvert X_j ^\top C^*X_j-\tr C^*\arrowvert^p\Big ] \\
   & =  \E\arrowvert X_1^\top C^*X_1-\tr C^*\arrowvert^p,  
   \nonumber 
   \end{align} 
   it is sufficient to deduce the bound for the right-hand side, that is 
   $$\E \arrowvert X_1^\top C^*X_1-\tr C^*\arrowvert^p 
    \leq 
   K_p(z)  
 \Big( m^{p/2}+
 \frac{m^{p+1}}{n}  \Big ) \ .
   $$
Because of $X_1^*,\dots, X_m^*\overset{iid}{\sim}\hat{\P}_n$, the matrix $C^*$  in \eqref{hd54}  depends on $X_1$ and therefore, standard results on centered random quadratic forms   as Lemma B.26 in \cite{baisilverstein2010} are not directly applicable. 
\medskip

For $i=1,\dots, m $ {and the resampling indices $I_1,\dots,I_m$ (i.e.\ $X_i^*=X_{I_i}$)} define 
 \begin{align}\label{eq: xsterntilde}
   \tilde{X}_i^* & =X_i^*I\{{I_i\not= 1}\} , \ \ 
   \tilde r_i^*  = {1\over \sqrt{m} } L_n  \tilde{X}_i^* ,
   \end{align}
   and write 
\begin{align}
\label{det34}
    \Delta_n^*=\sharp\big\{i\in\{2 ,\dots,m\}: {I_i=1}\big\}~, ~~  \bar \Delta_n^*=\sharp\big\{i\in\{1 ,\dots,m\}: {I_i=1}\big\}.
\end{align}
    We note that
    \begin{align}
        \label{hd50}
        \E ^*[ \Delta_n^* ] = {m - 1  \over n} ~, ~~ 
          \E ^*[ \bar \Delta_n^* ] = {m    \over n} .
    \end{align}
For any matrix $A$ built from $X_1^*,\dots, X_m^*$, we write $\tilde A$ for the corresponding matrix which arises by replacing $X_i^*$ by $\tilde X_i^*$, $i=1,\dots, m$, in the definition of $A$. Furthermore, we introduce 
\begin{align}
    \label{hd40}
 C 
 ^*  =L_n ^{\top}D_1^*(z)^{-1} 
  B 
   L_n  \\
       \label{hd40a}
 \check{ C} 
 ^*  =L_n ^{\top}D_1^*(z)^{-1} 
 \tilde B 
   L_n  
 \end{align}
 where 
\begin{align}
\label{hd52}
B &=  ( \E^*\underline{m}_n^*(z)
    \tilde\Sigma_n+I_q)^{-1} \\ 
\tilde B &=  ( \widetilde{\E^*\underline{m}_n^*(z)
   } \tilde\Sigma_n+I_q)^{-1}
   \label{hd53}
\end{align}
   and  
   $$
 \widetilde{\E^*\underline{m}_n^*(z)} 
 =  \E^* 
  {q \over m} \Big \{ {1 \over q} 
 {\rm tr} \Big [ 
 \sum_{j=1}^m \tilde r_j^*\tilde r_j^{*\top}-zI_q  \Big ]
 ^{-1} 
 \Big \}  - {1- q/m \over z }~.
 $$
 Note that  the difference between the terms 
 $\E^*\underline{m}_n^*(z)$
 and $ \widetilde{\E^*\underline{m}_n^*(z)} $
 consists in the fact that the sum in the 
 latter term does not contain the variable
 $X_1$ anymore. 
 In a first step, we replace  $ C 
 ^* $   by $ \check{ C} 
 ^* $  
 with an error of order $O(m/n)$.  For this purpose,
 we use the identity $A^{-1}_1 -A_2^{-1} = A^{-1}_1  (A_2-A_1) A_2^{-1}  $ together with the Sherman-Morrison formula to obtain 
 \begin{align*}
      \E  &\big\arrowvert  X_1^{\top}( C^* -  \check  C^*)    X_1\big\arrowvert^p     \\
     &=   
      \E   
      \Big | 
     X_1^{\top} L_n^{\top}D_1^*(z)^{-1} ( \widetilde{ \E^*\underline{m}_n^*(z)}
      \tilde\Sigma_n+I_q)^{-1}\\
      &\ \ \ \ \ \ \ \ 
     \ \cdot\E^*
      \Big [
     {1 \over m^2}    \frac{\bar \Delta_n^* X_1^\top  L_n^\top  \tilde D^*(z)^{-2}L_nX_1}{1+\frac{{\bar\Delta_n^*}}{m}X_1^\top L_n^\top \tilde{D}^*(z)^{-1}L_nX_1}\Big ]
    \tilde \Sigma_n
     ( \E^*\underline{m}_n^*(z)\tilde\Sigma_n+I_q)^{-1}
      L_n  X_1 \Big |^p \\
       & \leq 
       {1 \over m^{2p} }
       \E   \Big [  | X_1^{\top} X_1 |^{2p}   (\bar  \Delta_n^*)^p   \| L_n \|_{S_\infty }^{6p} 
       \max \Big (\frac{4\Arrowvert L_n\Arrowvert_{S_{\infty}}^2}{\Im(z)},2 \Big )^{2p} \frac{{\arrowvert z\arrowvert}^p}{\Im(z)^{4p}}
       \Big ]
       \\
       & \leq  K_p(z)  {m \over n}   ,
 \end{align*}
 where we have used  Lemma 2.3   
\cite{silverstein1995} together with $\Arrowvert \tilde D^*(z)^{-1}\Arrowvert_{S_{\infty}}\leq 1/\Im(z)$ and $\arrowvert 1/(1+\frac{1}{m}X_1^\top L_n^\top \tilde{D}^*(z)^{-1}L_nX_1)\arrowvert \leq \arrowvert z\arrowvert/\Im(z)$ for  the  first inequality, \eqref{hd50}, and the fact that the $X_i$  and $X_i^*$  are 
 $q= O(q) = O(m)$ dimensional vectors 
 with   uniformly bounded  components
 (by the arguments in Section \ref{sec52}
 and \ref{sec52aa}). 
 Similarly,  we have  
  \begin{align*}
    &  \E \big\arrowvert  {\rm tr } ( C^* -  \check  C^*) \big\arrowvert^p    \leq K_p(z)  {m \over n}  .
 \end{align*}
 and it follows that 
\begin{align}
    \label{hd51}
    \E \big  \arrowvert X_1^{\top} C ^*X_1 -\tr C ^*
    -\big (  X_1^{\top} \check C ^*X_1 -\tr \check C ^* \big ) 
    \big \arrowvert^p 
      \leq K_p(z)  {m \over n}   .
  \end{align}
   Consequently, it is sufficient to show the assertion for the matrix 
 \begin{align}
    \label{hd54}
       C^*  =L_n ^{\top}D_1^*(z)^{-1} M L_n ~,
   \end{align}
      where $M$  a potentially random matrix, which is, conditionally on $X_1, \ldots , X_n$ independent of $X_1^*$,  depending only on  $X_2, \ldots , X_n$ 
      with almost surely bounded spectral norm 
      (uniformly in $n$).   We then apply the result for     the matrices 
 $M=I_q$ and $M= \tilde B $  in \eqref{hd53} (note the latter has a uniformly bounded spectral norm by Lemma 2.3  in 
\cite{silverstein1995}).

   By the Sherman-Morrison formula,
  \begin{align}\label{eq: property}
   C^*&= \tilde{C}^*
 - {\frac{\frac{1}{m}\Delta_n^*\, G^*X_1X_1^{\top}\tilde C^*}{1+\frac{1}{m}\Delta_n^*X_1^{\top}G^*X_1},\qquad G^*=L_n^\top\tilde D_1^*(z)^{-1}L_n,}
 \end{align}
    where $\tilde{D}_1^*$ is defined as $D_1^*$ with $X_2^*,\dots, X_m^*$ replaced by $\tilde{X}_2^*,\dots, \tilde{X}_m^*$
    and 
 \begin{align}
     \label{det36}
    \tilde  C^*=  \tilde C(\tilde X_2^*,\dots, \tilde X_m^*)   = {L_n^\top  \tilde{D}_1^*(z)^{-1} M L_n} .
 \end{align}
  Note that the matrix $ \tilde  C^*$ does not depend on the random variable $X_1$ anymore. Therefore, inserting the conditional expectation with respect to $\tilde{X}_2^*,\dots, \tilde{X}_m^*$, Lemma B.26 in \cite{baisilverstein2010}  reveals
   \begin{align} \label{eq: step1}
 \E \big\arrowvert X_1^\top
 \tilde  C^* X_1-\tr  \tilde  C^* 
 \big\arrowvert^p \le  K _p (z) m^{p/2} ~.
\end{align}
Consequently, it remains  to derive a bound 
for  the difference 
\begin{align}
\label{hd12}
  \tilde  C^* - C^*   = \frac{ {1 \over {m}}  \Delta_n^* \,  G^*  X_1X_1^{\top}  \tilde  C^* }{1+{1 \over m} \Delta_n^*X_1^{\top} G^* X_1} , 
 \end{align}
 that is, a bound on 
    \begin{align} \label{hd55}
 \E \big\arrowvert X_1^\top (
 \tilde C^* -   C^*  ) X_1-\tr ( \tilde  
 C^* -    C^*  )
 \big\arrowvert^p  .
\end{align}
Employing the  estimate (3.4) in 
\cite{baisil1998}
for the denominator yields 
\begin{align}
\label{ar5}
\bigg\arrowvert \frac{\frac{\Delta_n^*}{m}X_1^{\top}G^*X_1}{1+\frac{\Delta_n^*}{m}X_1^{\top}G^*X_1}\bigg\arrowvert\leq 2\Big (1+ \frac{\arrowvert z\arrowvert}{ \Im (z) } \Big )
\end{align}
(distinguish the cases $\frac{\Delta_n^*}{m} \arrowvert X_1^{\top}G^*X_1\arrowvert >2$ and $\le 2$),
and we obtain (using \eqref{hd12} twice) 
 \begin{align} 
 \E \Big\arrowvert & X_1^\top 
( C^* -  \tilde  C^*)  X_1-\tr   (C^* -  \tilde  C^*)
 \Big\arrowvert^p \nonumber \\
 &\leq 2^{p-1} \E \big\arrowvert  X_1^{\top}( C^* -  \tilde  C^*)  X_1\big\arrowvert^p+ 2^{p-1}\E \big \arrowvert \tr   (C^* -  \tilde  C^*)\big\arrowvert ^p\nonumber\\
&= 2^{p-1}\E \bigg[\bigg\arrowvert\frac{\frac1m\Delta_n^*X_1^\top G^*X_1}{1+\frac1m\Delta_n^*X_1^\top G^*X_1}\bigg\arrowvert^p\big\arrowvert X_1^\top\tilde C^*X_1\big\arrowvert^p\bigg]\nonumber\\
&\hspace{20mm}+2^{p-1}\frac1{m^p}\E\bigg\arrowvert\frac{\Delta_n^*X_1^\top\tilde C^*G^*X_1}{1+\frac1m\Delta_n^*X_1^\top G^*X_1}\bigg\arrowvert^p\nonumber\\
&\le 2^{2p-2}\bigg(1+\frac{\arrowvert z\arrowvert }{ \Im (z) }\bigg)^{p-1}\frac{\arrowvert  z\arrowvert}{ \Im (z) } { 1  \over m} \E  \Big [  \Delta_n^*\big\arrowvert X_1^{\top}G^* X_1\big\arrowvert\big\arrowvert X_1^{\top}\tilde C^* X_1\big\arrowvert^{p}\Big ]
\label{hol1a}\\
&\hspace{20mm}+ 2^{p-1} {|z|^p  \over \Im  ^p (z) } { 1 \over m^p}\E \Big [(\Delta_n^*)^p\big \arrowvert X_1^\top \tilde  C^*  G^*  X_1\big\arrowvert^p\Big]\nonumber\\
& \leq K_p(z) \Big( \frac{m^{p+1}}{n} \Big),
\end{align}
where we used the fact that $X_1$ has uniformly bounded components and that the spectral norms of $\tilde C^*$, $G^*$ and $M$ are also uniformly bounded.
Combining this result with \eqref{eq: step1} completes the proof. 
\end{proof} 

\subsection{Remaining part of the  proof}
\label{sec53} 

Let $\mu^{\widehat \Sigma_n^*}$ denote the spectral measure of the matrix $\widehat \Sigma_n^*$ and denote by 
$m_n^*: \mathbb{C}^+ \to \mathbb{C}^+ $
the corresponding Stieltjes transform,
that is 
\begin{align}
\label{hd56} 
m_n^*(z)&=\frac{1}{q}\tr\Big[\big(\hat \Sigma_n^{*} -z I_q \big)^{-1}\Big]
\end{align}
We define 
$$
\underline{m}_n^*(z)=\frac{q}{m}m_n^*(z)-\Big(1-\frac{q}{m}\Big)\frac{1}{z}
$$
and denote by 
$\mu^{\tilde \Sigma_n}$  the spectral distribution of the matrix $\tilde{\Sigma}_n$ defined in \eqref{hol1}.
Finally,  we define   
\begin{align}
    \label{det102}
\underline{\tilde m}_n^0=\underline{m}_{{\frac{q}{m}},\mu^{\tilde \Sigma_n} }^0 
\end{align}
as  the solution of the equation \eqref{eq: mbarnull} with $G = \mu^{\tilde \Sigma_n}$ and $\gamma={q/m}$.
In order to show
\begin{equation}\label{eq: final:1}
\mu^{\widehat \Sigma_n^*}-\mu^{\widehat \Sigma_n}\Longrightarrow 0 \ \ \text{in probability}
\end{equation}
we will 
prove in Subsection \ref{sec54}  and \ref{sec55}  that, conditionally on $\Pi_n$, 
\begin{align}\label{eq: 5.1b}
\big\arrowvert\E^* \underline{m}_n^*(z)-\underline{\tilde m}_n^0(z)\big\arrowvert
& = \mathcal{O}_{\P}\Big(\frac1m+\sqrt{\frac{m}{n}} \Big)~, \\
\label{eq: 5.1}
\big\arrowvert\E^* \underline{m}_n^*(z)-
\underline{m}_n^*(z)
\big\arrowvert
& =\mathcal{O}_{\P}\Big(\frac{1}{\sqrt{m}}\Big)~.
\end{align}
As a consequence of the previous two steps, 
\begin{align}
   \label{ar8} 
\big\arrowvert\underline{m}_n^*(z)-\underline{\tilde m}_n^0(z)\big\arrowvert={o}_{\P}(1)
\end{align}
for any $z\in\C^+$ conditionally on $\Pi_n$.  Note that  both terms in this expression depend on the random projection $\Pi_n$.

Due to condition \eqref{eq: similarity3} in the Representative Subpopulation Condition,  we have 
\begin{align}
\label{det103}
\big |\underline{m}_{\frac{p}{n},\mu^{\Sigma_n}}^0 (z) 
-\underline{\tilde m}_n^0(z)\big | = o_{\P}
(1) 
\end{align}
for all $z \in \mathbb{C}^+$,  because the solution of the MP-equation  \eqref{hd10} is continuous in $H$ (with respect to the topology of weak convergence) {and in $\gamma$ (note $q/m-p/n\to0$)}; see  formula  (3.10) and the discussion in the lines below in \cite{baisil1998}.
Therefore, 
(see equation \eqref{eq: mbarnull} with $\gamma =p/n$) we arrive at 
\begin{align}
   \label{ar8a} 
\big\arrowvert\underline{m}_n^*(z)-
\underline{m}_{\frac{p}{n},\mu^{\Sigma_n}}
^0(z) \big\arrowvert =
o_{\P}
(1) ~.
\end{align}

Let $\C_0^+=\{z_1,z_2,\dots\}$  be  a countable dense subset of $\C^+$ and denote by  $(k_n)_{n\in\N}$  an arbitrary subsequence of $(n)_{n\in\N}$. Due to the characterization of stochastic convergence in terms of almost sure convergence, there exists some subsubsequence $(k_n')_{n\in\N}$ such that
$$
\big\arrowvert\underline{m}_{k_n'}^*(z_1)
-
\underline{m}^0_{\mu^{\Sigma_{k_n'}}}
(z_1)\big\arrowvert\rightarrow 0\ \text{a.s.} 
$$
where here and subsequently, the dependence on $\gamma = {p(k_n') / k_n'}$ is   suppressed.
Denote the exceptional null set by $N_1\subset\Omega$. Due to {\eqref{ar8a}} again, there exists a subsequence $(k_n'')_{n\in\N}$ of $(k_n')_{n\in\N}$ such that
$$
\big\arrowvert\underline{m}_{k_n''}^*(z_2)-\underline{m}^0_{\mu^{\Sigma_{k_n''}}}
(z_2)\big\arrowvert\rightarrow 0
$$
outside a null set $N_2$. Continuing inductively and applying finally  the Cantor diagonalization principle, we extract a subsequence $(\tilde{k}_n)_{n\in\N}$ of $(k_n)_{n\in\N}$ such that   
$$
\big\arrowvert\underline{m}_{\tilde{k}_n}^*(z)-
\underline{m}^0_{\mu^{\Sigma_{\tilde k_n}}}
(z)\big\arrowvert\rightarrow 0\ \ \forall\, z\in\C_0^+
$$
outside the null set $N=\bigcup_{j\in\N}N_j$. For any $\ell \in\N$, let $\C_\ell^+:=\{z\in\C^+: \Im(z)>1/\ell,\arrowvert z\arrowvert\leq \ell \}$. Then $\arrowvert\underline{m}^0_{\mu^{\Sigma_{\tilde k_n}}}
(z)\arrowvert\leq l$ and $| \underline{m}_{\tilde{k}_n}^*(z)\arrowvert\leq l$ for all $z\in \C_l^+$. By Vitali's convergence theorem,  
$$
\big\arrowvert\underline{m}_{\tilde{k}_n}^*(z)-
\underline{m}^0_{\mu^{\Sigma_{\tilde k_n}}}
(z)\big\arrowvert\rightarrow 0\ \ \forall\, z\in\C_{l}^+\ \ \text{a.s.}
$$
As this convergence is true for every $l\in \N$, we conclude
$$
\big\arrowvert\underline{m}_{\tilde{k}_n}^*(z)-
\underline{m}^0_{\mu^{\Sigma_{\tilde k_n}}}
(z)\big\arrowvert\rightarrow 0\ \ \forall\, z\in\C^+\ \ \text{a.s.}
$$
But as $\underline{m}^0_{{p \over n},\mu^{\Sigma_n}}
(z)\rightarrow \underline{m}_{c,H}^0(z)$ $\forall\, z\in\C^+$ and 
 $m_{c,H}^0$ is the Stieltjes transform of a probability 
 measure $\mu_{c,H}^0$ with compact support, this implies 
 weak convergence of 
 $$(\mu^{\widehat \Sigma_{\tilde k_n}^*})_{n \in \N} \Longrightarrow
 \mu_{c,H}^0 $$ almost surely.  Since $(k_n)_{n\in\N}$ was an arbitrary subsequence,
$
\mu^{\widehat \Sigma_n^*}\Rightarrow  \mu_{c,H}^0 \ \ \text{in probability.}
$
Finally, by the triangle inequality, 
Lemma \ref{prop1}  and $\mu^{0}_{{p\over n} , \Sigma_n}\Rightarrow \mu_{c,H}^0$,  we have  
$$ d_{BL}\big(\mu^{\widehat \Sigma_{{\tilde k_n}}^*}, \mu^{\widehat \Sigma_{\tilde k_n}}\big)\longrightarrow 0\ \ \text{a.s.} $$
and therefore  
$
d_{BL}\big(\mu^{\widehat \Sigma_{n}^*}, \mu^{\widehat \Sigma_{n}}\big)\longrightarrow 0
$
in probability.

\subsection{Proof of \eqref{eq: 5.1b}} 
\label{sec54}

 Recall the definition of the population 
covariance matrix $\Sigma_n$ and   that the matrix
$
\tilde{\Sigma}_n=\Pi_n \Sigma_n\Pi_n^\top 
=  L_n L_n^\top 
$  
in \eqref{hol1} is the population 
covariance matrix
corresponding to sub-sampling process,  where $L_n \in \mathbb{R}^{q \times q'}$. Note 
 that $\tilde{\Sigma}_n$ can be a random object which is independent of 
$X_1, \ldots , X_n$. With the notation from Section \ref{sec51} we can rewrite $\widehat{\Sigma}_{n}^{*}$ as
\begin{align*}
\widehat \Sigma_n^{*}  &=\sum_{j=1}^mr_j^*r_j^{*\top}\in\R^{q\times q}.
\end{align*}
Next, we define (for $\ z\in\C^+$) the Stieltjes transform$$
\underline{m}_n^*(z)=\frac{q}{m}m_n^*(z)-\Big(1-\frac{q}{m}\Big)\frac{1}{z} ~,
$$
where ${m}_n^*(z)$ denotes the Stieltjes transform 
of the spectral measure $\mu^{\widehat \Sigma_n^*}$
of the  matrix $\widehat \Sigma_n^*$ defined in \eqref{hd56} 
(note that the supports of the measures  corresponding to $\underline{m}_n^*(z)$ and ${m}_n^*(z)$
 differ by $\arrowvert m-q\arrowvert$ zeros only). 

\medskip

 As in expression (5.2) of 
\cite{baisil1998}, we obtain the identity
\begin{align}
\label{r3}
\begin{split}
\frac{q}{m}&\int\frac{\dd \mu^{\tilde \Sigma_n} (t)}{1+t\E^*\underline{m}_n^*(z)}+z\frac{q}{m}\E^*\big(m_n^*(z)\big)\\
&=\E^*\bigg\{\beta_1^* (z)\bigg[r_1^{*\top}D_1^*(z)^{-1}(\E^*\underline{m}_n^*(z)\tilde\Sigma_n+I_q)^{-1}r_1^*\\
&\hspace{30mm}-\frac{1}{m}\tr(\E^*\underline{m}_n^*(z)\tilde\Sigma_n+I_q)^{-1}\tilde\Sigma_n\E^*\big(D^*(z)^{-1}\big)\bigg]\bigg\}.
\end{split}
\end{align}
Rewriting the left-hand side
\begin{align*}
\frac{q}{m}&\int\frac{\dd \mu^{\tilde \Sigma_n} (t)}{1+t\E^*\underline{m}_n^*(z)}+z\frac{q}{m}\E^*\big(m_n^*(z)\big)\\
&=\frac{q}{m}\bigg(1-\int\frac{t\E^*\underline{m}_n^*(z)}{1+t\E^*\underline{m}_n^*(z)}\dd \mu^{\tilde \Sigma_n} (t)\bigg)+z\E^*\big(\underline{m}_n^*(z)\big)+\Big(1-\frac{q}{m}\Big)\\
&=\E^*\big(\underline{m}_n^*(z)\big)\bigg[z-\frac{q}{m}\int\frac{t}{1+t\E^*\underline{m}_n^*(z)}\dd \mu^{\tilde \Sigma_n} (t)+\frac{1}{\E^*\underline{m}_n^*(z)}\bigg]
\end{align*}
and recalling that 
\begin{equation}
    \label{hol2}
    z-\frac{q}{m}\int\frac{t}{1+t\underline{\tilde m}_n^0(z)}d\mu^{\tilde \Sigma_n}  (t) +\frac{1}{\underline{\tilde m}_n^0(z)}=0
\end{equation}
by \eqref{eq: mbarnull}, we start with establishing the estimate 
\begin{align}
\E \bigg\arrowvert \E^*&\bigg\{\beta_1^*(z)\bigg[r_1^{*\top}D_1^*(z)^{-1}(\E^*\underline{m}_n^*(z)\tilde\Sigma_n+I_q)^{-1}r_1^*\nonumber\\
&\hspace{30mm}-\frac{1}{m}\tr\Big((\E^*\underline{m}_n^*(z)\tilde\Sigma_n+I_q)^{-1}\tilde\Sigma_n\E^*\big(D^*(z)^{-1}\big)\Big)\bigg]\bigg\}\bigg\arrowvert\nonumber\\
&=\mathcal{O}\Big(\frac{1}{m} + {\sqrt{\frac{m}{n}}} \Big).\label{eq: Viii}
\end{align}
This is carried out in the subsequent Steps (i) -- (iii). Their proofs are given at the end of this paragraph. 
\begin{itemize}
\item[(i)] 
We shall prove the bound
\begin{align}
\frac{1}{m}\E\bigg\arrowvert \tr&\Big(\big(\E^* \underline{m}_n^*(z) \tilde \Sigma_n+I_q\big)^{-1}\tilde \Sigma_n \E^*\big(D^*(z)^{-1}\big)\Big)\nonumber\\
&-\tr\Big(\big(\E^*\underline{m}_n^*(z) \tilde \Sigma_n+I_q\big)^{-1}\tilde \Sigma_n \E^*\big(D_1^*(z)^{-1}\big)\Big)\bigg\arrowvert=\mathcal{O}(m^{-1}).\label{eq: 1.1}
\end{align}
(uniformly with respect to the projection $\Pi_n$).
\item[(ii)]The next aim is to verify
\begin{align}
\E &\Big\arrowvert \frac{1}{m}\tr \Big(\big(\E^* \underline{m}_n^*(z)\tilde\Sigma_n+I_q\big)^{-1}\tilde \Sigma_nD_1^*(z)^{-1}\Big)\nonumber\\
&\hspace{30mm}-\frac{1}{m}\tr \Big(\big(\E^* \underline{m}_n^*(z)\tilde \Sigma_n+I_q\big)^{-1}\tilde \Sigma_n\E^*\big(D_1^*(z)^{-1}\big)\Big)\Big\arrowvert^2\nonumber\\
&=\mathcal{O}(m^{-1}).\label{eq: 5.8}
\end{align}
\item[(iii)]We shall deduce \eqref{eq: Viii}.
\end{itemize}
Finally, we prove \begin{equation}\label{eq: (iv)}
\arrowvert \E^* \underline{m}_n^*(z)\arrowvert\geq c(z)+{\mathcal{O}_{\P}\big(m^{-1/2}+\sqrt{m/n}\big)}.
\end{equation}
for some constant $c(z)>0$. Because of the identity
\begin{align}
    \label{ar9}
    z\underline{m}_n^*(z)=-\frac{1}{m}\sum_{j=1}^m\beta_j^*(z)
\end{align}
\citep[cf. identity (2.2) in ][]{silverstein1995}, we have $\arrowvert \E^* \underline{m}_n^*(z)\arrowvert =\arrowvert z\arrowvert^{-1}\arrowvert \E^*\beta_1^*(z)\arrowvert$. But it follows from \eqref{eq: bound gamma} and the inequalities $\arrowvert b_n^*(z)\arrowvert,\arrowvert \beta_1^*(z)\arrowvert\leq \arrowvert z\arrowvert/\Im(z)$ that 
\begin{align}
\E \arrowvert \E^*\beta_1^*(z)-b_n^*(z)\arrowvert
&= \E\Big\arrowvert \E^*\big( b_n^*(z)\beta_1^*(z)\gamma_1^*(z)\big)\Big\arrowvert\nonumber
 \\
 &
\leq \frac{\arrowvert z\arrowvert^2}{\big(\Im(z)\big)^2}\E^{1/2}_X\big\arrowvert \gamma_1^*(z)\big\arrowvert^2 
=\mathcal{O}
\Big ( { 1 \over m^{1/2} }
+ \sqrt{\frac{m}{n}}  \Big)
,\label{eq: VHa}
\end{align} 
while
\begin{align*}
\arrowvert b_n^*(z)\arrowvert \geq \frac{1}{1+{\frac{q}{m}}\Arrowvert\tilde{\Sigma}_n\Arrowvert_{S_{\infty}}/\Im(z)},
\end{align*}
which is bounded away from zero uniformly 
in $\Pi_n$ and  $n\in\N$. Hence, \eqref{eq: (iv)} is verified.

\medskip
Having established \eqref{eq: Viii} and \eqref{eq: (iv)}, we conclude that 
\begin{equation}
\label{hol12}
\arrowvert \omega_n(z)\arrowvert=\mathcal{O}_\P  \big ( m^{-1} + {\sqrt{m/n}} ) 
\end{equation}
with
\begin{align}
\label{hol3}
\omega_n(z)&:=z-\frac{q}{m}\int\frac{t}{1+t\E^*\underline{m}_n^*(z)}\dd \mu^{\tilde \Sigma_n} (t)+\frac{1}{\E^*\underline{m}_n^*(z)},
\end{align}
that is, $\E^*\underline{m}_n^*(z)$ is an approximate solution to the fixed point equation \eqref{hol2} for $\underline{\tilde m}_n^0$. Next, we may rewrite
\eqref{hol3}
\begin{align}\label{eq: 3.9}
\E^* \underline{m}_n^*(z)= -\bigg(z-\frac{q}{m}\int\frac{t}{1+t\E^*\underline{m}_n^*(z)}\dd \mu^{\tilde \Sigma_n} (t)-\omega_n(z)\bigg)^{-1}. 
\end{align}
From \eqref{eq: 3.9} and the 
equation \eqref{hol2} 
 we get the identity
\begin{align}
\label{hol11}
\E^* \underline{m}_n^*(z)-\underline{\tilde{m}}_n^0(z)= \big(\E^* \underline{m}_n^*(z)-\underline{\tilde{m}}_n^0
(z) \big)\kappa_n(z) - \omega_n(z)\underline{\tilde{m}}_n^0(z)\E^* \underline{m}_n^*(z)
\end{align}
with
\begin{align*}
\kappa_n(z)=\frac{q}{m}\frac{\int\frac{t^2}{(1+t\E^*\underline{m}_n^*(z))(1+t\underline{\tilde m}_n^0(z))}\dd
\mu^{\tilde \Sigma_n} (t)}{\Big(-z+\frac{q}{m}\int\frac{t}{1+t\E^*\underline{m}_n^*(z)}\dd\mu^{\tilde \Sigma_n} (t)+\omega_n(z)\Big)\Big(-z+\frac{q}{m}\int\frac{t}{1+t\underline{\tilde m}_n^0(z)}\dd\mu^{\tilde \Sigma_n}(t)\Big)}.
\end{align*}
An application of the Cauchy-Schwarz inequality, the identity
$$
\Im(\underline{\tilde m}_n^0(z))=\frac{\Im(z)+\Im(\underline{\tilde m}_n^0(z))\frac{q}{m}\int\frac{t^2 \dd \mu^{\tilde \Sigma_n} (t)  }{\arrowvert 1+t\underline{\tilde m}_n^0(z)\arrowvert^2
} }{\Big\arrowvert -z+\frac{q}{m}\int\frac{t}{1+t\underline{\tilde m}_n^0(z)}\dd \mu^{\tilde \Sigma_n} (t)\Big\arrowvert^2}
$$
(which follows from \eqref{eq: mbarnull}), and a similar identity for the second factor (which follows from \eqref{hol3})
yields 
\begin{align*}
    \arrowvert \kappa_n(z)\arrowvert &\leq\left[\frac{\frac{q}{m}\Im(\E^*\underline{ m}_n^*(z))\int\frac{t^2}{\arrowvert 1+t\E^*\underline{ m}_n^*(z)\arrowvert^2}\dd
   \mu^{\tilde \Sigma_n} (t)}{\Im(z)+\Im(\E^* \underline{m}_n^*(z))\frac{q}{m}\int\frac{t^2}{\arrowvert 1+t\E^*\underline{ m}_n^*(z)\arrowvert^2}\dd
   \mu^{\tilde \Sigma_n} (t)  - \Im(\omega_n(z))}\right]^{1/2}\\
    &\quad\quad\quad\quad\quad\quad\quad\quad\quad\quad\quad\quad \times\left[\frac{\frac{q}{m}\Im(\underline{\tilde m}_n^0(z))\int\frac{t^2}{\arrowvert 1+t\underline{\tilde m}_n^0(z)\arrowvert^2}\dd
    \mu^{\tilde \Sigma_n}
    (t)}{\Im(z)+\Im(\underline{\tilde m}_n^0(z))\frac{q}{m}\int\frac{t^2}{\arrowvert 1+t\underline{\tilde m}_n^0(z)\arrowvert^2}\dd
   \mu^{\tilde \Sigma_n} (t) }\right]^{1/2}~.
    \end{align*}
 In the case $\arrowvert \Im(\omega_n(z))/\Im(z)\arrowvert <1$ this  in turn can be   bounded by
    $$
    \left[\frac{\frac{q}{m}\Im(\underline{\tilde m}_n^0(z))\int\frac{t^2}{\arrowvert 1+t\underline{\tilde m}_n^0(z)\arrowvert^2}\dd\mu^{\tilde \Sigma_n} (t)}{\Im(z)+\Im(\underline{\tilde m}_n^0(z))\frac{q}{m}\int\frac{t^2}{\arrowvert 1+t\underline{\tilde m}_n^0(z)\arrowvert^2}\dd
    \mu^{\tilde \Sigma_n} (t) }\right]^{1/2}.
    $$
Therefore,
\begin{align*}
&\liminf_n\, \P\Big(\arrowvert \kappa_n(z)\arrowvert<1\big)\\
&\ \ \geq \liminf_n\P\bigg(\left[\frac{\frac{q}{m}\Im(\underline{\tilde m}_n^0(z))\int\frac{t^2}{\arrowvert 1+t\underline{\tilde m}_n^0(z)\arrowvert^2}\dd
\mu^{\tilde \Sigma_n}
(t)}{\Im(z)+\Im(\underline{\tilde m}_n^0(z))\frac{q}{m}\int\frac{t^2}{\arrowvert 1+t\underline{\tilde m}_n^0(z)\arrowvert^2}\dd
\mu^{\tilde \Sigma_n}
(t) }\right]^{1/2}<1,\\ &\hspace{60mm} \arrowvert \Im(\omega_n(z))/\Im(z)\arrowvert <1\bigg)\\
&\ \ =1
\end{align*}
because $\Im(\omega_n(z))\rightarrow_{\P}0$ and 
$$
\frac{\Im(\underline{\tilde m}_n^0(z))\frac{q}{m}\int\frac{t^2}{\arrowvert 1+t\underline{\tilde m}_n^0(z)\arrowvert^2}\dd
\mu^{\tilde \Sigma_n}
(t)}{\Im(z)+\Im(\underline{\tilde m}_n^0(z))\frac{q}{m}\int\frac{t^2}{\arrowvert 1+t\underline{\tilde m}_n^0(z)\arrowvert^2}\dd
\mu^{\tilde \Sigma_n}
( t) }\longrightarrow_{\P}  \frac{\Im(\underline{ m}_0(z))c\int\frac{t^2}{\arrowvert 1+t\underline{m}_0(z)\arrowvert^2}\dd H (t)}{\Im(z)+\Im(\underline{m}_0(z))c\int\frac{t^2}{\arrowvert 1+t\underline{m}_0(z)\arrowvert^2}\dd {H}(t) }<1
$$
(note that the support of  $\mu^{\tilde \Sigma_n}$ is uniformly bounded 
because $\sup_{n \in \mathbb{N}} \| L_n \|_{S_\infty } < \infty$, by assumption).
At the same time, $\arrowvert \E^*\underline{m}_n^*(z)\arrowvert $ and $\arrowvert \underline{\tilde m}_n^0(z)\arrowvert$ are bounded from above by $1/\Im(z)$.
Hence,  \eqref{eq: 5.1b} follows from
\eqref{hol12} and  \eqref{hol11} .

\begin{itemize}
\item {\it Proof of (i).} By the Sherman-Morrison formula applied to the matrix  $D^*(z)-D_1^*(z)$, we may rewrite the left-hand side of \eqref{eq: 1.1} as
\begin{align}
\frac{1}{m}&\E\bigg\arrowvert \tr\big(\E^*  \underline{m}_n^*(z) \tilde \Sigma_n+I_q\big)^{-1}\tilde \Sigma_n \E^*\big(D^*(z)^{-1}\big)\nonumber\\
&\hspace{10mm}-\tr\big(\E^* \underline{m}_n^*(z) \tilde \Sigma_n+I_q\big)^{-1}\tilde \Sigma_n \E^*\big(D_1^*(z)^{-1}\big)\bigg\arrowvert\nonumber\\
&=\frac{1}{m}\E\bigg\arrowvert \E^*\Big(\beta_1^*(z)r_1^{*\top}D_1^*(z)^{-1}\big(\E^* \underline{m}_n^*(z) \tilde \Sigma_n+I_q\big)^{-1}\tilde \Sigma_n D_1^*(z)^{-1}r_1^*\Big)\bigg\arrowvert.\label{eq: Vi}
\end{align}
Next, as $\underline{m}_n^*(.)$ is a Stieltjes transform and the class of Stieltjes transforms is closed under convex combination, $\E^*\underline{m}_n^*$ is a Stieltjes transform again, such that Lemma 2.3  in 
\cite{silverstein1995} implies
\begin{equation}\label{eq: lemma 2.3}
\Big\Arrowvert (\E^*\underline{m}_n^*(z)\tilde\Sigma_n+I_q)^{-1}\Big\Arrowvert_{S_{\infty}}\leq \max\bigg(\frac{4\Arrowvert\tilde\Sigma_n\Arrowvert_{S_{\infty}}}{\Im(z)},2\bigg).
\end{equation}
Using additionally the estimates $\arrowvert \beta_1^*(z)\arrowvert\leq \arrowvert z\arrowvert/\Im(z)$, $\Arrowvert D_1^*(z)^{-1}\Arrowvert_{S_{\infty}}\leq 1/\Im(z)$,  we find that \eqref{eq: Vi} can be  bounded by
\begin{align*}
\frac{\arrowvert z\arrowvert}{(\Im(z))^3}\max\bigg(\frac{4\Arrowvert\tilde\Sigma_n\Arrowvert_{S_{\infty}}^3}{\Im(z)},2\Arrowvert\tilde\Sigma_n\Arrowvert_{S_{\infty}}^2\bigg)\mathcal{O}(m^{-1}).
\end{align*}
\item {\it Proof of (ii).} Using the representation by a telescope sum  and recalling the notation \eqref{ar20} yields 
\begin{align*}
D_1^*(z)^{-1}-\E^*\big[D_1^*(z)^{-1}\big]&=\sum_{j=2}^m\big(\E_j^*-\E_{j-1}^*\big)D_1^*(z)^{-1}\\
&=\sum_{j=2}^m \big(\E_j^*-\E_{j-1}^*\big)\Big(D_1^*(z)^{-1}-D_{1j}^*(z)^{-1}\Big)\\
&=-\sum_{j=2}^m  \big(\E_j^*-\E_{j-1}^*\big)\Big(D_{1j}^*(z)^{-1}r_j^*r_j^{*\top}D_{1j}^*(z)^{-1}\beta_{1j}^*(z)\Big)
\end{align*}
and the fact that these $(m-1)$ summands are orthogonal with respect to $\E^*$, we obtain
\begin{align}
\E&\Big\arrowvert \frac{1}{m}\tr \Big(\big(\E^* \underline{m}_n^*(z)\tilde\Sigma_n+I_q\big)^{-1}\tilde \Sigma_n D_1^*(z)^{-1}\Big)\nonumber\\
&\hspace{30mm}-\frac{1}{m}\tr \Big(\big(\E^* \underline{m}_n^*(z)\tilde \Sigma_n+I_q\big)^{-1}\tilde \Sigma_n\E^*\big(D_1^*(z)^{-1}\big)\Big)\Big\arrowvert^2\nonumber\\
&= \frac{1}{m^2}\sum_{j=2}^m\E\Big\arrowvert  \big(\E_j^*-\E_{j-1}^*\big)\beta_{1j}^*(z)r_j^{*\top}D_{1j}^*(z)^{-1}\big(\E^* \underline{m}_n^*(z)\tilde\Sigma_n+I_q\big)^{-1}\tilde \Sigma_n D_{1j}^*(z)^{-1}r_j^*\Big\arrowvert^2\nonumber
\\
&\leq \frac{4}{m^3}{\frac{\arrowvert z\arrowvert^2}{\Im(z)^2}}\E \Big\arrowvert X_1^\top  L_n^\top D_{12}^*(z)^{-1}\big(\E^* \underline{m}_n^*(z)\tilde\Sigma_n+I_q\big)^{-1}\tilde \Sigma_n D_{12}^*(z)^{-1}  L_nX_1 \Big\arrowvert^2 
\nonumber 
\\
\nonumber 
 &\leq  \frac{4}{m^3}{\frac{\arrowvert z\arrowvert^2}{\Im(z)^6}}
 \E [ \|  X_1 \| _2^4]  \cdot 
 \| L_n \|_{S_\infty}^8
\max \Big (\frac{4 \| \tilde \Sigma_n \| _{S_\infty}}{\Im (z)}, 2 \| \tilde \Sigma_n  \| _{S_\infty}  \Big )^{2} \\
\nonumber
&= \mathcal{O} (m^{-1} ) 
\end{align}
where we have used  again Lemma 2.3  in 
\cite{silverstein1995}.

        \item {\it Proof of (iii).} It follows from (i) that the left-hand side in \eqref{eq: Viii} is bounded by
        \begin{align}
        \nonumber 
        \E&\bigg\arrowvert \E^*\bigg\{\beta_1^*(z)\bigg[r_1^{*\top}D_1^*(z)^{-1}(\E^*\underline{m}_n^*(z)\tilde\Sigma_n+I_q)^{-1}r_1^*\\
&\hspace{20mm}-\frac{1}{m}\tr\Big(\E^*\big(D_1^*(z)^{-1}\big)(\E^*\underline{m}_n^*(z)\tilde\Sigma_n+I_q)^{-1}\tilde\Sigma_n\Big)\bigg]\bigg\}  \bigg\arrowvert + \mathcal{O}(m^{-1})
\label{hd64}
\end{align}
(uniformly with respect to $\Pi_n$).
        Employing the identity
        \begin{align}
        \label{ar11b}
\beta_1^*(z)-b_n^*(z)&=-\beta_1^*(z) b_n^*(z)\gamma_1^*(z)
=-b_n^*(z)^2\gamma_1^*(z)+b_n^*(z)^2\beta_1^*(z)\gamma_1^*(z)^2
\end{align}   
with  the definition of $\gamma_1^*$ in \eqref{gamma1} (note that $L_nL_n^\top  =\tilde \Sigma_n $ ),
we may rewrite
\begin{align}
 \E &
\bigg \arrowvert \E^*\bigg\{\beta_1^*(z)\bigg[r_1^{*\top}D_1^*(z)^{-1}(\E^*\underline{m}_n^*(z)\tilde\Sigma_n+I_q)^{-1}r_1^*\nonumber\\
&\hspace{20mm}-\frac{1}{m}\tr\Big(\E^*\big(D_1^*(z)^{-1}\big)(\E^*\underline{m}_n^*(z)\tilde\Sigma_n+I_q)^{-1}\tilde\Sigma_n\Big)\bigg]\bigg\}  \bigg \arrowvert 
 \nonumber\\
&\le\E\big\arrowvert b_n^*(z)\,\E^*Y\big\arrowvert\nonumber\\
&\hspace{15mm} +\E
 \bigg [
 \arrowvert b_n^*(z)\arrowvert^2\bigg\arrowvert\E^*\bigg\{\Big(\gamma_1^*(z)-\beta_1^*(z)\gamma_1^*(z)^2\Big)\bigg[r_1^{*\top}D_1^*(z)^{-1}(\E^*\underline{m}_n^*(z)\tilde\Sigma_n+I_q)^{-1}r_1^* 
 \nonumber\\
&\hspace{15mm}-\frac{1}{m}\tr\Big(\E^*\big(D_1^*(z)^{-1}\big)(\E^*\underline{m}_n^*(z)\tilde\Sigma_n+I_q)^{-1}\tilde\Sigma_n\Big) 
\bigg ] \bigg\}  \bigg\arrowvert  
\bigg ] 
.\label{eq: Viiid}
\end{align}
{Here $Y$ denotes the expression in square brackets, and the first term on the right-hand side of \eqref{eq: Viiid} accounts for the fact that $Y$ is not centered with respect to $\E^*$: since $r_1^*$ is conditionally independent of $D_1^*(z)$, but $\E^*[r_1^*r_1^{*\top}]=\frac1mL_n\hat SL_n^\top$ with $\hat S=\frac1n\sum_{i=1}^nX_iX_i^\top\neq I_{q'}$, we have $\E^*Y=\frac1m\tr\big(G(\hat S-I_{q'})\big)$ with $G=L_n^\top\E^*\big(D_1^*(z)^{-1}\big)(\E^*\underline m_n^*(z)\tilde\Sigma_n+I_q)^{-1}L_n$. As $\|G\|_{S_\infty}$ is bounded and $\E\|\hat S-I_{q'}\|_{S_2}^2=\frac1n\sum_{j,k=1}^{q'}\Var(X_{1j}X_{1k})=O(q'^2/n)$ (bounded entries), we obtain $\E|b_n^*(z)\E^*Y|\le\frac{|z|}{\Im z}\frac{\sqrt{q'}}{m}\|G\|_{S_\infty}\E\|\hat S-I_{q'}\|_{S_2}=O\big(\sqrt{m/n}\big)$.}
Using the bounds $\arrowvert b_n^* (z) \arrowvert , \arrowvert\beta_1^*(z)\arrowvert \leq \arrowvert z\arrowvert/\Im(z)$,   the Cauchy-Schwarz inequality and (ii) shows that \eqref{hd64} is bounded  by
\begin{align*}
&\frac{\arrowvert z\arrowvert^2}{\big(\Im(z)\big)^2}\bigg(2\E \arrowvert \gamma_1^* (z) \arrowvert^2 + \frac{2\arrowvert z\arrowvert^2}{\big(\Im(z)\big)^2}\E \arrowvert\gamma_1^*(z)\arrowvert^4\bigg)^{1/2}\\
&\hspace{10mm}\times \bigg[\E^{1/2}_X\bigg\arrowvert r_1^{*\top}D_1^*(z)^{-1}(\E^*\underline{m}_n^*(z)\tilde\Sigma_n+I_q)^{-1}r_1^*\\
&\hspace{25mm}-\frac{1}{m}\tr\Big(D_1^*(z)^{-1}(\E^*\underline{m}_n^*(z)\tilde\Sigma_n+I_q)^{-1}\tilde\Sigma_n\Big)\bigg\arrowvert^2 +\mathcal{O}\big(m^{-1/2}\big)\bigg].
\end{align*}
Note that, conditional on $X_1,\dots, X_n$, the random variable  $D_1^*(z)^{-1}(\E^*\underline{m}_n^*(z)\tilde\Sigma_n+I_q)^{-1}$ is independent of $r_1^*$. Hence,
by Proposition \ref{lemma: formel 2.1} 
(with the matrix $C^*$ in \eqref{eq: C_ia1}), we obtain for the second factor
\begin{align}
\nonumber 
&
\E \bigg\arrowvert r_1^{*\top}D_1^*(z)^{-1}(\E^*\underline{m}_n^*(z)\tilde\Sigma_n+I_q)^{-1}r_1^* -\frac{1}{m}\tr\Big(D_1^*(z)^{-1}(\E^*\underline{m}_n^*(z)\tilde\Sigma_n+I_q)^{-1}\tilde\Sigma_n \Big)\bigg\arrowvert^2 \\
&   ~~~~~~~~~~~~~~~~~~~~~~~~~~~~
= \mathcal{O} \Big  (  {1  \over m } +  {m  \over n }  \Big ). 
\label{hd60}
\end{align}

In order to complete the proof of (iii), we continue to show
\begin{align}\label{eq: bound gamma}
\E \arrowvert \gamma_1^* (z) \arrowvert^p=\mathcal{O}
\Big ( { 1 \over m^{p/2} }
+ \frac{m}{n} \Big)
\ \ \ \text{for any $p\geq 2$}.
\end{align}
To this aim, note first that 
\begin{align}
\label{hd61}
\E \Big\arrowvert {r_1^{*\top }}D_1^*(z)^{-1}r_1^*-\frac{1}{m}\tr\big(\tilde{\Sigma}_nD_1^*(z)^{-1}\big)\Big\arrowvert^p\leq
{K_p (z) }
\Big ( \frac{1}{m^{p/2}} + \frac{m}{n} \Big)
\end{align}
by Proposition  \ref{lemma: formel 2.1}, where the constant depends only on $z$ and $p$. 
But the same consideration as at the beginning of step (ii) provides the  identity
\begin{align*}
\E &\Big\arrowvert \frac{1}{m}\tr\big(\tilde{\Sigma}_nD_1^*(z)^{-1}\big)-\frac{1}{m}\tr\Big(\tilde\Sigma_n\E^*\big( D_{1}^*(z)^{-1}\big)\Big)\Big\arrowvert^p\\
&=\E \bigg\arrowvert \frac{1}{m}\sum_{j=2}^m\big(\E_j^*-\E_{j-1}^*\big) \beta_{1j}^*(z){r_j^*\top} D_{1j}^*(z)^{-1}\tilde\Sigma_n D_{1j}^*(z)^{-1} (z)r_j^*\bigg\arrowvert^p,
\end{align*}
which is bounded by the (discrete) Burkholder-Davis-Gundy inequality, the inequality $\arrowvert  \beta_{1j}^*(z)\arrowvert\leq \arrowvert z\arrowvert/\Im(z)$, Jensen's inequality
\begin{align}
\label{hd62}
K_p&\frac{1}{m^p}\E \bigg(\sum_{j=2}^m\bigg\arrowvert \big(\E_j^*-\E_{j-1}^*\big) \beta_{1j}^*(z)r_j^{*\top} D_{1j}^*(z)^{-1}\tilde{\Sigma}_nD_{1j}^*(z)^{-1} (z)r_j^*\bigg\arrowvert^2\bigg)^{p/2}\\
\nonumber 
&\leq K_p \frac{1}{m^{p/2}}\frac{2^p\arrowvert z\arrowvert^p}{\Im(z)^p}\frac{1}{m}\sum_{j=2}^m\E \bigg\arrowvert r_j^{*\top} D_{1j}^*(z)^{-1}\tilde{\Sigma}_n
D_{1j} ^*(z)^{-1}
(z)r_j^*\bigg\arrowvert^p\\
\nonumber 
&=\mathcal{O}\big(m^{-p/2}\big).
\end{align}
Combining \eqref{hd61} and \eqref{hd62} yields 
\eqref{eq: bound gamma}.
 Hence, we obtain for \eqref{hd64}
 the bound $\mathcal{O} \big ( m^{-1} + {\sqrt{m/n}} ) $ {(including the additional term above)}, which proves (iii).

\end{itemize}

\subsection{Proof of \eqref{eq: 5.1}} \label{sec55}
The proof follows the martingale arguments of the almost sure convergence of the random part for classical covariance matrices, replacing the expectations involved there by corresponding conditional expectations in the bootstrap world. We present the adapted reasoning for the sake of completeness.

Inserting and subtracting conditional expectations, we rewrite
\begin{align*}
m_n^*(z)-\E^*m_n^*(z)&=\frac{1}{q}\sum_{j=1}^m \E_j^*\tr \big (  D^*(z)^{-1}  \big ) -\E_{j-1}^*\tr  \big ( D^*(z)^{-1}  \big ) 
=: \frac{1}{q}\sum_{j=1}^m\rho_j^* .
\end{align*}
Next, we get by conditional independence of $D^*_j(z) $  from $X_j^*$, the Sherman-Morrison formula and invariance of the trace under cyclic permutation
\begin{align*}
\rho_j^*&=\big(\E_j^*-\E_{j-1}^*\big)\Big[\tr \big(   D^*(z)^{-1}  \big)-\tr \big(
  D^*_j(z)^{-1}  \big) \Big]
=\big(\E_j^*-\E_{j-1}^*\big)\frac{r_j^{*\top}
  D^*_j(z)^{-2} 
r_j^*}{1+r_j^{*\top}   D^*_j(z)^{-1} r_j^*}.
\end{align*}
Moreover, with the diagonal representation $\sum_{i\neq j} r_i^*r_i^{*\top}
=U\Lambda {U}^\top $ for some diagonal matrix $\Lambda$ and orthogonal  matrix $U$, 
\begin{align*}
    \bigg\arrowvert \frac{r_j^{*\top}  D^*_j(z)^{-2} r_j^*}{1+r_j^{*\top}
      D^*_j(z)^{-1} 
    r_j^*}\bigg\arrowvert &=\bigg\arrowvert\frac{({U}^\top r_j^*)^\top (\Lambda-zI_{q})^{-2}{U}^\top r_j^*}{1+({U}^\top r_j^*)^\top (\Lambda-zI_{q })^{-1} {U}^\top r_j^*}\bigg\arrowvert\\
    &\leq \frac{({U}^\top r_j^*)^\top ((\Lambda-\Re(z)I_{q})^{2}+(\Im(z)I_{q })^2)^{-1} {U}^\top r_j^*}{\Im(1+( {U}^\top r_j^*)^\top (\Lambda-zI_{q})^{-1} {U}^\top r_j^*)} =\frac{1}{\Im(z)},
\end{align*}
where we have used the identity 
$$
\Im\Big(\frac{1}{\lambda-z}\Big)=\frac{\Im(z)}{(\lambda-\Re(z))^2+\Im(z)^2}
$$
for $\lambda \in \R $ in the last identity. Therefore, the family $\rho_1^*,\dots, \rho_m^*$ forms a bounded martingale difference sequence, and writing the expectation $\E$ as expected conditional expectation $\E \E^*$, an application of Burkholder's inequality reveals
\begin{align*}
    \E \big\arrowvert m_n^*(z)-\E^*m_n^*(z)\big\arrowvert^4&\leq \frac{K_4}{q^4}\E \Big( \E^* \sum_{j=1}^m\arrowvert\rho_j^* \arrowvert^2\Big)^2
    \leq\frac{{16}K_4m^2}{\Im(z)^4q^4}=\mathcal{O}\Big(\frac{1}{m^2}\Big).
\end{align*}
Because of
$$
{\frac{q}{m}}\big(m_n^*(z)-\E^*m_n^*(z)\big)= \underline{m}_n^*(z)-
\E^*\underline{m}_n^*(z)
$$
the assertion  \eqref{eq: 5.1} follows.

\section{Proof of Theorem \ref{thm: EW1}}
\label{sec7} 

  \def\theequation{C.\arabic{equation}}	
	\setcounter{equation}{0}

We begin with the proof of Lemma \ref{multi}, which will be  crucial for the proof of Theorem \ref{thm: EW1}. 

\begin{proof}[Proof of   Lemma \ref{multi}]
For $r \in \mathbb{N}_0$ let $(x)_r = x(x-1)\ldots(x-r+1)$ then
\begin{equation}\label{sti1}
  x^r = \sum^r_{j=0} a_{r,j} (x)_j
\end{equation}
where $\{ a_{r,j} \mid j=0,\ldots,r\}$ are the Stirling numbers of the second kind \citep[see][]{riordan1958}.
Using this representation  twice  we obtain
\begin{eqnarray} \label{sti2}
\mathbb{E} \Big [ \prod^k_{\ell = 1} W^{s_\ell}_\ell \Big ]  &=& \sum^{s_1}_{j_1=0} \ldots \sum^{s_k}_{j_k=0} a_{s_1,j_1} \ldots a_{s_k,j_k} \mathbb{E} \Big             [ \prod^k_{\ell=1} (W_\ell)_{j_\ell} \Big ] \\  \nonumber 
&\leq & \sum^{s_1}_{j_1=0} \ldots \sum^{s_k}_{j_k=0} a_{s_1,j_1} \ldots a_{s_k,j_k}  \prod^k_{\ell=1} \mathbb{E} \Big             [ (W_\ell)_{j_\ell} \Big ] \\\nonumber
& = & \prod^k_{\ell=1} \Big( \sum^{s_\ell}_{j_\ell=0} a_{s_\ell,j_\ell} \mathbb{E} \Big             [ (W_\ell)_{j_\ell} \Big ]  \Big )  \\
& =& \prod^k_{\ell=1}  \mathbb{E} \Big             [ W_\ell^{s_\ell}  \Big ]  
\nonumber
\end{eqnarray}
where the inequality follows  evaluating the factorial moments of the multinomial distribution,
which gives 
\begin{eqnarray*}
  \mathbb{E} \Big [ \prod^k_{\ell =1} (W_\ell)_{j_\ell} \Big ]&=&  \Big( \frac {1}{n}\Big)^m \sum_{i_1 \geq j_1 \ldots i_k \geq j_k} \sum_{i_{k+1} \ldots   i_n \geq 0}
  \frac {m ! \mathds{1} \{\sum^n_{\ell =1}i_\ell=m\}}{\prod^k_{\ell=1}(i_\ell - j_\ell)! i_{k+1}! \ldots i_n!} \\
  \\
   &=&  \frac {m!}{(m-j_1-\ldots -j_k)!} \Big( \frac {1}{n}\Big)^m n^{m-j_1 -\ldots - j_k} \\
  &\leq & \prod_{\ell=1}^k 
\frac {m!}{(m-j_\ell )!} \Big( \frac {1}{n}\Big)^m n^{m-j_\ell }  = 
\prod_{\ell=1}^k 
\mathbb{E} \Big             [ (W_\ell)_{j_\ell} \Big ] ~.
\end{eqnarray*}
Therefore, the assertion will follow from the estimate for $s \geq 1$
\begin{equation}
    \label{ang11}
  \max_{ s  \leq k_m }  
   \Big(\frac {n}{m}\Big)
  \mathbb{E}[W^{s}_\ell] \leq  
1 +   c_\gamma {m^{1+\gamma }  \over n }  ~, 
\end{equation}
which we will prove in the following. {Since $\E[(W_\ell)_j]=m(m-1)\cdots(m-j+1)n^{-j}\le(m/n)^j$, \eqref{sti1} gives for $1\le s\le k_m$ (note that $a_{s,0}=0$, $a_{s,1}=1$)
$$
\frac nm\,\E\big[W_\ell^s\big]\le1+\sum_{j=2}^s a_{s,j}\Big(\frac mn\Big)^{j-1}.
$$
We use $a_{s,j}\le j^s/j!$ (as $j!\,a_{s,j}$ is the number of surjections from $\{1,\dots,s\}$ onto $\{1,\dots,j\}$), $j^s\le j^{\gamma\log m}=m^{\gamma\log j}$ and $m/n\le m^{-\gamma/4}$. Hence
$$
\frac nm\,\E\big[W_\ell^s\big]\le1+\frac mn\sum_{j=2}^s\frac{m^{\gamma(\log j-(j-2)/4)}}{j!}\le1+\frac mn\,m^{\gamma}\sum_{j\ge2}\frac1{j!}=1+(e-2)\,\frac{m^{1+\gamma}}n,
$$
because $\log j-(j-2)/4\le\log4-\frac12<1$ for all $j\ge2$. This proves \eqref{ang11} with $c_\gamma=e-2$, uniformly in $s\le k_m$.}
\end{proof}

\begin{proof}[Proof of Theorem \ref{thm: EW1}] 

We start with proving  (b). First, we derive a bound on the moments $\E \Arrowvert \widehat{\Sigma}_{n}^{*}\Arrowvert_{S_{\infty}}^{k_m}$ for the sequence  $k_m = \lfloor \gamma  \log m \rfloor  $ (with $\gamma > 0 $ to be specified later).
In what follows, we suppress the $m$-dependence of $k=k_m$. Let $(W_1,\cdots, W_n)$ denote a multinomial  distributed vector with parameter $(m,(1/n,\cdots,1/n))$. Then (note that $\Sigma_n=I_p$)
\begin{align*}
\E\big\Arrowvert \widehat{\Sigma}_{n}^{*}\big\Arrowvert_{S_{\infty}}^{k}&\leq \E\tr\Big(
\widehat{\Sigma}_{n}^{*k}
\Big)\\
&=\frac{1}{m^k}\sum_{\substack{i_1,\dots, i_k\in\{1,\dots, q\}\\ j_1,\dots, j_k\in\{1,\dots, n\}}} 
\E \big [ W_{j_1}\dots W_{j_k} \big ] 
 \E \big [ X_{i_1j_1}X_{i_2j_1}X_{i_2j_2}\cdots X_{i_kj_k}X_{i_1j_k}
\big ] .
\end{align*}
The difference to the analysis of 
\cite{yinbaikri1988}
for the matrix  $\widehat \Sigma_n$ are the additional factors $\E [ W_{j_1}\dots W_{j_k}] $ as well as the range of indices $\{1,\dots, q\}, \{1,\dots, n\}$ instead of $\{1,\dots, p\}, \{1,\dots, n\}$ in the above expression. Note that $n/p=\mathcal{O}(1)$ while $n/q\rightarrow\infty$.

Nevertheless, due to the similarity of our expression to the corresponding expectation analyzed in \cite{yinbaikri1988}, we may adopt their strategy of decomposing the summation as follows.  Drawing two parallel lines, the so-called $I$-line and $J$-line, we can construct a directed multigraph   by plotting for a given sequence $(i_1,j_1,i_2,j_2,\dots, i_k,j_k)$ the indices $i_1,\dots, i_k\in\{1,\dots, q\}$ on the $I$-line, the indices $j_1,\dots, j_k\in\{1,\dots, n\}$ on the $J$-line and interpret them as vertices on two disjoint classes on the two parallel lines. Edges will be the directed segments $i_1j_1,j_1i_2,\dots,j_ki_1$. They are $2k$ in number and they are regarded as different from each other, even if they have the same initials and ends. Two edges are said to coincide if they have the same vertex set. If not every edge coincides at least with one other edge, then  
$$
\E\big [ X_{i_1j_1}X_{i_2j_1}X_{i_2j_2}\cdots X_{i_kj_k}X_{i_1j_k} \big ] =0.
$$
In order to treat the remaining terms, we have to distinguish between different types of edges within canonical graphs, meaning graphs that satisfy $i_1=1$, $j_1=1$, $i_k\leq \max \{i_{k-1},\dots, i_1 \} +1 $ and $j_k\leq \max \{{j_{k-1}},\dots, j_1 \} +1 $ ($k\geq 2$). In the terminology of \cite{yinbaikri1988}, an edge is called innovation if its right vertex does not occur before. Depending on whether the right vertex belongs to the $I$-line or $J$-line, it is called row- or column-innovation. An edge is called $T_3$-edge, if there is exactly one innovation before which coincides with it. An edge will be called $T_4$-edge, if it is neither an innovation nor $T_3$. Equipped with these notions, the remaining sum can be split into the sums $\sum'\sum''\sum'''$ 
\begin{align*}
\E \big [ \tr\big (\widehat{\Sigma}_n^{*k}\big) \big ] &=\frac{1}{m^k}\sum'\sum''\sum'''\E [ W_{j_1}\dots W_{j_k} ] \E\big [ X_{i_1j_1}X_{i_2j_1}X_{i_2j_2}\cdots X_{i_kj_k}X_{i_1j_k}\big ].
\end{align*} 
Here, the $\sum'$-summation is over different arrangement of the four different types of edges (row innovation, column innovation, $T_3$ and $T_4$) at the $2k$ positions, the $\sum''$-summation is running over different canonical graphs  with given arrangement of the four types for $2k$ positions, and the $\sum'''$-summation over those constellations for which the graph is isomorphic to the given canonical graph. 

\medskip

If each edge coincides at least with one other edge and if $r$ denotes the number of row innovations and $l$ the number of $T_3$-edges, then there are $l-r$ column innovations and $(2k-2l)$ $T_4$-edges. 
As shown in \cite{yinbaikri1988}
page 518 ff,  the number of summands in the first sum is bounded by
\begin{align*}
\sum'&\leq \sum_{l=1}^k\sum_{r=1}^l\binom{k}{r}\binom{k}{l-r}\binom{2k-l}{l},
\intertext{the number of summands in the third sum can be estimated from above by}
\sum'''&\leq q^{r+1}n^{l-r}
\intertext{if the canonical graph corresponding to $\sum'''$ possesses $r$ row innovations and $l$ $T_3$-edges, and finally, if $t$ denotes the number of non-coincident $T_4$-edges,}  
\sum''&\leq k^{2t}(t+1)^{6k-6l},
\end{align*}
where $t$ ranges from $0$ to $2k-2l$.

\medskip

It remains to evaluate the corresponding summands
$$
\E \big [ W_{j_1}\dots W_{j_k}\big ] \E \big [ X_{i_1j_1}X_{i_2j_1}X_{i_2j_2}\cdots X_{i_kj_k}X_{i_1j_k}\big ] 
$$
when there are $r$ row innovations, $l$ $T_3$-edges and $t$ non-coincident $T_4$-edges. As argued in 
\cite{yinbaikri1988},
$$
 \Big\arrowvert\E \big [ X_{i_1j_1}X_{i_2j_1}X_{i_2j_2}\cdots X_{i_kj_k}X_{i_1j_k} \big ] \Big\arrowvert \leq k^t\big(\delta_m \sqrt{m}\big)^{2k-2l-t},
 $$
 while our Lemma \ref{multi} implies that
 $$
 \Big\arrowvert \E \big [ W_{j_1}\dots W_{j_k}\Big ] \Big\arrowvert \leq \Big(\frac{m}{n}\Big)^{l-r}
  \Big  (    
1 +  c_\gamma   {m^{1+\gamma}   \over n }  
\Big )  ^{k}   
 $$
 as there are $l-r$ different indices among $j_1, \ldots , j_k$.
 Putting these ingredients together, we obtain
 \begin{align*}
 \E \big [ \tr ( \widehat{\Sigma}_n^{*k} ) \big ]  
 &\leq \frac{1}{m^k}  
  \Big  (    
1 +  c_\gamma   {m^{1+\gamma}   \over n }  
\Big )  ^{k}   
 \sum_{l=1}^k\sum_{r=1}^l\bigg\{\binom{k}{r}\binom{k}{l-r}\binom{2k-l}{l}q^{r+1}n^{l-r}\\
 &\hspace{25mm}\times \sum_{t=0}^{2k-2l}k^{2t}(t+1)^{6k-6l}K\Big(\frac{m}{n}\Big)^{l-r}k^t\big(\delta_m \sqrt{m}\big)^{2k-2l-t}\bigg\}\\
 &\leq  q 
 \Big  (    
1 +  c_\gamma   {m^{1+\gamma}   \over n }  
\Big )  ^{k}  
 \sum_{l=1}^k\sum_{r=1}^l\binom{k}{r}\binom{k}{l-r}\binom{2k-l}{l} \Big( \frac{q}{m}\Big)^r  
 \\
 &\hspace{25mm}\times 
 \sum_{t=0}^{2k-2l}k^{3t}(t+1)^{6k-6l}\big(\delta_m \sqrt{m}\big)^{-t} \delta_m ^{2(k-l)}.
 \end{align*}
 Using now the same arguments as in \cite{yinbaikri1988},
 pages 519 -- 520 (replacing there $n$ by $m$ and $p$ by $q$)  we finally obtain
 \begin{align*}
 \E \big [ {\rm tr } ( \widehat{\Sigma}_n^{*k} ) \big ]  
 &\leq 
  \Big  (    
1 +  c_\gamma   {m^{1+\gamma}   \over n }  
\Big )  ^{k}  
\Big [  \big ( 2mq \big )^{1/k}
\big ( 1 + \sqrt{\delta_m  } \big )^2 \Big \{ 
\Big (   1 + \sqrt{ q \over m}
\Big ) ^2
+ \big ( {18 \delta_m ^{1/6}\gamma }\big )^6
\Big \} 
\Big ]^{k}
\end{align*}
Note that $ \big  (     1 +  c_\gamma   {m^{1+\gamma}   \over n }   \big )  ^{k}  \to  1 $ (if $\gamma<1$ and $m^2\le n$, or if $m=o(\log n)$) and
$$
\big ( 1 + \sqrt{\delta_m  } \big )^2 \big \{ \big (   1 + \sqrt{ q \over m}
\big ) ^2
+ \big ( {18 \delta_m ^{1/6}\gamma }\big )^6
\big \}  \to (1 + \sqrt{c})^2$$ as $m\to \infty$.
Now let $m = o ( \sqrt{n} )$. With $k=\lfloor\gamma\log m\rfloor$, $\gamma<1$ (so that $(1+c_\gamma m^{1+\gamma}/n)^k\to1$, as $m^2\le n$), Markov's inequality and the bound for $\E\tr(\widehat\Sigma_n^{*k})$ give
\begin{align*}
\P\big(\|\widehat\Sigma_n^*\|_{S_\infty}>K\big)&\le\frac{\E\tr(\widehat\Sigma_n^{*k})}{K^k}\le2mq\Big(\frac{(1+\sqrt c)^2+o(1)}{K}\Big)^k\\
&=O\Big(m^{2-\gamma\log(K/(1+\sqrt c)^2)+o(1)}\Big).
\end{align*}
Hence $\P(\|\widehat\Sigma_n^*\|_{S_\infty}>K)=o(m^{-l})$ for every $K>(1+\sqrt c)^2e^{(l+2)/\gamma}$; the constant $K_{\rm right}$ in (b1) depends on $l$. If $m=o(\log n)$, $\gamma$ may be chosen arbitrarily large (Lemma \ref{multi} applies since $n\ge m^{1+\gamma/4}$ eventually), which yields the second statement of (b1) for every $l$.

Part (c) follows from the bound on $\E \big [ \tr ( \widehat{\Sigma}_n^{*k} ) \big ]$ as well: for fixed $\ell\in\N$ and $k=k_m\ge\ell$, Jensen's inequality gives $$
\E\Arrowvert\widehat\Sigma_n^*\Arrowvert_{S_\infty}^\ell\le\big(\E \tr ( \widehat{\Sigma}_n^{*k} )\big)^{\ell/k}\le\big((2mq)^{1/k}\{(1+\sqrt c)^2+o(1)\}\big)^{\ell},
$$ 
and $(2mq)^{1/k}\to e^{2/\gamma}$. This completes the proof of Theorem \ref{thm: EW1}(b1) and (c).

 \smallskip
 
{For the proof of part (b2), let $\mathcal D=\{i\le n: W_i\ge1\}$ denote the set of distinct indices in the bootstrap sample and $D=|\mathcal D|$. Since $W_i\ge1$ for $i\in\mathcal D$,
$$
\widehat\Sigma_n^*=\frac1m\sum_{i=1}^nW_i(\Pi_nX_i)(\Pi_nX_i)^\top\succeq\frac Dm\,S_{\mathcal D},\qquad S_{\mathcal D}=\frac1D\sum_{i\in\mathcal D}(\Pi_nX_i)(\Pi_nX_i)^\top,
$$
hence $\lambda_{\min}(\widehat\Sigma_n^*)\ge\frac Dm\lambda_{\min}(S_{\mathcal D})$. The number $m-D$ of draws which hit an index drawn before is stochastically dominated by a binomial distribution with parameters $m$ and $m/n$ (the $t$-th draw hits one of the at most $t-1$ previously drawn indices with probability at most $m/n$). Therefore, with $j_m=\lceil\sqrt m\,\rceil$,
$$
\P(m-D\ge j_m)\le\binom m{j_m}\Big(\frac mn\Big)^{j_m}\le\frac{1}{j_m!}\Big(\frac{m^2}n\Big)^{j_m}=o(m^{-l})\qquad\text{for every }l\in\N,
$$
since $m^2\le n$ and $j_m!$ grows faster than any power of $m$. Conditionally on $\mathcal D$, the vectors $\Pi_nX_i$, $i\in\mathcal D$, are iid with iid (truncated and standardized) entries, because $W$ is independent of $X_1,\dots,X_n$ and of $\Pi_n$. On $\{D\ge m-j_m\}$ we have $q/D\to c$ and $|X_{ij}|\le\sqrt m\,\delta_m\le\sqrt{2D}\,\delta_m$. Hence Theorem~9.13 in \cite{baisilverstein2010}, the bound for the smallest eigenvalue for triangular arrays of independent entries bounded by $\eta_n\sqrt n$ with $\eta_n\to0$ and $\eta_nn^{1/4}\to\infty$, applied with sample size $D$ and $\eta_D=\sqrt2\,\delta_m$ (note that $\delta_mm^{1/8}\to\infty$ by (T)), yields, for every $K'<(1-\sqrt c)^2$ and $l\in\N$, $\P(\lambda_{\min}(S_{\mathcal D})<K'\,|\,\mathcal D)=o(m^{-l})$ uniformly on $\{D\ge m-j_m\}$ - otherwise this bound would fail along a strictly increasing sequence of sample sizes, contradicting that theorem. For $0<K<(1-\sqrt c)^2$, choosing $K'\in(K,(1-\sqrt c)^2)$ and using $D/m\ge1-j_m/m\to1$ on $\{D\ge m-j_m\}$ gives
$$
\P\big(\lambda_{\min}(\widehat\Sigma_n^*)<K\big)\le\P(m-D\ge j_m)+\sup_{\mathcal D:\,D\ge m-j_m}\P\big(\lambda_{\min}(S_{\mathcal D})<K'\,\big|\,\mathcal D\big)=o(m^{-l})
$$
for every $l\in\N$ and all sufficiently large $n$, which proves (b2).}

For the Proof of (a)  it is sufficient to apply the same arguments as  in  Lemma 2.2 and 2.3  of 
\cite{yinbaikri1988}  and the reasoning at beginning of Section 2 in  \cite{baiyin1993}
 to the bootstrap matrix $\widehat \Sigma_n^*$, which gives 
\begin{align}
    \label{revang1}
    \| \widehat \Sigma_n^* - \check \Sigma_n^* \|_{S_\infty} = o_{\P}(1),
\end{align}
 where  $ \check \Sigma_n^*$ denote sample covariance matrix built from the truncated and re-normalized  components  satisfying 
 $| X_{ij} | \leq \sqrt{m} \delta_m $.  The statement (a) is then a consequence of (b). 

 \end{proof}

\begin{proof}[Proof of Corollary \ref{thm: EW1a}]
We begin {with} part (b). {We consider $\widehat\Sigma^*_{n,L_n}=\frac1mL_nX^*X^{*\top}L_n^\top$ (cf.\ the formulation of the corollary); this} 
gives 
\begin{align*}
    {x^\top  \widehat \Sigma_{n,L_n}^* x \over  \|x \|^2 }
    &\geq 
    \lambda_{\rm min} \Big ( {1 \over m} \sum_{i=1}^m X_i^*X_i^{*\top} \Big) 
       {x^\top  L_nL_n^\top  x \over  \|x \|^2 }
       \geq 
    \lambda_{\rm min} \Big ( {1 \over m} \sum_{i=1}^m X_i^*X_i^{*\top} \Big) 
     {\lambda_{\rm min} ( L_nL_n^\top)}. 
\end{align*}
{The assertion follows by applying the proof of Theorem \ref{thm: EW1}(b2) to the $q' \times q'$-matrix ${1 \over m} \sum_{i=1}^m X_i^*X_i^{*\top}$ (along subsequences with $q'/m\to c''\le c'$; for $c''\ge1$ the lower bound is trivial, which explains the indicator), together with
\begin{align*}
\lambda_{\min}(L_nL_n^\top)&\ge\lambda_{\min}(\Pi_n\Sigma_n\Pi_n^\top)-\|\Pi_n\Sigma_n\Pi_n^\top-L_nL_n^\top\|_{S_\infty}\\
&\ge\lambda_{\min}(\Sigma_n)-2\|L_n\|_{S_\infty}\|R_n\|_{S_2}-\|R_n\|_{S_2}^2,
\end{align*}
which follows from $\Pi_n\Sigma_n\Pi_n^\top=(L_n+R_n)(L_n+R_n)^\top$ and Cauchy's interlacing theorem. (Note that $L_nL_n^\top\neq\Pi_n\Sigma_n\Pi_n^\top$ in general; \eqref{hol1} is only an approximation in the sense of spectral distributions.) Part (a) follows from 
$$
\|\widehat\Sigma^*_{n,L_n}\|_{S_\infty}\le\|L_n\|_{S_\infty}^2\|\frac1m\sum_{i=1}^mX_i^*X_i^{*\top}\|_{S_\infty},
$$
Theorem \ref{thm: EW1}(b1) and 
$$
\|L_n\|_{S_\infty}\le\|\Pi_nA_n\|_{S_\infty}+\|R_n\|_{S_2}\le\|\Sigma_n\|_{S_\infty}^{1/2}+\varepsilon_n
$$
(note that $\|L_n\|_{S_\infty}\le\|\Pi_nA_n\|_{S_\infty}$ need not hold, since setting entries to zero may increase the spectral norm).}
\end{proof}

\section{Proof of Theorem   \ref{thm: new}}
\label{sec8} 

  \def\theequation{D.\arabic{equation}}	
	\setcounter{equation}{0}

Throughout this section, we assume that Assumptions (A1) -- (A3+) are satisfied.
All proofs have in common the truncation steps in \eqref{ar100}   and \eqref{eq: trunk0} discussed in the following  Section \ref{sec82} and \ref{sec82aa}, respectively. 

 \subsection{Reduction to $L_n$} \label{sec82} 

Recalling the notation from
Section \ref{sec52}  we will first prove that  we can replace the matrix $\Pi_n A_n$ in the decomposition \eqref{decompproj}
by the matrix $L_n$, that is 
\begin{align}
\label{ar100}
 &  \hat{T}_n^*(f)-\hat {T}_{n,L_n}^*(f)\rightarrow_{\P}0   
\end{align}
where   
 $ \hat {T}_{n,L_n}^*(f)$ 
denotes the linear spectral statistics corresponding to the matrix
$$
\widehat \Sigma_{n,L_n}^*:=\frac{1}{m} L_n X^*X^{*\top} L_n^\top, 
$$ 
and   $ L_n X^* = (L_nX^*_1, \ldots ,  L_nX^*_m) \in \mathbb{R}^{q \times m} $. Define 
\begin{align*}
\mathcal{D}_n =\Big \{  
\lambda_{\min }  (\widehat \Sigma_{n}^* ) > K_{\rm left} ,~
\lambda_{\min }  (\widehat \Sigma_{n,L_n}^* ) > K_{\rm left} ,~
 \| \widehat \Sigma_{n,L_n}^*   \|_{S_\infty } <  K_{\rm  right} ,~
  \| \widehat \Sigma_{n}^*   \|_{S_\infty } <  K_{\rm  right} \Big \} 
\end{align*}
where $K_{\rm left}$ and $K_{\rm  right}$ are the constants from Corollary~\ref{thm: EW1a}, with $K_{\rm right}$ enlarged by some $\eta>0$ such that $f$ is still analytic in a neighborhood of $[K_{\rm left},K_{\rm right}]$. Since the variables have not been truncated at this stage, Corollary~\ref{thm: EW1a} is not directly applicable. However, we may apply the proof of Theorem~\ref{thm: EW1}(a) to the $q'\times q'$ matrix $\frac1m\sum_{i=1}^mX_i^*X_i^{*\top}$ along subsequences with $q'/m\to c''\le c'$ and combine it with the two norm bounds from the proof of Corollary~\ref{thm: EW1a}. This shows that
$$
\lambda_{\min}(\widehat\Sigma^*_{n,L_n})>K_{\rm left}'\quad\text{and}\quad\|\widehat\Sigma^*_{n,L_n}\|_{S_\infty}<K_{\rm right}'
$$
with probability tending to one for some constants $K_{\rm left}'>K_{\rm left}$ and $K_{\rm right}'<K_{\rm right}$. Such constants exist because $K_{\rm left}$ lies strictly below the threshold in Corollary~\ref{thm: EW1a}(b) and $K_{\rm right}$ has been enlarged. To transfer these bounds to $\widehat\Sigma^*_n$, we use the estimates in Section~\ref{sec52}, which give
\begin{align}
\E\big\|\widehat\Sigma^*_{n,L_n}-\widehat\Sigma^*_n\big\|_{S_1}&\le\E_{\Pi_n}\|R_n\|_{S_2}^2+2\,\E^{1/2}_{\Pi_n}\|R_n\|_{S_2}^2\,\E^{1/2}_{\Pi_n}\|L_n\|_{S_2}^2\nonumber\\
&=o(q^{-1})+o(q^{-1/2})\,O(q^{1/2})=o(1)\label{ar102}
\end{align}
under the additional condition $\E_{\Pi_n}\|R_n\|_{S_2}^2=o(q^{-1})$ of Theorem~\ref{thm: new}. By Weyl's inequality and \eqref{ar102}, we obtain
$$
\max_{1\le i\le q}\big|\lambda_i(\widehat\Sigma^*_n)-\lambda_i(\widehat\Sigma^*_{n,L_n})\big|\le\big\|\widehat\Sigma^*_n-\widehat\Sigma^*_{n,L_n}\big\|_{S_\infty}\le\big\|\widehat\Sigma^*_n-\widehat\Sigma^*_{n,L_n}\big\|_{S_1}=o_\P(1),
$$
and consequently $\P(\mathcal D_n^c)\to0$.
 By the Lipschitz continuity of $f$, the $1$-Wielandt-Hoffman inequality,  and Assumption (A1), it follows that  
\begin{align*}
\big\arrowvert 
\hat {T}_{n,L_n}^* (f ) 
- \hat{T}_n^*(f)
\big\arrowvert
 &
 \leq {\max}_{ \lambda \in [K_{\rm left}, K_{\rm  right}] }  | f'(\lambda ) | 
\sum_{j=1}^q 
\big\arrowvert {\hat \lambda}_{j,L_n}^*-\hat{\lambda}_j^*\arrowvert  + o_\mathbb{P}  (1)\\
& 
\leq {\max}_{ \lambda \in [K_{\rm left}, K_{\rm  right}] }  | f'(\lambda ) | 
\big\Arrowvert \widehat \Sigma_{n,L_n}^{*}-\widehat \Sigma_n^*\big\Arrowvert_{S_1}
+ o_\mathbb{P}  (1) ~,
\end{align*}
 $\hat \lambda_{j,L_n}^*$ is the $j$th eigenvalue of the matrix $ \widehat \Sigma_{n,L_n}^{*}$. {By \eqref{ar102}, the last line is $o_\P(1)$, which proves \eqref{ar100}. By the same estimate, $\sum_{i=1}^q|\lambda_i(L_nL_n^\top)-\lambda_i(\Pi_n\Sigma_n\Pi_n^\top)|\le\|L_nL_n^\top-\Pi_n\Sigma_n\Pi_n^\top\|_{S_1}=o_\P(1)$. Since $\underline m^0_{\gamma,G}(z)$ depends on $G$ only through $\int t(1+t\underline m)^{-1}\dd G(t)$, whose integrand is Lipschitz in $t$ on a neighborhood of the spectra for $z$ on the contour, this yields 
$$
q\big|\int f\,\dd\mu^0_{q/m,\mu^{L_nL_n^\top}}-\int f\,\dd\mu^0_{q/m,\mu^{\Pi_n\Sigma_n\Pi_n^\top}}\big|=o_\P(1),
$$ 
so that the centering below may be computed with $\tilde\Sigma_n=\Pi_n\Sigma_n\Pi_n^\top$.}

Therefore, we will assume in the following discussion  
that given the random projection $\Pi_n$
the matrix $\widehat{\Sigma}_{n}^{*} $ can be represented as
$$
\widehat{\Sigma}_{n}^{*}  = 
{1 \over m } L_nX^*X^{*\top} L_n^\top ~,
$$
 where $L_n$  is a $q \times q'$ matrix 
 satisfying $\| L_n \|_{S_\infty} \leq \alpha < \infty $ (for all $n \in \mathbb{N}$)
 and $X^* $ is an $  q'\times m$ matrix and $q' =  O(q)$.
Note that these arguments only require the existence of moments of order $4$.

\subsection{Reduction to truncated components}
\label{sec82aa}
We will continue truncating the random variables  $X_{ij}$. For this purpose we formulate the following lemma.

  \begin{lemma}\label{lemma: delta_n}
  There exists a sequence $(\delta_m)_{m\in\N}$ converging decreasingly to zero such that
\begin{align} \label{hol101}
  \delta_m^{-4}\E\Big(I_{\{\arrowvert X_{11}\arrowvert\geq\delta_mm^{1/2}\} }X_{11}^4\Big)\longrightarrow 0\ \ \text{as}\ \ n\rightarrow\infty.  
\end{align}
 \end{lemma}
 
 \begin{proof}
The proof of \eqref{hol101} is given on page 559 in \cite{baisil2004}, that we repeat here for the reader's convenience. First observe that there exists a strictly  increasing sequence $(m_k)_{k\in\N}$ with
 $$
 k^4\E\Big( I\big\{\arrowvert X_{11}\arrowvert\geq m_k^{1/2}/k\big\}X_{11}^4\Big)<\frac{1}{2^k}
 $$
 by dominated convergence, because $\E X_{11}^4<\infty$. Choose $\delta_m=1/k$ for $m\in[m_k,m_{k+1})$, $\delta_m=1$ for $m<m_1$. Then, $\delta_m\searrow 0$ and
 $$
 \delta_m^{-4}\E\Big(I_{\{\arrowvert X_{11}\arrowvert\geq\delta_mm^{1/2}\}}  X_{11}^4\Big)\le k^4\E\Big( I\big\{\arrowvert X_{11}\arrowvert\geq m_k^{1/2}/k\big\}X_{11}^4\Big)<\frac1{2^k}\ \ \text{for}\ \ m\in[m_k,m_{k+1}).
 $$ 
 Hence \eqref{hol101} holds as $m\to\infty$ and, since $m=m_n\to\infty$, as $n\to\infty$.
 \end{proof}

We choose the sequence  $(\delta_m)$ 
such that  Lemma  \ref{lemma: delta_n} holds and additionally 
such that (replace $\delta_m$ by $\max(\delta_m,m^{-1/16})$, which preserves \eqref{hol101}, since its left-hand side is nonincreasing in $\delta_m$) 
\begin{align}
\label{ar7} 
\delta_m m^{1/8}\rightarrow\infty ~.
\end{align}
With this sequence, we   
show that it is sufficient 
to consider random variables which satisfy 
\begin{align}
\arrowvert  X_{ij}\arrowvert&\leq \delta_m\sqrt{m}\ \ i=1,\dots n,\ j=1,\dots, p\label{eq: trunk1}\\
\E \ X_{11}&=0\\
\Var ( X_{11})&=1\\
\E   X_{11}^4 &\rightarrow  3\label{eq: trunk4}.
\end{align}
For this purpose, we introduce the notation
\begin{eqnarray*}
\check{X}_{ij} = 
\frac{
X_{ij} I \big \{ |X_{ij}| \leq \delta_m \sqrt{m} \big \} - \mathbb{E} \big[ X_{ij} I \big  \{|X_{ij}| \leq  \delta_m \sqrt{m} \big \} \big ] }
{ \sqrt{{\rm Var } \big (  X_{ij} I \big \{ |X_{ij}| \leq  \delta_m \sqrt{m} \big \} 
\big )  } }. 
\end{eqnarray*}
We write 
$$
 \check {\Sigma}_n^{*}:=\frac{1}{m}\sum_{i=1}^m L_n \check  X_i^{*}{\check {X}^{*\top}_i} L_n^\top  
$$ 
and denote by  
 $\check {\lambda}_i^*$  ($i=1,\dots, q$)   its eigenvalues in decreasing order and by 
$$
\check {T}_n^*(f)=\sum_{i=1}^{q}f(\check {\lambda}_i^*)
$$
the corresponding linear spectral statistic.

 \begin{lemma}[Bootstrap truncation lemma]
 \label{lemma: BTL} Grant Assumptions (A1)--(A3+). Then 
  \begin{align}
& \hat{T}_n^*(f)-\check {T}_n^*(f)\rightarrow_{\P}0.\label{eq: trunk0}   
  \end{align}  
\end{lemma}

 \begin{proof}[Proof of Lemma  \ref{lemma: BTL}]

For the sequence  $(\delta_m)$ specified  above we 
set
$$
\tilde{X}_{ij}:=X_{ij}I\{\arrowvert X_{ij}\arrowvert <\delta_m\sqrt{m}\} 
$$ 
and denote by $\tilde{X}_1^*,\dots, \tilde {X}_{{m}}^*$ the corresponding iid sample from $n^{-1}\sum_{i=1}^n\delta_{{\tilde{X}_i}}$. With
\begin{align}
\label{hol201}
\widehat \Sigma_n^*&=\frac{1}{m}\sum_{i=1}^mY_i^*Y_i^{* \top }=\frac{1}{m}\sum_{i=1}^m
L_n {X}_i^*{X}_i^{* \top }L_n^\top  
\end{align}
and
\begin{align}
\label{hol202}
\tilde{B}_n^*&=\frac{1}{m}\sum_{i=1}^m L_n \tilde {X}_i^*{\tilde {X}_i^{*\top}} L_n^\top  
\end{align}
we get by the union bound and Lemma \ref{lemma: delta_n} that
\begin{align}
\label{det112222}
\P(\widehat \Sigma_n^*\not=\tilde {B}_n^*)&\leq\P\Big (\tilde{X}_{ij}^*\not= X_{ij}^* \text{ for some } (i,j)\Big ) \leq mq'\,\P\big(\tilde {X}_{11}^*\not={X}_{11}^*\big)\\
&= mq' \,\E\Big (\frac{1}{n}\sum_{i=1}^nI\{\arrowvert X_{i1}\arrowvert\geq\delta_m\sqrt{m}\}\Big )
 = mq' \,\P(\arrowvert X_{11}\arrowvert>\delta_m\sqrt{m})\nonumber \\
&\leq K\delta_m^{-4}\int_{\{\arrowvert X_{11}\arrowvert\geq\delta_m\sqrt{m}\}}\arrowvert X_{11}\arrowvert^4 d \mathbb{P} =o(1). \nonumber
\end{align}
Now, we are passing over
from $\tilde {X}_1,\dots, \tilde {X}_n$
to  the centered and standardized modifications
$\check {X}_1,\dots, \check{X}_n$, where
$$
\check{X}_{ij}:=\frac{\tilde {X}_{ij} -\E\tilde{X}_{ij}}{\sigma_n}\ \ \text{ with }\sigma_n:=\sqrt{\Var(\tilde{X}_{ij})},
$$
and denote by $\check {X}_1^*,\dots,\check {X}_m^*$ the corresponding iid sample from $n^{-1}\sum_{i=1}^n\delta_{{\check{X}_i}}$. 
Further, we introduce 
 $\check {\X}^*=(\check X_1^*, \dots , \check  {X}_m^*)$ and $ {\tilde\X}^*=(\tilde X_1^* , \dots , \tilde X_m^*)$, 
$$
\check {\Sigma}_n^*=\frac{1}{m}\sum_{i=1}^mL_n \check {X}_i^*\check {X}_i^{*\top}L_n^\top 
= {\frac1m L_n \check \X ^*  \check \X ^{*\top} L_n^\top} ,
$$
and denote by   $\check {\lambda}_1^*\geq \check {\lambda}_2^*\geq \dots\geq \check{\lambda}_q^*$  its eigenvalues
with corresponding linear spectral statistic 
$$
\check{T}_n^*(f)=\sum_{j=1}^q f(\check\lambda _j).
$$
By employing the Lipschitz continuity of $f$, the $1$-Wielandt-Hoffman inequality and the reasoning of the proof of  Lemma 2.7 in 
 \cite{bai1999}, we deduce the upper bound
\begin{align} 
\nonumber 
\big\arrowvert\check  {T}_n^*(f)- \tilde {T}_n^*(f)\big\arrowvert&\leq 
{\max}_{\lambda \in  [K_{\rm left}, K_{\rm  right}]} 
 | f'(\lambda) | \cdot   
\sum_{j=1}^q \big\arrowvert \check {\lambda}_j^*-\tilde {\lambda}_j^*\arrowvert + o_\P (1) \\
&\leq {2}
{\max}_{\lambda \in  [K_{\rm left}, K_{\rm  right}] }
 | f'(\lambda) | \cdot    \Big(\frac{1}{m}\tr\Big[L_n ({\tilde  \X}-\check{\X})({\tilde{\X}}-\check{\X})^\top L_n ^\top\Big]\Big)^{1/2}
\nonumber \\
&\hspace{50mm}\times\Big(\tr\big(\tilde{B}_n^* \big)+\tr\big(\check {\Sigma}_n^*\big)\Big)^{1/2} + o_\P (1)
,
\label{eq: b1}
\end{align}
where  $
\tilde {T}_n^*(f)=\sum_{j=1}^q f(\tilde \lambda _j) $ denotes the linear spectral statistic  corresponding to the matrix  $\tilde{B}_n^* = {\frac1m L_n \tilde {\X}^* \tilde {\X}^{* \top } L_n^\top}$.
In order to bound the latter expression, observe that
\begin{align*}
\frac{1}{m}& \tr\Big[L_n (\tilde{\X}-\check{\X})(\tilde{\X}-\check{\X})^\top  L_n ^\top \Big]\\
&\leq {2\Big(1-\frac{1}{\sigma_n}\Big)^2 q \Arrowvert \tilde {B}_n^*\Arrowvert_{S_\infty} + \frac{2}{\sigma_n^2}\big|L_n \E \tilde{X}_{1}\big|^2}\\
&\leq {2q} \frac{(\sigma_n^2-1)^2}{\sigma_n^2(1+\sigma_n)^2}\Arrowvert \tilde {B}_n^*\Arrowvert_{S_\infty} + \frac{{2q'}}{\sigma_n^2}(\E \tilde {X}_{11})^2 \Arrowvert L_n\Arrowvert_{S_\infty} ^2.
\end{align*}
But
\begin{align}
\arrowvert \sigma_n^2-1\arrowvert&\leq 2\E\Big(I\{\arrowvert X_{11}\arrowvert\geq \delta_m\sqrt{m}\}\arrowvert X_{11}\arrowvert^2\Big)\nonumber =o(\delta_m^{2}m^{-1})\nonumber
\intertext{and}
\arrowvert \E\tilde {X}_{11}\arrowvert&=o(\delta_mm ^{-3/2}),\label{eq: Edelta}
\end{align}
such that
\begin{align*}
\frac{1}{m}& \tr\Big[L_n (\tilde{\X}-\check{\X})(\tilde{\X}-\check{\X})^\top L_n ^\top \Big]
={o(\delta_m^4m^{-1})}\Arrowvert \tilde {B}_n^*\Arrowvert_{S_\infty}+{o(\delta_m^2m^{-2})}\Arrowvert L_n\Arrowvert_{S_\infty}^2 .
\end{align*}
Plugging this bound into \eqref{eq: b1}, we find
\begin{align}
\big\arrowvert\check  {T}_n^*(f)- \tilde {T}_n^*(f)\big\arrowvert&
\leq \mathcal{O}_{\P} 
 (1)  {\big (o(\delta_m^2m^{-1/2})\sqrt{\Arrowvert \tilde{B}_n^*\Arrowvert_{S_\infty}}+o(\delta_m m^{-1})\big)} 
O_{\P}(\sqrt{q} ) , 
\label{eq: tro}
\end{align}
where we used the fact that, by  Theorem \ref{thm: EW1} and \eqref{det112222}, 
 $  
 {\max}_{\lambda \in [K_{\rm left}, K_{\rm  right}]} 
 | f'(\lambda) |
 = \mathcal{O}_{\P} (1) $  and 
\begin{align*}
\tr\big(\tilde {B}_n^* \big ) 
+ \tr\big(\check{\Sigma}_n^*\big)  =\mathcal{O}_{\P}({q}).
\end{align*}
{Since $\|\tilde B_n^*\|_{S_\infty}=O_\P(1)$ and $q=O(m)$, the right-hand side of \eqref{eq: tro} is $o_\P(\delta_m^2)+o_\P(\delta_m m^{-1/2})=o_\P(1)$.} This proves \eqref{eq: trunk0}. Note that these arguments only require the existence of the moments  order $4$. 
  \end{proof}

  Summarizing the discussion of Section \ref{sec82} and \ref{sec82aa}, we will from now assume that the random variables $X_{ij}$ {satisfy Condition (T) of Section~\ref{sec3} (with $\delta_m$ replaced by $2\delta_m$), and in particular} 
 satisfy \eqref{eq: trunk1} - \eqref{eq: trunk4}, that the vectors $X_i$ have $q'=O(q)$ components and that  the matrix $L_n$ is of dimension $q \times q' $.
 Note that the matrix $L_n$ can be a random matrix which is independent of $X_1, \ldots X_n$.

\subsection{Passing over to   the  bootstrapped  process of Stieltjes transforms}
\label{sec83}

Define
  \begin{align*}
{M}_n^*(z) &:= q \big(m_{\mu^{\widehat \Sigma_n^*}}(z)-m_n^0 (z) \big), 
\end{align*}
where  $m_n^0 (z) $ denotes the Stieltjes transform of the measure {$\mu^0_{q/m,\mu^{\tilde\Sigma_n}}$ (conditionally on $\Pi_n$)}. Moreover, 
$
\mathcal{D}_n =\big \{  
\lambda_{\min }  (\widehat \Sigma_{n}^* ) > K_{\rm left} ~,~
  \| \widehat \Sigma_{n}^*   \|_{S_\infty } <  K_{\rm  right} \big \}.
$ 
 By the relation
\begin{align}
    \label{eq: red 1}
\hat{T}_n^*(f) - 
q \int f d \mu^0_{{q/m},\mu^{{\tilde\Sigma_n}}}  = -\frac{1}{2\pi i}\ointctrclockwise f (z){M}_{n}^*(z)  \dd z 
\end{align}
provided by the Cauchy integral formula, it follows from Corollary \ref{thm: EW1a}  that the result is deduced from the corresponding limit theorem for
$$
\ointctrclockwise f(z){M}_{n} ^*(z) \mathds{1}_{\mathcal{D}_n} \dd z,
$$
where the curve integral is along any closed curve within a
 region on which $f$ is analytic and which encloses the interval $[K_{\rm left},K_{\rm  right}]$. The latter indeed boils down to proving a conditional Donsker-type theorem of a truncated version of the bootstrapped process $M_n^*(.)$, denoted by $\widehat{M}_n^*(.)$, 
see Section~\ref{sec84}. For a precise definition of $\widehat{M}_n^*(.)$ let $x_l$, $x_r$ be two real numbers with
\begin{align*}
x_l\in\begin{cases}
(0, K_{\rm left} ) & \text{if }K_{\rm left}>0\\
(-\infty,~0 ) & \text{if } K_{\rm left}\le0
\end{cases}
\end{align*}
and $
x_r> K_{\rm  right}
$, where $K_{\rm left}$ and $K_{\rm  right}$ are the constants introduced in Corollary \ref{thm: EW1a} and $$c' = \limsup_{n\to \infty}\frac{ q'}{m}  .
$$
Moreover, define
$
\mathcal{C}_u=\{ x+iv_0:~x\in [x_l,x_r] \}
$
and
\begin{align*}
\mathcal{C}=\{x_l+iv:v\in[0,v_0]\}\cup \CC_u \cup \{x_r+iv: v\in [0,v_0]\}
\end{align*}
such that the closed curve $\CC \cup \bar \CC$ is contained in a region where $f$ is analytic. Further, for some null sequence $(\varepsilon_n)$ satisfying
\begin{align}
\label{det105}
\varepsilon_n\geq m^{-\alpha}
\end{align}
for some $\alpha\in (0,1)$,  we introduce 
\begin{align}
\CC_{l,n}&:=\begin{cases}
\{x_l+iv:~v\in[m^{-1}\varepsilon_n,v_0]\}, & \text{if }x_l>0\\
 \{x_l+iv:~v\in[0,v_0]\} & \text{if }x_l\leq 0,
\end{cases}\\
\nonumber
\CC_{r,n}&:=\{x_r+iv:~v\in[m^{-1}\varepsilon_n,v_0]\}
\intertext{and}
\label{det11}
\CC_n&:=\CC_{l,n}\cup\CC_u\cup\CC_{r,n}.
\end{align}
Lastly, we define 
\begin{align*}
\widehat M_n^*(z)=\begin{cases}
M_n^*(z), &\text{for }z\in\CC_n,\\
M_n^*(x_r+im^{-1}\varepsilon_n), & \text{for }z=x_r+iv \text{ with }v\in [0,m^{-1}\varepsilon_n] \\
M_n^*(x_l+im^{-1}\varepsilon_n), & \text{for }z=x_l+iv\text{ with }x_l>0 \text{ and }v \in[0,m^{-1}\varepsilon_n] ,
\end{cases}
\end{align*}
and because of $m(\bar{z})=\overline{m(z)}$ for any Stieltjes transform $m$, we have 
$
\widehat{M}_n^*(\bar{z}):=\overline{\widehat M_n^*(z)}$ for $z\in\mathcal{C}$. 
Since 
\begin{align} 
\widehat{T}_n^*(f)\mathds{1}_{\mathcal{D}_n}& - 
q \int f d \mu^0_{{q/m},\mu^{{\tilde\Sigma_n}}} \mathds{1}_{\mathcal{D}_n}  
\nonumber \\
=&{-\frac{1}{2\pi i}}\ointctrclockwise_{\CC \cup \bar\CC} f(z) M_{n}^*(z) \mathds{1}_{\mathcal{D}_n}
\dd z \nonumber\\
=& {-\frac{1}{2\pi i}}\ointctrclockwise_{\CC \cup \bar\CC} f(z) \widehat M_{n}^*(z) \mathds{1}_{\mathcal{D}_n} \dd z\nonumber\\
&+\mathcal{O}\bigg(8 \varepsilon_n {q \over m}   \| f \|_{\CC\cup\bar\CC} \Big( \big | 
K_{\rm  right}
-x_r \big |^{-1}+
\big | 
K_{\rm left}
-x_l  \big |^{-1} \Big)\bigg)\nonumber\\
=&{-\frac{1}{2\pi i}}\ointctrclockwise_{\CC \cup \bar\CC} f(z) \widehat M_n^*(z)\dd z + o_{\P}(1)\label{eq: AR1005}
\end{align}
by Corollary \ref{thm: EW1a},  is sufficient to consider $\widehat M_n^*$ in what follows.  
The essential part of the proof of Theorem \ref{thm: new} consists of verifying the following Donsker-type result.  Gaussianity of  \eqref{eq: red 1} 
then follows with \eqref{eq: AR1005} from the continuous mapping Theorem. 

\begin{proposition}[Functional CLT for the conditional process $\widehat {M}_n^{*}$ in probability]\label{prop}
Grant the conditions of  Theorem \ref{thm: new}, then 
\begin{align*}
d_{BL}\bigg\{\LL\Big(\big( \widehat {M}_n^{*}(z) - \E^* \big [ \widehat {M}_n^{*}(z) \big ] \big)_{z\in\CC}\Big\arrowvert X_1,\dots,X_n, \Pi_n \Big),\, \LL(Z)\bigg\}\longrightarrow_{\P} 0
\end{align*}
with a  centered Gaussian process $(Z)$ on $\CC$ satisfying $Z(\bar z)=\overline{Z(z)}$ and
\begin{align*}
\E\big[Z(z_1) Z(z_2)\big]= 2 \frac{(\underline{m}_{c,H}^0)'(z_1)(\underline{m}_{c,{H}}^0)'(z_2)}{\big(\underline{m}_{c,H}^0(z_1)-\underline{m}_{c,{H}}^0
(z_2)\big)^2}-\frac{2}{(z_1-z_2)^2} \ \text{ for }z_1,z_2\in\mathcal{C}
\end{align*}
(understood as its continuous extrapolation for the removable singularities at $z_1=z_2$).
\end{proposition}

\begin{proposition} \label{revprop1}
Grant the conditions of  Theorem \ref{thm: new}, then  
$$
\sup_{z \in {\cal C}_n }
\bigg |  
 \E^* \big [ \widehat {M}_n^{*}(z) \big ]   
- c  \int \frac{ (\underline{m}_{c,H}^0(z) )^3 t^2 \, d 
{H}(t) }{(1 + t  \underline{m}_{c,H}^0(z) )^3 }
\Big [ 
 1 - c 
 \int \frac{ (\underline{m}_{c,H}^0(z) )^2 t^2 \, d {H}(t) }{(1 + t  \underline{m}_{c,H}^0(z) )^2 }   \Big]^{-2} \bigg | = o_{\P} (1). 
$$
    
\end{proposition}

A central  tool in the proof of this result is  an
analog of Proposition \ref{lemma: formel 2.1} for the truncation in Section~\ref{sec82aa} with a bound which is not depending on $z\in\mathcal{C}_n$.
These results will be presented first in the following section. The proof of Proposition \ref{prop}
is given in Section \ref{sec84}, while Proposition \ref{revprop1} is proved in Section \ref{proofrevprop1}. 

{To conclude the proof, let $d(f)$ denote the limit of $d_n(f)$ from Theorem~\ref{thm: new}, i.e.\ $d_n(f)$ with $(p/n,\mu^{\Sigma_n})$ replaced by $(c,H)$; by Proposition~\ref{revprop1}, $\E^*\big[-\frac1{2\pi i}\ointctrclockwise f(z)\widehat M_n^*(z)\,\dd z\big]\to_\P d(f)$. Together with \eqref{eq: red 1}, \eqref{eq: AR1005} and Proposition~\ref{prop},
$$
d_{BL}\Big[\LL\Big(\hat T_n^*(f)-q\int f\,\dd\mu^0_{q/m,\mu^{\tilde\Sigma_n}}-d(f)\,\Big|\,Y_1,\dots,Y_n,\Pi_n\Big),\ \mathcal N\big(0,\sigma^2(f)\big)\Big]\longrightarrow_\P0,
$$
where $\sigma^2(f)=-\frac{1}{4\pi^2}\ointctrclockwise\ointctrclockwise f(z_1)f(z_2)\,\E[Z(z_1)Z(z_2)]\,\dd z_1\,\dd z_2$ is the limiting variance of \cite{baisil2004}. The CLT of \cite{baisil2004} (for model \eqref{eq: Annahme_an_Y} see \cite{zou2022clt}) yields 
$$
\hat T_n(f)-p\int f\,\dd\mu^0_{p/n,\mu^{\Sigma_n}}-d_n(f)\Rightarrow\mathcal N(0,\sigma^2(f))
$$ and, conditionally on $
\Pi_n$, 
$$T^{(j)}_{n,m}(f,\Pi_n)-q\int f\,\dd\mu^0_{q/m,\mu^{\tilde\Sigma_n}}\Rightarrow\mathcal N(d(f),\sigma^2(f))
$$ for each block. Theorem~\ref{thm: new} follows once
$$
\overline T^{\Pi_n}_{n,m}(f)-q\int f\,\dd\mu^0_{q/m,\mu^{\tilde\Sigma_n}}-d(f)=o_\P(1)
$$
is established. To see this, write $Z_{n,j}=T^{(j)}_{n,m}(f,\Pi_n)-q\int f\,\dd\mu^0_{q/m,\mu^{\tilde\Sigma_n}}$, $j=1,\dots,m$. Conditionally on $\Pi_n$, these are iid (disjoint blocks of iid observations) with distribution function $F_n$, and $F_n\Rightarrow G$ in probability, where $G$ is the distribution function of $\mathcal N(d(f),\sigma^2(f))$. Since $\overline T^{\Pi_n}_{n,m}(f)-q\int f\,\dd\mu^0_{q/m,\mu^{\tilde\Sigma_n}}$ is the $\lceil m/2\rceil$-th order statistic of $Z_{n,1},\dots,Z_{n,m}$, we have for $\eta>0$, with the empirical distribution function $\hat F_m$ of $Z_{n,1},\dots,Z_{n,m}$,
\begin{align*}
&\Big\{\Big|\overline T^{\Pi_n}_{n,m}(f)-q\int f\,\dd\mu^0_{q/m,\mu^{\tilde\Sigma_n}}-d(f)\Big|>\eta\Big\}\\
&\qquad\subset\Big\{\hat F_m\big(d(f)-\eta\big)\ge\tfrac12\Big\}\cup\Big\{\hat F_m\big(d(f)+\eta\big)\le\tfrac12\Big\}.
\end{align*}
By Chebyshev's inequality, $\P\big(|\hat F_m(x)-F_n(x)|>\varepsilon\,\big|\,\Pi_n\big)\le(4m\varepsilon^2)^{-1}$ for every $x$, and $F_n(d(f)\pm\eta)\to G(d(f)\pm\eta)$ in probability, since $d(f)\pm\eta$ are continuity points of $G$ (also if $\sigma^2(f)=0$). As $G(d(f)-\eta)<\frac12<G(d(f)+\eta)$, both events have conditional probability tending to zero in probability, which proves the claim.}

\section{Proofs of Propositions \ref{prop} and \ref{revprop1}}

 \def\theequation{E.\arabic{equation}}	

\subsection{Non-standard results on quadratic forms} \label{sec82a} 

For the statement of  an
analog of Proposition \ref{lemma: formel 2.1} for the truncation in Section \ref{sec82aa} we study the following matrices  in the quadratic form
   \begin{eqnarray}
   \label{eq: C_i0}
   C^*  &=&
    C^*  (z) =
   L_n^\top  D_1^*(z)^{-1} L_n  ,   \\
   \label{eq: Cende}
    C^*  &=&
    C^*  (z) =
   L_n^\top  D_1^*(z)^{-1}M L_n  ,   \\     \label{eq: C_i}
   C^*  &=&
    C^*  (z) =
   L_n^\top  D_1^*(z)^{-2} L_n  ,   \\
 \label{eq: C_ia}  
    C^*  & =& 
      C^*  (z_1,z_2) =
      L_n ^\top D_{12}^*(z_1)^{-1} 
 L_nL_n^\top D_{12}^*(z_2)^{-1} 
  L_n  ~,
  \\
 \label{eq: C_ib}  
    C^*  & =& 
      C^*  (z_1,z_2) = L_n ^\top  \E_j^* \big [  D_{1j}^*(z_1)^{-1} \big ]
 L_n L_n^\top D_{1j}^*(z_2)^{-1} 
  L_n  ~, \\
    \label{eq: C_ic}
   C^*  &=&   C^*  (z) =
   L_n^\top  D_1^*(z)^{-2} L_n  ,   \\
    C^*  &=&
     \label{eq: C_id}  
    C^*  (z_1,z_2) =
   L_n^\top  D_1^*(z_1)^{-2}   D_1^*(z_2)^{-1} L_n ,  \\
      C^*  &=&
     \label{eq: C_ie}  
    C^*  (z_1,z_2) =
   L_n^\top  D_1^*(z_1)^{-1}   D_1^*(z_2)^{-1} L_n  
   \end{eqnarray}
 for $z, z_1,z_2\in\mathcal{C}_n$.  $M\in \C^{q\times q}$ is deterministic and of bounded spectral norm, uniformly in $n$. Recall that the notation $\E_X$ means integration with respect to $X=(X_1, \ldots , X_n)$. In other words (as the projection is independent of  $X$), the expectation is taken conditional on $\Pi_n $.

 \begin{proposition}
 \label{lemma: formel 2.1a}
For any 
 $p\geq 2$, there exists some constant $K_p>0$, such that for any $n\in \N$, \begin{align*}
 \E_X
 \big\arrowvert X_1^{*\top}
 C^*X_1^*-\tr C^* \big\arrowvert^p  \leq  
  K_p  
 \Big( m^{p-1} \delta_m^{2p-4} + 
 \frac{m^{p+1}}{n}
 \Big)
 ~, 
\end{align*}
where $K_p$ is a constant depending 
only on $p$ and the matrix 
$C^*$ is   given by  one of the matrices in   {\eqref{eq: C_i0}} - \eqref{eq: C_ie}. (The constant $\ell_0$ in \eqref{det33} below has to be chosen depending on $p$ - only finitely many values of $p$ are needed in the sequel.)
\end{proposition}

\begin{proof}
We 
denote by 
  $\widehat \Sigma_{n, 1, \ldots ,j}^* $ the matrix   which  is obtained  from $\widehat \Sigma_{n}^*$ by omitting the terms involving $X_1^* , \ldots , X_j^*$
  (for $j=0$ this is ${\widehat\Sigma_n^*}$) and the matrix 
  $\tilde \Sigma_{n, 1}^*$
  is the empirical covariance matrix of the  vectors $\tilde{X}_2^*,\dots, \tilde{X}_m^*$ with normalizing  factor $1/m$ which  are defined in \eqref{eq: xsterntilde}.  By Corollary \ref{thm: EW1a}  (with an adaptation of its proof to the matrices $\widehat \Sigma_{n, 1}  ^*, \widehat \Sigma_{n, 1,2}  ^*$
we find constants  $K_{\rm left}$ and $K_{\rm  right}$  
such that 
 the event
\begin{align}
    \label{det21}
    \mathcal{A}_n:= 
    \Big \{  
    &   \lambda_{\min} (\tilde \Sigma_{n, 1}^*) \geq K_{\rm left},
    ~  \| \tilde \Sigma_{n, 1}  ^* \|_{S_\infty }\leq K_{\rm  right},  ~
    \\
   & \nonumber
    ~~~~~~~~  ~~~~~~~~  ~~~~~~~~ 
 \lambda_{\min} (\widehat \Sigma_{n, 1, \ldots ,j}  ^*) \geq K_{\rm left},~
    \| \widehat \Sigma_{n, 1, \ldots ,j}  ^* \|_{S_\infty}   \leq K_{\rm  right} ~\text{for } j=0,1,2 
    \Big \} ,
\end{align}
satisfies {for a fixed $\ell_0\in\N$, chosen large enough for the polynomial factors in the estimates below (with $K_{\rm right}=K_{\rm right}(\ell_0)$ from Corollary~\ref{thm: EW1a}(a); a single $K_{\rm right}$ for all $\ell$ is not available in general)}
  \begin{align}
      \label{det33}
      \mathbb{P}( \mathcal{A}_n^c)  =  o ( m^{-{\ell_0}}) ~.
  \end{align}
{In the sequel, bounds of the form $o(m^{-\ell})$ or $O(m^{-\ell})$ involving $\mathcal A_n^c$ are meant for $\ell\le\ell_0$.} 
We will begin proving the statement for matrices of the form \eqref{eq: C_i0}. Observing the arguments as given in the proof of Proposition \ref{lemma: formel 2.1} we obtain 
\begin{align}
 \E_X
 \big\arrowvert X_1^{*\top}
 C^*X_1^*-\tr C^* \big\arrowvert^p  & \leq 
 2^{p-1} \Big (  \E_X \big\arrowvert X_1^\top
 \tilde  C^* X_1-\tr  \tilde  C^* 
 \big\arrowvert^p  \label{det35} \\
& ~~~~~~~~~~~
+ 
 \E_X \big\arrowvert X_1^\top (
 \tilde C^* -   C^*  ) X_1-\tr ( \tilde  
 C^* -    C^*  )
 \big\arrowvert^p 
 \Big ) 
 ~,  \nonumber
\end{align}
where matrix $\tilde C^*$ is defined by \eqref{det36} with $M=I_q$.
By the same reasoning leading to equation (3.2) in \cite{baisil2004} with $a(v) =1 $ and  $B(v) = 
\tilde C^*$ it follows that  
$$ \E_X \big\arrowvert X_1^\top
 \tilde  C^* X_1-\tr  \tilde  C^* 
 \big\arrowvert^p 
 \leq {c\, m^p {\delta_m^{(2p-4) \vee 0} \over m^{1 \wedge p} } =  c\, m^{p-1}\delta_m^{2p-4 }}~.
 $$
 {(Note that $X_1$ is not normalized by $m^{-1/2}$, in contrast to $r_l=N^{-1/2}T^{1/2}X_l$ in (3.2) of \cite{baisil2004}; this gives the factor $m^p$.)} Note that the matrix $\tilde  C^*$ satisfies the corresponding assumption for such an  estimate, that is 
\begin{align}
    \label{det32}
\| \tilde  C^*    \|_{S_\infty}  \leq  \| L_n  
\|_{S_\infty}^2
\| ( \tilde  D^*_1 (z))^{-1}   \|_{S_\infty}  \leq  c( 
\mathds{1}_{\mathcal{A}_n} + 
\mathds{1}_{\mathcal{A}_n^c} 
m^{1+ \alpha}  ) 
\leq  c( 
1  + 
\mathds{1}_{\mathcal{A}_n^c}m^{1+ \alpha}  ) ,
\end{align}
   where $\tilde{D}_1^*(z)$ is defined as $D_1^*(z)$ with $X_2^*,\dots, X_m^*$ replaced by $\tilde{X}_2^*,\dots, \tilde{X}_m^*$.

   Now we turn to  the second term in \eqref{det35} and consider
   \begin{align*}
    \E_X \big | \tr ( \tilde  
 C^* -    C^*  )
 \big | ^p       & \leq  m^p   \E_X \big\|  \tilde  
 C^* -    C^*  
 \big\|_{S_\infty} ^p    \leq 
 m^p  \| L_n  
\|_{S_\infty}^{2p}  \E_X \big\|   
D_1^* (z)^{-1} -    \tilde D_1^* (z)^{-1}
 \big\|_{S_\infty} ^p  \\
& \leq  c  m^p  \E_X \big [ \|   
D_1^* (z)^{-1} \|_{S_\infty} ^p   \|    \tilde D_1^* (z)^{-1}
 \|_{S_\infty} ^p  \|
 \widehat \Sigma_{n,1}^* -  \tilde \Sigma_{n,1}^*
\|_{S_\infty} ^p \big ] ,
   \end{align*}
   where the last inequality follows from the formula $B^{-1} - A^{-1}  = A^{-1} (A-B) B^{-1}$.
Note that we have
\begin{align}
    \label{ang21}
    \|   
D_1^* (z)^{-1} \|_{S_\infty}  &\leq 
\max_{i=1}^q {1\over | z - \lambda_i ( \widehat \Sigma_{n,1}^* ) | } 
\leq 
\max_{i=1}^q {1\over | z - \lambda_i ( \widehat \Sigma_{n,1}^* ) | }\mathds{1}_{\mathcal{A}_n}+ \frac{\arrowvert z\arrowvert}{\Im(z)}\mathds{1}_{\mathcal{A}_n^c}  
\\
\nonumber
& \leq \max \Big \{ {1\over | x_r - K_{\rm  right}| },  {1\over | x_l - K_{\rm left}| },{1\over v_0}
\Big \}  + \frac{\arrowvert z\arrowvert}{\Im(z)}\mathds{1}_{\mathcal{A}_n^c}
\end{align}
where $K_{\rm left}$ and $K_{\rm  right}$ are the constants from Corollary  \ref{thm: EW1a}. Moreover, 
$$
\| 
 \widehat \Sigma_{n,1}^* -  \tilde \Sigma_{n,1}^*
\|_{S_\infty} \leq   \| \Sigma_n \|_{S_\infty} \|X_1\|^2 {\Delta_n^* \over m } ,
$$
and 
\begin{align}
    \label{det37a}
{1 \over m^{p/2} } \E_X   \|X_1\|^p   & \leq c ,  \\
    \label{det37b}
\E_X   \| \Delta_n^*\|^p    & \leq c_p {m \over n } 
\end{align}
(note that \eqref{det37a} follows from the fact that here the random variable $X_1$ is of dimension $q'$ and has independent components  bounded by $\delta_m \sqrt{m}$).
Combining these estimates and 
and using the corresponding  bound for the quantity  $\|   \tilde 
D_1^* (z)^{-1} \|_{S_\infty}$ 
and  obtain  (observing \eqref{det33}), we arrive at 
  \begin{align}
  \nonumber 
   \E_X \big | & \tr ( \tilde  
 C^* -    C^*  )
 \big | ^p \\      
& \leq  c  m^p  \Big \{  \E_X  \big [ \mathds{1}_{\mathcal{A}_n} \| 
 \widehat \Sigma_{n,1}^* -  \tilde \Sigma_{n,1}^*
\|_{S_\infty} ^p \big ] 
+  \big |{ z \over  \Im (z)  } \big |^{2p} \E_X  \big [ \mathds{1}_{\mathcal{A}_n^c} \| 
 \widehat \Sigma_{n,1}^* -  \tilde \Sigma_{n,1}^*
\|_{S_\infty} ^p \big ] 
\big \}  \nonumber
 \\
\nonumber 
& \leq  c  m^p  \Big \{ {1 \over m^p}  \E_X 
|\Delta_n^*|^p 
\E_X \|X_1 \|^{2p}
+  \big | { z \over  \Im (z)  } \big |^{2p} \big ( \mathbb{P} ( {\mathcal{A}_n^c} ) \big )^{1/2}
\big (  \E_X \| 
 \widehat \Sigma_{n,1}^* -  \tilde \Sigma_{n,1}^*
\|_{S_\infty} ^{2p}  \big )^{1/2} 
\big \} 
\nonumber 
\\
\nonumber 
& \leq  c  m^p  \Big \{ {m \over n} 
+  m^{2p(1+\alpha)} (\mathbb{P} ( {\mathcal{A}_n^c} ) \big )^{1/2} \sqrt{m\over n} 
\Big \}  \\
& \leq c \Big({m^{p+1}\over n}+1\Big)~.
\label{det38}
   \end{align}
Similarly, we obtain 
\begin{align*}
     \E_X {\big|} X_1^\top (
 \tilde C^* -   C^*  ) X_1
 \big| ^p & \leq  \E_X \big [ \| X_1\|{^{2p}}  \| {\tilde C^*} -   C^*  \|^p_{S_\infty } \big ] \\
 & \leq c \E_X \Big [ \| X_1\|^{2p}
 \|   
D_1^* (z)^{-1} \|_{S_\infty} ^p   \|    \tilde D_1^* (z)^{-1}
 \|_{S_\infty} ^p   \| X_1\|^{2p} 
 \big ( {\Delta_n^* \over m }\big )^p
\Big  ] \\
& \leq {c  \over  m^p}
 \Big \{  \E_X  \big [ \mathds{1}_{\mathcal{A}_n} |\Delta_n^*|^p 
\|X_1 \|^{4p} \big ] 
+  \big |{ z \over  \Im (z)  } \big |^{2p} \E_X  \big [ \mathds{1}_{\mathcal{A}_n^c}   |\Delta_n^*|^p 
\|X_1 \|^{4p}  \big ] 
\Big \}  \\
& \leq  {c  \over  m^p} \Big \{ {m^{2p+1} \over n } + m^{(1+\alpha)2p} 
\big ( \mathbb{P} ( {\mathcal{A}_n^c}) \big )^{1/2} \big ( \E_X [ |\Delta_n^*|^{2p} 
\|X_1 \|^{8p} ]
\big )^{1/2}
\Big \} \\
& \leq c \Big( {  m^{p+1} \over  n}+1\Big),
\end{align*}
by \eqref{det33} (for $\ell_0\ge 4p(2+\alpha)$ the terms involving $\mathbb{P} ( {\mathcal{A}_n^c})$ are $o(1)$ - they do not carry a factor $1/n$, but $1\le m^{p-1}\delta_m^{2p-4}$ for large $m$ by \eqref{ar7}). Combining this estimate with \eqref{det38} and \eqref{det32} yields the statement of Proposition \ref{lemma: formel 2.1a} for the matrix \eqref{eq: C_i0}.

The statement for the other matrices follow by similar arguments, which are omitted for the sake of brevity. For, example, for the matrix \eqref{eq: C_ic} we use the identities
\begin{align*}
    \tilde D_1^* (z)^{-2}  - D_1^* (z)^{-2} &=  D_1^* (z)^{-2} \big (
     D_1^* (z)^{2} -\tilde D_1^* (z)^{2}
    \big ) \tilde D_1^* (z)^{-2} \\
  D_1^* (z)^{2} -\tilde D_1^* (z)^{2}    & = \big (  D_1^* (z) -  \tilde D_1^* (z) \big )\tilde D_1^* (z) + D_1^* (z) \big ( D_1^* (z) - \tilde D_1^* (z)\big ) 
\end{align*}
  which gives 
\begin{align*}
 \|  \tilde D_1^* (z)^{-2}  - D_1^* (z)^{-2} \|_{S_\infty} & \leq c
  \|  \tilde D_1^* (z)^{-2}   \|_{S_\infty}  \|    D_1^* (z)^{-2} \|_{S_\infty}  \\
  & ~~~~~~~~~~~~~~~ \times 
  \big ( 2 |z| + 
  \|   \widehat \Sigma_{n,1}^*   \|_{S_\infty} 
 +   \|   \tilde  \Sigma_{n,1}^*    \|_{S_\infty} \big ) 
    \|  \widehat \Sigma_{n,1}^*  -\tilde  \Sigma_{n,1}^*  \|_{S_\infty} .
\end{align*}
We now proceed in the same way as before 
multiplying with $( \mathds 1_{\mathcal{A}_n} +\mathds 1_{\mathcal{A}_n}^c)$, where we use 
$$ 
\|   \widehat \Sigma_{n,1}^*   \|^2_{S_\infty}  \leq  \|   \widehat \Sigma_{n,1}^*   \|^2_{S_2} \leq
{1 \over m^2 } \sum_{i,j=1}^m \|X_i\|^2 \| X_j\|^2 
$$
on $\mathcal{A}_n^c$.
\end{proof}

\begin{remark}\label{rem: randomM}
{\rm
Proposition \ref{lemma: formel 2.1a} extends to matrices $M$ which depend on the data through finitely many scalars. Let $d\in\N$, $w\in\C^d$, $\eta>0$, let $M(\omega)\in\C^{q\times q}$, $\omega\in\C^d$, be deterministic given $\Pi_n$ and analytic on $P=\{\omega:\max_{l}|\omega_l-w_l|\le2\eta\}$ with $\sup_{\omega\in P}\|M(\omega)\|_{S_\infty}\le K$, and let $\hat\omega$ be a $\C^d$-valued random variable. Then, for the matrix \eqref{eq: Cende} with $M$ replaced by $M(\hat\omega)$,
$$
\E_X\Big[\mathds 1_{\{\max_l|\hat\omega_l-w_l|\le\eta\}}\big\arrowvert X_1^{*\top}C^*X_1^*-\tr C^*\big\arrowvert^p\Big]\le 2^{dp}K_p\Big(m^{p-1}\delta_m^{2p-4}+\frac{m^{p+1}}n\Big),
$$
where $K_p$ depends only on $p$ and $K$.
Indeed, let 
$$
Q(\omega)=X_1^{*\top}L_n^\top D_1^*(z)^{-1}M(\omega)L_nX_1^*-\tr(L_n^\top D_1^*(z)^{-1}M(\omega)L_n)
$$ 
and $\omega_\theta=(w_l+2\eta e^{i\theta_l})_{l=1,\dots,d}$, $\theta\in[0,2\pi]^d$. On $\{\max_l|\hat\omega_l-w_l|\le\eta\}$, Cauchy's integral formula in each coordinate gives $|Q(\hat\omega)|\le\pi^{-d}\int_{[0,2\pi]^d}|Q(\omega_\theta)|\,\dd\theta$, hence$$|Q(\hat\omega)|^p\le2^{dp}(2\pi)^{-d}\int_{[0,2\pi]^d}|Q(\omega_\theta)|^p\,\dd\theta.
$$ 
by Jensen's inequality. As $M(\omega_\theta)$ is deterministic given $\Pi_n$ with $\|M(\omega_\theta)\|_{S_\infty}\le K$, Proposition \ref{lemma: formel 2.1a} applies for each $\theta$.
}
\end{remark}

\begin{remark} \label{det101}
{\rm 
    Note that Proposition \ref{lemma: formel 2.1a} will replace Proposition \ref{lemma: formel 2.1} in the following discussion. Moreover,  in the case $p=2$ both results yield the same estimate. This fact will be of importance as we will use some of the 
    estimates for Section \ref{sec53}  - \ref{sec55} in the following discussion which also hold under Assumption (A3+) (instead of (A3)) and the truncation scheme considered in this section. 
    }
\end{remark}

\begin{proposition}
     \label{lemrev5}
    \begin{align*}
&  \sup_{z_1,z_2 \in  {\cal C}_n} 
\E^*  \big\arrowvert X_1^{*\top}
 C^*(z_1,z_2) X_1^*-\tr C^* (z_1,z_2) \big\arrowvert^2  =  
O_{\mathbb P}
 \Big( m  + 
 \frac{m^{3}}{n}
 \Big) \\
 &   \sup_{z_1,z_2 \in  {\cal C}_n} 
\E^*  \big\arrowvert X_1^{*\top}
 C^*(z_1,z_2) X_1^*-\tr C^* (z_1,z_2) \big\arrowvert^4  =  
O_{\mathbb P}
 \Big( m^3 \delta_m^{4} + 
 \frac{m^{5}}{n}
 \Big)
 ~, 
\end{align*}
where 
\begin{align*}
C^*  (z_1,z_2) & ={L_n^\top}(D_1^*(z_1))^{-1}{L_n} \\
C^*  (z_1,z_2) & ={L_n^\top}(D_{12}^*(z_1))^{-1}{L_n} \\
C^*  (z_1,z_2) &  = {L_n^\top}( \E^* \underline{m}_{n}^* (z_1)  \tilde \Sigma_n 
    + I_q \big )^{-1}{L_n} \\
C^*  (z_1,z_2) &  = {L_n^\top}(D_1^*(z_1))^{-1} ( \E^* \underline{m}_{n}^* (z_1)  \tilde \Sigma_n 
    + I_q \big )^{-1}\tilde \Sigma_n (D_1^*(z_1))^{-1}{L_n} \\
     C^* (z_1,z_2)   &=  
    C^*  (z_1,z_2) =
   L_n^\top  D_1^*(z_1)^{-2}   D_1^*(z_2)^{-1} L_n ,  \\
    C^* (z_1,z_2)   &=  
    C^*  (z_1,z_2) =
   L_n^\top  D_1^*(z_1)^{-1}   D_1^*(z_2)^{-1} L_n ,\\
C^*  (z_1,z_2) & {} = {L_n^\top}(D_1^*(z_1))^{-1} \big( \E^* \underline{m}_{n}^* (z_1)  \tilde \Sigma_n 
    + I_q \big )^{-1}{L_n}
\end{align*}
\end{proposition}
\begin{proof}
Exemplary, we consider a matrix with $z=z_1=z_2$, for which we use the notation $C^*(z):=C^*(z,z)$ (the other cases can be treated similarly). For the matrices of the list which contain $\E^*\underline m_n^*(z)$, let $C^*(z,\omega)$ denote $C^*(z)$ with $\E^*\underline m_n^*(z)$ replaced by $\omega\in\C$, let $\omega_\theta(z)=\underline{\tilde m}_n^0(z)+2\eta e^{i\theta}$ and $B_n=\{\sup_{z\in\mathcal C_n}|\E^*\underline m_n^*(z)-\underline{\tilde m}_n^0(z)|\le\eta\}$, so that $\P(B_n)\to1$ by Lemma \ref{lemrev0}. As in Remark \ref{rem: randomM}, on $B_n$,
\begin{align*}
\sup_{z\in\mathcal C_n}\E^*\big|X_1^{*\top}C^*(z)X_1^*&-\tr C^*(z)\big|^4\\
&\le\frac{16}{2\pi}\int_0^{2\pi}\sup_{z\in\mathcal C_n}\E^*\big|X_1^{*\top}C^*(z,\omega_\theta(z))X_1^*-\tr C^*(z,\omega_\theta(z))\big|^4\,\dd\theta ,
\end{align*}
provided that $|1+t\underline{\tilde m}_n^0(w)|\ge2\kappa$ for some $\kappa>0$, all $w$ in the set $\bar U$ chosen below and all eigenvalues $t$ of $\tilde\Sigma_n$ (which holds with probability tending to one), and $2\eta\|\tilde\Sigma_n\|_{S_\infty}\le\kappa$. Given $\Pi_n$, $\omega_\theta$ is deterministic and analytic on $U$ with $\|(\omega_\theta(w)\tilde\Sigma_n+I_q)^{-1}\|_{S_\infty}\le1/\kappa$, and in the remainder of the proof $C^*(z)$ stands for $C^*(z,\omega_\theta(z))$. All bounds below are uniform in $\theta$.
 Let $\hat I =\{ i_1^*, \ldots , i_m^* \}$ denote the random subset of chosen indices by the bootstrap and note that $\# \hat I \leq m $, then 
\begin{align}  
\nonumber 
  \E \Big  [    \sup_{z \in {\cal C}_n}
\E^*  \big\arrowvert X_1^{*\top}
 C^*(z) X_1^*-\tr C^* (z) \big\arrowvert^4  
 \Big  ] 
 &\leq \E \bigg  [ 
 {1 \over n } 
  \sup_{z \in {\cal C}_n} 
  \sum_{i =1}^n   \big | X_i^\top  C^*(z) X_i - {\rm tr} ( C^*(z) ) \big |^4 \bigg ] 
  \\ 
  \nonumber 
&\leq \E \bigg  [  \sup_{z \in {\cal C}_n}  {1 \over n }  \sum_{i \in \hat I}   \big | X_i^\top  C^*(z) X_i - {\rm tr} ( C^*(z) ) \big |^4 \bigg ]  \\
& + \E \bigg  [  \sup_{z \in {\cal C}_n}  {1 \over n } \sum_{i \in \hat I^c }   \big | X_i^\top  C^*(z) X_i - {\rm tr} ( C^*(z) ) \big |^4 \bigg ] 
\nonumber \\
& =: \E [S_1] +  \E [S_2]~.
\label{rev11}
\end{align}
We will now consider both terms separately starting with $  \E  [S_1]$.
\begin{align*}
 \E [S_1 ]  &= \E \Big[  \E \Big[  
\sup_{z \in {\cal C}_n}  {1 \over n } \sum_{i \in \hat I}   \big | X_i^\top  C^*(z) X_i - {\rm tr} ( C^*(z) ) \big |^4 
 \Big | \hat I \Big ]
\Big ]   \\ 
&\leq  \E \Big[ 
 {1 \over n } \sum_{i \in \hat I} 
 \E \Big[  
\sup_{z \in {\cal C}_n}    \big | X_i^\top  C^*(z) X_i - {\rm tr} ( C^*(z) ) \big |^4 
 \Big | \hat I \Big ]
\Big ]  
\\ 
&\leq  \E \Big[ 
 {1 \over n } \sum_{i \in \hat I} 
 \E \Big[  
  \big ( ( \| X_i\|^2 + m )  \sup_{z \in {\cal C}_n}   \| C^*(z)\|_{{ S}_\infty}
\big )^4
 \Big | \hat I \Big ]
\Big ]
\end{align*}
Recalling the definition of the set ${\cal A}_n$ in \eqref{det21}  it follows
that 
\begin{align*}
 \E [S_1 \mathds{1}_{{\cal A}_n} ] & \lesssim
\E \Big [  {1 \over n } \sum_{i \in \hat I} \E \Big [
 \sum_{k_1,k_2, k_3,k_4=1}^{q'} X_{ik_1}^2
 X_{ik_2}^2X_{ik_3}^2X_{ik_4}^2
 \Big ] 
 \Big | \hat I \Big ] \\
 &\leq {m \over n} 
 \sum_{k_1,k_2, k_3,k_4=1}^{q'}  \E 
 \big [X_{1k_1}^2
 X_{1k_2}^2X_{1k_3}^2X_{1k_4}^2
 \big ]  \lesssim 
 {m^5 \over n} 
\end{align*}
On the other hand, on the set ${\cal A}_n^c$, 
$\sup_{z \in {\cal C}_n} \| C^*(z) \|    \lesssim {1 \over  | \Im (z) | } \leq m^{1+ \alpha }
$ and $X_{ik}^2 \leq \delta_n^2 m $
while $m^{\ell} \P ({\cal A}_n^c) = o(1) $  for any $\ell \in \N$. This gives 
$$
\E [S_1 \mathds{1}_{{\cal A}_n^c} ] = O( m^{- \ell} ) 
$$
for any $\ell \in \N$. 
We now turn to the term $\E S_2$. The part on ${\cal A}_n^c$ is treated as for $S_1$. On ${\cal A}_n$, all eigenvalues of the matrices entering $C^*(z)$ lie in $[K_{\rm left},K_{\rm right}]$ (with $K_{\rm left}$ replaced by $0$ if $x_l\le0$). Choose an open set $U\supset{\cal C}\cup\bar{\cal C}$ whose closure has distance at least $\eta'>0$ from $[K_{\rm left},K_{\rm right}]$ and whose boundary $\Gamma$ consists of finitely many closed curves of total length $\ell(\Gamma)<\infty$ with distance at least $\eta'$ from ${\cal C}\cup\bar{\cal C}$. Then $z\mapsto\mathds 1_{{\cal A}_n}Q_i(z)$, $Q_i(z)=X_i^\top C^*(z)X_i-\tr C^*(z)$, is analytic on $U$, and Cauchy's integral formula together with Jensen's inequality gives on ${\cal A}_n$
$$
\sup_{z\in{\cal C}_n}|Q_i(z)|^4\le\frac{\ell(\Gamma)^3}{(2\pi\eta')^{4}}\ointctrclockwise_\Gamma|Q_i(w)|^4\,|\dd w| .
$$
Hence
$$
\E\big[S_2\mathds 1_{{\cal A}_n}\big]\le\frac{\ell(\Gamma)^4}{(2\pi\eta')^4}\sup_{w\in\Gamma}\E\Big[\frac1n\sum_{i\in\hat I^c}\mathds 1_{{\cal A}_n}\big|X_i^\top C^*(w)X_i-\tr C^*(w)\big|^4\Big].
$$
For fixed $w\in\Gamma$ and $i\in\hat I^c$, the vector $X_i$ is, conditionally on the resampling indices and $\{X_j:j\in\hat I\}$, independent of $\mathds 1_{{\cal A}_n}C^*(w)$, and $\|\mathds 1_{{\cal A}_n}C^*(w)\|_{S_\infty}\le K$, so that $\tr\big(\mathds 1_{{\cal A}_n}C^*(w)C^*(w)^*\big)\le Km$. Lemma B.26 of \cite{baisilverstein2010} and $\E X_{11}^8\le\delta_m^4m^2\,\E X_{11}^4$ give
$$
\E\Big[\mathds 1_{{\cal A}_n}\big|X_i^\top C^*(w)X_i-\tr C^*(w)\big|^4\,\Big|\,\hat I,\{X_j:j\in\hat I\}\Big]\le K\big(m^2+\delta_m^4m^3\big)\le Kr_n ,
$$
where $r_n=m^3\delta_m^4+m^5/n$ and we used $\delta_m^4m\ge1$ for large $m$ by (T). Combining the bounds for $S_1$ and $S_2$ with Markov's inequality yields the assertion for the fourth moment. The second moment is treated in the same way, with $K(m+m^3/n)$ in place of $Kr_n$.
\end{proof}

\subsection{Bootstrap version of the Martingale-CLT} 
\begin{theorem}[Bootstrap Martingale-CLT]\label{thm: bmclt}
Let $(X_j)_{j\in\N}$ be an iid-sequence, $(\Pi_n)_{n\in\N}$ be some  further sequence of  random variables independent of $(X_j)_{j\in\N}$, and  $m\leq n$ with $m=m(n)\rightarrow\infty$. Conditional on  $X_1,\dots, X_n$,  let 
$$
X_1^*,\dots, X_m^*\overset{iid}{\sim}\hat{\P}_n=\frac{1}{n}\sum_{k=1}^n\delta_{{X_k}}
$$ denote the'$m$ out of $n$' bootstrap sample. Suppose that conditional on $X_1,\dots, X_n$ and $\Pi_n$, $$Y_{n,1}^*,\dots, Y_{n,m}^*$$ is a real square integrable martingale difference sequence with respect to the bootstrap canonical filtration $(\FF_{n,k}^*)_{k=1}^m$ with $\FF_{n,k}^*=\sigma(X_1^*,\dots, X_k^*\arrowvert X_1,\dots, X_n, \Pi_n)$. 
Assume that, as $n\rightarrow\infty$, 
\begin{align}\label{eq: cond 1}
\sum_{k=1}^{m}\E^*\big(\arrowvert Y_{n,k}^*\arrowvert^2\big\arrowvert \FF_{n,k-1}^*\big)\overset{\P}{\longrightarrow}\sigma^2
\end{align}
for some constant $\sigma^2{\ge}0$ and
\begin{align}\label{eq: cond 2}
\sum_{k=1}^m\E^*\Big[\arrowvert Y_{n,k}^*\arrowvert^2\ind_{\{\arrowvert Y_{n,k}^*\arrowvert\geq\varepsilon\}}\Big]\overset{\P}{\longrightarrow}0
\end{align}
for each $\varepsilon>0$. Then, as $n\rightarrow\infty $,
\begin{align}
\LL\bigg(\sum_{k=1}^mY_{n,k}^*\bigg\arrowvert X_1,\dots,X_n,\Pi_n\bigg)\Longrightarrow\NN(0,\sigma^2)\ \ \text{in probability}.
\end{align}
\end{theorem}

\begin{proof}
Preliminary, we assume that there exists some constant $c>0$ with
\begin{align}\label{eq: step 11}
\sup_{n}\sum_{k=1}^m \E^*\big(Y_{n,k}^{*2}\big\arrowvert {\FF_{n,k-1}^*}\big) \leq c.
\end{align}

{\sc Claim I.} With
$
Z_n^*=\sum_{k=1}^mY_{n,k}^*, 
$
\begin{align}\label{eq: step 1}
\sup_{t\in K}\Big\arrowvert \E^*\exp(it Z_n^*)-\exp\Big(-\frac{1}{2}t^2\sigma^2\Big)\Big\arrowvert \overset{\P}{\longrightarrow}0
\end{align}
for any compact subset  $K\subset\R$.

\medskip
{\it Proof of Claim I.}
The proof follows the lines in the proof of the classical martingale CLT, replacing all expectations  by conditional expectations $\E^*$ and the canonical filtration correspondingly. However, we state here locally uniform stochastic convergence rather than pointwise  stochastic convergence of the characteristic functions, which requires some extra care with the transfer of arguments. 

\medskip

Write 
\begin{align*}
\sigma_{n,l}^{*2}&=\E^*(Y_{n,l}^{*2}\arrowvert {\FF_{n,l-1}^*})\ \ \text{for }2\leq l\leq m,\ \ \sigma_{n,1}^{*2}=\E^*Y_{n,1}^{*2},\\
\Sigma_{n,l}&=\sum_{k=1}^l\sigma_{n,k}^{*2}\ \ \text{ for } 0\leq l\leq m, \ \text{and}\\
Z_{n,l}^*&=\sum_{k=1}^lY_{n,k}^* \ \ \text{for } 0\leq l\leq m.
\end{align*} 
Let $K\subset\R$ be compact. Then
\begin{align*}
\sup_{t\in K}\bigg\arrowvert\E^*&\bigg[\exp\Big(itZ_n^*\Big)-\exp\Big(-\frac{1}{2}t^2\sigma^2\Big)\bigg]\bigg\arrowvert\\
&\leq \sup_{t\in K}\E^*\bigg[\bigg\arrowvert 1-\exp\Big(\frac{1}{2}t^2\Sigma_{n,m}\Big)\exp\Big(-\frac{1}{2}t^2\sigma^2\Big)\bigg]\\
&\quad\quad + \sup_{t\in K}\bigg\arrowvert \E^*\bigg[\exp\Big(\frac{1}{2}t^2\Sigma_{n,m}\Big)\exp\Big(itZ_n^*\Big)-1\bigg]\bigg\arrowvert \\
&=: A_K+B_K.
\end{align*}
Choose some $0\leq t_K\in\R$ satisfying $-t_K\leq x\leq t_K$ for all $x\in K$. 
As $\Sigma_{n,m}\rightarrow_{\P}\sigma^2$ by assumption and \eqref{eq: step 11},  
\begin{align*}
\E A_K\leq \E\bigg(\exp\Big(\frac{1}{2}t_K^2\arrowvert \Sigma_{n,m}-\sigma^2\arrowvert\Big)-1\bigg)\rightarrow 0.
\end{align*}
As concerns $B_K$, 
\begin{align*}
B_K&=\sup_{t\in K}\bigg\arrowvert \sum_{k=1}^m\E^*\bigg[\exp\Big(it Z_{n,{k-1}}^*\Big)\exp\Big(\frac{1}{2}t^2\Sigma_{n,k}\Big)\bigg(\exp(itY_{n,k}^*)-\exp\Big(-\frac{1}{2}t^2\sigma_{n,k}^{*2}\Big)\bigg)\bigg]\bigg\arrowvert\\
&\leq \exp\Big(\frac{1}{2}t_K^2 c\Big)\sup_{t\in K}\sum_{k=1}^m\E^*\bigg[\bigg\arrowvert\E^*\bigg(\exp\big(itY_{n,k}^*\big)-\exp\Big(-\frac{1}{2}t^2\sigma_{n,k}^{*2}\Big)\bigg\arrowvert \FF_{n,k-1}^*\bigg)\bigg\arrowvert\bigg].
\end{align*}
Still assuming the temporary condition {\eqref{eq: step 11}}, it remains to prove that the latter expression converges to $0$ in probability. Taylor's approximation reveals
\begin{align}\label{eq: step 12}
\exp\Big(it Y_{n,k}^*\Big)=1+ it Y_{n,k}^*-\frac{1}{2}t^2Y_{n,k}^{*2} +\theta_{n,k}^*(t)
\end{align}
with
\begin{align*}
\sup_{t\in K}\arrowvert \theta_{n,k}^*(t)\arrowvert &\leq \sup_{t\in K}\min\Big\{\arrowvert t Y_{n,k}^*\arrowvert^3, \arrowvert tY_{n,k}^*\arrowvert^2\Big\}\\
&\leq (t_K^2+t_K^3)\Big(Y_{n,k}^{*2}\ind_{\{\arrowvert Y_{n,k}^{*}\arrowvert\geq\varepsilon\}}+\varepsilon Y_{n,k}^{*2}\Big)
\end{align*}
for any $\varepsilon>0$ as well as
\begin{align}\label{eq: step 13}
\exp\Big({-}\frac{1}{2}t^2\sigma_{n,k}^{*2}\Big)=1-\frac{1}{2}t^2\sigma_{n,k}^{*2}+\theta_{n,k}'(t)
\end{align}
with
\begin{align*}
\sup_{t\in K}\arrowvert\theta_{n,k}'(t)\arrowvert& \leq \sup_{t\in K}\Big(\frac{1}{2}t^2\sigma_{n,k}^{*2}\Big)^2\exp\Big(\frac{1}{2}t^2\sigma_{n,k}^{*2}\Big) \leq t_K^4\sigma_{n,k}^{*4}\exp\Big(\frac{1}{2}t_K^2c\Big).
\end{align*}
Therefore, with 
$$
c_K=t_K^2+t_K^3+t_K^4\exp\Big(\frac{1}{2}t_K^2c\Big),
$$
we arrive at
\begin{align}
\sup_{t\in K}\sum_{k=1}^m&\E^*\bigg[\bigg\arrowvert\E^*\bigg(\exp\big(itY_{n,k}^*\big)-\exp\Big(-\frac{1}{2}t^2\sigma_{k,n}^{*2}\Big)\bigg\arrowvert \FF_{n,k-1}^*\bigg)\bigg\arrowvert\bigg]\nonumber\\
&\leq  c_K\sum_{k=1}^m\bigg[\E^*\Big(Y_{n,k}^{*2}\ind_{\{\arrowvert Y_{n,k}^{*}\arrowvert\geq\varepsilon\}}\Big)+\varepsilon\E^*(\sigma_{n,k}^{*2})+\E^*(\sigma_{n,k}^{*4})\bigg]\nonumber\\
&\leq c_K\bigg[\varepsilon c+c\E^*\Big(\max_{k=1,\dots, {m}}\sigma_{n,k}^{*2}\Big)+\sum_{k=1}^m \E^*\Big(Y_{n,k}^{*2}\ind_{\{\arrowvert Y_{n,j}^{*}\arrowvert\geq\varepsilon\}}\Big)\bigg].\label{eq: step 17}
\end{align}
Because of $\sigma_{n,l}^{*2}\leq \varepsilon^2+\sum_{k=1}^m\E^*(Y_{n,k}^{*2}\ind_{\{\arrowvert Y_{n,l}^{*}\arrowvert\geq\varepsilon\}})$ for any $1\leq l\leq m$, \eqref{eq: cond 2} reveals that \eqref{eq: step 17} is upper bounded by
\begin{align*}
c_K(\varepsilon c+c\varepsilon^2)+ o_{\P}(1).
\end{align*}
Since $\varepsilon>0$ was chosen arbitrarily, this proves  \eqref{eq: step 1}.

\medskip

{\sc Claim II.}  The sequence $(Z_n^*)$ is tight in probability, i.e. for any $\varepsilon>0$, there exists some compact subset $K_{\varepsilon}\subset\R$, such that
\begin{align}\label{eq: step 2}
\limsup_{n\to\infty}\P\Big(\P\big( Z_n^*\not\in K_{\varepsilon}\big\arrowvert X_1,\dots, X_n,\Pi_n\big)> \varepsilon\Big)=0.
\end{align}

\medskip
{\it Proof of Claim II.} Using the identity
\begin{align*}
\int_{-u}^u\big(1-\exp(itx)\big)\dd t = 2u -\frac{\exp(iux)-\exp(-iux)}{ix}=2u-\frac{2\sin(ux)}{x},
\end{align*}
we obtain by the Theorem of Fubini
\begin{align*}
\frac{1}{u}&\int_{-u}^u\big(1-\E^*\exp(itZ_n^*)\big)\dd t\\
&=\int\Big(2-\frac{2\sin(ux)}{ux}\Big)\dd \P^{Z_n^*\arrowvert X_1,\dots, X_n,\Pi_n}(x)\\
&\geq 2\int_{\{\arrowvert x\arrowvert\geq 2/u\}}\Big(1-\Big\arrowvert\frac{\sin(ux)}{ux}\Big\arrowvert\Big)\dd \P^{Z_n^*\arrowvert X_1,\dots, X_n,\Pi_n}(x)\\
&\geq \P\Big(\arrowvert Z_n^*\arrowvert\geq\frac{2}{u}\Big\arrowvert X_1,\dots, X_n,\Pi_n\Big),
\end{align*}
where we used $\arrowvert \sin(v)/v\arrowvert\leq 1$ in the first inequality and $\arrowvert \sin(ux)\arrowvert\leq 1$ in the last line. Fix now $\varepsilon>0$ and choose $u>0$ such that
\begin{align*}
\frac{1}{u}\int_{-u}^u\bigg(1-\exp\Big(-\frac{1}{2}t^2\sigma^2\Big)\bigg)\dd t\leq\frac{\varepsilon}{2}. 
\end{align*}
With $K_{\varepsilon}=[-2/u, 2/u]$, we obtain 
\begin{align*}
\limsup_{n\rightarrow\infty}&\, \P \Big(\P\big( Z_n^*\not\in K_{\varepsilon}\big\arrowvert X_1,\dots, X_n,\Pi_n\big)> \varepsilon\Big)\\
&\leq \limsup_{n\rightarrow\infty}\P\bigg(\frac{1}{u}\int_{-u}^u\Big(1-\E^*\exp(itZ_n^*)\Big)\dd t>\varepsilon\bigg)\\
&\leq \limsup_{n\rightarrow\infty}\P\bigg(\frac{1}{u}\int_{-u}^u\Big(\exp\Big(-\frac{1}{2}t^2\sigma^2\Big)-\E^*\exp(itZ_n^*)\Big)\dd t>\frac{\varepsilon}{2}\bigg).
\end{align*}
The last expression is equal to zero by \eqref{eq: step 1}, which proves \eqref{eq: step 2}. 

\medskip

{\sc Claim III.} For any bounded Lipschitz function $f$,
\begin{align}\label{eq: step 3}
\int f(Z_n^*)\dd \hat{\Q}_n  \overset{\P}{\longrightarrow}\int f\dd \NN(0,\sigma^2)\ \ \ (n\rightarrow\infty),
\end{align}
where $\hat{\Q}_n$ denotes the random distribution $\P(~\cdot ~\arrowvert X_1,\dots, X_n,\Pi_n)$.

\medskip
\noindent
{\it Proof of Claim III.}
For any $\varepsilon>0$, fix $K_{\varepsilon}=[-c_{\varepsilon},c_{\varepsilon}]$ which satisfies $\NN(0,\sigma^2)(K_{\varepsilon})\geq 1-\varepsilon$ and \eqref{eq: step 2}.  Next, for any bounded Lipschitz function $f$  and any $\varepsilon>0$, there exists some bounded Lipschitz function $\tilde{f}_{\varepsilon}$ with
$$
f_{\arrowvert K_{\varepsilon}}=\tilde{f}_{\varepsilon\arrowvert K_{\varepsilon}}\ \ \ \text{and}\ \ \ \tilde{f}_{\varepsilon}(x)=0\ \text{for all}\ x\in [- c_{\varepsilon}-\Arrowvert f\Arrowvert_{\sup}, c_{\varepsilon}+\Arrowvert f\Arrowvert_{\sup}]^c,
$$
and
\begin{align*}
\bigg\arrowvert& \int f(Z_n^*)\dd \hat{\Q}_n -\int f\dd \NN(0,\sigma^2)\bigg\arrowvert\\
&\leq 2 \Arrowvert f\Arrowvert_{\sup} \hat{\Q}_n\big(Z_n^*\not\in K_{\varepsilon}\big)+ 2\Arrowvert f\Arrowvert_{\sup}\NN(0,\sigma^2)\big(K_{\varepsilon}^c\big)\\
&\quad\quad+\bigg\arrowvert \int \tilde{f}_{\varepsilon}(Z_n^*)\dd \hat{\Q}_n -\int \tilde{f}_{\varepsilon}\dd \NN(0,\sigma^2)\bigg\arrowvert\\
&\leq \bigg\arrowvert \int \tilde{f}_{\varepsilon}(Z_n^*)\dd \hat{\Q}_n -\int \tilde{f}_{\varepsilon}\dd \NN(0,\sigma^2)\bigg\arrowvert +4\Arrowvert f\Arrowvert_{\sup}\varepsilon+o_{\P}(1).
\end{align*}
Identifying the endpoint $-c_{\varepsilon}-\Arrowvert f\Arrowvert_{\sup}$  with $c_{\varepsilon}+\Arrowvert f\Arrowvert_{\sup}$ yields the torus $T_{\varepsilon}$ on which $f_{\varepsilon}$ is continuous. As the complex linear combinations of the monomials 
$$
m_j:x\mapsto \exp \left(i \frac{2\pi}{{2c_{\varepsilon}+2\Arrowvert f\Arrowvert_{\sup}}}jx\right), \ j\in\Z,
$$ restricted to $T_{\varepsilon}$ form a point-separating self-adjoint $\C$-algebra of functions on $T_{\varepsilon}$ which includes constants, the Stone-Weierstra{\ss} theorem reveals that they are dense in the space of continuous complex functions on $T_{\varepsilon}$ with respect to the topology of uniform convergence. Hence,  there exists some linear combination $P_{f,\varepsilon}=\sum_{j=1}^m\alpha_jm_j$ such that
\begin{align}\label{eq: step 3a}
\sup_{-c_{\varepsilon}-\Arrowvert f\Arrowvert_{\sup}\leq x\leq c_{\varepsilon}+\Arrowvert f\Arrowvert_{\sup}}\Big\arrowvert \tilde{f}_{\varepsilon}(x)-P_{f,\varepsilon}(x)\Big\arrowvert<\varepsilon.
\end{align}
Moreover, since $\tilde{f}_{\varepsilon}$ is bounded in absolute value by ${\Arrowvert f\Arrowvert_{\sup}}$ and $P_{f,\varepsilon}$ is ${(2c_{\varepsilon}+2\Arrowvert f\Arrowvert_{\sup})}$-periodic, \eqref{eq: step 3a} reveals $\Arrowvert P_{f,\varepsilon}\Arrowvert_{\sup}\le \Arrowvert f\Arrowvert_{\sup} +\varepsilon$.
Claim I and Claim II then imply
\begin{align*}
\bigg\arrowvert& \int \tilde{f}_{\varepsilon}(Z_n^*)\dd \hat{\Q}_n -\int \tilde{f}_{\varepsilon}\dd \NN(0,\sigma^2)\bigg\arrowvert\\
&\leq \int_{-c_{\varepsilon}-\Arrowvert f\Arrowvert_{\sup}}^{c_{\varepsilon}+\Arrowvert f\Arrowvert_{\sup}}\Big\arrowvert \tilde{f}_{\varepsilon}(x)-P_{f,\varepsilon}(x)\Big\arrowvert\dd\Big( \hat{\Q}_n^{Z_n^*}+ \NN(0,\sigma^2)\Big)\\
&\quad\quad +(\Arrowvert f\Arrowvert_{\sup}+\varepsilon)\hat{\Q}_n\big(Z_n^*\not\in K_{\varepsilon}\big)+ (\Arrowvert f\Arrowvert_{\sup}+\varepsilon)\NN(0,\sigma^2)\big(K_{\varepsilon}^c\big)\\
&\quad\quad +\bigg\arrowvert \int P_{f,\varepsilon}(Z_n^*)\dd \hat{\Q}_n -\int P_{f,\varepsilon}\dd \NN(0,\sigma^2)\bigg\arrowvert\\
&\leq 2\varepsilon +2(\Arrowvert f\Arrowvert_{\sup}+\varepsilon)\varepsilon +o_{\P}(1).
\end{align*}
Summarizing,
\begin{align*}
\bigg\arrowvert& \int f(Z_n^*)\dd \hat{\Q}_n -\int f\dd \NN(0,\sigma^2)\bigg\arrowvert\leq 6\Arrowvert f\Arrowvert_{\sup}\varepsilon+2\varepsilon(1+\varepsilon)+o_{\P}(1).
\end{align*}
Since $\varepsilon>0$ is arbitrary, this proves \eqref{eq: step 3}.

\medskip

{\sc Claim IV.} As $n\rightarrow\infty$,
\begin{align}\label{eq: step 4}
 d_{BL}\Big(\LL\big(Z_n^*\big\arrowvert X_1,\dots,X_n,\Pi_n\big),\,\NN(0,\sigma^2)\Big)\overset{\P}{\longrightarrow}0.
 \end{align}
{\it Proof of Claim IV.} As in the proof of claim III, for any $\varepsilon>0$, fix $K_{\varepsilon}=[-c_{\varepsilon},c_{\varepsilon}]$ which satisfies $\NN(0,\sigma^2)(K_{\varepsilon})\geq 1-\varepsilon$ and \eqref{eq: step 2}.
Denote the closed unit ball of bounded Lipschitz functions as introduced in Subsection \ref{sec51} by
$
B=\big\{f\in BL: \Arrowvert f\Arrowvert_{BL}\leq 1\big\}
$
and define 
$$
B_{K_{\varepsilon}}=\big\{g: K_{\varepsilon}\rightarrow\R: g=f_{\arrowvert K_{\varepsilon}}\ \text{for some}\ f\in B\big\}.
$$
Then the set $B_{K_{\varepsilon}}$ is closed with respect to the topology of uniform convergence and therefore compact by the Arzel\`a-Ascoli theorem. Hence, for any $\varepsilon>0$, there exist $N\in N$ and $f_1,\dots, f_N\in B$, such that for any $f\in B$, there exists some $g_f\in\{f_1,\dots,f_N\}$ with
\begin{align*}
\sup_{x\in K_{\varepsilon}}\arrowvert f(x)-g_f(x)\arrowvert <\varepsilon.
\end{align*}
As a consequence,
\begin{align*}
\sup_{f\in B}\bigg\arrowvert \int &(f-g_f)(\dd \hat{\Q}_n^{Z_n^*}-\dd  \NN(0,\sigma^2))\bigg\arrowvert\\
&\leq \sup_{f\in B}\int_{K_{\varepsilon}}\arrowvert f-g_f\arrowvert\dd(\hat{\Q}_n^{Z_n^*} +  \NN(0,\sigma^2)) + 2\hat{\Q}_n^{Z_n^*}(K_{\varepsilon}^c) +  2\NN(0,\sigma^2)(K_{\varepsilon}^c)\\
&\leq 6\varepsilon+o_{\P}(1)
\end{align*}
and therefore,
\begin{align*}
d_{BL}&\Big(\LL\big(Z_n^*\big\arrowvert X_1,\dots,X_n,\Pi_n\big),\, \NN(0,\sigma^2)\Big)\\
&\leq \max_{j=1,\dots, N}\bigg\arrowvert \int f_j(\dd \hat{\Q}_n^{Z_n^*}-\dd  \NN(0,\sigma^2))\bigg\arrowvert+6\varepsilon + o_{\P}(1)\\
&=6\varepsilon + o_{\P}(1)
\end{align*}
by \eqref{eq: step 3}. As $\varepsilon>0$ was arbitrary, this proves \eqref{eq: step 4}.

\medskip

To remove assumption \eqref{eq: step 11}, the same argument as in the proof of the classical martingale-CLT can be applied.
\end{proof}

\subsection{Proof of Proposition \ref{prop}}
\label{sec84}

\subsubsection{Weak convergence of finite dimensional distributions in probability}
\label{sec841}

Recall  that $\E_j^*$ denotes 
the conditional expectation operator corresponding to $\hat{\P}_n$ with respect to the $\sigma$-field generated by $X_1^*,\dots, X_j^*$ (conditional on $\Pi_n$)
Similar calculations as in Section 2 (p. 569-570) of  \cite{baisil2004} 
and an application of Proposition \ref{lemma: formel 2.1a}  for the matrices $L_n^\top  D_1^*(z)^{-1} L_n$ and $L_n^\top  D_1^*(z)^{-2} L_n$
yield for $ z\in \CC_n $
\begin{align*}
\widehat M_n^{*}(z) - \E^* \big [ 
\widehat M_n^{*}(z) \big ] & := 
q\big(m_{\mu^{\widehat \Sigma_n^*}}(z)-\E^* m_{\mu^{\widehat \Sigma_n^*}}(z)\big) \\
& = \sum_{j=1}^m Y_j^*+O_{\P} \Big ( \delta_m^2  + \sqrt{m^2 \over n} \Big )
= \sum_{j=1}^m Y_j^*+o_{\P} (1))~, 
\end{align*}
where 
\begin{align}
Y_j^*&=-(\E_j^*-\E_{j-1}^*)\Big(\bar{\beta}_j^*(z)\delta_j^*(z)-\bar{\beta}_j^*(z)^2\varepsilon_j^*(z)\frac{1}{m}\tr\big(L_nL_n^\top  D_j^*(z)^{-2}\big)\Big)\nonumber\\
&=-(\E_j^*-\E_{j-1}^*)\Big(\frac{\dd}{\dd z}\bar{\beta}_j^*(z)\varepsilon_j^*(z)\Big).\nonumber
\end{align}
For example, by the algebraic manipulations on page 569 in this reference we obtain
$$
\widehat M_n^{*}(z) - \E^* \big [ \widehat M_n^{*}(z) \big ] = \sum_{j=1}^m Y_j^*
+  \sum_{j=1}^m Z_j^*
$$
where 
$$
Z_j^* = (\E_j^*-\E_{j-1}^*) \big [  
\bar{\beta}_j^*(z)^2 (\varepsilon_j^*(z) 
\delta_j^*(z)- {\beta}_j^*(z) r_j^{*\top} D_j^{-2} (z) r_j^* \varepsilon_j^*(z)^2)  \big].
$$
The $L^2$-norm of the first term is now  estimated as follows
\begin{align}
\nonumber 
\E \Big [ \E^* \Big | \sum_{i=1}^m
(\E_j^*-\E_{j-1}^*) (
\bar{\beta}_j^*(z)^2 \varepsilon_j^*(z) 
\delta_j^*(z) ) \Big |^2
\Big ] & = 
\sum_{i=1}^m
\E  \big |
(\E_j^*-\E_{j-1}^*) (
\bar{\beta}_j^*(z)^2 \varepsilon_j^*(z) 
\delta_j^*(z) ) \big |^2
\Big ]  \\
& \leq  4 \sum_{i=1}^m
\E  \big |
\bar{\beta}_j^*(z)^2 \varepsilon_j^*(z) 
\delta_j^*(z)  \big |^2
\Big ]\nonumber\\
&= {O \Big ( \delta_m^4  + {m^2 \over n} \Big )}~,
\label{det111}
\end{align}
where the last estimate follows from Proposition \ref{lemma: formel 2.1a}
and the bounds $\arrowvert\bar{\beta}_j^*(z)\arrowvert\leq \arrowvert z\arrowvert/\Im(z)$. 
The second term can be estimated by similar arguments and is of order
$O_{\P}  ( \delta_m^4  + m^2/ n  )$.
Note that one crucial difference to the analysis of \cite{baisil2004} is now caused by the fact that the random variables $\varepsilon_j^*$ and $\delta_j^*$ are not centered anymore with respect to $\E_{X_j^*}^*$:
\begin{align*}
\E_{X_j^*}^*\big(r_j^{*\top}D_j^*(z)^{-k}r_j^*\big)&= {\frac{1}{m}\tr\Big( L_n\hat SL_n^\top   D_j^*(z)^{-k}\Big)}\not=  \frac{1}{m}\tr\Big(L_nL_n^\top   D_j^*(z)^{-k}\Big),\ \ k=1,2.
\end{align*}
{with $\hat S=\frac1n\sum_{i=1}^nX_iX_i^\top$.} In view of the limiting distribution result, it is therefore sufficient to study linear combinations
\begin{align*}
\sum_{i=1}^r\sum_{j=1}^m\alpha_iY_j^*(z_i)\ \ \text{with }\alpha_1,\dots, \alpha_r\in\C,\ r\in\N,
\end{align*}
due to the Cram\'er-Wold device (since the real parts of these linear combinations are running over all real linear combinations of $\Re(\sum_{j=1}^mY_j^*(z_i)),\Im(\sum_{j=1}^mY_j(z_i))$, $i=1,\dots, r$, as $\alpha_1,\dots, \alpha_r$ varies over $\C$). 
 Note furthermore that it is sufficient to consider the case $\Im (z_i) > 0$ ($i=1, \ldots  , r$), because the distribution of any $C ( \CC , \R^2) $-valued random variable $Z$ is uniquely determined by its finite dimensional distributions $\mathcal{L}(Z(z_1),\dots, Z(z_k))$ with $z_1,\dots, z_k$ belonging to a dense subset of $\CC$ and $k\in\N$. 
For this purpose, we shall prove that the conditions of 
 Theorem \ref{thm: bmclt}
are 
satisfied for $$
\Re\Big (\sum_{i=1}^r\sum_{j=1}^m\alpha_iY_j^*(z_i)\Big ).
$$
By 
 the bounds $\arrowvert\bar{\beta}_j^*(z)\arrowvert\leq \arrowvert z\arrowvert/\Im(z)$ and 
 the estimate
$$
\Big\arrowvert\frac{1}{m}\tr\big(L_nL_n^\top   D_j^*(z)^{-2}\big)\Big\arrowvert\leq \frac{q}{m}
\Arrowvert L_n \Arrowvert^2_{S_{\infty}}\cdot\Arrowvert D_j^*(z)^{-2}\Arrowvert_{S_{\infty}}\leq  c(1+o(1))\frac{
\Arrowvert L_n \Arrowvert^2_{S_{\infty}}
}{\Im(z)^2},
$$
we find by Proposition \ref{lemma: formel 2.1a} 
(with $p={4}$) that 
\begin{align*}
\E\arrowvert Y_j^*(z)\arrowvert^4&\leq K\Big (\frac{\arrowvert z\arrowvert^4}{\Im(z)^4}\E\arrowvert\delta_j^*(z)\arrowvert^4+\frac{\arrowvert z\arrowvert^8}{\Im(z)^{16}}c^4(1+o(1))\E\arrowvert\varepsilon_j^*(z)\arrowvert^4\Big)  =\mathcal{O}\Big(\frac{\delta_m^4}{m}+\frac{m}{n}\Big)
\end{align*}
as $n\rightarrow\infty$. Consequently,
\begin{align}
\nonumber 
\E \bigg [ 
\sum_{j=1}^m\E^*\bigg(\bigg\arrowvert \sum_{i=1}^r\alpha_iY_j^*(z_i)\bigg\arrowvert^2 I\bigg\{\sum_{i=1}^r\alpha_iY_j^*(z_i)\arrowvert\geq\varepsilon\bigg\}\bigg)
\bigg ] & \leq\frac{1}{\varepsilon^2}\sum_{j=1}^m\E\bigg\arrowvert\sum_{j=1}^r\alpha_iY_j^*(z_i)\bigg\arrowvert^4 \\
& 
=  \mathcal{O}\Big(\delta_m^4
+\frac{m^2}{n}\Big)
= o(1 ) 
\label{det112}
\end{align}
as $n\rightarrow\infty$, and condition \eqref{eq: cond 2} of Theorem \ref{thm: bmclt} is fulfilled. 

In order to verify  condition~\eqref{eq: cond 1}, it is sufficient to show that for $z_1,z_2$ with $\Im(z_1)\not=0$, $\Im(z_2)\not=0$, 
\begin{equation}\label{eq: 3.10}
\sum_{j=1}^m\E^*_{j-1}Y_j^*(z_1)Y_j^*(z_2)
\longrightarrow_\mathbb{P} 
\begin{cases}
\frac{1}{3}\frac{(\underline{m}_{c,H}^0)'''(z_1)}{(\underline{m}_{c,H}^0)'(z_1)}
-\frac{1}{2}\Big(\frac{(\underline{m}_{c,H}^0)''(z_1)}{(\underline{m}_{c,H}^0)'(z_1)}\Big)^2 &\text{ if }z_1=z_2\\ 
~&~\\
2\frac{(\underline{m}_{{c,H}}^0)'(z_1)(\underline{m}_{{c,H}}^0)'(z_2)}{\big(\underline{m}_{c,H}^0(z_1)-\underline{m}_{{c,H}}^0
(z_2)\big)^2}-\frac{2}{(z_1-z_2)^2}&\text{ if }z_1\not=z_2
\end{cases}
\end{equation}
(note that $\overline{Y_j^*(z)}=Y_j^*(\bar{z})$).
By the theorem of dominated convergence,
\begin{align*}
{\frac{\partial^2}{\partial z_2\partial z_1}\eqref{eq: 3.11}=\eqref{eq: 3.10}}
\end{align*}
with 
\begin{align}\label{eq: 3.11}
\sum_{j=1}^m\E_{j-1}^*\Big[ (\E_j^*-\E_{j-1}^*)\big(\bar{\beta}_j^*(z_1)\varepsilon_j^*(z_1)\big) (\E_j^*-\E_{j-1}^*)\big( \bar{\beta}_j^*(z_2)\varepsilon_j^*(z_2)\big)\Big].
\end{align}
As for the classical CLT of linear spectral statistics, it follows from Vitali's convergence theorem that the convergence  of \eqref{eq: 3.10} in probability  follows from the corresponding stochastic convergence of \eqref{eq: 3.11}.
For analyzing \eqref{eq: 3.11}  we shall prove the following claims. 

\medskip

{\sc Claim I.} 
\begin{align}
\sum_{j=1}^m&\,\E_{j-1}^*\Big[(\E_j^*-\E_{j-1}^*)\big(\bar{\beta}_j^*(z_1)\varepsilon_j^*(z_1)\big)(\E_j^*-\E_{j-1}^*)\big(\bar{\beta}_j^*(z_2)\varepsilon_j^*(z_2)\big)\Big]\nonumber\\
&=\sum_{j=1}^m\E_{j-1}^*\Big[\E_j^*\big(\bar{\beta}_j^*(z_1)\varepsilon_j^*(z_1)\big)\E_j^*\big(\bar{\beta}_j^*(z_2)\varepsilon_j^*(z_2)\big)\Big]+o_{\P}(1).%
\end{align}

\medskip

{\sc Claim II.}

\begin{align}
\sum_{j=1}^m&\, \E_{j-1}^*\Big[\E_j^*\big(\bar{\beta}_j^*(z_1)\varepsilon_j^*(z_1)\big)\E_j^*\big(\bar{\beta}_j^*(z_2)\varepsilon_j^*(z_2)\big)\Big]\nonumber\\
&=\sum_{j=1}^mb_n^*(z_1)b_n^*(z_2)\E_{j-1}^*\Big[\E_j^*\big(\varepsilon_j^*(z_1)\big)\E_j^*\big(\varepsilon_j^*(z_2)\big)\Big] + o_{\P}(1).
\end{align}

\medskip

{\sc Claim III.}

\begin{align*}
\sum_{j=1}^m&\, b_n^*(z_1)b_n^*(z_2)\E_{j-1}^*\big[\E_j^*\big(\varepsilon_j^*(z_1)\big)\E_j^*\big(\varepsilon_j^*(z_2)\big)\big]\nonumber\\
&=2 b_n^*(z_1)b_n^*(z_2)\frac{1}{m^2}\sum_{j=1}^m\tr \Big(L_n^\top  \E_j^*\big(D_j^*(z_1)^{-1}\big)
L_nL_n^\top  
\E_j^*\big(D_j^*(z_2)^{-1}\big)L_n
\Big)
+o_{\P}(1).
\end{align*}

\medskip

In order to prove stochastic convergence and to determine the limit in probability of the right-hand side in Claim III, we shall prove the representation
\begin{align*}
\frac{1}{m^2}&\sum_{j=1}^m\tr \Big(L_n^\top  
\E_j^*\big(D_j^*(z_1)^{-1}\big)
L_n L_n^\top   
\E_j^*\big(D_j^*(z_2)^{-1}\big)L_n
\Big)\\
&= \frac{a_n(z_1,z_2)}{z_1z_2\,\underline{\tilde m}_n^0(z_1)\underline{\tilde m}_n^0(z_2)}\,\frac{1}{m}\sum_{j=1}^m \Big(1-\frac{j-1}{m}a_n(z_1,z_2)\Big)^{-1} + o_{\P}(1)
\end{align*}
for some function $a_n(z_1,z_2)$ that will be specified in Claim VI. This is the most involved part of the proof. Claims IV and V are intermediate steps on this way. For this purpose, we 
recall the notation
of $D_{ij}^*(z)$, $\beta_{ij}^*(z)$ and $b_1^*(z)$  in \eqref{ar11},
\eqref{ar20}, and \eqref{ar13}, respectively, which will be used intensively in the following discussion.

\medskip

{\sc Claim IV.} There exists some constant $K>0$, such that for all $n\in\N$ and any $j\leq m$
\begin{align}
\tr&\Big(\E_j^*\big(D_j^*(z_1)^{-1}\big)
L_n L_n^\top  
\E_j^*\big(D_j^*(z_2)^{-1}\big)
L_n L_n^\top  
\Big)\nonumber\\
&\ \ \ \times \bigg[1{-}\frac{j-1}{m^2}b_1^*(z_1)b_1^*(z_2)\tr\bigg(
\Big(z_2 I-\frac{m-1}{m}b_1^*(z_2) L_n L_n^\top  \Big)^{-1}
L_n L_n^\top  
\nonumber\\
&\hspace{45mm}\times\Big(z_1 I-\frac{m-1}{m}b_1^*(z_1)
L_n L_n^\top  
\Big)^{-1}
L_n L_n^\top  
\bigg)\bigg]\nonumber\\
&= \tr\bigg[\Big(z_1 I-\frac{m-1}{m}b_1^*(z_1) L_n L_n^\top  \Big)^{-1}
L_n L_n^\top  
\Big(z_2 I-\frac{m-1}{m}b_1^*(z_2) L_n L_n^\top  \Big)^{-1}
L_n L_n^\top  
\bigg] \label{eq: IV}\\
&\hspace{120mm}~+R(z_1,z_2)\nonumber
\end{align}
with $\E\arrowvert R(z_1,z_2)\arrowvert\mathds 1_{B_n}\leq K\sqrt{m}$, where $B_n=\{|b_1^*(z_l)+z_l\underline{\tilde m}_n^0(z_l)|\le\eta,\ l=1,2\}$ for some $\eta>0$, so that $\P(B_n)\to1$ by \eqref{eq: VHb}.

\medskip

\medskip

{\sc Claim V.} Recall the notation
of $\underline{\tilde m}_n^0$ 
in \eqref{det102}. For  any $j\leq m$, 
\begin{align}
&
\tr\Big(\E_j^*\big(D_j^*(z_1)^{-1}\big)
L_nL_n^\top  
\E_j^*\big(D_j^*(z_2)^{-1}\big)L_n L_n^\top  \Big)\nonumber\\
&\quad\quad
\times\bigg[1-\frac{j-1}{m^2}\underline{\tilde m}_n^0(z_1)\underline{\tilde m}_n^0(z_2)\nonumber\\
&\hspace{40mm}\cdot\tr\Big((I+ \underline{\tilde m}_n^0(z_2)
L_nL_n^\top  )^{-1}
L_nL_n^\top  
(I+\underline{\tilde m}_n^0(z_1)L_nL_n^\top  )^{-1}L_nL_n^\top  \Big)\bigg]\nonumber\\
&=\frac{1}{z_1z_2}\tr\Big((I+ \underline{\tilde m}_n^0(z_2)L_nL_n^\top  )^{-1}L_nL_n^\top  (I+\underline{\tilde m}_n^0(z_1)L_nL_n^\top  )^{-1}L_nL_n^\top  \Big)+R'(z_1,z_2)\label{eq: ClaimV}
\end{align}
with $\arrowvert R'(z_1,z_2)\arrowvert =\mathcal{O}_{\P}(\sqrt{m})$.

\medskip

\medskip
{\sc Claim VI:} We shall conclude the stochastic convergence in \eqref{eq: 3.10}:
\begin{equation*}
\sum_{j=1}^m\E^*_{j-1}Y_j^*(z_1)Y_j^*(z_2)
\longrightarrow_\mathbb{P} 
\begin{cases}
\frac{1}{3}\frac{(\underline{m}_{c,H}^0)'''(z_1)}{(\underline{m}_{c,H}^0)'(z_1)}
-\frac{1}{2}\Big(\frac{(\underline{m}_{c,H}^0)''(z_1)}{(\underline{m}_{c,H}^0)'(z_1)}\Big)^2 &\text{ if }z_1=z_2\\ 
~&~\\
2\frac{(\underline{m}_{{c,H}}^0)'(z_1)(\underline{m}_{{c,H}}^0)'(z_2)}{\big(\underline{m}_{c,H}^0(z_1)-\underline{m}_{{c,H}}^0
(z_2)\big)^2}-\frac{2}{(z_1-z_2)^2}&\text{ if }z_1\not=z_2.
\end{cases}
\end{equation*}
Note that the expression for $z_1=z_2$ is the continuous extrapolation of the one for $z_1\not=z_2$ for the removable singularities at $z_1=z_2$.

\subsection*{Proofs of Claim I -- Claim VI} \label{sec952}
$ $

\medskip

{\it Proof of Claim I.}
Due to the identity
\begin{align*}
\E_{j-1}^*&\,\Big[(\E_j^*-\E_{j-1}^*)\big(\bar{\beta}_j^*(z_1)\varepsilon_j^*(z_1)\big)(\E_j^*-\E_{j-1}^*)\big(\bar{\beta}_j^*(z_2)\varepsilon_j^*(z_2)\big)\Big]\\
&=\E_{j-1}^*\Big[\E_j^*\big(\bar{\beta}_j^*(z_1)\varepsilon_j^*(z_1)\big)\E_j^*\big(\bar{\beta}_j^*(z_2)\varepsilon_j^*(z_2)\big)\Big]\\
&\quad\quad\quad\quad- \E_{j-1}^*\big(\bar{\beta}_j^*(z_1)\varepsilon_j^*(z_1)\big)\E_{j-1}^*\big(\bar{\beta}_j^*(z_2)\varepsilon_j^*(z_2)\big),
\end{align*}
the claim follows if
\begin{align*}
\sum_{j=1}^m \E_{j-1}^*\big(\bar{\beta}_j^*(z_1)\varepsilon_j^*(z_1)\big)\E_{j-1}^*\big(\bar{\beta}_j^*(z_2)\varepsilon_j^*(z_2)\big)\longrightarrow_{\P}0.
\end{align*}
Employing the operator identity $\E_{j-1}^*=\E_{j-1}^*\E_{X_j^*}^*$, the independence of $\bar{\beta}_j^*(z)$ and $D_j^*(z)$ from $X_j^*$ and the bound
$$
\arrowvert \bar{\beta}_j^*(z)\arrowvert  \leq \frac{\arrowvert z\arrowvert}{\Im(z)},
$$
we deduce  
with $A_j(z)=\E_{j-1}^*\big(\bar\beta_j^*(z)D_j^*(z)^{-1}\big)$ and
Lemma \ref{lemma: neu1} that 
\begin{align}
\nonumber 
\E\bigg\arrowvert &\sum_{j=1}^m \E_{j-1}^*\big(\bar{\beta}_j^*(z_1)\varepsilon_j^*(z_1)\big)\E_{j-1}^*\big(\bar{\beta}_j^*(z_2)\varepsilon_j^*(z_2)\big)\bigg\arrowvert\\
\nonumber 
&\leq \sum_{j=1}^m\E\bigg\arrowvert  \E_{j-1}^*\big(\bar{\beta}_j^*(z_1)\E_{X_j^*}^*\varepsilon_j^*(z_1)\big)\E_{j-1}^*\big(\bar{\beta}_j^*(z_2)\E_{X_j^*}^*\varepsilon_j^*(z_2)\big)\bigg\arrowvert\\
\nonumber 
&
 = \sum_{j=1}^m\E\Big\arrowvert \E_{X_j^*}^*\big(r_j^{*\top}A_j(z_1)r_j^*-{1 \over m} \tr A_j(z_1)\big)\E_{X_j^*}^*\big(r_j^{*\top}A_j(z_2)r_j^*-{1 \over m}  \tr A_j(z_2)\big)\Big\arrowvert
\\
&\leq K(z_1,z_2) \Big({\delta_m^2 m  \over \sqrt{n}} +\frac{m^2}{n}\Big)
\label{det113} 
\end{align}

\medskip
{\it Proof of Claim II.}
Inserting the conditional expectation operator $\E^*$, the proof follows by representing the difference as a martingale difference sum and Burkholder's inequality with the exponent $2$ and Lemma \ref{lemma: ueber 4.3}. 

\medskip

{\it Proof of Claim III.}
Since $\arrowvert b_n^*(z)\arrowvert \leq \arrowvert z\arrowvert/\Im(z)$, $\E_{j-1}^*D_j^*(z)^{-1}=\E_{j}^*D_j^*(z)^{-1}$ and 
\begin{align*}
\E\bigg\arrowvert  \sum_{j=1}^m&\bigg\{\E_{j-1}^*\Big[\E_j^*\big(\varepsilon_j^*(z_1)\big)\E_j^*\big(\varepsilon_j^*(z_2)\big)\Big]\\
&\quad\quad\quad -\frac{2}{m^2}\tr \Big(L_n^\top  \E_j^*\big(D_j^*(z_1)^{-1}\big)
L_n L_n^\top 
\E_j^*\big(D_j^*(z_2)^{-1}\big)L_n^\top  \Big)\bigg\}\bigg\arrowvert \\
&\leq \sum_{j=1}^m\E\bigg\arrowvert \E_{j-1}^*\Big[\E_j^*\big(\varepsilon_j^*(z_1)\big)\E_j^*\big(\varepsilon_j^*(z_2)\big)\Big]\\
&\quad\quad\quad\quad -\frac{2}{m^2}\tr \Big(L_n^\top  \E_j^*\big(D_j^*(z_1)^{-1}L_n L_n^\top  \E_j^*\big(D_j^*(z_2)^{-1}\big) L_n^\top 
\Big)\bigg\arrowvert\\
&= \sum_{j=1}^m \E \bigg\arrowvert \E_{X_j^*}^*\bigg[\big(X_j^{*\top}A_{j} (z_1)X_j^*-\tr A_{j}(z_1) )\big(X_j^{*\top}A_{j}(z_2) X_j^*- \tr A_{j}(z_2)) \\
& \hspace{60mm}
-2 \tr(A_{j} (z_1) A_{j} (z_2) )\bigg]\bigg\arrowvert
\end{align*}
 with
$$
A_{j} (z_k)= \frac{1}{m}L_n^\top  \E_{j-1}^*\big(D_j^*(z_k)^{-1}\big)L_n 
,\ \ k=1,2,
$$
the proof is an immediate consequence of Lemma \ref{lemma: 1.15 neu}. 

\medskip
{\it Proof of Claim IV.}
The algebraic manipulations in \cite{baisil2004} on page 572 provide the representation
\begin{align}\label{eq: Dj}
D_j^*(z_1)^{-1}&= -\Big(z_1 I-  \frac{m-1}{m}b_1^*(z_1) L_n L_n^\top  \Big)^{-1} + b_1^*(z_1)A^*(z_1)+ B^*(z_1)+C^*(z_1)
\end{align} 
with
\begin{align*}
A^*(z_1)&=\sum_{\substack{1\leq i\leq m\\i\not=j}}\Big(z_1I-\frac{m-1}{m}b_1^*(z_1)
L_n L_n^\top  \Big)^{-1}
\Big(r_i^* {r_i^{*}}^\top -\frac{1}{m} L_n L_n^\top  \Big)D_{ij}^*(z_1)^{-1},\\
B^*(z_1)&= \sum_{\substack{1\leq i\leq m\\i\not=j}}\big(\beta_{ij}^*(z_1)-b_1^*(z_1)\big)\Big(z_1I -\frac{m-1}{m}b_1^*(z_1)
L_n L_n^\top  \Big)^{-1} r_i^* {r_i^{*}}^\top D_{ij}^*(z_1)^{-1}
\intertext{and}
C^*(z_1)&=\frac{1}{m}b_1^*(z_1)\Big(z_1 I -\frac{m-1}{m}b_1^*(z_1)
L_n L_n^\top  \Big)^{-1}
L_nL_n^\top \\
&\hspace{40mm}\times \sum_{\substack{1\leq i\leq m\\i\not=j}}\Big(D_{ij}^*(z_1)^{-1}-D_j^*(z_1)^{-1}\Big).
\end{align*}
In order to prove Claim IV, we establish the following steps.

\begin{itemize}
\item[(i)] 
If a possibly random $q\times q$-matrix $M$ satisfies $\Arrowvert M\Arrowvert_{S_{\infty}}\leq c$, then
\begin{align*}
\E\big\arrowvert \tr\big(B^*(z_1)M\big)\big\arrowvert\leq Kc\frac{\arrowvert z_1\arrowvert^2(1+\tfrac qm\Arrowvert\tilde\Sigma_n\Arrowvert_{S_\infty}/\Im(z_1))}{\Im(z_1)^5}\sqrt{m}.
\end{align*}
\item[(ii)]If a possibly random $q\times q$-matrix $M$ satisfies $\Arrowvert M\Arrowvert_{S_{\infty}}\leq c$, then
\begin{align*}
\E\big\arrowvert \tr\big(C^*(z_1)M\big)\big\arrowvert\leq Kc\frac{\arrowvert z_1\arrowvert(1+\tfrac qm\Arrowvert\tilde\Sigma_n\Arrowvert_{S_\infty}/\Im(z_1))}{\Im(z_1)^3}.
\end{align*}
\item[(iii)] We have
\begin{align}
\tr \Big(\E_j^*&\big(A^*(z_1)\big) L_n L_n^\top 
D_j^*(z_2)^{-1}
L_n L_n^\top  
\Big)\nonumber\\
&= \tr \Big(\E_j^*\big(\check{A}^*(z_1)\big)
L_n L_n^\top 
D_j^*(z_2)^{-1}
L_n L_n^\top  \Big)+\check{R}(z_1,z_2),\nonumber
\end{align}
with 
\begin{align*}
\check{A}^*(z_1)&=\sum_{1\leq i < j}\Big(z_1I  -\frac{m-1}{m}b_1^*(z_1) L_nL_n^\top  \Big)^{-1} \Big(r_i^* {r_i^{*}}^\top -\frac{1}{m}
L_n L_n^\top  \Big)D_{ij}^*(z_1)^{-1}
\end{align*} 
and $\E\arrowvert\check{R}(z_1,z_2)\arrowvert\leq K(z_1,z_2)\sqrt{m}$.
\item[(iv)] Moreover,
\begin{align*}
\tr \Big(\E_j^*\big(\check{A}^*(z_1)\big)L_nL_n^\top D_j^*(z_2)^{-1}L_nL_n^\top \Big)= A_1^*(z_1,z_2)+ R''(z_1,z_2),
\end{align*}
where $\E\arrowvert R''(z_1,z_2)\arrowvert\leq K\sqrt{m}$ and 
\begin{align*}
A_1^*(z_1,z_2)
&=-\sum_{1\leq i<j}\beta_{ij}^*(z_2) {r_i^{*}}^\top \E_j^*\Big(D_{ij}^*(z_1)^{-1}\Big)L_nL_n^\top D_{ij}^*(z_2)^{-1}r_i^* {r_i^{*}}^\top D_{ij}^*(z_2)^{-1}\\
&\hspace{30mm}\times L_nL_n^\top \Big(z_1I-\frac{m-1}{m}b_1^*(z_1)L_nL_n^\top \Big)^{-1}r_i^*.
\end{align*}
\item[(v)] Each summand of $A_1^*(z_1,z_2)$ satisfies the approximation
\begin{align*}
\beta_{ij}^*(z_2) &{r_i^{*}}^\top \E_j^*\Big(D_{ij}^*(z_1)^{-1}\Big)L_nL_n^\top D_{ij}^*(z_2)^{-1}r_i^* {r_i^{*}}^\top D_{ij}^*(z_2)^{-1}\\
&\hspace{25mm}\times L_nL_n^\top \Big(z_1I-\frac{m-1}{m}b_1^*(z_1)L_nL_n^\top \Big)^{-1}r_i^*\\
& =
b_1^*(z_2)\frac{1}{m^2}\tr\Big[\E_j^*\Big(D_{ij}^*(z_1)^{-1}\Big)L_nL_n^\top D_{ij}^*(z_2)^{-1} L_nL_n^\top \Big]\\
&\hspace{25mm}\times\tr\bigg[D_{ij}^*(z_2)^{-1} L_nL_n^\top \Big(z_1I-\frac{m-1}{m}b_1^*(z_1)L_nL_n^\top \Big)^{-1}L_nL_n^\top \bigg]\\
&\quad + R'''(z_1,z_2)
\end{align*}
with $\E\arrowvert R'''(z_1,z_2)\arrowvert\mathds 1_{B_n}\leq Km^{-1/2}$.
\end{itemize}

\medskip

Having established these five steps, the proof of Claim IV is conducted as follows. By the triangle and Jensen inequality,
\begin{align*}
\E&\bigg\arrowvert \tr\Big(\E_j^*\big(D_j^*(z_1)^{-1}\big)L_nL_n^\top \E_j^*\big(D_j^*(z_2)^{-1}\big)L_nL_n^\top \Big)\\
&\quad - \tr \bigg[\Big(z_1I-\frac{m-1}{m}b_1^*(z_1)L_nL_n^\top \Big)^{-1}L_nL_n^{\top}\E_j^*\big(D_j^*(z_2)^{-1}\big)L_nL_n^\top \bigg]\\
&\quad - b_1^*(z_1)\tr\Big(\E_j^*\big(A^*(z_1)\big)L_nL_n^\top \E_j^*\big(D_j^*(z_2)^{-1}\big)L_nL_n^\top \Big)\bigg\arrowvert\\
&= \E\bigg\arrowvert \E_j^*\Big[\tr \Big(B^*(z_1)L_nL_n^\top \E_j^*\big(D_j^*(z_2)^{-1}\big)L_nL_n^\top \Big)\Big]\\
&\quad \quad\quad+  \E_j^*\Big[\tr \Big(C^*(z_1)L_nL_n^\top \E_j^*\big(D_j^*(z_2)^{-1}\big)L_nL_n^\top \Big)\Big]\bigg\arrowvert\\
&\leq \E\Big\arrowvert\tr \Big(B^*(z_1)L_nL_n^\top \E_j^*\big(D_j^*(z_2)^{-1}\big)L_nL_n^\top \Big)\Big\arrowvert\\
&\quad +   \E\Big\arrowvert\tr \Big(C^*(z_1)L_nL_n^\top \E_j^*\big(D_j^*(z_2)^{-1}\big)L_nL_n^\top \Big)\Big\arrowvert\\
&\leq K(z_1,z_2)\sqrt{m},
\end{align*}
where the last inequality is established in steps (i) and (ii). By Jensen's inequality, the bound $\arrowvert b_1^*(z_1)\arrowvert\leq \arrowvert z_1\arrowvert/\Im(z_1)$ and steps (iii) and  (iv), 
\begin{align*}
\E&\bigg\arrowvert  \tr\Big(b_1^*(z_1)\E_j^*\big(A^*(z_1)\big)L_nL_n^\top \E_j^*\big(D_j^*(z_2)^{-1}\big)L_nL_n^\top \Big)- b_1^*(z_1)\E_j^*\big(A_1^*(z_1,z_2)\big)
\bigg\arrowvert\\
&\leq K(z_1,z_2)\sqrt{m}.
\end{align*}
It remains to prove the approximation 
\begin{align}
\E&\bigg\arrowvert b_1^*(z_1)\E_j^*\big(A_1^*(z_1,z_2)\big) +\nonumber \tr\Big(\E_j^*\big(D_j^*(z_1)^{-1}\big)L_nL_n^\top \E_j^*\big(D_j^*(z_2)^{-1}\big)L_nL_n^\top \Big)\nonumber\\
&\hspace{3mm} \times \frac{j-1}{m^2}b_1^*(z_1)b_1^*(z_2)\tr\bigg(\E_j^*\big(D_j^*(z_2)^{-1}\big)L_nL_n^\top 
\nonumber \Big(z_1 I_q-\frac{m-1}{m}b_1^*(z_1)L_nL_n^\top \Big)^{-1}L_nL_n^\top \bigg)\bigg\arrowvert\nonumber\\
&\leq K(z_1,z_2)\sqrt{m}.\label{eq: intermed}
\end{align}
Applying the approximation of step (v), we deduce
\begin{align*}
\E\Big\arrowvert b_1^*(z_1)\E_j^*\big(A_1^*(z_1,z_2)\big)-b_1^*(z_1)\E_j^*\big(\tilde{A}_1^*(z_1,z_2)\big)\Big\arrowvert\leq K(z_1,z_2)\sqrt{m},
\end{align*}
with
\begin{align*}
\tilde{A}_1^*&(z_1,z_2)\\
&=-\sum_{1\leq i<j\leq m}b_1^*(z_2)\frac{1}{m^2}\tr\Big[\E_j^*\Big(D_{ij}^*(z_1)^{-1}\Big)L_nL_n^\top D_{ij}^*(z_2)^{-1} L_nL_n^\top \Big]\\
&\hspace{5mm}\times\tr\bigg[D_{ij}^*(z_2)^{-1} L_nL_n^\top \Big(z_1I-\frac{m-1}{m}b_1^*(z_1)L_nL_n^\top \Big)^{-1}L_nL_n^\top \bigg].
\end{align*}
By Lemma 2.10 of \cite{baisil1998}, we may successively replace $\E_j^*\big(D_{ij}^*(z)^{-1}\big)$ by $\E_j^*\big(D_j^*(z)^{-1}\big)$ within the traces
\begin{align*}
&\bigg\arrowvert \tr\Big[\E_j^*\Big(D_{ij}^*(z_1)^{-1}\Big)L_nL_n^\top D_{ij}^*(z_2)^{-1} L_nL_n^\top \Big]\\
&\hspace{5mm}\times\tr\bigg[D_{ij}^*(z_2)^{-1} L_nL_n^\top \Big(z_1I-\frac{m-1}{m}b_1^*(z_1)L_nL_n^\top \Big)^{-1}L_nL_n^\top \bigg]\\
&\ \  -  \tr\Big[\E_j^*\Big(D_{j}^*(z_1)^{-1}\Big)L_nL_n^\top D_{j}^*(z_2)^{-1} L_nL_n^\top \Big]\\
&\hspace{5mm}\times\tr\bigg[D_{j}^*(z_2)^{-1} L_nL_n^\top \Big(z_1I-\frac{m-1}{m}b_1^*(z_1)L_nL_n^\top \Big)^{-1}L_nL_n^\top \bigg]\bigg\arrowvert\\
&\leq K(z_1,z_2)\cdot m,
\end{align*}
which proves \eqref{eq: intermed} and therefore the equation
\begin{align}\label{eq: D.5200}
\tr&\Big(\E_j^*\big(D_j^*(z_1)^{-1}\big)
L_n L_n^\top  
\E_j^*\big(D_j^*(z_2)^{-1}\big)
L_n L_n^\top  
\Big)\nonumber\\
&\ \ \ \times \bigg[1+\frac{j-1}{m^2}b_1^*(z_1)b_1^*(z_2)\tr\bigg(
\E_j^*\big(D_{j}^*(z_2)^{-1}\big) L_n L_n^\top  
\Big(z_1 I-\frac{m-1}{m}b_1^*(z_1)
L_n L_n^\top  
\Big)^{-1}
L_n L_n^\top  
\bigg)\bigg]\nonumber\\
&= -\tr\bigg[\Big(z_1 I-\frac{m-1}{m}b_1^*(z_1) L_n L_n^\top  \Big)^{-1}
L_n L_n^\top  
\E_j^*\big(D_{j}^*(z_2)^{-1}\big)
L_n L_n^\top  
\bigg] 
~+\tilde R(z_1,z_2)
\end{align}
with $\E\arrowvert \tilde R(z_1,z_2)\arrowvert\mathds 1_{B_n}\leq K\sqrt{m}$. Now, inserting   the representation \eqref{eq: Dj} into \eqref{eq: D.5200}, this time for $D_j^*(z_2)^{-1}$, and using (i) and (ii) together with the bound
\begin{align*}
\E\big\arrowvert \tr\big(A(z_2)M\big)\big\arrowvert\mathds 1_{B_n}&\leq \sum_{i\not=j}\E^{1/2}\mathds 1_{B_n}\bigg\arrowvert {r_i^*}^{\top}D_{ij}^*(z_2)^{-1}M\Big(z_2I-\frac{m-1}{m}b_1^*(z_2)L_nL_n^\top \Big)^{-1}r_i^*\\
&\ \ \ \ \ \ -\frac{1}{m}\tr\Big(D_{ij}^*(z_2)^{-1}M\Big(z_2I-\frac{m-1}{m}b_1^*(z_2)L_nL_n^\top \Big)^{-1}\Big)\bigg\arrowvert^2\\
&\leq K(z_2,\Arrowvert M\Arrowvert_{S_{\infty}})\sqrt{m}\Big(1+\frac{m}{\sqrt{n}}\Big)
\end{align*}
 for matrices $M$ of uniformly bounded spectral norm by   Proposition \ref{lemma: formel 2.1a}, Claim~IV is verified. Here and in step (v) below, the matrices in the quadratic forms depend on $X_1,\dots,X_n$ through $b_1^*(z_1)$ and $b_1^*(z_2)$, and Remark \ref{rem: randomM} is used with $\hat\omega=(b_1^*(z_1),b_1^*(z_2))$ and $w=(-z_1\underline{\tilde m}_n^0(z_1),-z_2\underline{\tilde m}_n^0(z_2))$. This is possible for $\eta$ small enough, since $|1+t\underline{\tilde m}_n^0(z_l)|$ is bounded away from zero for $t$ in the spectrum of $\tilde\Sigma_n$, so that $(z_lI-\frac{m-1}{m}\omega_lL_nL_n^\top)^{-1}$ is analytic and bounded for $|\omega_l+z_l\underline{\tilde m}_n^0(z_l)|\le2\eta$. Accordingly, the bounds in the proof of Claim~IV hold after multiplication with $\mathds 1_{B_n}$.

\medskip
\begin{itemize}
\item {\it Proof of (i).} First, applying the same reasoning as for inequality (2.10) in \cite{baisil2004}, we obtain the spectral norm bound
\begin{align}\label{eq: snb}
\Big\Arrowvert\Big(z_1I-\frac{m-1}{m}b_1^*(z_1)L_nL_n^\top \Big)^{-1}\Big\Arrowvert_{S_{\infty}}\leq \frac{1+\tfrac qm\Arrowvert\tilde\Sigma_n\Arrowvert_{S_\infty}/\Im(z_1)}{\Im(z_1)}.
\end{align}
By the Cauchy-Schwarz inequality, the upper bounds $\arrowvert \beta_{12}^*(z)\arrowvert,\arrowvert b_1^*(z)\arrowvert\leq \arrowvert z\arrowvert/\Im(z)$ and  $\Arrowvert D_{ij}^*(z)^{-1}\Arrowvert_{S_{\infty}}\leq 1/\Im(z)$, \eqref{eq: snb}, 
  and Lemma \ref{lemma: ueber 4.3},
\begin{align*}
 \E\big\arrowvert \tr\big(B^*(z_1)M\big)\big\arrowvert 
& \leq m\E^{1/2}\big\arrowvert\beta_{12}^*(z_1)-b_1^*(z_1)\big\arrowvert^2\\
&\hspace{10mm}\times \E^{1/2}\Big\arrowvert  {r_i^{*}}^\top D_{ij}^*(z_1)^{-1}M\Big(z_1I-\frac{m-1}{m}b_1^*(z_1)L_nL_n^\top \Big)^{-1}r_i^*\Big\arrowvert^2\\
&
\leq \frac{\arrowvert z_1\arrowvert^2}{(\Im(z_1))^2}
  \E^{1/2}\bigg[\big\arrowvert r_1^{*\top}r_1^*\big\arrowvert^2\big\Arrowvert D_{ij}^*(z_1)^{-1}\big\Arrowvert_{S_{\infty}}^2\Arrowvert M\Arrowvert_{S_{\infty}}^2\\
&\hspace{30mm}\times\Big\Arrowvert \Big(z_1I-\frac{m-1}{m}b_1^*(z_1)L_nL_n^\top \Big)^{-1}\Big\Arrowvert_{S_{\infty}}^2\bigg]\\
&\leq   K  \frac{\arrowvert z_1\arrowvert^2(1+\tfrac qm\Arrowvert\tilde\Sigma_n\Arrowvert_{S_\infty}/\Im(z_1))}{\Im(z_1)^5}\sqrt{m}.
\end{align*}
\item {\it Proof of (ii).} By the cyclic invariance of the trace, the bound $\arrowvert b_1^*(z)\arrowvert\leq\arrowvert z\arrowvert/\Im(z)$, the submultiplicativity of $\Arrowvert\cdot\Arrowvert_{S_{\infty}}$, \eqref{eq: snb}, and Lemma 2.6 in \cite{silversteinbai1995},
\begin{align*}
\E\big\arrowvert \tr\big(C^*(z_1)M\big)\big\arrowvert
& \leq \sum_{\substack{1\leq i\leq m\\i\not=j}}\E\bigg[\frac{1}{m}\big\arrowvert b_1^*(z_1)\big\arrowvert  \bigg\arrowvert  \tr\bigg( \Big(D_{ij}^*(z_1)^{-1}-D_j^*(z_1)^{-1}\Big)M\\
&\hspace{40mm}\times  \Big(z_1I-\frac{m-1}{m}b_1^*(z_1)L_nL_n^\top \Big)^{-1}L_nL_n^\top \bigg)\bigg\arrowvert\bigg]\\
&\leq Kc\frac{\arrowvert z_1\arrowvert(1+\tfrac qm\Arrowvert\tilde\Sigma_n\Arrowvert_{S_\infty}/\Im(z_1))}{\Im(z_1)^3}.
\end{align*}

\item {\it Proof of (iii)}. Because of $\E_j^*=\E_j^*\E_{X_i^*}^*$ for $i>j$, the triangle and Cauchy-Schwarz inequality,
\begin{align*}
 \E\bigg\arrowvert &\tr \Big(\E_j^*\big(A^*(z_1)\big)L_nL_n^\top D_j^*(z_2)^{-1}L_nL_n^\top \Big) -  \tr \Big(\E_j^*\big(\check{A}^*(z_1)\big)L_nL_n^\top D_j^*(z_2)^{-1}L_nL_n^\top \Big)\bigg\arrowvert\\
& =\E\bigg\arrowvert \sum_{i>j}\tr\Big[\Big(z_1I-\frac{m-1}{m}b_1^*(z_1)L_nL_n^\top \Big)^{-1}\E_{X_i^*}^*\Big(r_i^* {r_i^{*}}^\top -\frac{1}{m}L_nL_n^\top \Big)\\
&\hspace{8mm}\times \E_j^*\big(D_{ij}^*(z_1)^{-1}\big)L_nL_n^\top D_j^*(z_2)^{-1}L_nL_n^\top \Big]\bigg\arrowvert\\
& \leq \frac{1}{m}\sum_{i>j}\E^{1/2}\Big\Arrowvert \frac{1}{n}\sum_{i=1}^n Y_iY_i^\top-L_nL_n^\top\Big\Arrowvert_{S_2}^2\\
&\hspace{8mm}\times
\E^{1/2}\Big\Arrowvert 
\E_j^*\big(D_{ij}^*(z_1)^{-1}\big)L_nL_n^\top D_j^*(z_2)^{-1}L_nL_n^\top \\ &\hspace{25mm}\times\Big(z_1I-\frac{m-1}{m}b_1^*(z_1)L_nL_n^\top \Big)^{-1}\Big\Arrowvert_{S_2}^2.
\end{align*}
Applying \eqref{eq: snb} and the estimates $\Arrowvert D_{ij}^*(z)\Arrowvert_{S_{\infty}}$, $ \Arrowvert D_{j}^*(z)\Arrowvert_{S_{\infty}}\leq 1/\Im(z)$, 
\begin{align*}
\E^{1/2}&\Big\Arrowvert 
\E_j^*\big(D_{ij}^*(z_1)^{-1}\big)L_nL_n^\top D_j^*(z_2)^{-1}L_nL_n^\top  \times\Big(z_1I-\frac{m-1}{m}b_1^*(z_1)L_nL_n^\top  \Big)^{-1}\Big\Arrowvert_{S_2}^2\\
&\leq K(z_1,z_2)\sqrt{m},
\end{align*}
while 
\begin{align}
\label{det114}
\E^{1/2}\Big\Arrowvert \frac{1}{n}\sum_{i=1}^n Y_iY_i^\top - L_nL_n^\top \Big\Arrowvert_{S_2}^2=\mathcal{O}\Big(\frac{m}{\sqrt{n}}\Big), 
\end{align}
since both matrices are of dimension $q \times q$.
This  proves (iii). 
\item{\it Proof of (iv).}
Adding and subtracting $D_{ij}^*(z_2)$ and applying the Sherman-Morrison formula to the difference $D_j^*(z_2)^{-1}-D_{ij}^*(z_2)^{-1}$ in the subsequent expression $A_1^*(z_1,z_2)$, we obtain the decomposition
\begin{align*}
\tr \Big(\E_j^*\big(\check{A}^*(z_1)\big)&L_nL_n^\top  D_j^*(z_2)^{-1}L_nL_n^\top  \Big) =A_1^*(z_1,z_2) + A_2^*(z_1,z_2)+A_3^*(z_1,z_2) 
\end{align*}
with
\begin{align*}
A_1^*(z_1,z_2)&=-\sum_{1\le i<j} \beta_{ij}^*(z_2) {r_i^{*}}^\top \E_j^*\big(D_{ij}^*(z_1)^{-1}\big)L_nL_n^\top  D_{ij}^*(z_2)^{-1}r_i^*\\
&\hspace{10mm}\times 
 {r_i^{*}}^\top D_{ij}^*(z_2)^{-1}L_nL_n^\top  \Big(z_1I_q-\frac{m-1}{m}b_1^*(z_1)L_nL_n^\top  \Big)^{-1}r_i^*\\
A_2^*(z_1,z_2)&= -\tr\bigg\{\sum_{1\leq i<j}\Big(z_1I-\frac{m-1}{m}b_1^*(z_1)L_nL_n^\top  \Big)^{-1}\frac{1}{m}L_nL_n^\top  \\
&\hspace{15mm}\times\E_j^*\big(D_{ij}^*(z_1)^{-1}\big)L_nL_n^\top  \Big(D_j^*(z_2)^{-1}-D_{ij}^*(z_2)^{-1}\Big)L_nL_n^\top  \bigg\}\\
A_3^*(z_1,z_2)&=\tr\bigg\{\sum_{1\leq i<j}\Big(z_1I-\frac{m-1}{m}b_1^*(z_1)L_nL_n^\top  \Big)^{-1}\Big(r_i^* {r_i^{*}}^\top -\frac{1}{m}L_nL_n^\top  \Big)\\
&\hspace{10mm}\times \E_j^*\big(D_{ij}^*(z_1)^{-1}\big)L_nL_n^\top  D_{ij}^*(z_2)^{-1}L_nL_n^\top  \bigg\}.
\end{align*}
By Lemma 2.6 in \cite{silversteinbai1995} and \eqref{eq: snb},
\begin{align*}
\big\arrowvert A_2^*(z_1,z_2)\big\arrowvert \leq K\frac{1+\tfrac qm\Arrowvert\tilde\Sigma_n\Arrowvert_{S_\infty}/\Im(z_1)}{(\Im(z_1))^2}.
\end{align*}
Analogously to the proof of (i), we find
\begin{align*}
\E\big\arrowvert A_3^*(z_1,z_2)\big\arrowvert \leq K\frac{1+\tfrac qm\Arrowvert\tilde\Sigma_n\Arrowvert_{S_\infty}/\Im(z_1)}{(\Im(z_1))^3}\sqrt{m}.
\end{align*}

\item {\it Proof of (v).} By the Cauchy-Schwarz inequality and 
 Lemma   \ref{lemma: ueber 4.3}, we have 
\begin{align*}
\hspace{-5mm}\E\bigg\arrowvert &\big(\beta_{ij}^*(z_2)-b_1^*(z_2)\big) {r_i^{*}}^\top \E_j^*\big(D_{ij}^*(z_1)^{-1}\big)L_n L_n^\top  D_{ij}^*(z_2)^{-1}r_i^*\\
&\hspace{20mm}\times  {r_i^{*}}^\top D_{ij}^*(z_2)^{-1}L_n L_n^\top  \Big(z_1I  -\frac{m-1}{m}b_1^*(z_1)L_nL_n^\top  \Big)^{-1}r_i^*\bigg\arrowvert\\
&
 \leq \E^{1/2}\big\arrowvert \beta_{ij}^*(z_2)-b_1^*(z_2)\big\arrowvert^2\cdot \E^{1/2}\bigg\arrowvert  {r_i^{*}}^\top \E_j^*\big(D_{ij}^*(z_1)^{-1}\big)L_nL_n^\top  D_{ij}^*(z_2)^{-1}r_i^*
\\
&\hspace{20mm}
\ \times  {r_i^{*}}^\top D_{ij}^*(z_2)^{-1}L_nL_n^\top  \Big(z_1I-\frac{m-1}{m}b_1^*(z_1)L_nL_n^\top \Big)^{-1}r_i^*\bigg\arrowvert^2
\\
&\leq K(z_1,z_2)\frac{1}{\sqrt{m}},
\end{align*}
where the last inequality follows by  the Cauchy-Schwarz inequality, 
\eqref{eq: snb}, the bound
$|{r_1^{*}}^\top  C r^*_1| \leq \| C\|_{S_\infty}  \| r_1^* \|^2$  and  the fact that $\E  \| r_1^* \|^p \leq  c $, which follows from \eqref{det37a}. 
Next,
\begin{align*}
\hspace{-5mm}\E\bigg\arrowvert &b_1^*(z_2) {r_i^{*}}^\top \E_j^*\big(D_{ij}^*(z_1)^{-1}\big)L_nL_n^\top  D_{ij}^*(z_2)^{-1}r_i^*\\
&\hspace{20mm}\times  {r_i^{*}}^\top D_{ij}^*(z_2)^{-1}L_nL_n^\top  \Big(z_1I-\frac{m-1}{m}b_1^*(z_1)L_nL_n^\top  \Big)^{-1}r_i^*\\
&- b_1^*(z_2)\frac{1}{m^2}\tr\Big[\E_j^*\big(D_{ij}^*(z_1)^{-1}\big)L_nL_n^\top  D_{ij}^*(z_2)^{-1} L_nL_n^\top  \Big]\\
&\hspace{5mm}\times\tr\bigg[D_{ij}^*(z_2)^{-1} L_nL_n^\top  \Big(z_1I-\frac{m-1}{m}b_1^*(z_1)L_nL_n^\top  \Big)^{-1}L_nL_n^\top  \bigg]\bigg\arrowvert\\
&\leq \E\bigg\arrowvert b_1^*(z_2)\frac{1}{m}\tr\Big[\E_j^*\big(D_{ij}^*(z_1)^{-1}\big)L_nL_n^\top  D_{ij}^*(z_2)^{-1} L_nL_n^\top  \Big]\\
&\hspace{10mm}\times\bigg(  {r_i^{*}}^\top D_{ij}^*(z_2)^{-1}L_nL_n^\top  \Big(z_1I-\frac{m-1}{m}b_1^*(z_1)L_nL_n^\top  \Big)^{-1}r_i^*\\
&\hspace{15mm}-\frac{1}{m}\tr\bigg[D_{ij}^*(z_2)^{-1} L_nL_n^\top  \Big(z_1I-\frac{m-1}{m}b_1^*(z_1)L_nL_n^\top  \Big)^{-1}L_nL_n^\top  \bigg]\bigg)\bigg\arrowvert\\
&
+ \E\bigg\arrowvert b_1^*(z_2)\bigg( {r_i^{*}}^\top \E_j^*\big(D_{ij}^*(z_1)^{-1}\big)L_nL_n^\top  D_{ij}^*(z_2)^{-1}r_i^*
\\
&\hspace{30mm}
-\frac{1}{m}\tr\Big[\E_j^*\big(D_{ij}^*(z_1)^{-1}\big)L_nL_n^\top  D_{ij}^*(z_2)^{-1} L_nL_n^\top  \Big]\bigg)
\\
&\hspace{10mm}
\times 
{r_i^{*}}^\top D_{ij}^*(z_2)^{-1}L_nL_n^\top  \Big(z_1I-\frac{m-1}{m}b_1^*(z_1)L_nL_n^\top  \Big)^{-1}r_i^*
\bigg\arrowvert. 
\end{align*}
By an application of the Cauchy-Schwarz inequality to the first term,  \eqref{eq: snb}, the bounds $\arrowvert b_1^*(z)\arrowvert\leq \arrowvert z\arrowvert/\Im(z)$ and $\Arrowvert D_{ij}^*(z)^{-1}\Arrowvert_{S_{\infty}}\leq 1/\Im(z)$,  and  Proposition  \ref{lemma: formel 2.1a}, the right-hand side of the last inequality is bounded by
$$
K(z_1,z_2)\frac{1}{\sqrt{m}}.
$$
\end{itemize}

\medskip
{\it Proof of Claim V.} 
We shall prove
\begin{equation}
\big\arrowvert b_1^*(z) + z\underline{\tilde m}_n^0(z)\big\arrowvert=\mathcal{O}_{\P}(m^{-1/2}).\label{eq: VHb}
\end{equation}
First,
\begin{align*}
\big\arrowvert b_1^*(z)-b_n^*(z)\big\arrowvert &=\Big\arrowvert b_1^*(z)b_n^*(z)\frac{1}{m}\E^*\tr\Big(L_n L_n^\top \big(D_1^*(z)^{-1}-D_{12}^*(z)^{-1}\big)\Big)\Big\arrowvert\\
&\leq \frac{\arrowvert z\arrowvert^2}{\big(\Im(z)\big)^3}\frac{\Arrowvert L_n \Arrowvert_{S_{\infty}}^2}{m}
\end{align*}
by the inequalities $\arrowvert b_1^*(z)\arrowvert, \arrowvert b_n^*(z)\arrowvert\leq \arrowvert z\arrowvert/\Im(z)$ and Lemma 2.6 in \cite{silversteinbai1995}. Next, note that the estimate 
\eqref{eq: VHa} is still valid under the truncation scheme used in section. This follows by showing the inequality
$\E \arrowvert \gamma_1^* (z) \arrowvert^2=\mathcal{O}
( { 1 \over m }
+ \frac{m}{n} )$
 by an application of Proposition \ref{lemma: formel 2.1a} with $p=2$ and using the same arguments following equation \eqref{eq: bound gamma}. Therefore, it follows that 
$$
\E\arrowvert b_n^*(z)-\E^*\beta_1^*(z)\arrowvert=\mathcal{O}(m^{-1/2}), 
$$
and by  \eqref{ar9}
we have 
$\E^*\beta_1^*(z)= - z  \E^* \underline{m}_n^*(z)$. Together with \eqref{eq: 5.1b}, 
which is also  valid under the truncation scheme used in section (see Remark \ref{det101}),
we conclude \eqref{eq: VHb}. Replacing successively $b_1^*(z_i)$ in the left-hand side of \eqref{eq: IV} by ${-} z_i  \underline{\tilde{m}}_n^0(z_i)$, $i=1,2$, and employing Lemma 2.6 in \cite{silversteinbai1995}, the proof of Claim V is completed. 

\medskip

{\it Proof of Claim VI.}
{Write $\underline{\tilde m}_i=\underline{\tilde m}_n^0(z_i)$ and $I_n(z_1,z_2)=\int t^2\big(1+t\underline{\tilde m}_1\big)^{-1}\big(1+t\underline{\tilde m}_2\big)^{-1}\dd\mu^{\tilde\Sigma_n}(t)$, so that the trace on the right-hand side of \eqref{eq: ClaimV} equals $qI_n(z_1,z_2)$. Dividing \eqref{eq: ClaimV} by $q$ gives
\begin{align*}
&{1 \over q}
\tr\Big(\E_j^*\big(D_j^*(z_1)^{-1}\big)
L_n L_n^\top 
\E_j^*\big(D_j^*(z_2)^{-1}\big)L_n L_n^\top  \Big)\Big[1-\frac{j-1}{m}\,a_n(z_1,z_2)\Big]\\
&\qquad= \frac{I_n(z_1,z_2)}{z_1z_2} + \mathcal{O}_{\P}(m^{-1/2})
\end{align*}
with
$$
a_n(z_1,z_2)=\frac qm\,\underline{\tilde m}_1\underline{\tilde m}_2\,I_n(z_1,z_2),
$$
which is the analogue of the expression inside the absolute value sign in (2.19) of \cite{baisil2004}. As for (2.19) in \cite{baisil2004}, $\limsup_n|a_n(z_1,z_2)|<1$. By \eqref{eq: VHb} and $|b_1^*(z)-b_n^*(z)|\le K/m$ we have $b_n^*(z_i)=-z_i\underline{\tilde m}_i+\mathcal O_\P(m^{-1/2})$. Hence, by Claims I--III, \eqref{eq: 3.11} can be written as
\begin{align*}
&2\,b_n^*(z_1)b_n^*(z_2)\frac1{m^2}\sum_{j=1}^m\tr\Big(L_n^\top\E_j^*\big(D_j^*(z_1)^{-1}\big)L_nL_n^\top\E_j^*\big(D_j^*(z_2)^{-1}\big)L_n\Big)+o_\P(1)\\
&\qquad=2\,a_n(z_1,z_2)\frac{1}{m}\sum_{j=1}^m \Big(1-\frac{j-1}{m}a_n(z_1,z_2)\Big)^{-1} + o_{\P}(1)
\end{align*}
(the factor $2$ comes from Claim III, i.e.\ from the real case $\E X_{11}^4=3$; compare page 578 of \cite{baisil2004}, where the same computation in the complex Gaussian case yields half of the limit).
We now use the Representative  Subpopulation Condition \ref{def: rsc} to 
conclude that
\begin{equation}\label{AR10000}
\begin{split}
&a_n(z_1,z_2  ) \longrightarrow a(z_1,z_2):=
 c\,\underline{ m}^{0}_{c,H} (z_1)\underline{ m}^{0}_{c,H} (z_2)\\
 &\qquad\times\int t^2\big(1+t\underline{ m}^{0}_{c,H} (z_1)\big)^{-1}\big(1+  t \underline{ m}^{0}_{c,H} (z_2)\big)^{-1}\,d H(t)
\end{split}
\end{equation}
in probability.} For this purpose,  we note that it follows from \eqref{eq: similarity3}, \eqref{det108}  and Condition \ref{def: rsc}
that $ \mu^{\tilde \Sigma_n}   \Rightarrow   H  $
in probability and  $\underline{\tilde m}_n^0(z) \to  \underline{ m}^{0}
_{c,H}(z)$ in probability, see equation \eqref{det103}.  Together with the fact that the sequence of  bounded continuous functions
$$
t\mapsto \frac{t^2}{\big(1+t\underline{ \tilde  m}_{n}^0 (z_1)\big)\big(1+  t \underline{\tilde  m}_n^0(z_2)\big)}
$$
converges in probability uniformly on compacts to its limit, this implies \eqref{AR10000}.  
The final steps in the proof of \eqref{eq: 3.10} are then the same as on page 578 in \cite{baisil2004}{: \eqref{eq: 3.11} converges in probability to 
$$
2a(z_1,z_2)\int_0^1(1-ta(z_1,z_2))^{-1}\dd t=2\int_0^{a(z_1,z_2)}(1-w)^{-1}\dd w,
$$
and $a(z_1,z_2)=1+\underline m^0_{c,H}(z_1)\underline m^0_{c,H}(z_2)(z_1-z_2)/(\underline m^0_{c,H}(z_2)-\underline m^0_{c,H}(z_1))$ yields the limit in \eqref{eq: 3.10}}. 

\subsubsection{Proof of Proposition \ref{prop} 
under conditional  tightness in probability  of $M_n^{*} $} \label{sec953}
We start with the following lemma whose assumption can be deduced from the unconditional tightness of the sequence $\mathcal{L} ({\widehat M_n^*} )$
by an  application of Markov's inequality.
Subsequently, we use the abbreviation $U:=(X_1,, X_2, \dots  \Pi_1, \Pi_2, \dots  )$ and denote by $\P^{\widehat{M}_n^*}(\cdot\,\arrowvert u)$ the conditional distribution $\mathcal{L}\big((\widehat{M}_n^*(z))_{z\in\mathcal{C}}\arrowvert  U=u  \big)$. 

\begin{lemma} \label{lemneu}
Assume that for every $\varepsilon>0$ and $\eta>0$, there exists some compact set $K$ such that
\begin{equation}\label{eq: tt}
\sup_n\P\Big(\P^{\widehat{M}_n^*}(K^c\arrowvert U )>\varepsilon\Big)< \eta.
\end{equation}   
Then there exists some array of measurable sets $(A_{m,n})_{m,n\in\N}$ satisfying $\sup_n\P(U\in A_{m,n}^c)\rightarrow 0$ as $m\rightarrow\infty$ such that the family
$$
\Big(\P^{\widehat{M}_n^*}(\cdot\arrowvert u): u\in A_{m,n},n\in\N \Big)
$$
is tight for every $m\in\N$. 
\end{lemma}

\begin{proof}
    Let
$\varepsilon\searrow 0$ be some null sequence and $\eta_k=2^{-k}$, $k\in\N$. Let $(K_k)$ be some increasing sequence of compacts such that $K_k$ satisfies \eqref{eq: tt} for $\varepsilon_k$ and $\eta_k$. Define
$$
A_{m,n}^c:=\bigcup_{k\geq m}\Big\{u: \P^{\widehat{M}_n^*}(K_k^c\arrowvert u)>\varepsilon_k\Big\}.
$$
Every measure $\P^{\widehat{M}_n^*}(\cdot\arrowvert u)$ with $u\in A_{m,n}$  satisfies $\P^{\widehat{M}_n^*}(K_k^c\arrowvert u)\leq\varepsilon_k$, $k\geq m$. Hence, the family 
$$
\Big(\P^{\widehat{M}_n^*}(\cdot\arrowvert u): u\in A_{m,n},n\in\N\Big)
$$ 
is tight for every $m\in\N$ and by the sigma-subadditivity,
$$
\sup_n\P(A_{n,m}^c )\leq\sum_{k\geq m}\frac{1}{2^k}\longrightarrow 0\ \ \text{as }m\rightarrow\infty.
$$
\end{proof}

Assume now that the condition of Lemma \ref{lemneu} holds. Then the weak convergence of the finite dimensional distribution in probability implies
the assertion of Proposition \ref{prop}, which can be seen as follows.
We define for any measure $R$ on the Borel field on the continuous function on $\CC$ the operation 
$$
R\cdot \mathds{1}_A (u)  = \begin{cases}
   R & \text{ if } u \in A \\
   \delta_0 & \text{  otherwise. } 
\end{cases}
$$
Then
\begin{align}
   \mathbb{P} \big ( d_{\rm BL}  &\big ( \mathbb{P}^{\widehat{M}_n^*}(\cdot | U ), \mathcal{L} (Z) \big ) > \varepsilon \big )\nonumber\\
   &  \leq 
    \mathbb{P} \big ( d_{\rm BL}  \big ( \mathbb{P}^{\widehat{M}_n^*}(\cdot | U ), \mathcal{L} (Z) \big ) > \varepsilon , ~ U \in  A_{m.n}  \big ) + 
     \mathbb{P} \big ( U \in  A_{m.n}^c\big ) \nonumber \\
      &  = 
    \mathbb{P} \big ( d_{\rm BL}  \big ( \mathbb{P}^{\widehat{M}_n^*}(\cdot | U ) \mathds{1}_{A_{m,n} }(U)  , \mathcal{L} (Z) \mathds{1}_{A_{m,n} }(U) \big ) > \varepsilon , ~  U \in  A_{m.n}  \big ) + 
     \mathbb{P} \big ( U \in  A_{m.n}^c\big ) \nonumber \\
  &  \leq 
    \mathbb{P} \big ( d_{\rm BL}  \big ( \mathbb{P}^{\widehat{M}_n^*}(\cdot | U ) \mathds{1}_{A_{m,n} }(U)  , \mathcal{L} (Z) \mathds{1}_{A_{m,n} }(U) \big ) > \varepsilon  \big ) + 
     \mathbb{P} \big (   U \in  A_{m.n}
     ^c\big )
     \label{det104}
     \end{align}
By Lemma \ref{lemneu}, the family $ \big \{ \mathbb{P}^{\widehat{M}_n^*}(\cdot | u ) \mathds{1}_{A_{m,n} }(u) \big \}_{n \in \mathbb{N}}$
is tight uniformly in $u$ for every $m \in \mathbb{N}$. The same holds true for the sequence $\big \{ \mathcal{L} (Z) \mathds{1}_{A_{m,n} }(u) \}_{n \in \mathbb{N}} $ which attains only the measures  $\mathcal{L} (Z)$ and $\delta_0$.
Now, let $(n_k)$ be an arbitrary subsequence. By the weak convergence of the finite dimensional distributions in probability (Section \ref{sec841}), there are a further subsequence $(\tilde n_k)$, which does not depend on $u$, and a set $A$ with $\P(U\in A)=1$ such that along $(\tilde n_k)$ the finite dimensional distributions of $\mathbb{P}^{\widehat{M}_{n}^*}(\cdot | u )$ at the points of a countable dense subset of $\CC$ converge to those of $Z$ for every $u\in A$. Fix $u\in A$ and a subsequence of $(\tilde n_k)$. Either $u\notin A_{m,n}$ infinitely often along it, or, by Lemma \ref{lemneu}, a further subsequence of $\mathbb{P}^{\widehat{M}_{n}^*}(\cdot | u )$ converges weakly, and its limit is $\mathcal{L}(Z)$, since the law of a process with continuous paths is determined by its finite dimensional distributions on a dense set. In both cases $d_{\rm BL}(\mathbb{P}^{\widehat{M}_{n}^*}(\cdot | u ),\mathcal{L}(Z))\mathds{1}_{A_{m,n}}(u)\to0$ along a further subsequence. Hence this sequence tends to zero along $(\tilde n_k)$ for every $u\in A$, and as $(n_k)$ was arbitrary,
$$
d_{\rm BL}  \Big( \mathbb{P}^{\widehat{M}_{n}^*}(\cdot | U )   ,\  \mathcal{L} (Z)  \Big)\mathds{1}_{A_{m,n}}(U)    \rightarrow_{\P} 0.
$$
 Thus, it follows from \eqref{det104} that
$$
\limsup_{n\to \infty} 
   \mathbb{P} \big ( d_{\rm BL}  \big ( \mathbb{P}^{\widehat{M}_n^*}(\cdot | U ), \mathcal{L} (Z) \big ) > \varepsilon \big ) 
\leq  \sup_{n \in \mathbb{N}}
  \mathbb{P} \big ( U \in  A_{m.n}^c\big )
$$
As the left-hand side does not depend on $m$, the assertion of Proposition \ref{prop}  now follows by taking the limit $ m \to \infty$.

\subsubsection{Conditional tightness of the process $\widehat M_n^{*} - \E^* [\widehat M_n^{*}]  $ in probability}
\label{sec842}

It is  sufficient to prove the conditional  moment condition  (12.51) in \cite{billingsley1999}, which follows from  \begin{align}
\label{det66}
    \sup_{n} \P \bigg ( 
    \sup_{ z_1,z_2}     
    \mathbb{E}^*   \Big [
    { |  \widehat M_n^{*}  (z_1) 
    - \E^* [  \widehat M_n^{*}  (z_1)] - ( \widehat M_n^{*}  (z_2) - \E^* [ \widehat M_n^{*}  (z_2)] |^{2}   \over |z_1 - z_2 |^{2}  }
    \Big ] \geq  K  \bigg )  = o(1)
\end{align}
for $K \to \infty$. 
Using similar arguments as in \cite{baisil2004}, p.582 we obtain
\begin{align}
 \label{det1} 
 & {   \widehat M_n^{*}  (z_1) 
   - \E^* [  \widehat M_n^{*}  (z_1)] - ( \widehat M_n^{*}  (z_2) - \E^* [ \widehat M_n^{*}  (z_2)] )   \over z_1 - z_2   } \\
    \nonumber  
& 
\hspace{56mm}=  H_{1n} (z_1,z_2) + H_{2n} (z_1,z_2) + H_{3n} (z_1,z_2)
 \end{align}
 where
 \begin{align*}
 H_{1n} (z_1,z_2) &= 
 \sum_{j=1}^m 
 \big ( \mathbb{E}_{j}^* -\mathbb{E}_{j-1}^* \big ) \beta_j^* (z_1) \beta_j^* (z_2) \big (r_j{^*}^\top (D_j^* (z_1))^{-1} (D_j^* (z_2))^{-1} r_j^* \big ) ^2 
 \\  
  H_{2n} (z_1,z_2) &= 
-   \sum_{j=1}^m 
 \big ( \mathbb{E}_{j}^* -\mathbb{E}_{j-1}^* \big ) \beta_j^* (z_1)  r_j{^*}^\top (D_j^* (z_1))^{-2} (D_j^* (z_2))^{-1} r_j^* 
 \\
 H_{3n} (z_1,z_2)  &= -   
\sum_{j=1}^m 
 \big ( \mathbb{E}_{j}^* -\mathbb{E}_{j-1}^* \big ) \beta_j^* (z_2)  r_j{^*}^\top (D_j^* (z_2))^{-2} (D_j^* (z_1))^{-1} r_j^* 
\end{align*}
We 
begin deriving a uniform conditional moment  bound for the quantity $H_{2n} (z_1,z_2) $. For this purpose it is crucial to define 
\begin{equation}
\label{rev1000}
\begin{split}
\tilde b _n^* (z)  & = \Big (  1 +  {1 \over m} {\rm tr} ( \E^*\big  [  \mathds{1}_{{\cal A}_n}   L_nL_n^\top D_1^*(z)^{-1}  \big ] \big )  \Big )^{-1} \\
\tilde  \gamma_1^* (z)  &= {r_1^*}^\top D_1^*(z)^{-1} r_1^*  - {1 \over m} {\rm tr} ( \E^*\big  [  \mathds{1}_{{\cal A}_n}   L_nL_n^\top D_1^*(z)^{-1}  \big ] \big )~
\end{split}
\end{equation}
and  obtain the identity 
\begin{align}
    \label{r4}
   \tilde  b _n^*(z) = \beta_1^*(z) +  \beta_1^*(z) \tilde  b _n^*(z)  \tilde  \gamma_1^*(z) 
\end{align}
where $\beta_1^*(z)$
 is  defined in \eqref{ar11a}. 
 Note that for the subsequent arguments, it is essential that the indicator $\mathds{1}_{{\cal A}_n} $
is included inside of the conditional expectation $\mathbb{E}^*$.
It  then follows by a  
straightforward calculation that 
$$
     H_{2n} (z_1,z_2) = -\tilde  b_n^*(z_1)W_1^*(z_1,z_2) + \tilde  b_n^*(z_1)W_2^*(z_1,z_2)
$$
where 
\begin{align*}
  W_1^* (z_1,z_2)   &=   \sum_{j=1}^m  \big (  \mathbb{E}^*_j 
   - \mathbb{E}^*_{j-1} \big)
  r_j{^*}^\top D_j^* (z_1)^{-2} D_j^* (z_2)^{-1} r_j^* ~,
  \\ 
  W_2^* (z_1,z_2) &= 
   \sum_{j=1}^m  \big (  \mathbb{E}^*_j 
   - \mathbb{E}^*_{j-1} \big)  \Big [
 \beta_j^*(z_1) 
   r_j{^*}^\top D_j^* (z_1)^{-2} D_j^* (z_2)^{-1} r_j^*  \tilde  \gamma_j^*(z_1)
  \Big   ] ~.
\end{align*}
Note that 
\begin{align}
    \label{det41}
   \E ^* | H_{2n}(z_1,z_2)  |^{2} \leq  C \big ( 
   | \tilde b_n^*(z_1) |^2  \E ^*  | W_1^*  (z_1,z_2)  |^{2}  +  |  \tilde b_n^*(z_1)| ^2 \E ^*  | W_2^* (z_1,z_2)  |^{2}\big ) ~,
\end{align}
and 
note that  by Lemma \ref{lemrev1} it follows that 
\begin{align}
    \label{det43a} 
 \sup_{z\in  {\cal C}_n} | \tilde  b _n^*(z) |  
\end{align}
is a tight sequence. 
We will show at the end of this section that  
\begin{align}
  \sup_{z_1,z_2 \in  {\cal C}_n} \E^*  | W_j^* (z_1,z_2)  |^{2}  ~~~\text{are tight sequences ($j=1,2$).}
      \label{det43b}
\end{align} 
 Therefore, observing \eqref{det41} it follows that $
 \sup_{z_1,z_2 \in {\cal C}_n} \E ^*|  H_{2n} (z_1,z_2) |^2  
$  is tight as well. The  term $H_{3n} (z_1,z_2)$ in  the decomposition \eqref{det1}  can be treated 
in the same way.

To prove tightness of the sequence  $
 \sup_{z_1,z_2 \in {\cal C}_n} \E ^*|  H_{1n} (z_1,z_2) |^2 $ 
 we use the same arguments as in \cite{baisil2004} and obtain the representation
\begin{align}
    \label{det65b}
H_{1n} (z_1,z_2) = 
\tilde b_n^*(z_1)\tilde  b_n^*(z_2)Y_1^*(z_1,z_2) -\tilde  b_n^*(z_2)Y_2^*(z_1,z_2) -\tilde  b_n^*(z_1)\tilde  b_n^*(z_2)Y_3^*(z_1,z_2),
\end{align}
where
\begin{align*}
    Y_1^*(z_1,z_2) &= \sum_{j=1}^m 
 \big ( \mathbb{E}_{j}^* -\mathbb{E}_{j-1}^* \big ) 
 \Big [ \big (r_j{^*}^\top (D_j^* (z_1))^{-1} (D_j^* (z_2))^{-1} r_j^* \big ) ^2  
 \\
 & ~~~~~~~~~~~~~~~~~~~~~
 - \Big ( {1 \over m }  {\rm tr } \big ( L_n^\top (D_j^* (z_1))^ {-1} (D_j^* (z_2) )^{-1}
 L_n  \big ) \Big )^2 \Big ], 
 \\
     Y_2^* (z_1,z_2) &=   \sum_{j=1}^m 
 \big ( \mathbb{E}_{j}^* -\mathbb{E}_{j-1}^* \big ) \beta_j^* (z_1) \beta_j^* (z_2) \big (r_j{^*}^\top (D_j^* (z_1))^{-1} (D_j^* (z_2))^{-1} r_j^* \big ) ^2 \tilde  \gamma_j^* (z_2),  \\
      Y_3^* (z_1,z_2) &=   \sum_{j=1}^m 
 \big ( \mathbb{E}_{j}^* -\mathbb{E}_{j-1}^* \big ) \beta_j^* (z_1) \big (r_j{^*}^\top (D_j^* (z_1))^{-1} (D_j^* (z_2))^{-1} r_j^* \big ) ^2 \tilde  \gamma_j^* (z_1). \\
\end{align*}
This yields for some constant $C>0$ 
\begin{align}
\mathbb{E}^* [ |H _{1n} (z_1,z_2)|^{2} ] 
& \leq C \big ( | \tilde  b_n^*(z_1) \tilde b_n^*(z_2) |^{2 }   
 \E^* \big [ | Y_1^* (z_1,z_2)|^{2} 
\big ]  +    |  \tilde  b_n^*(z_2) |^{2 } \big ]    
 \E^* \big [ | Y_2^* (z_1,z_2)|^{2} 
\big ]  \big )  
\nonumber \\
\label{det65c}
&\ \ \ \ \  + 
 | \tilde  b_n^*(z_1)\tilde  b_n^*(z_2) |^{2}    
 \E^*  \big [ | Y_3^* (z_1,z_2) |^{2} 
\big ]    .
\end{align}
At the end of this proof  we will show
\begin{align}
        \sup_{z_1,z_2 \in  {\cal C}_n}   \mathbb{E}^* \big [ |Y_j^* (z_1,z_2) |^2 \big ]  ~~~ \text{ are tight sequences $(j=1,2,3)$.}
          \label{det65e} 
\end{align} 
It then follows that the  sequence  $
 \sup_{z_1,z_2 \in {\cal C}_n} \E ^*|  H_{1n} (z_1,z_2)  |^2 $ 
is tight as well and combining this result with \eqref{det1} 
yields \eqref{det66}.

Therefore, it remains  to show the estimates, \eqref{det43b} and \eqref{det65e}, which will be done next. 

\medskip

{\bf Proof of \eqref{det43b}:} Note that  
\begin{align*}
  W_1^* (z_1,z_2) &=   \sum_{j=1}^m 
 \big (  \mathbb{E}^*_j 
   - \mathbb{E}^*_{j-1} \big)
    \Big \{ 
  r_j{^*}^\top D_j^* (z_1)^{-2} D_j^* (z_2)^{-1} r_j^*  \\
  \nonumber \\
   & \hspace{36mm} -  {1 \over m} \tr  \big (  
  L_n L_n^\top  D_j^* (z_1)^{-2}
   D_j^* (z_2)^{-1} \big )
   \Big 
   \} ~, 
\end{align*}
  and Proposition \ref{lemrev5}  with the matrix \eqref{eq: C_id}  gives 
\begin{align}
   \sup_{z_1,z_2 \in  {\cal C}_n}  \mathbb{E}^* [ 
| W_1^* (z_1,z_2) |^2] &\ =  \sup_{z_1,z_2 \in  {\cal C}_n}  \sum_{j=1}^m\mathbb{E}^*
\Big [ 
 \Big |
 \big (  \mathbb{E}^*_j 
   - \mathbb{E}^*_{j-1} \big) \Big \{ r_j{^*}^\top D_j^* (z_1)^{-2} D_j^* (z_2)^{-1} r_j^*  \nonumber \\
   & \hspace{36mm}
- {1 \over m } {\rm tr}  \big ( 
L_n^\top D_j^* (z_1)^{-2} D_j^* (z_2)^{-1}L_n
\big ) \Big ) \Big |^2 
\Big ]
\nonumber \\
 & 
\nonumber
\ \leq  4 m   \sup_{z_1,z_2 \in  {\cal C}_n}  
 \mathbb{E}^* 
\Big [ \Big | 
r_1{^*}^\top D_1^* (z_1)^{-2} D_1^* (z_2)^{-1} r_1^*   \nonumber \\
   & \hspace{36mm}
- {1 \over m } {\rm tr}  \big ( 
L_n^\top D_1^* (z_1)^{-2} D_1^* (z_2)^{-1}L_n
\big ) \Big |^2
\Big ] \nonumber \\
&\  =  4 m O_\P  \Big ( {1\over m}  + {m \over n } \Big )  
 =  O_\P (1)  . 
\label{det115}
\end{align} 
As concerns the term $W_2^*$ $(z_1,z_2)$, we find 
\begin{align*}
       \sup_{z_1,z_2 \in  {\cal C}_n}  & \mathbb{E}^* 
       [ | W_2^*(z_1,z_2) |^2 ]  \\
       &=     \sup_{z_1,z_2 \in  {\cal C}_n}     \sum_{j=1}^m  \mathbb{E}^* 
      \Big [ \Big |
\big (  \mathbb{E}^*_j 
   - \mathbb{E}^*_{j-1} \big)  \big [
 \beta_j^*(z_1) 
   r_j{^*}^\top D_j^* (z_1)^{-2} D_j^* (z_2)^{-1} r_j^*  \tilde \gamma_j^*(z_1)
  \big   ] \Big |^2
  \Big ] \\
  &\  \leq      4m  \sup_{z_1,z_2 \in  {\cal C}_n} \mathbb{E}^*  \big |
 \beta_1^*(z_1) 
   r_1{^*}^\top D_1^* (z_1)^{-2} D_1^* (z_2)^{-1} r_1^*  \tilde \gamma_1^*(z_1)
\big |^2 ~.
\end{align*}
Note that on the set ${\mathcal{A}_n} $, 
\begin{align}    
\|r_1^* \|^2  & =  \|  r_1^* {r_1^*}^\top  \|_{S_\infty} 
\leq  \Big \|  \sum_{j=1}^m r_j^* {r_j^*}^\top  \Big \|_{S_\infty}  +\Big \|  \sum_{j=2}^m r_j^* {r_j^*}^\top  \Big \|_{S_\infty}  \leq 2 K_{\rm  right},
\label{det62aa}
\end{align}
and,  recalling the definition of $D^*$ in \eqref{ang2}  and  \eqref{ang21}, yields on ${\cal A}_n$
\begin{equation}
\label{det62}
|  \beta_1^*(z) | = | 1 -{r_j^*}^\top  D^*(z)^{-1}{r_j^*} |  \leq 1  + \|  D^*(z)^{-1} \|_{S_\infty}
\| {r_j^*} \|^2 < c .
\end{equation}
We now use the decomposition $\mathds{1} = \mathds{1}_{\mathcal{A}_n} +  \mathds{1}_{\mathcal{A}_n^c}$ and   \eqref{det33}  
to obtain 

\begin{align*}
      \sup_{z_1,z_2 \in  {\cal C}_n}  \mathbb{E}^* 
       [ | W_2^*(z_1,z_2) |^2 ]  & \leq  4m 
        \sup_{z_1 \in  {\cal C}_n}  \mathbb{E} ^* \big |
    \mathds{1}_{\mathcal{A}_n} 
   \tilde  \gamma_1^*(z_1)
\big |^2   + o_\P (1)  .
\end{align*}
Recalling the definition of $\varepsilon_1^* (z) $ in \eqref{ang3}, Burkholder's inequality,  we obtain  the bound (for any $\ell \in \N$)
\begin{align*}
\sup_{z \in {\cal C}_n} & \E^* \big [    |  \tilde  \gamma_1^*(z)    |^2   \big  ] \lesssim \sup_{z \in {\cal C}_n} \E^* \big [    |   \gamma_1^*(z)  - \varepsilon_1^* (z)   |^2  \big ]
+ \sup_{z \in {\cal C}_n}   \E^* \big [ |\varepsilon^*_1 (z) |^2 \big ]   \\
&  \hspace{60mm} 
+ \sup_{z \in {\cal C}_n} \Big | 
{1 \over m} {\rm tr} ( \E^*\big  [  \mathds{1}_{{\cal A}_n^c}   \tilde \Sigma_n D_1^*(z)^{-1}  \big ] \Big |^2
\\
 &  \lesssim \sup_{z \in {\cal C}_n} {1 \over m^2 } \E^* \Big | \sum_{j=2}^m (\E^*_{j} -\E^*_{j-1} ) \beta_{1j}^*(z) {r_j^*}^\top  D_{1j}^*(z)^{-1}\tilde \Sigma_n  
 D_{1j}^*(z)^{-1}{r_j^*}
     \Big |^2 + \sup_{z \in {\cal C}_n}   \E^* \big [ |\varepsilon^*_1 (z) |^2 \big ] \\
     &  \hspace{60mm}
      + O_{\P} \big ( m^{ 2(1+\alpha)-\ell}\big ) 
 \\
 &
 \lesssim {1 \over m } \sup_{z \in {\cal C}_n} 
 \E^* \big [ \big |  
 \beta_{12}^*(z) {r_2^*}^\top  D_{12}^*(z)^{-1}\tilde \Sigma_n 
 D_{12}^*(z)^{-1}
{r_2^*}
\big  |^2 \big ] 
    +  \sup_{z \in {\cal C}_n}    \E^* \big [ |\varepsilon^*_1 (z) |^2 \big ] \\
&  \hspace{60mm}
 + O_{\P} \big ( m^{ 2(1+\alpha)-\ell}\big ) 
\\
 &
 \lesssim {1 \over m } \sup_{z \in {\cal C}_n} 
 \E^* \big [ \big |  
 \beta_{12}^*(z) {r_2^*}^\top  D_{12}^*(z)^{-1}\tilde \Sigma_n 
 D_{12}^*(z)^{-1}
{r_2^*}
\big  |^2   \mathds{1}_{{\cal A}_n}  \big ] 
+  \sup_{z \in {\cal C}_n}    \E^* \big [ |\varepsilon^*_1 (z) |^2 \big ] \\
     &  \hspace{60mm}
      + O_{\P} \big ( m^{ 2(1+\alpha)-\ell}\big ) 
\\
& = O_{\P} \Big ( {1\over m} \Big) + O_{\P} \big ( m^{ 2(1+\alpha)-\ell} \big )  + \sup_{z \in {\cal C}_n}   \E^* \big [ |\varepsilon^*_1 (z) |^2 \big ] ,
\end{align*}
where $\beta_{12}^*(z) $ is defined in \eqref{ar20}
and is bounded on  the set ${\mathcal{A}_n} $, 
which follows by similar calculations as used for the derivation of \eqref{det62}). Moreover, by Proposition  \ref{lemrev5}, we have 
 $$
 \sup_{z \in {\cal C}_n}   \E^* \big [ |\varepsilon^*_1 (z) |^2 \big ]  = O_{\mathbb P}
 \Big( {1 
 \over m} + 
 \frac{m}{n}
 \Big),
 $$
 which gives 
 \begin{align}
     \label{rev100}
    m  \sup_{z \in {\cal C}_n} & \E^* \big [    |  \tilde  \gamma_1^*(z)    |^2   \big  ] = 
O_{\P}  (  1 ) +   O_{\mathbb P}
 \Big(  
 \frac{m^2}{n}
 \Big) = 
O_{\P}  (  1 )
 \end{align}
 and proves \eqref{det43b}.

\medskip

{\bf Proof of \eqref{det65e}:}  For the sake of brevity, we 
restrict ourselves to the term $ Y_1^*$. Corresponding results  for $Y_2^* $ and $Y_3^*$ are derived in a similar way.
 Note that 
\begin{align*}
 \sup_{z \in {\cal C}_n}  \mathbb{E}^*& |Y_1^* (z_1,z_2) |^2  = \sup_{z \in {\cal C}_n}
  \sum_{j=1}^{m} \E^* 
  \Big |  \big ( \mathbb{E}_{j}^* -\mathbb{E}_{j-1}^* \big ) 
 \Big [ \big (r_j{^*}^\top (D_j^* (z_1))^{-1} (D_j^* (z_2))^{-1} r_j^* \big ) ^2  \\
 & \hspace{48mm}
 - \Big ( {1 \over m }  {\rm tr } \big ( L_n^\top (D_j^* (z_1))^ {-1} (D_j^* (z_2) )^{-1}
 L_n  \big ) \Big )^2 \Big ]
  \Big |^2 \\
 &  \leq 
  4 m \sup_{z \in {\cal C}_n} \E^* 
  \Big | 
 \big (r_1{^*}^\top (D_1^* (z_1))^{-1} (D_1^* (z_2))^{-1} r_1^* \big ) ^2\\
 &\hspace{60mm}- \big ( {1 \over m }  {\rm tr } \big ( L_n^\top (D_1^* (z_1))^ {-1} (D_1^* (z_2) )^{-1}
 L_n  \big ) \big )^2 
  \Big |^2
  \end{align*}
We recall the definition of the set $\mathcal{A}_n$ defined in \eqref{det21}, the decomposition 
$\mathds{1} = \mathds{1}_{\mathcal{A}_n} +  \mathds{1}_{\mathcal{A}_n^c}$ 
  and the identity
 $|a^2-b^2|^2 = |a-b|^4 + 2( \bar a b + a\bar  b)
 |a-b|^2$, which give  (observing \eqref{det33})
 \begin{align*}
   m  \sup_{z \in {\cal C}_n} \E^*   \mathds{1}_{\mathcal{A}_n^c}&
  \Big | 
 \big (r_1{^*}^\top (D_1^* (z_1))^{-1} (D_1^* (z_2))^{-1} r_1^* \big ) ^2 \\ &\qquad  - \big ( {1 \over m }  {\rm tr } \big ( L_n^\top (D_1^* (z_1))^ {-1} (D_1^* (z_2) )^{-1}
 L_n  \big ) \big )^2  \Big |^2\\
 & =  O_\P (1)  ,  
 \end{align*} 
  and 
  \begin{align*}
 &  \sup_{z \in {\cal C}_n} \E^*   \mathds{1}_{\mathcal{A}_n}
  \Big | 
 \big (r_1{^*}^\top (D_1^* (z_1))^{-1} (D_1^* (z_2))^{-1} r_1^* \big ) ^2   - \big ( {1 \over m }  {\rm tr } \big ( L_n^\top (D_1^* (z_1))^ {-1} (D_1^* (z_2) )^{-1}
 L_n  \big ) \big )^2  \Big |^2 \\
  & \leq   \sup_{z \in {\cal C}_n} \mathbb{E}^*   \Big | 
  r_1{^*}^\top (D_1^* (z_1))^{-1} (D_1^* (z_2))^{-1} r_1^*    -  {1 \over m }  {\rm tr } \big ( L_n^\top (D_1^* (z_1))^ {-1} (D_1^* (z_2) )^{-1}
 L_n  \big ) \Big |^4 \\ 
 & + 4 \sup_{z \in {\cal C}_n}  \mathbb{E}^*   \Big [ \Big |
  r_1{^*}^\top (D_1^* (z_1))^{-1} (D_1^* (z_2))^{-1} r_1^*    -  {1 \over m }  {\rm tr } \big ( L_n^\top (D_1^* (z_1))^ {-1} (D_1^* (z_2) )^{-1}
 L_n  \big ) \Big |^2 
 \\
 & \times
  \mathds{1}_{\mathcal{A}_n}
 \| (D_1^* (z_1))^ {-1} (D_1^* (z_2) )^{-1} \|_{{S}_\infty}^2  {q \over m} \|r_1^* \|^2
 \Big ]  \\
&  = O_\P  \Big (  {1 \over m} \Big ), 
\end{align*}
where we have used Proposition \ref{lemrev5} for the matrix   \eqref{eq: C_ie} and the estimate \eqref{ang21} on the set $\mathcal{A}_n$.  The assertion \eqref{det65e} for $ Y_1^*$ now follows.

\subsubsection{Auxiliary results for the proof of Proposition \ref{prop}}
 $ $

  \medskip 

Recall that in the next two Lemmas, we consider 
truncated, centered and normalized $q'$-dimensional random vector $X_1 , \ldots, X_n$ 
(see the discussion in Section \ref{sec82} and \ref{sec82aa}).

    \begin{lemma}\label{lemma: neu1}
For $k=1, \ldots , m $  let 
\begin{align*}
A_{k} (z)=L_n^\top  \E_{k-1}^*\big( \bar\beta_k^*(z) D_k^* (z)^{-1}\big) L_n  \ .
\end{align*}
 Then 
  \begin{align}
  \label{det107}
&  \E \bigg\arrowvert \E_{X_k^*}^*\Big(X_k^{*\top}A_k (z_1) X_k^*-\tr(A_k(z_1))\Big)\E_{X_k^*}^*\Big(X_k^{*\top}A_k(z_2)X_k^*-\tr(A_k(z_2))\Big)\bigg\arrowvert \\
 \nonumber 
&   \hspace{60mm}  \leq K (z_1,z_2) 
\Big({\delta_m^2m^{2} \over \sqrt{n}} +\frac{m^3}{n}\Big)
  \end{align}
  \end{lemma}
  
 \begin{proof}
Evaluating $\E_{X_k^*}^*$ provides the upper bound  
 \begin{align}
  \E\Big\arrowvert &\E_{X_k^*}^*\Big(X_k^{*\top}A_k (z_1) X_k^*-\tr(A_k  (z_1) )\Big)\E_{X_k^*}^*\Big(X_k^{*\top}A_k  (z_2)X_k^*-\tr(A_k  (z_2) )\Big)\Big\arrowvert \nonumber\\
  &\leq  \E\bigg\arrowvert \frac{1}{n^2}\sum_{i=1}^n\big(X_i^\top A_k  (z_1) X_i-\tr(A_k  (z_1))\big)\big(X_{i}^\top A_k  (z_2) X_{i}-\tr(A_k  (z_2))\big)\bigg\arrowvert \nonumber\\ 
  &\quad\quad +  \E\bigg\arrowvert \frac{1}{n^2}\sum_{\substack{i,i'\in\{1,\dots n\} :\\ i\not= i'}}\big(X_i^\top A_k  (z_1)X_i-\tr(A_k  (z_1))\big)\big(X_{i'}\top A_k  (z_2) X_{i'}-\tr(A_k  (z_2))\big)\bigg\arrowvert.\label{eq: term2!}
    \end{align} 
By the Cauchy-Schwarz inequality, we have 
  \begin{align}
  \label{det106} 
   \E\bigg\arrowvert &\frac{1}{n^2}\sum_{i=1}^n\big(X_i^\top A_k  (z_1) X_i-\tr(A_k  (z_1))\big)\big(X_{i}^\top A_k  (z_2) X_{i}-\tr(A_k  (z_2))\big)
   \bigg\arrowvert\\
   \nonumber
   &\leq \frac{1}{n}\E^{1/2}\big\arrowvert X_1^\top A_k  (z_1)X_1-\tr(A_k  (z_1))\big\arrowvert^2\E^{1/2}\big\arrowvert X_1^\top A_k  (z_2)X_1-\tr(A_k  (z_2))\big\arrowvert^2.
   \end{align}
   We now investigate one factor  using Jensen's inequality for the conditional expectation and  the bound 
   $
   | \bar\beta_j^*(z)| \leq \frac{\arrowvert z \arrowvert }{\Im(z ) } 
   $:
\begin{align*}
  \E \big\arrowvert &X_1^\top A_k  (z_1)X_1-\tr(A_k  (z_1))\big\arrowvert^2\\  
  & \leq
 \Big (  \frac{\arrowvert z_1\arrowvert }{\Im(z_1)} \Big )^2
   \E \big\arrowvert X_1^\top  L_n^\top D_k^* (z_1)^{-1}  (z_1) L_n X_1-\tr( L_n^\top  D_k^* (z_1)^{-1}L_n )\big\arrowvert^2 
   \\
   & \leq K \Big ( m + {m^3 \over n} \Big ),
\end{align*}
where last inequality follows by the same arguments as given after formula \eqref{det35} in the proof Proposition \ref{lemma: formel 2.1a}. Therefore, the right-hand side of \eqref{det106} is bounded by
$$
 K \frac{\arrowvert z_1\arrowvert \arrowvert z_2\arrowvert}{\Im(z_1)\Im(z_2)} \Big(1+\frac{m^2}{n}\Big)\frac{m}{n}.  
 $$

We define $i^*_\ell $  as the index $j \in \{ 1, \ldots , n \} $ with $X_\ell ^* = X_j$
For $k=1, \ldots , m $, we denote $\hat{I}_k=\{1,\dots,n\}\cap\{i_1^*,\ldots,  i_{k-1}^*,  i_{k+1}^*, \ldots , i_m^*\}$ and observe  that $\sharp\hat{I}_k\leq m$  and that  $\hat{I}_k, X_1,\dots,X_n$ are jointly independent. With the notation
 $$
 R_{ki} (z) = X_i^\top A_k (z) X_i-\tr(A_k (z) ),\ \ i=1,\dots, n ,\ k=1,  \ldots , m, 
 $$
 and with $\tilde R_{ki}(z)=X_i^\top C_k(z)X_i-\tr(C_k(z))$, $C_k(z)=L_n^\top\bar\beta_k^*(z)D_k^*(z)^{-1}L_n$, so that $R_{ki}(z)=\E^*_{k-1}\tilde R_{ki}(z)$, the Cauchy--Schwarz inequality and Jensen's inequality for $\E^*_{k-1}$ give
\begin{align*}
\E\bigg\arrowvert \frac{1}{n^2}\sum_{\substack{i,i'\in\{1,\dots n\} :\\ i\not= i'}} R_{ki}(z_1) R_{ki'}(z_2) \bigg\arrowvert
&\leq \prod_{j=1,2}\E^{1/2}\bigg\arrowvert \frac{1}{n}\sum_{i=1}^n R_{ki}(z_j)\bigg\arrowvert^2+\E\bigg\arrowvert \frac{1}{n^2}\sum_{i=1}^n R_{ki}(z_1)R_{ki}(z_2)\bigg\arrowvert\\
&\leq \prod_{j=1,2}\E^{1/2}\bigg\arrowvert \frac{1}{n}\sum_{i=1}^n \tilde R_{ki}(z_j)\bigg\arrowvert^2+\E\bigg\arrowvert \frac{1}{n^2}\sum_{i=1}^n R_{ki}(z_1)R_{ki}(z_2)\bigg\arrowvert,
\end{align*}
where the last term has been bounded in \eqref{det106}, and for $z\in\{z_1,z_2\}$
$$
\E\bigg\arrowvert \frac{1}{n}\sum_{i=1}^n \tilde R_{ki}(z)\bigg\arrowvert^2\leq 2\,\E\bigg\arrowvert \frac{1}{n}\sum_{i\in\hat{I}_k} \tilde R_{ki}(z)\bigg\arrowvert^2+2\,\E\bigg\arrowvert \frac{1}{n}\sum_{i\in\hat{I}_k^c} \tilde R_{ki}(z)\bigg\arrowvert^2 .
$$
By the Cauchy-Schwarz inequality, we have 
\begin{align*}
\E\bigg\arrowvert \frac{1}{n}\sum_{i\in\hat{I}_k}\tilde R_{ki}(z)\bigg\arrowvert^2&\leq \frac{1}{n^2}\sum_{i,i'=1}^n\E^{1/4} ( I {\{ i\in\hat{I}_k  \}} )\E^{1/4} ( I {\{ i'\in\hat{I}_k  \}} )\E^{1/4} \arrowvert \tilde R_{ki}(z) \arrowvert^4 \E^{1/4} \arrowvert \tilde R_{ki'}(z)\arrowvert^4 \\
&
\leq \sqrt{K} \Big({\delta_m^2m^{2} \over \sqrt{n}} +\frac{m^3}{n}\Big).
\end{align*}
For the last inequality, we use the bound $| \bar\beta_k^*(z)| \leq \arrowvert z \arrowvert /\Im(z)$, the same arguments as given after formula \eqref{det35} in the proof Proposition \ref{lemma: formel 2.1a} and the fact that $\E ( I {\{ i\in\hat{I}_k \}} ) \leq {m\over n}$, which follows because $\E ( I {\{ i\in\hat{I}_k \}} )$ is independent of $i$ and $\sum_{i=1}^n\E ( I {\{ i\in\hat{I}_k  \}} )  \leq m$.
Observing that the conditional expectations 
$$\E [ \tilde R_{ki}  (z) \arrowvert \hat{I}_k, \{X_j:j\in\hat{I}_k\} ] $$ vanish for $i \in \hat I^c_k$, that $\tilde R_{ki}(z)$ and $\tilde R_{ki'}(z)$ are conditionally uncorrelated for $i\not=i'$, $i,i'\in\hat{I}^c_k$, given $\hat{I}_k$ and $\{X_j:j\in\hat{I}_k\}$, and Lemma B.26 of \cite{baisilverstein2010}, we obtain
$$
\E\bigg\arrowvert \frac{1}{n}\sum_{i\in\hat{I}_k^c} \tilde R_{ki}(z)\bigg\arrowvert^2=\E\,\frac{1}{n^2}\sum_{i\in\hat{I}_k^c}\E\Big[\arrowvert\tilde R_{ki}(z)\arrowvert^2 \Big\arrowvert \hat{I}_k, \{X_j:j\in\hat{I}_k\}\Big]\leq K(z)\frac{m}{n}.
$$
Summarizing these  calculations yield \eqref{det107}, which completes the proof. 
  \end{proof}

\medskip
\begin{lemma}\label{lemma: 1.15 neu}
For any $k=1, \ldots , m $  define 
\begin{align*}
&A_{k}(z)= \frac{1}{m}
L_n^\top  \E_{k-1}^*\big(D_k^*(z)^{-1}\big)L_n  = \frac{1}{m}L_n^\top  \E_{k}^*\big(D_k^*(z)^{-1}\big)L_n ~.
\end{align*}
  Then there exists some positive constant $K(z_1,z_2)$ (independent of $k$)  such that  
\begin{align*}
\E&\Big\arrowvert \E_{X_k^*}\Big[\big(X_k^{*\top}A_k (z_1) X_k^*-\tr A_k (z_1) \big)\big(X_k^{*\top}A_k (z_2) X_k^*-\tr A_k (z_2) \big)\\
&\hspace{80mm}- 2 \tr(A_k (z_1) A_k (z_2 ))\Big]\Big\arrowvert\\
&\leq  K (z_1,z_2) \Big ( \frac{\delta_m^2}{\sqrt{n}}+\frac{m}{n} +  o \Big(\frac{1}{m}\Big) \Big )  .
\end{align*}
\end{lemma}

\begin{proof}
Recall the notation of $\hat I_k$ in the proof of Lemma \ref{lemma: neu1}. Let $\breve D_k^*(z)$ be defined as $D_k^*(z)$ with $(X_\ell^*)_{\ell>k}$ replaced by a (conditionally on $X_1,\dots,X_n$) independent copy $(\breve X_\ell^*)_{\ell>k}$, and let
$$
\tilde A_k(z_1)=\frac1mL_n^\top D_k^*(z_1)^{-1}L_n,\qquad \tilde A_k(z_2)=\frac1mL_n^\top \breve D_k^*(z_2)^{-1}L_n.
$$
Let $\breve i_\ell^*$ denote the index $j$ with $\breve X_\ell^*=X_j$. In this proof, we set $\hat{I}_k=\{1,\dots,n\}\cap\{i_1^*,\ldots,i_{k-1}^*,i_{k+1}^*,\ldots,i_m^*,\breve i_{k+1}^*,\ldots,\breve i_m^*\}$, so that $\sharp\hat{I}_k\leq 2m$.
Evaluating $\E_{X_1^*}^*$ provides the upper bound 
\begin{align}
\E&\Big\arrowvert \E_{X_k^*}\Big[\big(X_k^{*\top}A_k (z_1) X_k^*-\tr A_k (z_1) \big)\big(X_k ^{*\top}A_k (z_2) X_k^*-\tr A_k (z_2)\big)\nonumber\\
&\hspace{80mm}-  2\tr(A_k (z_1) A_k (z_2) )\Big]\Big\arrowvert\nonumber\\
&=\E\Big\arrowvert \frac{1}{n}\sum_{i=1}^n \big(X_i^\top A_k (z_1) X_i-\tr A_k (z_1) \big)\big(X_i^\top A_k (z_2 ) X_i-\tr A_k (z_2) \big)
\nonumber \\
&\hspace{80mm}- 2\tr(A_k (z_1) A_k (z_2))\Big\arrowvert\nonumber\\
&\leq\E\Big\arrowvert \frac{1}{n}\sum_{i=1}^n \big(X_i^\top \tilde A_k(z_1) X_i-\tr \tilde A_k(z_1) \big)\big(X_i^\top \tilde A_k(z_2) X_i-\tr \tilde A_k(z_2) \big)
\nonumber \\
&\hspace{80mm}{}- 2\tr(\tilde A_k(z_1) \tilde A_k(z_2))\Big\arrowvert\nonumber
\end{align}
by Jensen's inequality: given $X_1,\dots,X_n$ and $X_1^*,\dots,X_{k-1}^*$, the matrices $\tilde A_k(z_1)$ and $\tilde A_k(z_2)$ are independent with conditional expectations $A_k(z_1)$ and $A_k(z_2)$, and each summand inside the absolute value is a product of a linear function of $\tilde A_k(z_1)$ and a linear function of $\tilde A_k(z_2)$. Hence the term inside the absolute value in the preceding line is the conditional expectation of the term inside the absolute value in the last line. Splitting the sum over $i$, we obtain
\begin{align}
&\leq \E\Big\arrowvert \frac{1}{n}\sum_{i\in\hat{I}_k} \big(X_i^\top \tilde A_k(z_1) X_i-\tr \tilde A_k(z_1) \big)\big(X_i^\top \tilde A_k(z_2) X_i-\tr \tilde A_k(z_2)\big)\nonumber\\
&\hspace{80mm}- 2 \tr(\tilde A_k(z_1)\tilde A_k(z_2) )\Big\arrowvert\label{eq: sum1}\\
&\quad\quad +  \E\Big\arrowvert \frac{1}{n}\sum_{i\in\hat{I}_k^c} \big(X_i^\top \tilde A_k(z_1) X_i-\tr \tilde A_k(z_1) \big)\big(X_i^\top \tilde A_k(z_2)X_i-\tr \tilde A_k(z_2) \big)\nonumber\\
&\hspace{80mm}- 2 \tr(\tilde A_k(z_1) \tilde A_k(z_2) )\Big\arrowvert\label{eq: sum2}
\end{align}
We start with bounding \eqref{eq: sum1}. Since   $\arrowvert \tr (\tilde A_k(z_1) \tilde A_k(z_2) )\arrowvert  \leq  m \| \tilde A_k(z_1) \|_{S_\infty} \| \tilde A_k(z_2) \|_{S_\infty} $,  $\| \tilde A_k(z_1) \|_{S_\infty}\leq K(z_1)/m$ and  $\sharp \hat{I}_k\leq 2m$, we have  
$$
\E\Big\arrowvert \frac{\sharp\hat{I}_k}{n}\tr(\tilde A_k(z_1)\tilde A_k(z_2))\Big\arrowvert\leq K (z_1, z_2) \frac{1}{n}.
$$
Next, by the Cauchy-Schwarz inequality and
by the same arguments as given after formula \eqref{det35} in the proof Proposition \ref{lemma: formel 2.1a},  it follows that  
\begin{align*}
\E&\Big\arrowvert \frac{1}{n}\sum_{i\in\hat{I}_k} \big(X_i^\top \tilde A_k(z_1) X_i-\tr \tilde A_k(z_1)  \big)\big(X_i^\top \tilde A_k(z_2) X_i-\tr \tilde A_k(z_2)  \big)\Big\arrowvert\\
&\leq \frac{1}{n}\sum_{i=1}^n\E^{1/2} \big(\mathds{1}_{\{i\in\hat{I}_k\}} \big) \E^{1/4}\big\arrowvert  X_i^\top \tilde A_k(z_1)  X_i-\tr \tilde A_k(z_1)  \big\arrowvert^4\\
&\hspace{60mm}\times \E^{1/4}\big\arrowvert  X_i^\top \tilde A_k(z_2) X_i-\tr \tilde A_k(z_2) \big\arrowvert^4\\
&
\leq K(z_1,z_2) \frac{1}{\sqrt{n}}\Big(\delta_m^2+\frac{m}{\sqrt{n}}\Big).
\end{align*}
As concerns \eqref{eq: sum2}, 
we use the same arguments as in the proof of Lemma \ref{lemma: neu1} and obtain
\begin{align*}
&  \E^2\bigg\arrowvert \frac{1}{n}\sum_{i\in\hat{I}_k^c} \big(X_i^\top \tilde A_k(z_1)  X_i-\tr \tilde A_k(z_1)  \big)\big(X_i^\top \tilde A_k(z_2) X_i-\tr \tilde A_k(z_2)  \big)- 2 \tr(\tilde A_k(z_1)  \tilde A_k(z_2) )\bigg\arrowvert\\
 &\leq \E \bigg \{ \E\bigg[\frac{1}{n^2}\sum_{i,i'\in\hat{I}_k^c}\Big[\big(X_i^\top \tilde A_k(z_1)  X_i-\tr \tilde A_k(z_1)  \big)\big(X_i^\top \tilde A_k(z_2) X_i-\tr \tilde A_k(z_2) \big)\\
 &\hspace{90mm}- 2\tr(\tilde A_k(z_1) \tilde A_k(z_2) )\Big]\\  &\hspace{20mm}\times \Big[\big(X_{i'}^\top \overline{\tilde A_k(z_1) }X_{i'}-\tr \overline{\tilde A_k(z_1)  }\big)\big(X_{i'}^\top \overline{\tilde A_k(z_2)  }X_{i'}-\tr \overline{\tilde A_k(z_2) }\big)\\
&\hspace{90mm}-2 \tr(\overline{\tilde A_k(z_1) \tilde A_k(z_2) })\Big]\\
&\hspace{107mm}
\bigg\arrowvert \hat{I}_k,\{X_j: j\in\hat{I}_k\}\bigg] \bigg \}
\\
 &= \E\frac{1}{n^2}\sum_{i\in\hat{I}_k^c}\E\bigg[\Big\arrowvert \big(X_i^\top \tilde A_k(z_1)  X_i-\tr \tilde A_k(z_1)  \big)\big(X_i^\top \tilde A_k(z_2) X_i-\tr \tilde A_k(z_2) \big)\\
 &\hspace{40mm}- 2 \tr(\tilde A_k(z_1)  \tilde A_k(z_2) )\Big\arrowvert^2\bigg\arrowvert \hat{I}_k,\{X_j: j\in\hat{I}_k\}\bigg] ~+ ~ K(z_1,z_2)o\Big(\frac{1}{m^2}\Big) \\
 &\leq\frac{2}{n}\E\bigg\{\E^{1/2} \Big[\big\arrowvert X_i^\top \tilde A_k(z_1)  X_i-\tr \tilde A_k(z_1)  \big\arrowvert^4\Big\arrowvert \hat{I}_k,\{X_j: j\in\hat{I}_k\}\Big]\\
&\quad\quad\quad\quad\quad\quad\quad\quad\times\E^{1/2}\Big[\big\arrowvert \big\arrowvert X_i^\top \tilde A_k(z_2)  X_i-\tr \tilde A_k(z_2) \big\arrowvert^4\Big\arrowvert \hat{I}_k,\{X_j: j\in\hat{I}_k\}\Big]\bigg\}\\
 &\hspace{66mm} + \frac{8}{n}\E\arrowvert \tr(\tilde A_k(z_1)  \tilde A_k(z_2) )\arrowvert^2  ~+ ~K(z_1,z_2)o\Big(\frac{1}{m^2}\Big) \\
 &\leq K(z_1,z_2) \Big( \frac{\delta_m^4}{mn} + \frac{m}{n^2}  +  o\Big(\frac{1}{m^2}\Big) \Big),
 \end{align*}
 where the last line follows from Lemma B.26 
 in \cite{baisilverstein2010} and the truncation at level $\delta_m\sqrt{m}$. In this estimate, the term  $K(z_1,z_2)o(1/m^2)$ is caused by the sum over the mixed products with indices $i\not=i'$ where the conditional expectation factorizes. Its bound is due to  \eqref{eq: trunk4}  and  formula (1.15) in \cite{baisil2004}, using that
 $$
 \Big | \sum_{i=1}^{q'} (\tilde A_k(z_1))_{ii} (\tilde A_k(z_2))_{ii}  \Big | \leq  \|\tilde A_k(z_1) \|_{S_2} 
  \|\tilde A_k(z_2) \|_{S_2} \leq \frac{K(z_1,z_2)}{m}. 
  $$

\end{proof}

\begin{lemma}\label{lemma: ueber 4.3}
There exists some constant $K(z) >0$, such that
   \begin{align*}
   \E\arrowvert \bar{\beta}_1^*(z)-b_n^*(z)\arrowvert^2\leq K(z) \frac{1}{m} ~, \\
   \E\arrowvert {\beta}_{12}^*(z)-b_1^*(z)\arrowvert^2\leq K(z) \frac{1}{m}~. 
   \end{align*}     
\end{lemma}    
\begin{proof}
  Note that by Proposition \ref{lemma: formel 2.1a}, we have
  \begin{align*}
      \E \big | \bar{\beta}_j^*(z) -{\beta}_j^*(z)   \big |^2  
      & \leq \E  \bigg | 
\frac{r_j^{*^\top}{D_j^*}(z)^{-1}r_j^* - 
{1\over m } \tr\big(L_n L_n^{^\top}{D_j^*}(z)^{-1}) }{(1+r_j^{*^\top}{D_j^*}(z)^{-1}r_j^*) (1+{1\over m }\tr\big(L_n L_n^{^\top}{D_j^*}(z)^{-1}\big)} \bigg |^2    \\
& \leq \Big ( { | z | \over \Im (z)} \Big )^{4}
\E  \Big | 
{r_j^{*^\top}{D_j^*}(z)^{-1}r_j^* - 
{1\over m } \tr\big(L_n L_n^{^\top}{D_j^*}(z)^{-1}) } \Big |^2 \\
& \leq  K(z) \Big ( {1 \over m  }  + {m \over n  } \Big )  \leq  K(z)  {1 \over m  } ~. 
  \end{align*}
  Moreover,
  \begin{align*}
      \E\arrowvert {\beta}_1^*(z)-b_n^*(z)\arrowvert^2   & = \E |  {\beta}_1^*(z) b_n^*(z) \gamma_1^*(z) |^2  \\
      & \leq
      \Big ( { | z | \over \Im (z)} \Big )^{4} 
      \E |   \gamma_1^*(z) |^2 \leq K(z) \Big ( {1 \over m  }  + {m \over n  } \Big ) \leq  K(z)  {1 \over m  },
  \end{align*}
  which follows as in \eqref{eq: bound gamma}, where use Proposition \ref{lemma: formel 2.1a} instead of Proposition \ref{lemma: formel 2.1} because of the different truncation scheme. The first assertion now follows by Minkowski's inequality.  The second inequality is obtained analogously.
\end{proof}

\subsection{Proof of Proposition \ref{revprop1}}

\label{proofrevprop1}

\begin{proof}  Note that $M_n ^*$ and  $\widehat {M}_n^{*} $ coincide on ${\cal C}_n$.   By the same calculations as on pages 588-589 \cite{baisil2004} we have 
\begin{align}
\nonumber 
 \E^* \big [ \widehat {M}_n^{*}(z) \big ] 
 &=   {m} ( \E^* \underline{m}_n^*(z)  -  \underline{\tilde m}_n^0(z))   \\
& = -\underline{\tilde m}_n^0(z) {m}  A_n^*(z) \nonumber \\ &\qquad\times\label{rev41}\Big [ 1 - {q \over m} \E^* \underline{ m}_n^*(z) \underline{\tilde m}_n^0(z) \int \frac{t^2 \, d \mu^{\tilde \Sigma_n} (t) }{(1 + t  \E^* \underline{m}_n^*(z) )(1 + t \underline{\tilde m}_n^0(z))}   \Big ]^{-1} , 
\end{align}
 where
\begin{align}
\begin{split}
A_n^*(z) & = {q\over m}  \int \frac{  d \mu^{\tilde \Sigma_n}  (t)  }{1 + t \E^* \underline{m}_n^*(z)}
- {q\over m}  + z \E^* \underline{m}_n^*(z) + 1 \\
&
= -\E^* \underline{m}_n^*(z) \left( -z - \frac{1}{\E^* \underline{m}_n^*(z) } + {q \over m} \int \frac{t \,  d \mu^{\tilde \Sigma_n}  (t)  }{1 + t \E^*\underline{m}_n^*(z)} \right). 
\end{split}
\label{r5}
\end{align} 
Recall that $\tilde{ \underline{m}}^0_n = \underline{m}_{{q/m}, \mu^{\tilde \Sigma_n}}^0(z) $ is the solution of the Mar\u{c}enko-Pastur equation \eqref{eq: mbarnull} for  $\mu^{\tilde \Sigma_n}$. 
As 
$$
\mu^{\tilde \Sigma_n} \Rightarrow H
$$ 
and ${\cal C}$ lies outside the compact support of $H$, the sequence $\tilde{ \underline{m}}^0_n$ converges uniformly on ${\cal C}_n$ to $ \underline{m}_H^0$ and by Lemma \ref{lemrev0}, the sequence 
$(\E^* \underline{m}_n^* )_{n \in \N} $
has the same uniform limit in probability. 
Therefore we obtain by Lemmas  \ref{lemrev0} , \ref{lemrev2}  and the same arguments as given on page 585 in \cite{baisil2004} that the sequence 
$$
 1 - {q \over m} \E^* \underline{ m}_n^*(z) \underline{\tilde m}_n^0(z) \int \frac{t^2 \, d \mu^{\tilde \Sigma_n} (t) }{(1 + t  \E^* \underline{m}_n^*(z) )(1 + t \underline{\tilde m}_n^0(z))}   
$$
is uniformly bounded away from $0$ in probability and its inverse converges uniformly to 
\begin{align}
    \label{rev40}
\Big [ 
 1 - c 
 \int \frac{ (\underline{m}_{c,H}^0(z) )^2 t^2 \, d {H}(t) }{(1 + t  \underline{m}_{c,H}^0(z) )^2 }   \Big]^{-1} .
\end{align}
Hence it is sufficient to prove uniform convergence in probability  of the sequence $({m}A_n^*(z) )_{n\in \N}$. To this aim we note that we  obtain from \eqref{r3}
\begin{align}
\begin{split}
    m A_n^*(z)= 
& m \Big ( {q \over m } \int { d \mu^{\tilde \Sigma_n}  (t) \over 1 + t \E^*\underline{m}^*_n(z)} + z {q \over m }  \E^*{m}^*_n(z) \Big )
\\
&
= m  \E^* \bigg \{ \beta_1^*(z)  \Big [ r_1^{\ast{\top}} D_1^*(z) ^{-1}
    \big ( \E^* \underline{m}_{n}^* (z) \tilde \Sigma_n  
    + I_q \big )^{-1}  r_1^*
   \\
   & - \frac{1}{m} 
\E^* \operatorname{tr} \big \{  \big ( \E^* \underline{m}_{n}^*(z)\tilde \Sigma_n  
    + I_q \big )^{-1}\tilde \Sigma_n  D^*(z)^{-1}
    \big \}
     \Big ] \bigg \} 
     \end{split} ~. 
     \label{r6}
\end{align}
Recalling the definition of 
$\tilde b _n^* (z)$, $\tilde  \gamma_1^* (z) $ in \eqref{rev1000} and the identity \eqref{r4}
we now investigate the difference
\begin{align*}
  &  \E^* \big[ \operatorname{tr} \big \{  \big ( \E^* \underline{m}_{n}^*(z) \tilde  \Sigma_n 
    + I_q \big )^{-1}\tilde  \Sigma_n D_1 ^*(z)^{-1}
    \big \} \big ] - \E^* 
    \big[  \operatorname{tr} \big \{  \big ( \E^* \underline{m}_{n}^*(z) \tilde  \Sigma_n
    + I_q \big )^{-1} \tilde  \Sigma_n D^*(z)^{-1}
    \big \} \big ]  \\
    &= \E^* \big [ \beta_1^*(z)   {\rm tr } 
     \big \{  \big ( \E^* \underline{m}_{n}^*(z) \tilde  \Sigma_n
    + I_q \big )^{-1} \tilde \Sigma_n D_1^*(z)^{-1} r_1^* {r_1^*}^\top D_1^*(z)^{-1}  \big \}  \big  ] \\
    &=    \tilde  b_n^*(z) 
    \E^* \big [ ( 1- \beta_1^*(z)   \tilde   \gamma_1^*(z) ) 
  {r_1^*}^\top   D_1^*(z)^{-1}  \big ( \E^* \underline{m}_{n}^*(z) \tilde  \Sigma_n
    + I_q \big )^{-1} \tilde \Sigma_n D_1^*(z)^{-1} r_1^* 
    \big ]
\end{align*}
where we used the Sherman-Morrison formula  and \eqref{r4} for the last identity. 
By Lemma  \ref{lemrev2}, Lemma   \ref{lemrev1}, \eqref{det62aa} and \eqref{det62}  and the fact that the spectral norm of $D_1^*(z)^{-1}$ is uniformly bounded on ${\cal A}_n$ it follows that 
\begin{align*}
 \sup_{z \in {\cal C}_n } & \Big |  
    \E^* \big [ \mathds{1}_{{\cal A}_n}  \beta_1^*(z)   \tilde \gamma_1^*(z)  
  {r_1^*}^\top   D_1^*(z)^{-1}  \big ( \E^* \underline{m}_{n}^*(z) \tilde  \Sigma_n
    + I_q \big )^{-1} \tilde \Sigma_n D_1^*(z)^{-1} r_1^* 
    \big ]    \Big | \\
    & \leq  C \sup_{z \in {\cal C}_n }  \| \big ( \E^* \underline{m}_{n}^*(z) \tilde  \Sigma_n
    + I_q \big )^{-1}  \|_{S_\infty}   
    \E^*   \big [ \mathds{1}_{{\cal A}_n} | \tilde    \gamma_1^*(z)  | 
    \big ]  \\ 
    &=o_{\mathbb{P} }(1)
\end{align*}
Similarly, on ${\cal A}_n^c $ we use the estimates 
$$
\|r_1^*\|^2 \leq \delta_n^2m ~, ~~\sup_{z \in {\cal C}_n} \| D_1^*(z)^{-1} \|_{S_\infty} \lesssim {m^{1+ \alpha}}~, ~~ \sup_{z \in {\cal C}_n} | \beta_1^*(z) | \lesssim { \sup_{z \in {\cal C}_n}  |z| \over |\Im (z)|  } \lesssim  {
m^{1+ \alpha}}, 
$$
and obtain 
\begin{align*}
 \sup_{z \in {\cal C}_n } & \Big |  
    \E^*  \mathds{1}_{{\cal A}_n^c} \big [ \beta_1^*(z)    \tilde  \gamma_1^*(z)  
  {r_1^*}^\top   D_1^*(z)^{-1}  \big ( \E^* \underline{m}_{n}^*(z) \tilde  \Sigma_n
    + I_q \big )^{-1} \tilde \Sigma_n D_1^*(z)^{-1} r_1^* 
    \big ]    \Big | \\
    & \lesssim \sup_{z \in {\cal C}_n }  \Big \| \big ( \E^* \underline{m}_{n}^*(z) \tilde  \Sigma_n
    + I_q \big )^{-1}  \Big \|_{S_\infty}   
    \E^*   \big [ \mathds{1}_{{\cal A}_n^c }  
    \big ]   {m^{10+4 \alpha}} \delta_n^4 \\
    &=o_{\mathbb{P}}(1)
\end{align*}
Consequently, it follows from \eqref{r4b} in Lemma \ref{lemrev1} that 
\begin{align}
\nonumber 
    \E^*& \big[  \operatorname{tr} \big \{  \big ( \E^* \underline{m}_{n}^*(z) \tilde  \Sigma_n 
    + I_q \big )^{-1}\tilde  \Sigma_n D_1 ^*(z)^{-1}
    \big \} \big ] \\
    & ~~~~~~~~~~~~~~~~~~~~~~~~
    - \E^* 
    \big[  \operatorname{tr} \big \{  \big ( \E^* \underline{m}_{n}^*(z) \tilde  \Sigma_n
    + I_q \big )^{-1} \tilde  \Sigma_n D^*(z)^{-1}
    \big \} \big ] 
  \label{r9} \\\nonumber
    &=    \tilde  b_n^*(z) 
    \E^* \big [ 
  {r_1^*}^\top   D_1^*(z)^{-1}  \big ( \E^* \underline{m}_{n}^*(z) \tilde  \Sigma_n
    + I_q \big )^{-1} \tilde \Sigma_n D_1^*(z)^{-1} r_1^* 
    \big ]\\
    &\ \ \ \ \ \ + o_{\mathbb{P},{\rm unif}}
 (1) ~,\nonumber
 \end{align}
 where $X_n(z) = o_{\mathbb{P},{\rm unif}} (1)$ means that $\sup_{z \in {\cal C}_n} | X_n(z) | = o_{\mathbb{P}} (1)$. 
Using this approximation in \eqref{r6} yields

 \begin{align}
\begin{split}
     m A_n^*(z)  & = 
 m   \Big ( {q \over m } \int { d \mu^{\tilde \Sigma_n}  (t) \over 1 + {t}\E^*\underline{m}^*_n(z)} + z {q \over m }  \E^*{m}^*_n(z) \Big )
\\
&
= m  \E^* \bigg \{ \beta_1^*(z)  \Big [ r_1^{\ast{\top}} D_1^*(z) ^{-1}
    \big ( \E^* \underline{m}_{n}^* (z) \tilde \Sigma_n  
    + I_q \big )^{-1}  r_1^*
   \\
   & - \frac{1}{m} 
\E^* \operatorname{tr} \big \{  \big ( \E^* \underline{m}_{n}^*(z)\tilde \Sigma_n  
    + I_q \big )^{-1}\tilde \Sigma_n  D^*_1(z)^{-1}
    \big \}
     \Big ] \bigg \} \\
     & + 
     \tilde   b _n^*(z) 
   \E^* [ \beta_1^*(z) ] \\ &\qquad\times \E^* \big [ 
  {r_1^*}^\top   D_1^*(z)^{-1}  \big ( \E^* \underline{m}_{n}^*(z) \tilde  \Sigma_n
    + I_q \big )^{-1} \tilde \Sigma_n D_1^*(z)^{-1} r_1^* 
    \big ]
      + o_{\mathbb{P},{\rm unif}}(1)~.
     \end{split}
     \label{r6a}
\end{align}
Replacing $\beta_1^*(z) =  \tilde  b _n^*(z) -  \tilde b _n^*(z)^2  \tilde \gamma_1^*(z)  + \beta_1^*(z)  \tilde   b _n^*(z)^2  \tilde \gamma_1^*(z)^2 $, where $ \tilde  \gamma_1^*(z)$  and $ \tilde  b _n^*(z)$ are defined in \eqref{rev1000}, we get for the first part
\begin{align}
\nonumber 
 &   m \E^* \bigg \{ \beta_1^*(z)  \Big [ r_1^{\ast{\top}} D_1^*(z) ^{-1}
    \big ( \E^* \underline{m}_{n}^* (z)  \tilde \Sigma_n 
    + I_q \big )^{-1}  r_1^* \\
    \nonumber
 & ~~~~~~~~   - \frac{1}{m} 
\E^* \operatorname{tr} \big \{  \big ( \E^* \underline{m}_{n}^*(z) \tilde \Sigma_n 
    + I_q \big )^{-1} \tilde \Sigma_n  D^*_1(z)^{-1}
    \big \}
     \Big ] \bigg \}
     \\
     \nonumber 
      &  =   \tilde   b_n^*(z)  m \E^*  \Big [ r_1^{\ast{\top}} D_1^*(z) ^{-1}
    \big ( \E^* \underline{m}_{n}^* (z)  \tilde \Sigma_n 
    + I_q \big )^{-1}  r_1^*
     \\
    \nonumber
 & ~~~~~~~~    
    - \frac{1}{m} 
\E^* \operatorname{tr} \big \{  \big ( \E^* \underline{m}_{n}^*(z) \tilde \Sigma_n 
    + I_q \big )^{-1} \tilde \Sigma_n  D^*_1(z)^{-1}
    \big \}
     \Big ]  
     \\
     \nonumber 
     &- m   \tilde  b _n^*(z)^2 \E^* \Big \{   \tilde  \gamma_1^*(z) \Big [ {r_1^{\ast}}^{\top} D_1^*(z) ^{-1}
    \big ( \E^* \underline{m}_{n}^* (z)  \tilde \Sigma_n 
    + I_q \big )^{-1}  r_1^*
      \\
    \nonumber
 & ~~~~~~~~   -  \frac{1}{m} 
\E^* \operatorname{tr} \big \{  \big ( \E^* \underline{m}_{n}^*(z) \tilde \Sigma_n 
    + I_q \big )^{-1} \tilde \Sigma_n  D^*_1(z)^{-1}
    \big \}
     \Big ]
    \Big  \} \\ 
    \nonumber 
     &  +   \tilde  b _n^*(z)^2 
   m \E^* \bigg \{  \beta_1^*(z)   \tilde  \gamma_1^*(z)^2 \bigg [ {r_1^{\ast}}^{\top} D_1^*(z) ^{-1}
    \big ( \E^* \underline{m}_{n}^* (z)  \tilde \Sigma_n 
    + I_q \big )^{-1}  r_1^* 
    \\
    \nonumber 
    &   -  {1 \over m }
\E^* \bigg \{ \operatorname{tr} \big \{  \big ( \E^* \underline{m}_{n}^*(z) \tilde \Sigma_n 
    + I_q \big )^{-1} \tilde \Sigma_n  D^*_1(z)^{-1}
    \big \}
      \bigg \} \bigg ] \bigg \}  \\
      \nonumber 
     & =     
     - m  \tilde  b _n^*(z)^2 \E^* \bigg \{   \tilde  \gamma_1^*(z) \Big [ {r_1^{\ast}}^{\top} D_1^*(z) ^{-1}
    \big ( \E^* \underline{m}_{n}^* (z)  \tilde \Sigma_n 
    + I_q \big )^{-1}  r_1^*
     \Big ]
    \bigg \} \\ 
    \nonumber 
          &  +   \tilde  b _n^*(z)^2 
     \bigg( m \E^* \bigg \{  \beta_1^*(z)   \tilde  \gamma_1^*(z)^2 \bigg [ {r_1^{\ast}}^{\top} D_1^*(z) ^{-1}
    \big ( \E^* \underline{m}_{n}^* (z)  \tilde \Sigma_n 
    + I_q \big )^{-1}  r_1^* 
    \\
    \nonumber 
    &   -  {1 \over m }
\E^* \bigg \{ \operatorname{tr} \big \{  \big ( \E^* \underline{m}_{n}^*(z) \tilde \Sigma_n 
    + I_q \big )^{-1} \tilde \Sigma_n  D^*_1(z)^{-1}
    \big \}
      \bigg \} \bigg ] \bigg \} \bigg ) 
     + o_{\mathbb{P},{\rm unif}} (1)  \\ 
     \label{rev20a}     
     & =        - m  \tilde  b _n^*(z)^2 \E^* \bigg \{   \tilde \gamma_1^*(z) \Big [ {r_1^{\ast}}^{\top} D_1^*(z) ^{-1}
    \big ( \E^* \underline{m}_{n}^* (z)  \tilde \Sigma_n 
    + I_q \big )^{-1}  r_1^*
     \Big ]
    \bigg \} \\
      \label{rev20b} 
          &  +  \tilde  b _n^*(z)^2 
     \bigg( m  \E^* \bigg \{  \beta_1^*(z)   \tilde  \gamma_1^*(z)^2 \bigg [ {r_1^{\ast}}^{\top} D_1^*(z) ^{-1}
    \big ( \E^* \underline{m}_{n}^* (z)  \tilde \Sigma_n 
    + I_q \big )^{-1}  r_1^* 
    \\
     \nonumber  
    &   -  {1 \over m }
\operatorname{tr} \big \{  \big ( \E^* \underline{m}_{n}^*(z) \tilde \Sigma_n 
    + I_q \big )^{-1} \tilde \Sigma_n  D^*_1(z)^{-1}
    \big \}
       \bigg ] \bigg \} \bigg ) 
       \nonumber 
      \\ 
     &+  \tilde  b _n^*(z)^2 {\rm Cov}^* \Big ( \beta_1^*(z)  \tilde   \gamma_1^*(z)^2, {\tr}  \big ( D_1^*(z) ^{-1}
     ( \E^* \underline{m}_{n}^* (z)  \tilde \Sigma_n 
    + I_q \big )^{-1} \tilde \Sigma_n \big) \Big )  + o_{\mathbb{P},{\rm unif}} (1)
    \label{rev20c}
\end{align}
uniformly for $z \in {\cal C}_n$, 
where we have used Lemma \ref{lemrev1} and  \ref{lemrev3}. We now prove that the  term \eqref{rev20c} is of order $o_{\P,{\rm unif}}(1)$.
Note first that by the Cauchy-Schwarz inequality 
\begin{align}
\nonumber 
&  
\sup_{z \in {\cal C}_n} \Big |   \tilde  b _n^*(z)^2 {\rm Cov}^* \Big ( \beta_1^*(z)   \tilde  \gamma_1^*(z)^2, {\tr}  \big ( D_1^*(z) ^{-1}
     ( \E^* \underline{m}_{n}^* (z)  \tilde \Sigma_n 
    + I_q \big )^{-1} \tilde \Sigma_n \big) \Big )   \Big |  \\
    \label{rev20}
    & ~~ \leq   \sup_{z \in {\cal C}_n}  |  \tilde  b_n^*(z)| \big  \{   \E^* [ | \beta_1^*(z)   |^2 |  \tilde   \gamma_1^*(z)| ^4
    ] 
   \big  \}^{1/2} \\ 
    & ~~ \times    \sup_{z \in {\cal C}_n}  |  \tilde  b_n^*(z)|
    \big \{ \E^* \big | {\tr}  \big ( D_1^*(z) ^{-1}
     ( \E^* \underline{m}_{n}^* (z)  \tilde \Sigma_n 
    + I_q \big )^{-1}\tilde \Sigma_n  \big) 
       \nonumber \\
    & ~~~~~~~~~~~~~~~~~~~~~~~~~-
    \E^* \big [ {\tr}  \big ( D_1^*(z) ^{-1}
     ( \E^* \underline{m}_{n}^* (z)  \tilde \Sigma_n 
    + I_q \big )^{-1} \tilde \Sigma_n \big) \big ] \big | ^2 \big \}^{1/2}
    \nonumber
\end{align}
First, note that  by Lemma \ref{lemrev1} it follows that 
$$
\sup_{z\in  {\cal C}_n} | \tilde  b _n^*(z) |  =O_{\mathbb{P}} (1)~.
$$
Moreover, it is easy to see that 
$  \E^* [ | \beta_1^*(z)   |^2 |   \tilde  \gamma_1^*(z) |^4  \mathds{1}_{{\cal A}_n^c} 
    ]  = o_{\P, {\rm unif} } (1)$.  On the other hand on ${\cal A}_n$ it follows from \eqref{det62}  that 
$$
\sup_{z \in {\cal C}_n}
\E^* [ | \beta_1^*(z)   |^2  \tilde  \gamma_1^*(z)^4|  \mathds{1}_{{\cal A}_n} 
    ]  \lesssim 
    \sup_{z \in {\cal C}_n}
    \E^* [ |  |   \tilde  \gamma_1^*(z)|  ^4\mathds{1}_{{\cal A}_n} 
    ]
$$
Recalling the definition of $\varepsilon_1^* (z) $ in \eqref{ang3}, Burkholder's inequality,  we obtain  the bound (for any $\ell \in \N$)
\begin{align*}
\sup_{z \in {\cal C}_n} & \E^* \big [    |  \tilde  \gamma_1^*(z)    |^4   \big  ] \lesssim \sup_{z \in {\cal C}_n} \E^* \big [    |   \gamma_1^*(z)  - \varepsilon_1^* (z)   |^4  \big ]
+ \sup_{z \in {\cal C}_n}   \E^* \big [ |\varepsilon^*_1 (z) |^4 \big ]   \\
&  \hspace{60mm} 
+ \sup_{z \in {\cal C}_n} \Big | 
{1 \over m} {\rm tr} ( \E^*\big  [  \mathds{1}_{{\cal A}_n^c}   \tilde \Sigma_n D_1^*(z)^{-1}  \big ] \Big |^4
\\
 &  \lesssim \sup_{z \in {\cal C}_n} {1 \over m^4 } \E^* \Big | \sum_{j=2}^m (\E^*_{j} -\E^*_{j-1} ) \beta_{1j}^*(z) {r_j^*}^\top  D_{1j}^*(z)^{-1}\tilde \Sigma_n  
 D_{1j}^*(z)^{-1}{r_j^*}
     \Big |^4 + \sup_{z \in {\cal C}_n}   \E^* \big [ |\varepsilon^*_1 (z) |^4 \big ] \\
     &  \hspace{60mm}
      + O_{\P} \big ( m^{ 4(1+\alpha)-\ell}\big ) 
 \\
 &
 \lesssim {1 \over m^2 } \sup_{z \in {\cal C}_n} 
 \E^* \big [ \big |  
 \beta_{12}^*(z) {r_2^*}^\top  D_{12}^*(z)^{-1}\tilde \Sigma_n 
 D_{12}^*(z)^{-1}
{r_2^*}
\big  |^4 \big ] 
    +  \sup_{z \in {\cal C}_n}    \E^* \big [ |\varepsilon^*_1 (z) |^4 \big ] \\
&  \hspace{60mm}
 + O_{\P} \big ( m^{ 4(1+\alpha)-\ell}\big ) 
\\
 &
 \lesssim {1 \over m^2 } \sup_{z \in {\cal C}_n} 
 \E^* \big [ \big |  
 \beta_{12}^*(z) {r_2^*}^\top  D_{12}^*(z)^{-1}\tilde \Sigma_n 
 D_{12}^*(z)^{-1}
{r_2^*}
\big  |^4   \mathds{1}_{{\cal A}_n}  \big ] 
+  \sup_{z \in {\cal C}_n}    \E^* \big [ |\varepsilon^*_1 (z) |^4 \big ] \\
     &  \hspace{60mm}
      + O_{\P} \big ( m^{ 4(1+\alpha)-\ell}\big ) 
\\
& = O_{\P} \Big ( {1\over m^2} \Big) + O_{\P} \big ( m^{ 4(1+\alpha)-\ell} \big )  + \sup_{z \in {\cal C}_n}   \E^* \big [ |\varepsilon^*_1 (z) |^4 \big ] ,
\end{align*}
where $\beta_{12}^*(z) $ is defined in \eqref{ar20}. Note that on the set ${\mathcal{A}_n} $, $\| r_2^*\|^2$ is bounded (see \eqref{det62aa}) and  we have $ \|  D^*_{12} (z)^{-1} \|_{S_\infty} < C$ and $  |  
 \beta_{12}^*(z) | \leq C $   (which follows by similar calculations as used for the derivation of \eqref{det62}). Moreover, by Proposition   \ref{lemrev5}, we have 
 $$
 \sup_{z \in {\cal C}_n}   \E^* \big [ |\varepsilon^*_1 (z) |^4 \big ]  = O_{\mathbb P}
 \Big( {\delta_m^{4} 
 \over m} + 
 \frac{m}{n}
 \Big),
 $$
 which gives 
 \begin{align}
     \label{rev100b}
     \sup_{z \in {\cal C}_n} & \E^* \big [    |  \tilde  \gamma_1^*(z)    |^4   \big  ] = 
O_{\P} \Big ( {1\over m^2} \Big) +   O_{\mathbb P}
 \Big( {\delta_m^{4} 
 \over m} + 
 \frac{m}{n}
 \Big)
 \end{align}
 This implies that the first factor in \eqref{rev20} converges to $0$ in probability. Moreover, we note for later purposes that this also reveals the approximation
 \begin{align}
     \label{rev29a}
   \sup_{z \in {\cal C}_n}  \E^* | \beta_1^* (z)  -  \tilde  b _n^*(z)| = o_{\P} (1) 
 \end{align}
 We now consider  the second factor in  \eqref{rev20} and  introduce the notation
$$
M(z) = ( \E^* \underline{m}_{n}^* (z)  \tilde \Sigma_n 
    + I_q \big )^{-1}\tilde \Sigma_n   . 
    $$
Let $Q_j(z)={r_j^*}^\top D_{1j}^*(z)^{-1}M(z)D_{1j}^*(z)^{-1}r_j^*$, $t_j(z)=\frac1m\tr\big(\tilde\Sigma_nD_{1j}^*(z)^{-1}M(z)D_{1j}^*(z)^{-1}\big)$, and let $\tilde b_{1j}^*(z)$, $\tilde\gamma_{1j}^*(z)$ be defined as $\tilde b_n^*(z)$, $\tilde\gamma_1^*(z)$ in \eqref{rev1000} with $D_{1j}^*(z)$ and $r_j^*$ in place of $D_1^*(z)$ and $r_1^*$. Similar arguments as used in equation (4.7) in \cite{baisil2004} give 
\begin{align}
\label{rev22}
& 
\sup_{z \in {\cal C}_n}  
     \E^* \big | {\tr}  \big ( D_1^*(z) ^{-1}
     M(z)  \big) -  \E^* \big [ {\tr}  \big ( D_1^*(z) ^{-1}
     M(z)  \big) \big ] \big | ^2    \\
     \nonumber
     &   ~~~~~
    \leq  \sum_{j=2}^m 
    \sup_{z \in {\cal C}_n}    \E^* \big | (\E_j^* - \E_{j-1}^*)  \big [ \beta_{1j}^*(z)  Q_j(z) \big ]\big |^2  \\
       &   ~~~~~ 
          \nonumber
          \leq 8      
     \sum_{j=2}^m  \Big \{
    \sup_{z \in {\cal C}_n}|\tilde b_{1j}^*(z)|^2 \sup_{z \in {\cal C}_n}    \E^* \big | Q_j(z)-t_j(z)\big |^2
      \\
         \nonumber
       &   ~~~~~ + 
 \sup_{z \in {\cal C}_n}|\tilde b_{1j}^*(z)|^2\sup_{z \in {\cal C}_n}    \E^* \big | \beta_{1j}^*(z)\tilde\gamma_{1j}^*(z)Q_j(z)\big |^2 
       \Big \} , 
\end{align}
where we used $\beta_{1j}^*=\tilde b_{1j}^*-\beta_{1j}^*\tilde b_{1j}^*\tilde\gamma_{1j}^*$ (cf.\ \eqref{r4}) and $(\E_j^*-\E_{j-1}^*)\,t_j(z)=0$, since $t_j(z)$ does not depend on $X_j^*$. By exchangeability, the summands do not depend on $j$, so that it suffices to consider $j=2$. By the arguments of Lemma \ref{lemrev1}, $$\sup_{z\in{\cal C}_n}|\tilde b_{12}^*(z)|$$ is tight. Proposition \ref{lemrev5} gives $\sup_{z\in{\cal C}_n}\E^*|Q_2(z)-t_2(z)|^2=O_\P(\frac1m+\frac mn)$, so that the first sum is of order $O_\P(1+\frac{m^2}n)=O_\P(1)$. For the second sum, it suffices to consider the estimates on ${\cal A}_n$ by \eqref{det33}. On ${\cal A}_n$, $\beta_{12}^*(z)$ and $Q_2(z)$ are bounded (see \eqref{det62aa} and \eqref{det62}), and \eqref{rev100} (with $D_{12}^*$ in place of $D_1^*$) yields $\sup_{z\in{\cal C}_n}\E^*|\tilde\gamma_{12}^*(z)|^2=O_\P(\frac1m)$. Hence the second sum is $O_\P(1)$ as well.  This shows that \eqref{rev22} is stochastically bounded,
that is 
\begin{align}
 \label{rev29}  
 \sup_{z \in {\cal C}_n}  
     \E^* \big | {\tr}  \big ( D_1^*(z) ^{-1}
     M(z)  \big) -  \E^* \big [ {\tr}  \big ( D_1^*(z) ^{-1}
     M(z)  \big) \big ] \big | ^2  = O_{\P}(1)
\end{align}
 This gives  for the term \eqref{rev20c}
\begin{align*}
 \sup_{z \in {\cal C}_n}  | \eqref{rev20c}   |  & = O_{\mathbb P}
 \Big( \Big(  {\delta_m^{4} 
 } + 
 \frac{m^2}{n}
 \Big)^{1/2}\Big )   
= o_{\P} (1) .
\end{align*}
Similarly, using  the Cauchy-Schwarz inequality together with Proposition   \ref{lemrev5}  and \eqref{rev100b} we obtain for  the term \eqref{rev20b}
and obtain by 
\begin{align*}
    \sup_{z \in  {\cal C}_n} |  \eqref{rev20b}| & \leq  m  \sup_{z \in  {\cal C}_n} | \tilde  b_n^*(z)|^2   
    \Big (  \sup_{z \in  {\cal C}_n}
           \E^* \big [   |\beta_1^*(z) |^2 | \tilde \gamma_1^*(z)|^4 \big ] \\
           & ~~~~~\times 
           \sup_{z \in  {\cal C}_n} \E^* 
           \Big  [ \big | {r_1^{\ast}}^{\top} D_1^*(z) ^{-1}
    \big ( \E^* \underline{m}_{n}^* (z)  \tilde \Sigma_n 
    + I_q \big )^{-1}  r_1^* 
  \\
  & ~~~~~~~~~~
  -  {1 \over m }
\operatorname{tr} \big \{  \big ( \E^* \underline{m}_{n}^*(z) \tilde \Sigma_n 
    + I_q \big )^{-1} \tilde \Sigma_n  D^*_1(z)^{-1}
    \big \}
       \big |^2  \Big ]  \Big )^{1/2} \\
       & = o_{\P} (1)
\end{align*}
Finally, we  derive a uniform approximation in probability for the term \eqref{rev20a}  we introduce the notation
$$
H(z) = \big ( \E^* \underline{m}_{n}^* (z)  \tilde \Sigma_n 
    + I_q  \big )^{-1}
$$
and  
observe the decomposition 
\begin{align}
\nonumber
 & m \E^* \Big [ \tilde    \gamma_1^*(z) \big [ {r_1^{\ast}}^{\top} D_1^*(z) ^{-1}
    H(z)   r_1^*
     \big ]
    \Big]  \\
    & \label{rev24}
    ~~~=
   m  \E^* \Big [  \Big (  {r_1^{\ast}}^{\top} D_1^*(z) ^{-1}r_1^{\ast} -
   {1 \over m } \tr \big ( D_1^*(z) ^{-1} \tilde \Sigma_n \big ) \Big ) \\
   \nonumber
   & ~~~~~~~~~~~~~~~~~~~~~~ \times \Big (  {r_1^{\ast}}^{\top} D_1^*(z)^{-1} H(z)  r_1^{\ast} -
   {1 \over m } \tr \big ( D_1^*(z)^{-1} H(z)  \tilde \Sigma_n \big ) \Big )
    \Big]    \\
  & 
  \label{rev26} ~~~ + {1 \over m } {\rm Cov \big ( \tr  ( D_1^*(z) ^{-1} \tilde \Sigma_n ) , \tr (   D_1^*(z) ^{-1} H(z) \tilde \Sigma_n  ) \big ) } \\
   \label{rev25}
    & ~~~
    - \tr \big ( \E^* \big[ \mathds{1}_{{\cal A}_n}D_1^*(z) ^{-1} \tilde \Sigma_n \big]  \big ) \E^*  \Big \{  
 {r_1^{\ast}}^{\top} D_1^*(z)^{-1} H(z)  r_1^{\ast}
 \\
\nonumber & ~~~~~~~~~~~~~~~~~~~~~~ - {1 \over m}  \tr \big (  D_1^*(z)^{-1} H(z) \tilde \Sigma_n
\big ) 
    \Big \} \\ 
        & 
         \label{rev27}
         ~~~ + 
    \E^* \Big [ \tr \big (   D_1^*(z) ^{-1}  \tilde \Sigma_n   \big ) \Big \{ 
 {r_1^{\ast}}^{\top} D_1^*(z)^{-1} H(z)  r_1^{\ast}
  - {1 \over m}  \tr \big (  D_1^*(z)^{-1} H(z) \tilde \Sigma_n
\big ) 
    \Big \} \Big ] 
    \\     
        \label{rev28}
             & ~~~ + 
    \E^* \Big [ \tr \big (   D_1^*(z) ^{-1}  H(z) \tilde \Sigma_n   \big ) \Big \{ 
 {r_1^{\ast}}^{\top} D_1^*(z)^{-1}  r_1^{\ast}
  - {1 \over m}  \tr \big (  D_1^*(z)^{-1}  \tilde \Sigma_n
\big ) 
    \Big \} \Big ] \\
\nonumber & ~~~ + {1 \over m}\tr \big ( \E^* \big[  \mathds{1}_{{\cal A}_n^c}D_1^*(z) ^{-1}  \tilde \Sigma_n \big]  \big )\, \E^* \big [ \tr \big (  D_1^*(z)^{-1} H(z)   \tilde \Sigma_n \big ) \big ] ~. 
\end{align}
By Lemma \ref{lemrev3} it follows that the terms \eqref{rev25} - \eqref{rev28}  are of order $o_{\P, {\rm unif}} (1)$. The last term of the decomposition (which accounts for the indicator in the definition \eqref{rev1000} of $\tilde\gamma_1^*$) is of order $o_{\P, {\rm unif}} (1)$ as well, because $\Arrowvert D_1^*(z)^{-1}\Arrowvert_{S_\infty}\le Km^{1+\alpha}$ for $z\in{\cal C}_n$, $\sup_{z\in{\cal C}_n}\Arrowvert H(z)\Arrowvert_{S_\infty}=O_\P(1)$ by Lemma \ref{lemrev2} and $\E\,\P^*({\cal A}_n^c)=\P({\cal A}_n^c)=o(m^{-\ell_0})$ by \eqref{det33}.
The  Cauchy-Schwarz inequality and the same arguments as given in the derivation of \eqref{rev29} show  that the term \eqref{rev26} is of order
$o_{\P, {\rm unif}} (1)$
as well, which gives
\begin{align}
\nonumber
 & m \E^* \Big [ \tilde    \gamma_1^*(z) {r_1^{\ast}}^{\top} D_1^*(z) ^{-1}
    H(z)   r_1^*
    \Big]  \\
    & \label{rev24a}
    ~~~=
   m  \E^* \Big [  \Big (  {r_1^{\ast}}^{\top} D_1^*(z) ^{-1}r_1^{\ast} -
   {1 \over m } \tr \big ( D_1^*(z) ^{-1} \tilde \Sigma_n \big ) \Big ) \\
   \nonumber
   & ~~~~~\times \Big (  {r_1^{\ast}}^{\top} D_1^*(z)^{-1} H(z)  r_1^{\ast} -
   {1 \over m } \tr \big ( D_1^*(z)^{-1} H(z)  \tilde \Sigma_n \big ) \Big )
    \Big]  + o_{\P, {\rm unif}} (1). 
    \end{align}
    Using \eqref{rev29a}
    and  this estimate in \eqref{rev20a} yields for \eqref{r6a}
 \begin{align}
    \nonumber
    mA_n^*(z) = 
& m \Big ( {q \over m } \int { d \mu^{\tilde \Sigma_n}  (t) \over 1 + {t}\E^*\underline{m}^*_n(z)} + z {q \over m }  \E^*{m}^*_n(z) \Big )
\\
&
= \tilde    b _n^*(z)^2 
    \E^* \big [ 
  {r_1^*}^\top   D_1^*(z)^{-1}  \big ( \E^* \underline{m}_{n}^*(z) \tilde  \Sigma_n
    + I_q \big )^{-1} \tilde \Sigma_n D_1^*(z)^{-1} r_1^* 
    \big ]
       \label{r6b} \\
     & - \tilde   b _n^*(z)^2
      m  \E^* \Big [  \Big (  {r_1^{\ast}}^{\top} D_1^*(z) ^{-1}r_1^{\ast} -
   {1 \over m } \tr \big ( D_1^*(z) ^{-1} \tilde \Sigma_n \big ) \Big )
    \nonumber     
        \\
   \nonumber
   & ~~~~~\times \Big (  {r_1^{\ast}}^{\top} D_1^*(z)^{-1} H(z)  r_1^{\ast} -
   {1 \over m } \tr \big ( D_1^*(z)^{-1} H(z)  \tilde \Sigma_n \big ) \Big )
    \Big]  + o_{\P, {\rm unif}} (1) \\
    &  = S_n(z) + o_{\P, {\rm unif}} (1) ,  
\end{align}
It now follows from Lemma \ref{lemrev1} and  Proposition  \ref{lemrev5} that the sequence $(S_n)_{n \in \N}$ is uniformly bounded in probability.  Moreover, similar arguments show that the sequence of derivatives on ${\cal C}_n$  is uniformly bounded in probability as well. 
Therefore, the sequence $(\tilde S_n) _{n \in \N}$ is  equicontinuous  on 
${\cal C} \cap \C^+$, where $\tilde S_n$ is the constant continuation of $S_n$ from ${\cal C}_n$ to 
${\cal C} \cap \C^+$.
Consequently, the uniform limit of $(mA_n^*(z))_{n \in \N}$
can be determined pointwise (by the Arzel\`a-Ascoli Theorem). 

This limit can now be obtained by exactly the same arguments as given on page $589$-$592$ in \cite{baisil2004}, where the expectation is replaced by the conditional expectation $\E^*$ and the convergence is correspondingly in probability. This gives 
$$
mA_n^*(z) \to 
-  {c} \int \frac{ (\underline{m}_{c,H}^0(z) )^2 t^2 \, d {H}(t) }{(1 + t  \underline{m}_{c,H}^0(z) )^3 }
\Big [ 
 1 - c 
 \int \frac{ (\underline{m}_{c,H}^0(z) )^2 t^2 \, d {H}(t) }{(1 + t  \underline{m}_{c,H}^0(z) )^2 }   \Big]^{-1}
$$
{(this is the limit of $N A_n(z)$ on pages 588--592 of \cite{baisil2004}, with $c_n$ included in $A_n$; with the normalization $\E^*\widehat M_n^*=m(\E^*\underline m_n^*-\underline{\tilde m}_n^0)$, which follows from $\widehat M_n^*=q(m_n^*-m_n^0)$ and $\underline m=\frac qm m-(1-\frac qm)\frac1z$, the factor $c$ in Proposition~\ref{revprop1} arises here and not from a factor $q/m$ in \eqref{rev41}).} Combining this limit with the limit in \eqref{rev40} and observing \eqref{rev41} finally gives
$$
\sup_{z \in {\cal C}_n }
\bigg |  \E^* \big [
\widehat {M}_n^{*}(z) \big ]      
- c  \int \frac{ (\underline{m}_{c,H}^0(z) )^3 t^2 \, d {H}(t) }{(1 + t  \underline{m}_{c,H}^0(z) )^3 }
\Big [ 
 1 - c 
 \int \frac{ (\underline{m}_{c,H}^0(z) )^2 t^2 \, d {H}(t) }{(1 + t  \underline{m}_{c,H}^0(z) )^2 }   \Big]^{-2} \bigg | = o_{\P} (1) 
$$
\end{proof}

\bigskip

\subsubsection{Auxiliary results for the proof of Proposition \ref{revprop1}}

\begin{lemma}
    \label{lemrev0}
    $$
    \sup_{z \in \mathcal{C}_n}
\big | 
 \E^* \underline{m}_n^* (z) 
- \underline{\tilde m}_n^0 (z)  \big |   
= o_{\mathbb P}(1)
$$
\end{lemma}
\begin{proof} The proof follows by  
    the same arguments as given in the derivation of equation (4.1) in \cite{baisil2004}
(see also equation \eqref{eq: 5.1b}),  noting that 
$ \E^*\big [ {\mu}^{\Sigma_n^*} \big ]   \Rightarrow \mu_H^0$  in probability by dominated convergence.
\end{proof}

\begin{lemma}
    \label{lemrev2}
    $$
     \sup_{z \in \mathcal{C}_n}   \| (\ \E^* \underline{m}_n^* (z) \tilde \Sigma_n   + I)^{-1}  \big \|_{S_\infty}
     = O_{\mathbb{P}} (1) $$
\end{lemma}
\begin{proof}%
    A simple calculation shows   \begin{align*}
 \sup_{z \in \mathcal{C}_n}   \| (\ \E^* \underline{m}_n^* (z) \tilde \Sigma_n   + I)^{-1}  \big \|_{S_\infty} \leq  \Big [ 
 \inf_{z \in \mathcal{C}_n} \min_{\|x\|_2=1}
 \big \| 
 \big ({\underline{\tilde m}_n^0} (z) \tilde \Sigma_n   + I) \big )x
 \big \|_2 - R_n
 \Big ]^{-1},
\end{align*}
where
$$
R_n =
\sup_{z \in \mathcal{C}_n}
\big | 
 \E^* \underline{m}_n^* (z) 
- \underline{\tilde m}_n^0 (z)  \big |  \cdot 
 \big \|  \tilde \Sigma_n    
\big \|_{S_\infty}
= o_{\mathbb P}(1), 
$$
where the last estimate follows by Lemma \ref{lemrev0}.
\end{proof}
\begin{lemma}\label{lemrev1}

\begin{align}
\label{r4a}
    \sup_{z  \in {\cal C}_n} |\tilde   \gamma_1^*(z) | & =o_{\mathbb{P}} (1) \\
    \sup_{z\in  {\cal C}_n} |\tilde   b _n^*(z) | & =O_{\mathbb{P}} (1)
    \label{r4b}
    \end{align}
\end{lemma}
\begin{proof}
Note that $ \tilde   \gamma_1^*$  depends on $n$ and in this proof we highlight this dependence by the notation $\tilde   \gamma_{1,n}^*$. Let $D$ be an open, connected and bounded set such that ${\cal C}\cup\bar{\cal C} \subset D$ and $\bar D \cap [K_{\rm left}, K_{\rm  right}]  =  \emptyset $ (a neighbourhood of the contour). Throughout this proof, $\tilde\gamma_{1,n}^*$ stands for $\mathds{1}_{{\cal A}_n}\tilde\gamma_{1,n}^*$, which is analytic on $D$. This is sufficient, as $\P({\cal A}_n^c)\to0$.
Recalling the definition of  the set ${\cal A}_n$ in \eqref{det21}  we have  
\begin{align*}
 \sup_{z\in  D}  \mathds{1}_{{\cal A}_n} \big |  {r_1^*}^\top  D_1^*(z)^{-1} r_1^*\big |  \leq \mathds{1}_{{\cal A}_n} \| {r_1^*}\|^2  \sup_{z\in  D} \|  D_1^*(z)^{-1}\|_{{\cal S}_\infty } & = O_\P (1)  \\  
 \sup_{z\in D }  \Big | 
{1 \over m} {\rm tr} \big ( \E^*\big  [  \mathds{1}_{{\cal A}_n}   L_nL_n^\top D_1^*(z)^{-1}  \big ]
 \big)  \Big | & = O_\P (1) 
\end{align*}
Consequently, 
$$
\sup_{z  \in D} | \tilde \gamma_{1,n}^*(z) |  =O_{\mathbb{P}} (1).
$$
Therefore, for any $\delta >0 $, there exists a constant $K_\delta >0 $ such that 
\begin{align}
    \label{r8a}
\limsup_{n \to \infty }
\mathbb{P} \big (B_{n,\delta} \big ) \leq \delta~,
\end{align}
where
\begin{align}
    \label{r8}
B_{n,\delta} = \Big \{ \omega \in \Omega ~\Big |~ \sup_{z \in D }   | \tilde  \gamma_{1,n}^*(z) | >  K_\delta 
\Big \} 
\end{align}
 As shown in  \eqref{rev100b}, we have  for each $z \in {\cal C}_u$ 
\begin{align}
   \tilde  \gamma_{1,n}^*(z) = o_{\mathbb{P}}(1).
\end{align}
Let $z_1, z_2, \ldots $ denote a sequence in ${\cal C}_u$ with accumulation point in $D$ and let $n_k$ be an arbitrary subsequence. Then there exists a further subsequence  $n_k'$ such that    
$$
   \tilde   \gamma_{1,n_k'}^*(z_1) =o(1)
$$
on a set $N_1^c \subset \Omega$ with $\mathbb{P} (N_1) =0 $.  By the same reasoning there exists a further subsequence  ${n_k''}$ of ${n_k'}$ such that 
$$
   \tilde   \gamma_{1,n_k''}^*(z_2) =o(1)
$$
on a set $N_2^c \subset \Omega$ with $\mathbb{P} (N_2) =0 $. By Cantor's diagonalization principle we obtain a sequence  denoted by $\tilde n_k$  such that
$$
   \tilde   \gamma_{1,\tilde n_k}^*(z_j) =o(1)
$$
for all $j \in \mathbb{N}$ and all $\omega$ outside the null set $N = \cup_{j \in \mathbb{N}} N_j \subset \Omega$. Moreover, the sequence of functions  $z \to     \tilde  \gamma^*_{1,\tilde n_k} (z) (\omega)  \mathds{1}_{B_{\tilde n_k,\delta}^c}(\omega)   $ is bounded for every $\omega \in N^c $.   By Vitali's Theorem  the sequence  $z \to     \tilde  \gamma^*_{1, \tilde n_k} (z) (\omega)  \mathds{1}_{B_{\tilde n_k,\delta}^c}(\omega)   $ converges to $0$ uniformly on compact subsets of $D$, in particular on ${\cal C}$, for  every $\omega \in N^c $.  As the initial sequence $n_k$ was arbitrary we conclude
\begin{align*}
    \lim_{n\to \infty } \mathbb{P} \Big ( \sup_{z \in {\cal C}} \big | 
     \tilde  \gamma^*_{1, n} (z)  \mathds{1}_{B_{ n,\delta}^c} 
    \big | > \eta 
    \Big )  =0 
\end{align*}
for any $\eta >0$. Finally, we obtain from the definition of $B_{ n,\delta}$
\begin{align*}
    \limsup _{n\to \infty } \mathbb{P} \Big ( \sup_{z \in {\cal C}} \big | 
    \tilde  \gamma^*_{1, n} (z)   
    \big | > \eta 
    \Big )  \leq
    \limsup_{n\to \infty } \mathbb{P} \Big ( \sup_{z \in {\cal C}} \big | 
    \tilde   \gamma^*_{1, n} (z)  \mathds{1}_{B_{ n,\delta}} 
    \big | > {\eta  \over 2} 
    \Big ) 
     \leq \delta ~. 
\end{align*}
As $\delta >0 $ was arbitrary the assertion \eqref{r4a} follows.

To prove \eqref{r4b} we note that the statement is obvious for the functions $z \to \tilde  b _n^* (z) \mathds{1}_{{\cal A}_n^c}$. On the other hand  we obtain on ${\cal A}_n$ from \eqref{r4}
$$
{\tilde b} _n^* (z) = { {\beta_1^*(z)} \over 
1-  \beta_1^*(z) \tilde  \gamma_{1,n}^* (z)}~{\text{(by \eqref{r4})}}.
$$
By \eqref{det62} $\beta_1^*$ is uniformly bounded on ${\cal A}_n$ and the assertion follows from \eqref{r4a}.
\end{proof}

\begin{lemma}
    \label{lemrev3}
\begin{align}
\sup_{z \in {\cal C}_n}  \E^*    \Big |
\E^*_{X_1^*} \Big [   {X_1^{\ast}}^\top C^*(z)    X_1^*
    -  
\operatorname{tr} \big \{  C^*(z)   
    \big \}
     \Big ] \Big |  
      = o_{\mathbb{P}} (1) 
     \label{r10}
     \end{align}
     where $C^*(z)$ is one of the matrices
   \begin{align*} 
     C^*(z) & =  L_n^\top  D_1^*(z) ^{-1}  \big ( \E^* \underline{m}_{n}^* (z)  \tilde \Sigma_n 
    + I_q \big )^{-1}  L_n \\
     C^*(z)   &=  {  \tr \big (   D_1^*(z) ^{-1}  \big ( \E^* \underline{m}_{n}^* (z)  \tilde \Sigma_n + I_q \big )^{-1}   \tilde \Sigma_n   \big )  \over m } L_n^\top   D_1^*(z) ^{-1}  L_n    \\
      C^*(z)   &=  {  \tr \big (   D_1^*(z) ^{-1}   \tilde \Sigma_n   \big )  \over m }   L_n^\top  D_1^*(z) ^{-1}  \big ( \E^* \underline{m}_{n}^* (z)  \tilde \Sigma_n + I_q \big )^{-1}  L_n
     \end{align*}
\end{lemma}
\begin{proof}[Proof of Lemma \ref{lemrev3}]
Let $C^*(z,\omega)$ denote the matrix $C^*(z)$ with $\E^*\underline m_n^*(z)$ replaced by $\omega\in\C$, and let $\omega_\theta(z)=\underline{\tilde m}_n^0(z)+2\eta e^{i\theta}$, $\theta\in[0,2\pi]$, with $\eta>0$ chosen below. On the event $B_n=\{\sup_{z\in{\cal C}_n}|\E^*\underline m_n^*(z)-\underline{\tilde m}_n^0(z)|\le\eta\}$, whose probability tends to one by Lemma \ref{lemrev0}, we have
\begin{align*}
&\sup_{z \in {\cal C}_n}  \E^*    \Big |
\E^*_{X_1^*} \Big [   {X_1^{\ast}}^\top C^*(z)    X_1^*
    -  
\operatorname{tr}  C^*(z)   
     \Big ] \Big |  \\
&\qquad\leq\frac1\pi\int_0^{2\pi}\E^*\sup_{z \in {\cal C}_n}\Big |
\E^*_{X_1^*} \Big [   {X_1^{\ast}}^\top C^*(z,\omega_\theta(z))    X_1^*
    -  
\operatorname{tr}  C^*(z,\omega_\theta(z)) 
     \Big ] \Big |\,\dd\theta
\end{align*}
by Cauchy's integral formula in $\omega$: as in the proof of Lemma \ref{lemrev2}, with probability tending to one, $|1+t\underline{\tilde m}_n^0(z)|\ge2\kappa$ for some $\kappa>0$, all $z\in{\cal C}_n$ and all eigenvalues $t$ of $\tilde\Sigma_n$. Hence, for $2\eta\|\tilde\Sigma_n\|_{S_\infty}\le\kappa$, the matrix $C^*(z,\omega)$ is analytic in $\omega$ on the disc $|\omega-\underline{\tilde m}_n^0(z)|\le2\eta$, and on $B_n$ the point $\E^*\underline m_n^*(z)$ has distance at least $\eta$ from its boundary. Given $\Pi_n$, $\omega_\theta(z)$ is deterministic and Lipschitz in $z\in{\cal C}_n$ with $\|(\omega_\theta(z)\tilde\Sigma_n+I_q)^{-1}\|_{S_\infty}\le1/\kappa$, uniformly in $\theta$. It therefore suffices to show that the expectation of the supremum on the right-hand side tends to zero uniformly in $\theta$, and in the remainder of the proof $C^*(z)$ stands for $C^*(z,\omega_\theta(z))$.

Let $\hat I =\{ i_1^*, \ldots , i_m^* \}$ denote the random subset of chosen indices by the bootstrap and note that $\# \hat I \leq m $, then
\begin{align}
\nonumber 
& 
  m  \E^*_{X_1^*}   \Big [ {r_1^{\ast}}^\top  C^*(z)  r_1^*
    - \frac{1}{m} 
\E^* \operatorname{tr} \big \{  C^*(z) \tilde \Sigma_n  
    \big \}
     \Big ] = {1 \over n} \sum_{i=1}^n 
     X_i^\top C^*(z) X_i - \E^* \operatorname{tr} \big \{  C^*(z) \tilde \Sigma_n  
    \big \} \\
    \label{r11}
   \quad   &= 
    {1 \over n} \sum_{i\in \hat I } \Big \{ \sum_{k=1}^{q'} (X_{ik}^2-1)C_{kk}^*(z) + \sum_{k\neq k'}^{q'} X_{ik}X_{ik'}C_{kk'}^*(z)
    \Big \} \\
    \quad  & + {1 \over n} \sum_{i\in \hat I^c } \Big \{ \sum_{k=1}^{q'} (X_{ik}^2-1)C_{kk}^*(z) + \sum_{k\neq k'}^{q'} X_{ik}X_{ik'}C_{kk'}^*(z)
    \Big \} 
    \nonumber 
\end{align}
We now investigate both terms separately:
\begin{align}
\nonumber
& 
\E \Big [ \sup_{z \in {\cal C}_n}  
\Big | 
   {1 \over n} \sum_{i\in \hat I } \Big \{ \sum_{k=1}^{q'} (X_{ik}^2-1)C_{kk}^*(z) + \sum_{k\neq k'}^{q'} X_{ik}X_{ik'}C_{kk'}^*(z)
    \Big \}   \Big |    \Big]  \\
    \nonumber
    & \quad \leq 
 \E \Big [ \E \Big [ \sup_{z \in {\cal C}_n}  \Big | 
   {1 \over n} \sum_{i\in \hat I } \Big \{ \sum_{k=1}^{q'} (X_{ik}^2-1)C_{kk}^*(z) + \sum_{k\neq k'}^{q'} X_{ik}X_{ik'}C_{kk'}^*(z)
    \Big \}  \Big |~ \big | \hat I   \Big ]    \Big]  \\ 
     & \quad \leq 
 \E \Big [ \E \Big [ \sup_{z \in {\cal C}_n}  \Big | 
   {1 \over n} \sum_{i\in \hat I } \Big \{ \sum_{k=1}^{q'} (X_{ik}^2-1)C_{kk}^*(z)  \Big | +  \sup_{z \in {\cal C}_n}  \Big | {1 \over n} \sum_{i\in \hat I } \sum_{k\neq k'}^{q'} X_{ik}X_{ik'}C_{kk'}^*(z)
    \Big \}  \Big |~ \big | \hat I   \Big ]    \Big] 
    \label{r11a}
\end{align}
For the first term we have
\begin{align}
\nonumber 
  &   \E \Big [  \E \Big [ \sup_{z \in {\cal C}_n}  \Big | 
  \sum_{k=1}^{q'}   {1 \over n}   \sum_{i\in \hat I }  (X_{ik}^2-1)C_{kk}^*(z) \Big | ~ \big | ~\hat I \Big ] \Big ]  \\
  \nonumber 
  & \leq 
  \E \Big [  \E \Big [  
 \Big \{ \sum_{k=1}^{q'}  \Big (  {1 \over n}   \sum_{i\in \hat I }  (X_{ik}^2-1) \Big )^2 \Big \} ^{1/2}
  \sup_{z \in {\cal C}_n}
  \Big \{ \sum_{k=1}^{q'}
| C_{kk}^*(z)|^2 \Big \} ^{1/2}
 \mathds{1}_{{\cal A}_n} \Big | ~ ~\hat I \Big ] \Big ] + o(1)  \\
 & \leq  C \sqrt{m} \Big \{ 
 \E \Big [  
\sum_{k=1}^{q'}  \Big (  {1 \over n}   \sum_{i\in \hat I }  (X_{ik}^2-1) \Big )^2
 \Big \}^{1/2}  \leq C {m^{3/2} \over n } + o(1)  = o(1) 
 \label{r12}
\end{align}
For the second term we obtain similarly
\begin{align*}
  &  \E \Big [ \E \Big [  \sup_{z \in {\cal C}_n}  \Big | {1 \over n} \sum_{k\neq k'} \sum_{i\in \hat I } ^{q'} X_{ik}X_{ik'}C_{kk'}^*(z)
    \Big \}  \Big |~ \big | \hat I   \Big ]    \Big]  \\
    & \quad   \leq  
    \E \Big [ \E \Big [    \Big \{ \sum_{k\neq k'} \Big ( {1 \over n} \sum_{i\in \hat I } ^{q'} X_{ik}X_{ik'} \Big )^2 \Big \} ^{1/2}  \sup_{z \in {\cal C}_n}  
    \Big \{ \sum_{k\neq k'}  | C_{kk'}^*(z)| ^2 \Big \} ^{1/2}\mathds{1}_{{\cal A}_n} 
   ~ \Big |  \hat I   \Big ]    \Big]  + o(1)  \\
   & \quad \leq C \sqrt{m} \E  \Big [  \Big \{ \E \Big [  \sum_{k\neq k'}
   {1 \over n^2} \sum_{i,j \in \hat I} X_{ik}X_{ik'}X_{jk}X_{jk'}
   \Big ] \Big \}^{1/2} ~\Big | ~\hat I \Big ]  \Big ]  + o(1) \\
   & \quad \leq C {m^2 \over n} + o(1) = o(1)~,  
\end{align*}
which proves that the first term in \eqref{r11} is of order $o(1)$. To derive a corresponding statement for the second term, we note that similar arguments as given in \eqref{r12} show that
\begin{align} 
  &   \E \Big [  \E \Big [ \sup_{z \in {\cal C}_n}  \Big | 
  \sum_{k=1}^{q'}   {1 \over n}   \sum_{i\in \hat I^c }  (X_{ik}^2-1)C_{kk}^*(z) \Big | ~ \big | ~\hat I \Big ] \Big ]  \leq C {m \over \sqrt{n} } + o(1)  = o(1) 
 \label{r13}
\end{align}
For the remaining term we use a chaining argument. For this purpose  we introduce  the notation $\Delta_{kk'} (z_1,z_2) = C_{kk'}(z_1) - C_{kk'}(z_2) $, defined $\Delta (z_1,z_2)$ as the corresponding matrix with these entries  and note that conditional on $\hat I$ the random variables 
$\mathds{1}_{{\cal A}_n}$ and 
$\Delta_{kk'} (z_1,z_2)$ is independent of $X_i$, whenever $i \in \hat I^c $. This yields
\begin{align*} 
&  \E \Big \{  \E \Big [ \mathds{1}_{{\cal A}_n}
\Big | 
 {1 \over n} \sum_{i \in \hat I^c} \sum_{k \neq k'}  X_{ik}  X_{ik'} \Delta_{kk'} (z_1,z_2)
\Big |^2  ~\Big | ~\hat I \Big ]  \Big \}  \\
& \quad =
 {1 \over n^2} \sum_{i \in \hat I^c}
  \sum_{k_1 \neq k_1'}  
   \sum_{k_2 \neq k_2'} 
  \E \Big \{  \E [X_{ik_1}  X_{ik_1'} X_{ik_2}  X_{ik_2'} ~ | ~\hat I ] \cdot \E [ \mathds{1}_{{\cal A}_n}\Delta_{k_1k_1'} (z_1,z_2) \bar \Delta_{k_2k_1'} (z_1,z_2)  ~ | ~\hat I
  ]  \Big \} \\
  & \quad \leq {2 \over n} 
 \E \big \{  \E \big [ \mathds{1}_{{\cal A}_n}
  \| \Delta (z_1,z_2) \|^2_{S_2} | \hat I \big ] \big \}  
  \leq 2 { m \over n} |z_1 -z_2|^2
\end{align*}
Now note that  we have 
\begin{align*}
    \Delta (z_1,z_2) & = C^*(z_1) -  C^*(z_2) \\
    & = (M(z_1) - M(z_2))D_1^*(z_1)^{-1} + M(z_2)  ( D_1^*(z_1)^{-1}  - D_1^*(z_2)^{-1}  ) 
\end{align*}
Therefore we obtain 
\begin{align*}
   \mathds{1}_{{\cal A}_n} \| \Delta (z_1,z_2) \|_{S_2} &  \leq  \sqrt{m}   \mathds{1}_{{\cal A}_n}  \| \Delta (z_1,z_2) \|_{S_\infty } \\
   & \leq C \sqrt{m}  \big \{ \| (M(z_1) - M(z_2) \|_{S_\infty }
   +  \mathds{1}_{{\cal A}_n} \| D_1^*(z_1)^{-1}  - D_1^*(z_2)^{-1}  \|_{S_\infty }
   \big \} 
   \\
    & \leq \sqrt{m}  C |z_1-z_2| ,
\end{align*}
where we have used that $z\mapsto\omega_\theta(z)$ is Lipschitz on ${\cal C}_n$, uniformly in $\theta$. This finally yields
\begin{align*} 
&    \E \Big [ \mathds{1}_{{\cal A}_n}
\Big | 
 {1 \over n} \sum_{i \in \hat I^c} \sum_{k \neq k'}  X_{ik}  X_{ik'} \Delta_{kk'} (z_1,z_2)
\Big |^2  \Big ]  
  \leq C { m \over n} |z_1 -z_2|^2~.
\end{align*}
Employing standard  chaining with the  Young-Orlicz  module $\phi (t)=t^2$ reveals the chaining bound 
$$
\int_{0}^{\ell \sqrt{m\over n}} \sqrt{\sqrt{m} \over \varepsilon \sqrt{n} } d\varepsilon \leq C \sqrt{ m \over n} 
$$
where $\ell$ denotes the length of the curve ${\cal C} $. Hence, it follows that
$$
\E \Big [ \mathds{1}_{{\cal A}_n} \sup_{z \in {\cal C}_n}  \Big | {1 \over n} \sum_{k\neq k'} \sum_{i\in \hat I^c } ^{q'} X_{ik}X_{ik'}C_{kk'}^*(z)
    \Big \}    \Big ]    \Big]  = O \Big ( \sqrt{m \over n} \Big ) .
$$
The term corresponding to 
$\mathds{1}_{{\cal A}_n^c}$ can be treated by the same arguments as given for the second term in the last line of \eqref{r11a}, which completes the proof.
\end{proof}

 \setlength{\bibsep}{2pt}

\begin{funding}
This work was supported by the DFG Research Unit 5381 {\it Mathematical Statistics in the Information Age}, project number 460867398.
\end{funding}

\begin{acks}[Acknowledgments]
The authors would like to thank Nina D\"ornemann for useful discussions and Patrick Bastian and Thomas Lam for their help with the numerical examples.
\end{acks}

\bibliographystyle{imsart-nameyear}
\bibliography{references_korr}

@article{10.1214/009117905000000233,
author = {Jinho Baik and G{\'e}rard Ben Arous and Sandrine P{\'e}ch{\'e}},
title = {{Phase transition of the largest eigenvalue for nonnull complex sample covariance matrices}},
volume = {33},
journal = {Ann. Probab.},
number = {5},
publisher = {Institute of Mathematical Statistics},
pages = {1643 -- 1697},
year = {2005},
doi = {10.1214/009117905000000233},
URL = {https://doi.org/10.1214/009117905000000233}
}

@article{Paul2007,
author="Paul, D. ",
title="Asymptotics of sample eignestructure for a large dimensional spiked covariance model ",
year=2007,
journal="Stat. Sin.",
volume=17,
pages="1617-1642"
}

@article{BAI1988166,
title = {A note on the largest eigenvalue of a large dimensional sample covariance matrix},
journal = {J. Multivar. Anal.},
volume = {26},
number = {2},
pages = {166-168},
year = {1988},
issn = {0047-259X},
doi = {https://doi.org/10.1016/0047-259X(88)90078-4},
url = {https://www.sciencedirect.com/science/article/pii/0047259X88900784},
author = {Z.D Bai and Jack W Silverstein and Y.Q Yin}
}

@article{johnstone2006,
  author = {Johnstone, I. M.},
  title = {High dimensional statistical inference and random matrices},
  journal = {Proc.  International  Congress  of  Mathematicians, 2006,  http://arxiv.org/abs/math/0611589},
  year = {2006}
}

@article{nagao1973,
  author = {Nagao, H.},
  title = {On some test criteria for covariance matrix},
  journal = {Annals of  Statistics},
  year = {1973},
  volume = {1},
  pages = {700-709}
}

@book{baisilverstein2010,
  author =   "Bai, Z. D. and Silverstein, J. W.",
  title =    "Spectral Analysis of Large Dimensional Random Matrices",
  address =  "New York",
  publisher =    "Springer",
  year =     2010
}

@article{silverstein1995,
  author =   "Silverstein, J. W.",
  title =    "Strong convergence of the empirical distribution of eigenvalues of large-dimensional random matrices.",
  journal =  "J. Multivar. Anal.",
  year =     1995,
  volume =   55,
  pages =     "331-339"
}

@article{baisil2004,
  author =   "Bai, Z. D. and Silverstein, J. W.",
  title =    "{CLT} for linear spectral statistics of large dimensional sample covariance matrices",
  journal =  "Ann. Probab.",
  year =     2004,
  volume =   32,
  pages =    "553-605",
}

@article{najimyao2016,
  author =   "Najim, J. and Yao, J.",
  title =    "Gaussian fluctuations for linear spectral statistics of large random covariance matrices.",
  journal =  "Ann. Appl. Probab..",
  year =     2016,
  volume =   26,
  number  =   3,
  pages =    "1837-1887",
}

@article{lopblaaue2019,
   author = {Lopes, Miles E and Blandino, Andrew and Aue, Alexander},
    title = "{Bootstrapping spectral statistics in high dimensions}",
    journal = {Biometrika},
    volume = {106},
    number = {4},
    pages = {781-801},
    year = {2019},
    month = {09},
    issn = {0006-3444},
    doi = {10.1093/biomet/asz040},
    url = {https://doi.org/10.1093/biomet/asz040},
    eprint = {https://academic.oup.com/biomet/article-pdf/106/4/781/30646672/asz040.pdf},
}

@article{karpur2016,
  author =   "El Karoui, N. and Purdom, E.",
  title =    "The bootstrap, covariance matrices and {PCA} in moderate and high-dimensions.",
  journal =  "arXiv:1608.00948",
  year =     2016
}

@article{chatterjee2009,
author = "Chatterjee, Sourav and Ledoux, Michel",
doi = "10.1214/ECP.v14-1504",
fjournal = "Electronic Communications in Probability",
journal = "Electron. Commun. Probab.",
pages = "495--500",
pno = "48",
publisher = "The Institute of Mathematical Statistics and the Bernoulli Society",
title = "An observation about submatrices",
url = "https://doi.org/10.1214/ECP.v14-1504",
volume = "14",
year = "2009"
}

@book{gresze1958,
  author =   "Grenander, U. and Szego, G.",
  title =    "Toeplitz Forms and their Applications",
  address =  "Berkeley and Los Angeles",
  publisher =    "University of California Press",
  year =     1958
}

@book{riordan1958,
  author =   "Riordan, J.",
  title =    "An Introduction to Combinatorial Analysis",
  address =  "New York",
  publisher =    "Wiley",
  year =     1958
}

@book{BrockwellDavis,
	Author = {P. Brockwell AND R. Davis},
	Title = {Time Series: Theory and Methods},
	Publisher = {Springer},
	Year = {1998},
	Adress = {New York}
}

@article{Marcenko1967,
	doi = {10.1070/sm1967v001n04abeh001994},
	year = 1967,
	month = {apr},
	publisher = {{IOP} Publishing},
	volume = {1},
	number = {4},
	pages = {457--483},
	author = {V A Mar{\v{c}}enko and L A Pastur},
	title = {Distribution of eigenvalues for some sets of random matrices
	},
	journal = {Mathematics of the {USSR}-Sbornik},
}

@article{delvaux2012,
author = {Delvaux, Steven},
title = {Equilibrium problem for the eigenvalues of banded block Toeplitz matrices},
journal = {Math. Nachr.},
volume = {285},
number = {16},
pages = {1935-1962},
doi = {10.1002/mana.201100221},
url = {https://onlinelibrary.wiley.com/doi/abs/10.1002/mana.201100221},
eprint = {https://onlinelibrary.wiley.com/doi/pdf/10.1002/mana.201100221},
year = {2012}
}

@article{bickel1997,
  author =   "Bickel, P.J. and G{\"{o}}tze, F. and van Zwet, W.R.",
  title =    "Resampling fewer than $n$ observations: gains, losses, and remedies for losses.",
  journal =  "Stat. Sin.",
  year =     1997,
  volume =   7,
  number  =  1,
  pages =    "1-31"
}

@article{durlopsaff1999,
  author =   "Duran, A.J. and Lopez-Rodriguez, P. and Saff, E.B.",
  title =    "Zero asymptotic behaviour for orthogonal matrix polynomials.",
  journal =  "Journal d'Analyse Math{\'{e}}matique",
  year =     1999,
  volume =   78,
  pages =    "37-60"
}

@article{fan2019,
author = "Fan, Zhou and Johnstone, Iain M.",
doi = "10.1214/18-AOS1767",
fjournal = "Annals of Statistics",
journal = "Ann. Statist.",
month = "10",
number = "5",
pages = "2855--2886",
publisher = "The Institute of Mathematical Statistics",
title = "Eigenvalue distributions of variance components estimators in high-dimensional random effects models",
url = "https://doi.org/10.1214/18-AOS1767",
volume = "47",
year = "2019"
}

@article{limath1999,
author = "Li, CK. and  Mathias, R.",
fjournal = "Numerische Mathematik ",
journal = "Numer. Math.",
pages = "377--413",
title = "The {L}idskii-{M}irsky-{W}ielandt theorem -- additive and multiplicative versions",
volume = "81",
year = "1999"
}

@article{baisil1998,
author = {Z. D. Bai and Jack W. Silverstein},
title = {{No eigenvalues outside the support of the limiting spectral distribution of large-dimensional sample covariance matrices}},
volume = {26},
journal = {Ann. Probab.},
number = {1},
publisher = {Institute of Mathematical Statistics},
pages = {316 -- 345},
year = {1998},
doi = {10.1214/aop/1022855421},
URL = {https://doi.org/10.1214/aop/1022855421}
}

@article{zou2022clt,
  title={CLT for linear spectral statistics of large dimensional sample covariance matrices with dependent data},
  author={Zou, Tingting and Zheng, Shurong and Bai, Zhidong and Yao, Jianfeng and Zhu, Hongtu},
  journal={Stat. Pap.},
  volume={63},
  number={2},
  pages={605--664},
  year={2022},
  publisher={Springer}
}

@article{yinbaikri1988,
	author = {Yin, Y. Q. and Bai, Z. D. and Krishnaiah, P. R.},
	date = {1988/08/01},
	doi = {10.1007/BF00353874},
	isbn = {1432-2064},
	journal = {Probab. Theory Relat. Fields},
	number = {4},
	pages = {509--521},
	title = {On the limit of the largest eigenvalue of the large dimensional sample covariance matrix},
	url = {https://doi.org/10.1007/BF00353874},
	volume = {78},
	year = {1988}}

@article{baiyin1993,
	author = {Z. D. Bai and Y. Q. Yin},
	doi = {10.1214/aop/1176989118},
	journal = {Ann. Probab.},
	number = {3},
	pages = {1275 -- 1294},
	publisher = {Institute of Mathematical Statistics},
	title = {{Limit of the Smallest Eigenvalue of a Large Dimensional Sample Covariance Matrix}},
	url = {https://doi.org/10.1214/aop/1176989118},
	volume = {21},
	year = {1993}}

@book{billingsley1999,
  address = {New York},
  author = {Billingsley, Patrick},
  edition = {Second},
  publisher = {John Wiley \& Sons Inc.},
  series = {Wiley Series in Probability and Statistics: Probability and MathematicalStatistics},
  title = {Convergence of probability measures},
  year = 1968
}

@article{bai1999,
  author = {Bai, Z.D.},
 ISSN = {10170405, 19968507},
 URL = {http://www.jstor.org/stable/24306604},
 journal = {Stat. Sin.},
 number = {3},
 publisher = {Institute of Statistical Science, Academia Sinica},
 title = {Methodologies in spectral analysis of large dimensional random matrices,
A review.},
 urldate = {2024-05-09},
 volume = {9},
pages = {611 -- 677},
 year = {1999}
}

@article{silversteinbai1995,
	author = {J.W. Silverstein and Z.D. Bai},
	doi = {https://doi.org/10.1006/jmva.1995.1051},
	issn = {0047-259X},
	journal = {J. Multivar. Anal.},
	number = {2},
	pages = {175-192},
	title = {On the Empirical Distribution of Eigenvalues of a Class of Large Dimensional Random Matrices},
	url = {https://www.sciencedirect.com/science/article/pii/S0047259X85710512},
	volume = {54},
	year = {1995}
}

@article{BickelSakov2008,
 author = {Peter J. Bickel and Anat Sakov},
 journal = {Stat. Sin.},
 number = {3},
 pages = {967--985},
 publisher = {Institute of Statistical Science, Academia Sinica},
 title = {On the choice of $m$ in the  $m$ of $n$ bootstrap and confidence bounds for extrema},
 volume = {18},
 year = {2008}
}

@article{PolitisRomano1994,
author = {Dimitris N. Politis and Joseph P. Romano},
title = {{Large sample confidence regions based on subsamples under minimal assumptions}},
volume = {22},
journal = {Ann. Statist.},
publisher = {Institute of Mathematical Statistics},
pages = {2031 -- 2050},
year = {1994}
}

@article{ding2023extreme,
      title={Extreme eigenvalues of sample covariance matrices under generalized elliptical models with applications}, 
      author={Xiucai Ding and Jiahui Xie and Long Yu and Wang Zhou},
      year={2023},
      journal={arXiv:2303.03532},
      archivePrefix={arXiv},
      primaryClass={stat.ME}
}

@article{HanXoZhou2018,
	author = {Fang Han and Sheng Xu and Wen-Xin Zhou},
	doi = {10.3150/16-BEJ912},
	journal = {Bernoulli},
	number = {3},
	pages = {1787 -- 1833},
	publisher = {Bernoulli Society for Mathematical Statistics and Probability},
	title = {{On Gaussian comparison inequality and its application to spectral analysis of large random matrices}},
	url = {https://doi.org/10.3150/16-BEJ912},
	volume = {24},
	year = {2018}}

@article{yao2022rates,
      title={Rates of bootstrap approximation for eigenvalues in high-dimensional PCA}, 
      author={Junwen Yao and Miles E. Lopes},
      year={2022},
      journal = { arXiv:2208.03050}}

@article{wanglopes2022,
	author = {Siyao Wang and Miles E. Lopes},
	doi = {10.1214/23-EJS2140},
	journal = {Electron. J. Stat.},
	number = {2},
	pages = {1848 -- 1892},
	publisher = {Institute of Mathematical Statistics and Bernoulli Society},
	title = {{A bootstrap method for spectral statistics in high-dimensional elliptical models}},
	url = {https://doi.org/10.1214/23-EJS2140},
	volume = {17},
	year = {2023}
}

@article{yuzhaozhou2024,
	author = {Yu, Long and  Zhao, Peng and Zhou, Wang },
	doi = {10.1080/01621459.2024.2346364},
	eprint = {https://doi.org/10.1080/01621459.2024.2346364},
	journal = { J. Am. Stat. Assoc.},
	number = {0},
	pages = {1--12},
	publisher = {Taylor \& Francis},
	title = {Testing the Number of Common Factors by Bootstrapped Sample Covariance Matrix in High-Dimensional Factor Models},
	url = {https://doi.org/10.1080/01621459.2024.2346364},
	volume = {0},
	year = {2024}}

@article{ledoitwolf2022,
	author = {Olivier Ledoit and Michael Wolf},
	doi = {10.1214/aos/1031689018},
	journal = {Ann. Statist.},
	number = {4},
	pages = {1081 -- 1102},
	publisher = {Institute of Mathematical Statistics},
	title = {{Some hypothesis tests for the covariance matrix when the dimension is large compared to the sample size}},
	url = {https://doi.org/10.1214/aos/1031689018},
	volume = {30},
	year = {2002}}

@article{dettekroll,
    author = {Dette, H and Kroll, M},
    title = {A simple bootstrap for Chatterjee’s rank correlation},
    journal = {Biometrika},
    volume = {112},
    number = {1},
    pages = {asae045},
    year = {2024},
    month = {08},
    issn = {1464-3510},
    doi = {10.1093/biomet/asae045},
    url = {https://doi.org/10.1093/biomet/asae045},
    eprint = {https://academic.oup.com/biomet/article-pdf/112/1/asae045/60211588/asae045.pdf},
}

@article{JurczakRohde,
author = {Kamil Jurczak and Angelika Rohde},
title = {{Spectral analysis of high-dimensional sample covariance matrices with missing observations}},
volume = {23},
journal = {Bernoulli},
number = {4A},
publisher = {Bernoulli Society for Mathematical Statistics and Probability},
pages = {2466 -- 2532},
year = {2017},
doi = {10.3150/16-BEJ815},
URL = {https://doi.org/10.3150/16-BEJ815}
}

@misc{Nina2025,
      title={Tracy-Widom, Gaussian, and Bootstrap: Approximations for Leading Eigenvalues in High-Dimensional {PCA}}, 
      author={Nina Dörnemann and Miles E. Lopes},
      year={2025},
      eprint={2503.23097},
      archivePrefix={arXiv},
      primaryClass={math.ST},
      url={https://arxiv.org/abs/2503.23097}, 
}

@article{panzhou2008,
  author = {Pan, G. M. and Zhou, W.},
  title = {Central limit theorem for signal-to-interference ratio of reduced rank linear receiver},
  journal = {Ann. Appl. Probab.},
  volume = {18},
  number = {3},
  pages = {1232--1270},
  year = {2008},
  doi = {10.1214/07-AAP477}
}

@article{delvauxdette2012,
  author = {Delvaux, Steven and Dette, Holger},
  title = {Zeros and ratio asymptotics for matrix orthogonal polynomials},
  journal = {J. Anal. Math.},
  volume = {118},
  pages = {657--690},
  year = {2012},
  doi = {10.1007/s11854-012-0047-x}
}

\end{document}